\documentclass[10pt,twoside, a4paper, english, reqno]{amsart}
\usepackage[dvips]{epsfig}
\usepackage{amscd}
\usepackage{stmaryrd}
\usepackage{amssymb}
\usepackage{amsthm}
\usepackage{amsmath}
\usepackage{graphicx,enumitem,dsfont}
\usepackage{latexsym}

\usepackage{upref}
\usepackage[colorlinks]{hyperref}
\usepackage{color}
\usepackage{subcaption}
\usepackage[usenames,dvipsnames]{xcolor}
\theoremstyle{plain}
\newtheorem{thm}{Theorem}[section]
\theoremstyle{plain}
\newtheorem{lem}[thm]{Lemma}
\newtheorem{prop}[thm]{Proposition}

\newtheorem{ex}{Example}[section]
\theoremstyle{definition}
\newtheorem{defi}{Definition}[section]
\newtheorem{rem}{Remark}[section]

\newtheorem*{maintheorem*}{Main Theorem}
\newcommand{\Dt}{{\Delta t}}
\usepackage{graphicx}
\usepackage{epsfig}

\newenvironment{Assumptions}
{
\setcounter{enumi}{0}

\begin{enumerate}}
{\end{enumerate} }

\newcommand{\Dx}{{\Delta x}}
\newcommand{\R}{\ensuremath{\mathbb{R}}}
\newcommand{\Z}{\ensuremath{\mathbb{Z}}}

\newcommand{\rd}{\ensuremath{\mathbb{R}^d}}

\newcommand{\goto}{\ensuremath{\rightarrow}}

\newcommand{\eps}{\ensuremath{\varepsilon}}

\numberwithin{equation}{section} \allowdisplaybreaks

\title[Splitting for stochastic balance laws]
{Convergence of an operator splitting scheme for fractional conservation laws with L\'{e}vy noise}

\date{}

\author[ S. R. Behera ]{Soumya Ranjan Behera}
\address[Soumya Ranjan Behera] {\newline 
Department of Mathematics,
Indian Institute of Technology Delhi,
Hauz Khas, New Delhi, 110016, India.}
\email[] {maz198759@iitd.ac.in}

\author[A. K. Majee]{Ananta K. Majee}
\address[Ananta K. Majee]{\newline
Department of Mathematics,
Indian Institute of Technology Delhi,
Hauz Khas, New Delhi, 110016, India. }
\email[]{majee@maths.iitd.ac.in}


\thanks{}

\thanks{}
\begin{document}
\begin{abstract}
In this paper, we are concerned with an operator splitting scheme
for linear fractional  and fractional degenerate stochastic conservation laws driven by multiplicative L\'{e}vy noise.
More specifically, using a variant of classical Kru\v{z}kov's doubling of variable approach, we show that the approximate solutions generated by the splitting scheme
converges to the unique stochastic entropy solution of the underlying problems. Finally, the convergence analysis is illustrated by several numerical examples.
\end{abstract}

\maketitle
\section {Introduction}
Let ($\Omega, \mathbb{P}, \mathcal{F}, \{ \mathcal{F}_t\}_{t \geq 0}$)
be a filtered probability space satisfying the usual hypothesis, i.e.,  $\{\mathcal{F}_t\}_{t\ge 0}$ is a right-continuous filtration such that $\mathcal{F}_0$ 
 contains all the $\mathbb{P}$-null subsets of $(\Omega, \mathcal{F})$. We are interested in numerical approximations of  $L^2(\mathbb{R}^d)$-valued predictable process $u(t,\cdot)$ which satisfies the following Cauchy problem
\begin{equation}{\label{eq:0.0}}
\begin{cases}
    \displaystyle      \frac{\partial }{\partial t}u(t,x) + A(u(t,x)) + \text{div}_x f(u(t,x)) = q(u(t,x)) \quad 
        in \quad \mathbb{Q}_T, \\
         u(0, x) = u_0(x),  \quad  x \in \mathbb{R}^d,
    \end{cases}
\end{equation}
where $\mathbb{Q}_T$ = $\mathbb{R}^d \times (0,T)$ with T $>$ 0 fixed. In \eqref{eq:0.0}, $A(u(t,x))$ is the associate  fractional term, $u_0(\cdot)$ is the given initial function, $f : \mathbb{R}  \rightarrow \mathbb{R}^d$ is the flux function, and $q(u(t,x))$ is the multiplicative L\'evy  noise. To be more precise, taking $A(u(t,x))$ as $\mathcal{L}_{\theta}[u(t,\cdot)](x)$, we intend to study numerical approximation of $L^2(\mathbb{R}^d)$-valued predictable process $u(t,\cdot)$ satisfying the SPDE
\begin{equation}\label{eq:fractional}
\begin{cases}
     \displaystyle     du(t,x) + \mathcal{L}_{\theta}[u(t,\cdot)](x)dt +\mbox{div}_x f(u(t,x))dt = \sigma(u(t,x))dW(t) + \int_{|z| > 0}  \eta(u(t,x); z) \widetilde{N}(dz,dt) \quad
        in \quad \mathbb{Q}_T, \\
         u(0, x) = u_0(x),  \quad  x \in \mathbb{R}^d,
    \end{cases}
\end{equation}
where $\mathcal{L}_{\theta}$[u] is the fractional Laplace operator $(-\Delta)^\theta$[u] of order $\theta\in (0,1)$, defined by
\[\mathcal{L}_{\theta}[\varphi](x) : = a_{\theta}\, P.V. \int_{|z| > 0} \frac{\varphi(x) - \varphi(x+z)}{|z|^{d +2 \theta}}dz,\]
for some constant $a_\theta$ $>$ 0, and a  sufficiently regular function $\varphi$. In \eqref{eq:fractional}, W(t) is a $\{\mathcal{F}_t\}_{t \geq 0}$-adapted one-dimensional standard Brownian noise, $\widetilde{N}(dz,dt) = N(dz,dt) - m(dz)dt$ and $N$ is a Poisson random measure  with intensity measure $m(dz)$, satisfying $\int_{|z| > 0} (|z|^2 \wedge 1)m(dz) < \infty$. Moreover, the noise coefficients $\sigma: \R \goto \R$ and $\eta: \R \times \R \goto \R$ are given Lipschitz continuous functions (see Section \ref{sec:technical} for the complete set of assumptions) signifying the multiplicative nature of the noise. The stochastic integral on the right-hand side of \eqref{eq:fractional}
is defined in the It\^{o}-L\'{e}vy sense.
\vspace{0.2 cm}

Taking $A(u(t,x))$ as $\mathcal{L}_{\theta}[\phi(u(t,\cdot))](x)$ for some nonlinear function $\phi$, we are also interested in studying
numerical approximations of $L^2(\mathbb{R}^d)$-valued predictable process $u(t,\cdot)$ for the fractional degenerate Cauchy problem
 \begin{equation}
 \label{eq:fdegenerate}
 \begin{cases} 
 \displaystyle du(t,x) +\mathcal{L}_{\theta}[\phi(u(t,\cdot))](x)dt + \mbox{div}_x f(u(t,x))\,dt = \sigma(u(t,x))dW(t) \quad in \quad \mathbb{Q}_T, \\
 u(0,x) = u_0(x), \quad  x\in \mathbb{R}^d\,\cdot
\end{cases}
\end{equation}
The basic assumption is that $\phi : \R \rightarrow \R,$ is non-decreasing with $\phi(0)=0$. In \eqref{eq:fdegenerate}, the stochastic integral is defined in the It\^{o} sense. 
\vspace{.1cm}

\subsection{Review of existing literature}
The equations of type \eqref{eq:fractional} and \eqref{eq:fdegenerate} can be viewed as a stochastic perturbation of the fractional convection-diffusion equations with nonlinear sources. In the absence of non-local term in \eqref{eq:fractional} and \eqref{eq:fdegenerate} together with $\sigma =0= \eta$, equations \eqref{eq:fractional} and \eqref{eq:fdegenerate} become standard conservation laws in $\R^d$. The well-posedness analysis of deterministic conservation laws is well documented in the literature; see e.g., \cite{dafermos, godu,kruzkov, Volpert} and references therein. For $\sigma = \eta=0$, the equation \eqref{eq:fractional} becomes a fractional conservation laws which was studied in \cite{Alibaud 2007}. The entropy solution theory of fractional degenerate convection-diffusion equations was established by Cifani et al. in \cite{Alibaud 2012,cifani}. 
\vspace{.2cm}

The study of stochastic balance laws driven by noise is  a comparatively new area of pursuit. Only recently, many authors
\cite{ Bauzet-2012,Bauzet-2015,Majee-2015,Majee-2014, BKM-2015,Majee-2019, Chen:2012fk,Hofmanova-2016, Vovelle2010,Vovelle-2018, nualart:2008, Karlsen-2017,Kim,Majee-2017,xu} are devoted towards understanding the effects of stochastic forcing on the solutions of nonlinear
Cauchy problems for partial differential equations. Due to more technical novelties, the study of well-posedness result in case of nonlinear fractional degenerate convection-diffusion equations with nonlinear stochastic forcing is more subtle. In \cite{frac lin}, the authors established the well-posedness of stochastic entropy solution of \eqref{eq:fractional} using the concept of measure-valued solution. The  uniqueness result for  stochastic entropy solution of \eqref{eq:fdegenerate} was analyzed by using a variant of Kružkov’s doubling of variables technique and the existence was proved as a by product of unique measure valued(mild) solution and a priori estimations by Koley et al. \cite{frac non}. Due to the nonlinear nature of the underlying problem, an
explicit solution formula is hard to obtain, and hence robust numerical schemes for approximating such equations are very
important. The first documented development in this direction for the deterministic counterpart of \eqref{eq:fdegenerate} is \cite{cifani fds}, where a monotone finite difference scheme was constructed and its convergence to the unique Kru\v{z}kov-type entropy solution was shown. 
\vspace{.2cm}

In the last decade, there has been a growing interest in numerical approximations and numerical experiments for entropy solutions to the related Cauchy problem driven by stochastic forcing. Within the existing literature, we refer to the paper by Holden et al. \cite{risebroholden1997}, where the authors successfully
implemented an operator-splitting method to prove the existence of a path-wise weak solution for such Cauchy problem driven by Brownian noise in one space dimension. Moreover, they presented some numerical examples to illustrate their theory. Recently, Bauzet \cite{Bauzet-2015-Splitting} generalized the operator splitting technique for  stochastic conservation laws, and by employing the Young measure theory, the author established the convergence of approximate solutions to an entropy solution of the underlying problem. We also refer to see \cite{Karlsen-2018}, where the time splitting method was analyzed for
more general noise coefficient in the spirit of Malliavin calculus and Young measure theory. In an another development \cite{kroker}, Kr\"{o}ker and Rohde established the
convergence of a monotone semi-discrete finite volume scheme using a stochastic compensated compactness method. In recent papers \cite{Bauzet-2020,Bauzet-2016a, Bauzet-2016b}, Bauzet
et al. have studied fully discrete schemes via flux-splitting and monotone finite volume schemes for stochastic conservation
laws driven by multiplicative Brownian noise and established its convergence by using Young measure techniques, see also \cite{Vovelle-2020}, where the authors employed a kinetic formulation approach and established the convergence of the explicit-in-time finite volume method for the approximations of scalar first-order conservation laws with compactly supported, general multiplicative noise. Very recently, in \cite{koley-2022} the authors have established the rate of convergence of the approximate solutions generated by a finite difference scheme  for fractional degenerate conservation laws driven by Brownian noise.
\vspace{.1cm}

Being a relatively new area of pursuit, numerical schemes for the stochastic scalar balance laws
driven by L\'{e}vy noise even more sparse than the Brownian noise case. In fact, to the best of the
authors knowledge, the first attempt to answer the quest for an efficient numerical scheme
for such equations were made in a recent paper by Koley et al. \cite{Majee-2018}. The authors have studied a semi-discrete finite difference scheme for conservation laws driven by a homogeneous multiplicative L\'{e}vy noise and showed the convergence of approximate solutions, generated
by the finite difference scheme, to the unique BV-entropy solution as the spatial discretization parameter $\Delta x \goto 0$. Moreover, they have established the rate of convergence, which is of order $\frac{1}{2}$.  In  \cite{Majee-2018-flux}, the author has studied a fully discrete flux-
splitting finite volume scheme for \eqref{eq:fdegenerate} with $\phi=0$, and addressed the convergence of the scheme.
\subsection{Aim and scope of the paper}
The above discussions clearly highlight the lack of  study of the numerical scheme and its convergence analysis for equations \eqref{eq:fractional} and \eqref{eq:fdegenerate}. In this paper, 
we intend to study an operator splitting methods  for \eqref{eq:fractional} and \eqref{eq:fdegenerate}, and wish to prove the convergence 
of approximate numerical solutions, generated by the operator splitting schemes, 
to the unique stochastic BV-entropy solution of Cauchy problems \eqref{eq:fractional} and \eqref{eq:fdegenerate} mainly for mathematical curiosity.  In \cite{Bauzet-2015-Splitting}, the author has deduced the entropy inequality for approximate solutions generated by an operator splitting scheme and then passed to the limit as $\Delta t \rightarrow 0$ in the Young measure sense to show its convergence. Nevertheless, this technique will not be compatible here as the main difficulty lies in passing to the limit in the sense of Young measure in the non-local fractional term. Moreover, one can also observe that in  \cite{frac lin, frac non}, the authors have not shown that the Young measure valued limit of viscous solutions satisfies the entropy formulation. In view of the above discussion, our strategy is to use a variant of Kru\v{z}kov's doubling of variables technique  to obtain the Kato's inequality and there by convergence result. However, our approach requires significant changes in the order of passing the limit with respect to the various parameters in the proof to obtain the Kato's inequality as compared to the existing hierarchy of passing limits in \cite{frac lin, frac non}. To be  more precise, we send the parameter $\Delta t \rightarrow 0$, before sending $\delta_0 \rightarrow 0$, (see Subsection \ref{Convergence Analysis}). The changes in order of passing to the limit in various parameters effect all the terms appearing in the entropy inequality and it requires rigorous estimations of all terms to obtain the stochastic Kato's inequality. In addition, the average time continuity of regularized viscous solution plays a crucial role in achieving our result (see, Lemma \ref{lem:average-time-cont-viscous}). 
\vspace{.1cm}

The remaining part of this paper is organized as follows. We state the assumptions, 
detail the technical framework, and state the main results
in Section~\ref{sec:technical}. In Section~\ref{sec:properties}, 
we discuss various properties enjoyed by the splitting operators, while Section~\ref{sec:estimations} deals with 
the {\em a-priori} estimates for the approximate solutions and average time continuity for regularized viscous solution. In Sections~\ref{sec:convergence u} and \ref{sec:6} , we have shown the convergence of approximate solutions to the unique $BV$-entropy solution of fractional and fractional degenerate  stochastic conservation laws respectively. Finally, in Section~\ref{sec:numerical}, we explore some numerical experiments in one space dimension along with the brief discussion of Euler-Maruyama method.

\section{Preliminaries And Technical Framework}\label{sec:technical}Throughout this paper, we use the letters $C,\,K$ etc. to denote various generic constants.
There are situations where constants may change from line to line, but the notation is kept
unchanged so long as it does not impact the central idea. 
The Euclidean norm on any $\R^d$-type space is denoted by $|\cdot|$, and
the norm in $BV(\R^d)$ is denoted by $|\cdot|_{BV(\R^d)}$. For any separable Hilbert Space $H$, we denote $N_w^2((0,T); H)$, as the space of all square integrable predictable $H$-valued process $u$ such that $\displaystyle\mathbb{E}\Big[ \int_0^T ||u(t)||_H^2\,dt\Big] < \infty.$
\vspace{.1cm}

\subsection{ Stochastic entropy formulation and well-posedness results}
It is well known that weak solutions may be discontinuous and they are not uniquely determined by their initial data. Consequently, an entropy condition must be imposed to single out the physically correct solution. To do so, for equation \eqref{eq:fractional}, 
 we re-write the non-local operator $\mathcal{L}_\theta$ as a sum of two operators $\mathcal{L}_\theta^{\bar{r}}[\varphi]$ and $\mathcal{L}_{\theta,\bar{r}}[\varphi]$ for some $\bar{r}>0$, where
\begin{align}
    & \mathcal{L}_\theta^{\bar{r}}[\varphi] := a_\theta \int_{|z| > \bar{r}} \frac{\varphi(x)-\varphi(x+z)}{|z|^{d + 2\theta}}dz, \notag \\
   & \mathcal{L}_{\theta,\bar{r}}[\varphi] := a_\theta \,P.V. \int_{|z| \le \bar{r}} \frac{\varphi(x)-\varphi(x+z)}{|z|^{d + 2\theta}}dz. \notag
\end{align}
 Since the notion of entropy solution is built 
around the so called entropy-entropy flux pair, we begin with the definition of entropy flux pair.
\begin{defi}[Entropy Flux pair]
A pair $(\beta,\zeta) $ is called an entropy flux pair if  $\beta \in C^2(\R)$, $\beta \ge 0$ and $\zeta = (\zeta_1,\zeta_2,....\zeta_d):\R \rightarrow \R^d $ is a vector field satisfying $ 
\zeta'(r) = \beta'(r)f'(r)$, for all $r \in \R$. Moreover, an entropy flux pair $(\beta,\zeta)$ is called convex if $ \beta^{\prime\prime}(\cdot) \ge 0$.  
\end{defi}
With the help of a convex entropy flux pair ($\beta, \zeta$), we now recall the stochastic entropy solution of \eqref{eq:fractional},  cf.~\cite[Definition $1.3$]{frac lin}.
\begin{defi}[Stochastic entropy solution] \label{defi:entropysol-f}
An element  $ u\in$ $N_w^2((0, T); L^2(\mathbb{R}^d))$ with initial data $u_0\in$ $L^2(\mathbb{R}^d)$ is called a stochastic entropy solution of \eqref{eq:fractional} if for any non-negative test function $\psi \in {C}_c^{1,2}([0,\infty) \times \mathbb{R}^d)$ and convex entropy flux pair $(\beta, \zeta)$, the following inequality holds:
\begin{align*}
    &\int_{\mathbb{R}^d} \psi(0,x)\beta(u_0(x))dx + \int_{\mathbb{Q}_T}\Big \{ \partial_t\psi(t,x)\beta(u(t,x)) + \zeta(u(t,x))\cdot \nabla\psi(t,x) \Big \}\,dx\,dt \notag \\ 
    &- \int_{\mathbb{Q}_T} \Big[\mathcal{L}_\theta^{\bar{r}}[u(t,\cdot)](x)\psi(t,x)\beta'(u(t,x)) + \beta(u(t,x))\mathcal{L}_{\theta,\bar{r}}[\psi(t, .)](x) \Big]\,dx\,dt \notag \\
   & +\int_{\mathbb{Q}_T} \sigma(u(t,x))\beta'(u(t,x))\psi(t,x)\,dx\,dW(t)+
   \frac{1}{2} \int_{\mathbb{Q}_T} \sigma^2(u(t,x))\beta''(u(t,x))\psi(t,x)\,dx\,dt \notag \\& +  \int_{\mathbb{Q}_T}\int_{|z| > 0} \Big(\beta \big(u(t,x) + \eta(u(t,x); z)\big) -\beta(u(t,x))\Big) \psi(t,x)\,\widetilde{N}(dz,dt)\,dx \notag \\& +
   \int_{\mathbb{Q}_T}\int_{|z| > 0} \int_0^1 (1- \lambda)\eta^2(u(t,x); z) \beta'' \Big (u(t,x) +\lambda \eta(u(t,x); z)\Big) \psi(t,x)\,d\lambda\,m(dz)\,dx \,dt \geq 0 \quad\mathbb{P}\text{-a.s.} 
\end{align*}
\end{defi}
Similarly, the notation of stochastic entropy solution of \eqref{eq:fdegenerate}  is defined as follows; see cf.~\cite[Definition $1.2$]{frac non}.
\begin{defi}[Stochastic entropy solution of \eqref{eq:fdegenerate}]\label{defi:fd}
An element $ u\in$ $N_w^2((0, T); L^2(\mathbb{R}^d))$ with initial data $u_0\in$ $L^2(\mathbb{R}^d)$ is called a stochastic entropy solution of \eqref{eq:fdegenerate} if  for any non-negative test function $\psi \in {C}_c^{1,2}([0,\infty) \times \mathbb{R}^d)$, convex entropy flux pair $(\beta, \zeta)$ and any $k\in \R$, the following inequality holds:
\begin{align*}
    0 \le &\int_{\mathbb{R}^d} \psi(0,x)\beta(u_0(x) -k)dx + \int_{\mathbb{Q}_T}\Big \{ \partial_t\psi(t,x)\beta(u(t,x )-k) + \zeta(u(t,x))\cdot \nabla\psi(t,x) \Big \}\,dx\,dt \notag \\ 
    &- \int_{\mathbb{Q}_T} \Big[\mathcal{L}_\theta^{\bar{r}}[\phi(u(t,\cdot))](x)\psi(t,x)\beta'(u(t,x)-k) + \phi_k^\beta(u(t,x))\mathcal{L}_{\theta,\bar{r}}[\psi(t, .)](x) \Big]\,dx\,dt \notag \\
   & +\int_{\mathbb{Q}_T} \sigma(u(t,x))\beta'(u(t,x)-k)\psi(t,x)\,dx\,dW(t)+
   \frac{1}{2} \int_{\mathbb{Q}_T} \sigma^2(u(t,x))\beta''(u(t,x)-k)\psi(t,x)\,dx\,dt, \notag \quad  \mathbb{P}\text{-a.s.},
\end{align*}
\end{defi}
where $\phi_k^\beta(a) = \displaystyle \int_k^a\beta'(r-k)\phi'(r)\,dr.$
\vspace{0.2cm}

We aim to show the convergence of approximate solutions, constructed via time splitting method (cf.~Subsection \ref{sec:time splitting}), to the unique entropy solution, and to do so we need the following assumptions:

\begin{Assumptions}
\item \label{A1}  The initial function $u_0: \R^d \rightarrow \R$ belongs to $L^2(\R^d) \cap L^{\infty}(\rd)\cap BV(\R^d)$.
\item \label{A2} $\phi:\R \goto \R$ is a non-decreasing Lipschitz continuous function with $\phi(0)=0$. 
\item \label{A3}  The flux function $f: \R \rightarrow \R^d$ is a Lipschitz continuous function with $f(0)=0$.
\item\label{A4} $\sigma : \R \rightarrow \R$ is Lipschitz continuous i.e., there  exists a constant $C>0$ such that, for all $u, v \in \mathbb{R}$, 
 \[|\sigma(u) - \sigma(v)|\le C|u-v|. \]
\item\label{A5} $\sigma(0) = 0$ and there exists $M>0$ such that $\sigma(u) = 0$ for all $ |u| > M $. 
\item \label{A6} $\eta: \R \times \R \rightarrow \R$ and there exist positive constants $0<\lambda^* <1$  and  $ C>0$, such that for all
$ u,v,z\in \R$,
 \begin{align*} | \eta(u;z)-\eta(v;z)|  & \le  \lambda^* |u-v|( |z|\wedge 1), \\
  \text{and}\quad|\eta(u;z)| & \le C(1+|u|)(|z|\wedge 1).
 \end{align*}
 \item \label{A7}$\eta(0;z)= 0$, for all $z \in \R$ and $\eta(u;z)=0$, for all $|u|> M$. 
\item \label{A8}  The L\'{e}vy measure $m(dz)$ which has a possible singularity at $z=0$, satisfies 
\begin{align*}
 \int_{|z|>0} (1\wedge |z|^2) \,m(dz) < + \infty.
\end{align*}
\end{Assumptions}
The well-posedness of stochastic entropy solution of \eqref{eq:fractional} is given by the following theorem, whose proof can be found in \cite{frac lin}.
\begin{thm}
\label{thm: fractional}
Under the assumptions \ref{A1}, \ref{A3}, \ref{A4}, \ref{A6} and \ref{A8}, there exists a unique stochastic entropy solution of \eqref{eq:fractional} in the sense of Definition \ref{defi:entropysol-f}.
\end{thm}
Regarding the existence and uniqueness of entropy solution of \eqref{eq:fdegenerate}, we arrive at the following theorem, cf.~\cite{frac non}.
\begin{thm}\label{thm: fdegenerate}
 Let the assumptions \ref{A1}-\ref{A4} be true. Then there exists a unique stochastic entropy solution of  \eqref{eq:fdegenerate} in the sense of Definition \ref{defi:fd}. 
\end{thm}

\subsection{ Time Splitting method}\label{sec:time splitting}
We will use splitting method to construct  approximate solutions to \eqref{eq:fractional} and \eqref{eq:fdegenerate}. We break the underlying equation into two sub-equations and then construct operator splitting solutions by solving the first sub-equation by using the solution of the second equation as initial datum. To be more specific, let $T > 0$ be fixed, $0 \le s \le T.$ Let  $R(t, s)$ be the operator which takes the value $u_s$ to the solution $u$ at time $s$ of the stochastic equation 
 \begin{equation}\label{eq:noise}
        \begin{cases}
    \displaystyle      du(t,x) = \sigma(u(t,x))dW(t) + \int_{|z| > 0} \eta(u(t,x); z) \widetilde{N}(dz,dt), \hspace{2mm} t \in [s, T]\,, \\
          u(t = s) = u_s\,,
         \end{cases}
\end{equation}
 \[  \text{i.e.}, \hspace{2mm} u(t) = R(t, s)u_s = u_s + \int_s^t \sigma(u(r,x))dW(r) +  \int_s^t \int_{|z| > 0} \eta(u(r,x); z) \widetilde{N}(dz,dr).\]

Again, let $S(t-s)$ be an operator which takes the initial function $u(s,x)$ at time $s$ to the weak entropy solution $u$ at time $t$ to the following fractional conservation laws (see, \cite[Theorem 3.1]{Alibaud 2007})
\begin{equation}\label{eq:operator-S}
        \begin{cases}
       \displaystyle  \frac{\partial}{\partial t}u(t,x)+ [\mathcal{L}_{\theta}(u(t,\cdot))(x) + \mbox{div}_x f(u(t,x))] = 0,  \hspace{2mm} t \in [s, T]\,, \\
         u( t = s, x) = u(s, x), \hspace{2mm} x \in \mathbb{R}^d\,, \\
        \end{cases}
\end{equation}
$$\text{i.e.,}\quad  u(t,x)= S(t-s)u(s,x).$$
In a similar way, we can find a solution operator $\bar{S}(t-s)$ for the deterministic fractional degenerate 
conservation laws (see, \cite[Section 3]{cifani}) i.e., $u(t,x)= \bar{S}(t-s)u(s,x)$, where $u$ is a weak solution of 
\begin{equation}
 \label{eq:operator Sbar}
 \begin{cases} 
 \frac{\partial}{\partial t}u(t,x)+ \mathcal{L}_\theta[\phi(u(t,\cdot))](x) + \mbox{div}_x f(u(t,x)) =0,\quad  t \in [s,T]\, \\
 u(t=s,x) = u(s,x), \quad x \in \R^d\,.
\end{cases}
\end{equation}
We discretize the time interval $[0,T]$ as follows. For $N\in \mathbb{N}$, we consider the uniform temporal discretization parameter $\Delta t = \frac{T}{N}$. Let $t_n=n\Delta t,~n\in \{0,1,2,\ldots, N\}$ be the uniform partition of $[0,T]$. Following the ideas of \cite{Bauzet-2015-Splitting, risebroholden1997}, we define the approximate solutions of \eqref{eq:fractional} in terms of operators $S(t)$ and $R(t,s)$ as follows:
\begin{equation}
 \label{approxi:solu} u_{\Delta t}(t,x)=
 \begin{cases} 
 u^n(x),\quad & \text{if}~~t=t_n, \\
 R(t,t_n)u^n(x), \quad &\text{if}~~ t\in (t_n, t_{n+1}),
\end{cases}
\end{equation}
where the sequence $\{u^n\}_{n\ge 0}$ is defined by 
\begin{equation}
 \label{eq:sequence}
 \begin{cases} 
 u^0(x)&=u_0(x) \\
 u^{n+1}(x)&= S(\Delta t) R(t_{n+1}, t_n) u^n(x).
\end{cases}
\end{equation}
Similarly, we can define approximate solutions for \eqref{eq:fdegenerate} in terms of operators $\bar{S}(t)$ and $\bar{R}(t,s)$ (taking $\eta = 0$ in \eqref{eq:noise}) with
\begin{align}
 u^{n+1}=\bar{S}(\Delta t)\bar{R}(t_{n+1}, t_n)u^n(x). \label{eq:sequence-nonlinear}
 \end{align}
For notational simplicity let us also define the approximate solutions of \eqref{eq:fdegenerate} as  $u_{\Delta t}(t,x).$
Roughly, we can say that operator splitting method is a two step method. In the first step, it solves the stochastic equation while in the second step, it solves the scalar conservation laws without source terms and with initial condition as obtained from the first step.
\vspace{0.1cm}

Now we are ready to state the main results of this article. 
\begin{thm}\label{thm: mainthmf}
Let the assumptions \ref{A1}, \ref{A3}-\ref{A8} be true. Then the approximate solutions $u_{\Delta t}(t, x)$ given by \eqref{approxi:solu} and \eqref{eq:sequence} converges to a unique entropy solution of \eqref{eq:fractional} in $L_{loc}^p(\mathbb{R}^d; L^p(\Omega \times (0,T)))$ for $1 \le p < \infty.$
\end{thm}

\begin{thm}\label{thm: mainthmd}
Under the assumptions \ref{A1}-\ref{A5}, the approximate solutions $u_{\Delta t}(t, x)$ given by \eqref{approxi:solu} and \eqref{eq:sequence-nonlinear} converges to a unique entropy solution of \eqref{eq:fdegenerate} in $L_{loc}^p(\mathbb{R}^d; L^p(\Omega \times (0,T)))$ for $1 \le p < \infty.$
\end{thm}

\section{ Properties of Splitting Operators}\label{sec:properties}
Recall that $S(t)$ and $\bar{S}(t)$ are operators which solves the scalar conservation laws as mentioned in Subsection \ref{sec:time splitting}. Regarding the properties of the solution operator $S(t)$, we have the following lemma due to \cite[Lemma $4.5$]{cifani} and \cite[Lemma $2.1$]{cifani fds}. 
\begin{lem}\label{lem: prpertiesof S}
Let the solution of \eqref{eq:operator-S} be defined as $u(t) = S(t)u_0$, where $u_0 \in L^\infty(\mathbb{R}^d) \cap BV(\mathbb{R}^d)$. Then the following estimates hold:
 \begin{itemize}
  \item[$(a)$] For almost every $t \in [0,T]$,
  \[ ||u(t)||_{L^\infty(\R^d)} \le ||u_0||_{L^\infty(\R^d)} \quad and \quad |u(t)|_{BV(\R^d)} \le|u_0|_{BV(\R^d)}.\]
  \item[$(b)$] There exists a constant $C>0$ such that for any $t_1,~t_2\in [0,T]$,
\[\int_{\mathbb{R}^d}|u(t_1, x) - u(t_2, x)|dx \le C|u_0|_{BV(\mathbb{R}^d)}\rho(t_1-t_2)\,, \]
 \end{itemize}
where  the function $\rho:[0,\infty)\goto [0,\infty)$ is given by
\begin{align*}
\rho(s) = 
\begin{cases}
         |s| \hspace{5mm} \text{if} \hspace{2mm} \theta \in (0,\frac{1}{2}),\\
         |s\ln s| \hspace{5mm} \text{if} \hspace{2mm} \theta = \frac{1}{2},\\
         |s|^{\frac{1}{\theta}} \hspace{5mm} \text{if} \hspace{2mm} \theta \in (\frac{1}{2},1).
\end{cases}
\end{align*}
The above result holds true for the solution  $u(t) = \bar{S}(t)u_0$ of equation \eqref{eq:operator Sbar}.
\end{lem}
To proceed further, we prove a useful property of $R(t,s)$ which is given in the following lemma.
\begin{lem}\label{lem:propertiesof R}
Let s $\in$ $[0, T]$ and $t\in [s, T]$. Then, $\mathbb{P}$-a.s., the operator 
$R(t,s)$ is the identity outside the interval $[-2M, 2M]$ and takes $[-2M, 2M]$ into itself i.e., $R(t,s)u_s = u_s$, if $|u_s|>2M$  and $R(t,s)u_s \in [-2M, 2M]$ if $u_s\in [-2M, 2M]$, where $M$ is defined by \ref{A5} and \ref{A7}.
\begin{proof}
 Let $\varphi$ be a smooth function on $\R$ such that $\varphi$ vanishes in $[-2M, 2M]$ and increasing in 
 $(-\infty, -2M] \cup [2M, \infty)$. Let $v(t)=R(t,s)v(s)$ i.e., $v(t)$ satisfies the following equation:
 $$v(t)= v(s)+ \int_s^t\sigma(v(r,x))\,dW(r) + \int_{s}^t \int_{|z|>0} \eta(v(r);z)\, \widetilde{N}(dz,dr).$$
 We apply It\^{o}-L\'{e}vy formula on $\varphi(v(t))$ and have 
 \begin{align*}
  \varphi(v(t)) = & \varphi(v(s)) + \int_s^t \sigma(v(r))\varphi^\prime(v(r))\,dW(r) +\frac{1}{2} \int_s^t \sigma^2(v(r))\varphi^{\prime\prime}(v(r))\,dr\notag  \\ &+ \int_{s}^t \int_{|z|>0} \int_{0}^1 \eta(v(r);z) \varphi^\prime\big( v(r)+ \lambda \eta(v(r);z)\big)
  \,d\lambda \,\widetilde{N}(dz,dr) \notag \\
  &  \quad + \int_{s}^t \int_{|z|>0} \int_{0}^1 (1-\lambda) \eta^2(v(r);z) \varphi^{\prime\prime}\big( v(r)+ \lambda \eta(v(r);z)\big)
  \,d\lambda \,m(dz)\,dr.
 \end{align*}
 Since $\sigma(u)=0$ and $\eta(u;z)=0$, for $|u|>M$ and $\varphi$ vanishes in $[-2M, 2M]$, we get that
 \begin{align}
  \varphi(v(t))= \varphi(v(s)), \quad  \forall t\in [s,T],\quad \mathbb{P}-\text{a.s.} \label{eqn:for phi}
 \end{align}
\noindent{\bf Case $1$}. Suppose $v(s) \in [-2M,2M]$. Then, thanks to equality \eqref{eqn:for phi}, one easily conclude that
$v(t) \in [-2M,2M]$. 
\vspace{.1cm}

\noindent{{\bf Case $2.$}}  Suppose $v(s) \in [-2M,2M]^{\complement}$(complement of  $[-2M, 2M]$). 
In this case, since outside $[-2M, 2M]$, $\varphi$ is increasing, 
we conclude from \eqref{eqn:for phi} that $\mathbb{P}$-a.s., $v(t)=v(s)$ and hence $R(t,s)\bar{u}=\bar{u}$
for any $|\bar{u}| > 2M$.  This completes the proof.
\end{proof}

\end{lem}

 It is shown in \cite{frac lin, Majee-2017} that if the noise coefficient does not depend upon the spatial variable $x$ explicitly and initial condition is in $BV$-class, then the solution of \eqref{eq:fractional} also lies in $BV$-class. A similar result can be obtained for the solution of \eqref{eq:noise} and we demonstrate it in the following lemma. 
\begin{lem}\label{lem:BV}
 Let $s\in [0,T]$ and $v_0$ be an $\mathcal{F}_s$-predictable process satisfying 
 $$ \mathbb{E} \big[ |v_0|_{BV(\R^d)}+ ||v_0||^2_2 \big] < + \infty.$$
 Define $v(t):= R(t,s)v_0$, for $t\in [s,T]$. Then for all $t\in [s,T]$
 \begin{align*}
  \mathbb{E}\Big[ |v(t)|_{BV(\R^d)}\Big] \le \mathbb{E}\Big[ |v_0|_{BV(\R^d)}\Big].
 \end{align*}
\end{lem}
\begin{proof}
Let $0\le \psi \in C_c^2(\R^d)$ be a cut-off function.  
Let $\beta:\R \rightarrow \R$ be a $C^\infty$ function satisfying 
 \begin{align*}
      \beta(0) = 0,\quad \beta(-r)= \beta(r),\quad \beta^\prime(-r) = -\beta^\prime(r),\quad \beta^{\prime\prime} \ge 0,
 \end{align*} and 
\begin{align*}
\beta^\prime(r)=
\begin{cases} 
-1,\quad &\text{when} ~ r\le -1,\\
\in [-1,1], \quad &\text{when}~ |r|<1,\\
+1, \quad &\text{when} ~ r\ge 1.
\end{cases}
\end{align*} For any $\xi > 0$, define  $\beta_\xi:\R \rightarrow \R$ by 
\begin{align*}
         \beta_\xi(r) := \xi \beta(\frac{r}{\xi}).
\end{align*} Then
\begin{align}\label{eq:approx to abosx}
 |r|-M_1\xi \le \beta_\xi(r) \le |r|\quad \text{and} \quad |\beta_\xi^{\prime\prime}(r)| \le \frac{M_2}{\xi} {\bf 1}_{|r|\le \xi},
\end{align} where
\begin{align*}
 M_1 := \sup_{|r|\le 1}\big | |r|-\beta(r)\big |, \quad M_2 := \sup_{|r|\le 1}|\beta^{\prime\prime} (r)|.
\end{align*}
We apply It\^{o}-L\'{e}vy formula on $\beta_{\xi}$, and get
  \begin{align}
  &  \mathbb{E}\Big[ \int_{\R^d}\beta_{\xi}(v(t))\psi(x)\,dx \Big] -  \mathbb{E}\Big[ \int_{\R^d}\beta_{\xi}(v_0)\psi(x)\,dx \Big]  = \frac{1}{2}\mathbb{E}\Big[\int_s^t\int_{\mathbb{R}^d}\sigma^2(v(r))\beta_\xi^{\prime\prime}(v(r))\psi(x)\,dx\,dr \Big] \notag \\
   &+  \mathbb{E}\Big[\int_{s}^t \int_{|z|>0} \int_{0}^1 \int_{\R^d} (1-\lambda) \eta^2(v(r);z) \beta_{\xi}^{\prime\prime}\big( v(r)+
   \lambda \eta(v(r);z)\big)\psi(x)\,dx\,d\lambda \,m(dz)\,dr \Big]=:\mathcal{E}_0(\xi) + \mathcal{E}_1(\xi).\label{eqn:bv-1}
  \end{align}
  Since $r^2 {\beta_{\xi}}^{\prime \prime}(r) \le C \xi$, we see that
  \begin{align*}
      \mathcal{E}_0(\xi) \le C(T,\psi)\xi \rightarrow 0 \quad \text{as} \quad \xi \rightarrow 0.
  \end{align*}
  By the assumptions \ref{A6} and \ref{A7}, we see that
  \begin{align*}
   |\eta(v(r);z)| &= |\eta(v(r);z)- \eta(0;z)|  \le \lambda^* |v(r)|.
  \end{align*}
  Since $\beta^{\prime\prime}$ is an even function, without loss of generality we may assume that $v(r)\ge 0$. Thus, we obtain
  $-\lambda^* v(r) \le \eta(v(r);z)$ and hence 
  \begin{align}
   0\le v(r) \le (1-\lambda^*)^{-1} \big( v(r) + \lambda \eta(v(r);z)\big) \quad \text{for all} \, \lambda \in [0,1].  \label{game:1}
  \end{align} 
  Again, in view of the assumptions  \ref{A6} and \ref{A7}, it is easy to see that
  \begin{align*}
   \eta^2(v(r);z) &\le |v(r)|^2 (1\wedge |z|^2) 
   \le (1-\lambda^*)^{-2} \big( v(r) + \lambda \eta(v(r);z)\big)^2 (1\wedge |z|^2) \quad (\text{by}~~\eqref{game:1}).
  \end{align*}
  Since $r^2 {\beta_{\xi}}^{\prime \prime}(r) \le C \xi$, we see that $ \eta^2(v(r);z)\beta_{\xi}^{\prime\prime}\big( v(r)+
   \lambda \eta(v(r);z)\big) \le C \xi\,(1-\lambda^*)^{-2}(1\wedge |z|^2)$ and hence 
  \begin{align*}
   \mathcal{E}_1(\xi) &\le C \xi\,(1-\lambda^*)^{-2}  \, \mathbb{E}\Big[\int_{s}^t \int_{|z|>0} \int_{0}^1 \int_{\R^d} (1-\lambda)
   (1\wedge |z|^2)\psi(x)\,dx\,d\lambda \,m(dz)\,dr \Big] \notag \\
   & \le C(T,\psi) \xi \goto 0 \quad \text{as}~~\xi \goto 0.
  \end{align*}
  Let us explain how one can assume, without loss of generality, that $v(r)\ge 0$. Observe that if $v(r)\le 0$, then by setting $\tilde{v}(r):=-v(r)\ge 0$, one can arrive at the following inequality 
  $$  0\le \tilde{v}(r) \le (1-\lambda^*)^{-1} \big( \tilde{v}(r) - \lambda \eta(v(r);z)\big)\quad \forall\, \lambda \in [0,1].  $$
  Since $\beta^{\prime\prime}(\cdot)$ is non-negative and even function, we then obtain 
  \begin{align*}
  & v^2(r)\beta^{\prime\prime}(v(r)+\lambda \eta(v(r);z))= \tilde{v}^2(r) 
  \beta^{\prime\prime}(\tilde{v}(r)-\lambda \eta(v(r);z))\notag \\
  & \le C(\lambda^*) \big( \tilde{v}(r) - \lambda \eta(v(r);z)\big)^2
   \beta^{\prime\prime}(\tilde{v}(r)-\lambda \eta(v(r);z)) \le C(\lambda^*)\xi\,\cdot
  \end{align*}
  One can give a similar argument, used for the case $v(r)\ge 0$, to conclude that $\mathcal{E}_1(\xi)\goto 0$ as $\xi \goto 0$.
  \vspace{.1cm}
  
  We first pass to the limit in \eqref{eqn:bv-1} as $\xi \goto 0$ and then send $\psi\goto \mathds{1}_{\R^d}$, where
 $\mathds{1}_{B}$ denotes the characteristic function of any set $B$.
  Thanks to \eqref{eq:approx to abosx} and the estimations of $\mathcal{E}_0(\xi)$ and
  $\mathcal{E}_1(\xi)$, we obtain 
  \begin{align}
    \mathbb{E} \Big[ \int_{\R^d} |v(t)|\,dx \Big] =   \mathbb{E} \Big[ \int_{\R^d} |v_0(x)|\,dx \Big].\label{eqn: L1 bound}
  \end{align}
  
To proceed further, we recall a classical result on approximations of $BV$ function; cf.~\cite{Evans}. For any $u \in BV(\R^d)$, there exists a sequence $\{u_\eps\}_{\eps>0}$ with $u_\eps \in W^{1,2}(\R^d)\cap BV(\R^d)$ such that 
  \begin{align*}
    \int_{\R^d} |u_\eps -u|dx \goto 0, \, \text{as} \,\,\eps \goto 0, \quad \text{and} \quad   
   |u_\eps |_{BV(\R^d)} \goto |u|_{BV(\R^d)}.
  \end{align*}
Note that $v_0 \in BV(\R^d)$ $\mathbb{P}$-a.s. Therefore, there exists a sequence $v_0^\eps \in W^{1,2}(\R^d)\cap BV(\R^d)$ such that 
  $v_0^\eps \goto v_0$ in $L^1(\R^d)$ almost surely  and 
  \begin{align*}
   \mathbb{E}\Big[TV_x(v_0^\eps)\Big] \le  \mathbb{E}\Big[TV_x(v_0)\Big].
  \end{align*}
  Let $v_\eps(t)= R(t,s)v_0^\eps$, for all $t\in [s,T]$ with $s\in [0,T]$. In order to obtain BV estimate of $v$, we need to estimate $\frac{\partial v_\eps}{\partial x_i}$ and then pass to the limit as $\eps$ tends to zero. To do so, we first note that $v_\eps - v$ satisfies the following equation:
\begin{equation*}
 \begin{cases} 
 \displaystyle d(v_\eps(t)-v(t)) = \Big(\sigma(v_\eps(t) -\sigma(v(t))\Big)dW(t) + \int_{|z|>0} \Big( \eta(v_\eps(t);z)-\eta(v(t);z)\Big) \widetilde{N}(dz,dt) \\
 v_\eps -v|_{t=s}= v_0^\eps -v_0.
\end{cases}
\end{equation*}

Let $0\le \psi \in C_c^2(\R^d)$ be a cut-off function.  Applying It\^{o}-L\'{e}vy formula to $\beta_{\xi}(v_\eps -v)$ and then taking
the expectation in the resulting equation, we get
\begin{align}
\mathbb{E}\Big[ \int_{\R^d}\beta_{\xi} & \big(v_\eps(t)-v(t)\big)\psi(x)\,dx \Big] -  \mathbb{E}\Big[ \int_{\R^d}\beta_{\xi}\big(v_0^\eps-v_0\big)\psi(x)\,dx \Big]  \notag \\
  &= \frac{1}{2}\mathbb{E}\Big[\int_s^t\int_{\mathbb{R}^d} \big|\sigma(v_\eps(r)) - \sigma(v(r))\big|^2\beta_\xi^{\prime\prime}\big(v_\eps(r) - v(r)\big)\psi(x)\,dx\,dr\Big]\notag \\
   &+ \mathbb{E}\Big[\int_{s}^t \int_{|z|>0} \int_{0}^1 \int_{\R^d} (1-\lambda) 
   \beta_{\xi}^{\prime\prime}\Big( v_\eps(r) -v(r) + 
   \lambda  \big(\eta(v_\eps(r);z)-\eta(v(r);z)\big) \Big) \notag \\
  & \hspace{5cm} \times \big|\eta(v_\eps(r);z)-\eta(v(r);z)\big|^2 \psi(x)\,dx\,d\lambda \,m(dz)\,dr \Big] \notag \\
   & =:  \mathcal{E}_1(\eps,\xi) + \mathcal{E}_2(\eps,\xi).\label{eqn:bv-2}
  \end{align}
  By using the Lipschitz continuity of $\sigma$  and the fact that $r^2 {\beta_{\xi}}^{\prime \prime}(r) \le C \xi$, we see that
  \begin{align}\label{approx Xi_1}
      \mathcal{E}_1(\eps, \xi) \le C(T,\psi)\xi.
  \end{align}
We now  estimate the term $\mathcal{E}_2(\eps,\xi)$. Observe that $\mathcal{E}_2(\eps,\xi)$ can be written as
\begin{align*}
 \mathcal{E}_2(\eps,\xi)=  \mathbb{E}\Big[\int_{s}^t \int_{|z|>0} \int_{0}^1 \int_{\R^d} (1-\lambda) 
   \beta_{\xi}^{\prime\prime}\big(a + \lambda h \big)h^2 \psi(x)\,dx\,d\lambda \,m(dz)\,dr \Big],
\end{align*}
where $a= v_\eps(r)-v(r)$ and $h=\eta(v_\eps(r);z)-\eta(v(r);z)$. In view of the assumption  \ref{A6}, we see that
$$ h^2\beta_{\xi}^{\prime\prime}(a+\lambda h) \le a^2 \beta_{\xi}^{\prime\prime}(a+\lambda h)(1\wedge |z|^2).$$
Therefore, to estimate $\mathcal{E}_2(\eps,\xi)$, we need to find an upper bound for $a^2 \beta_{\xi}^{\prime\prime}(a+\lambda h)$. As before, we assume without loss of generality 
that $a\ge 0$. Thus, by the assumption \ref{A6}, we have 
$$0\le a \le  (1-\lambda^*)^{-1} \big(a+\lambda h\big)\quad \forall\,\lambda\in [0,1]$$ and hence
\begin{align}
 \mathcal{E}_2(\eps,\xi) & \le   \mathbb{E}\Big[\int_{s}^t \int_{|z|>0} \int_{0}^1 \int_{\R^d} (1-\lambda) 
  a^2 \beta_{\xi}^{\prime\prime}\big(a + \lambda h \big) \psi(x)\,dx\,d\lambda \,m(dz)\,dr \Big] \notag \\
  & \le  \mathbb{E}\Big[\int_{s}^t \int_{|z|>0} \int_{0}^1 \int_{\R^d}(1-\lambda^*)^{-2}
  \big(a+\lambda h\big)^2 \beta_{\xi}^{\prime\prime}\big(a + \lambda h \big)(1\wedge |z|^2)  \psi(x)\,dx\,d\lambda \,m(dz)\,dr \Big] \notag \\
  & \le C\,\xi  \mathbb{E}\Big[\int_{s}^t \int_{|z|>0} \int_{\R^d}(1\wedge |z|^2)  \psi(x)\,dx \,m(dz)\,dr \Big] \notag \\
  & \le  C(T,\psi) \xi. \label{eqn:bv-2-error}
\end{align}
By using \eqref{approx Xi_1} and \eqref{eqn:bv-2-error}, we pass to the limit as $\xi \goto 0$ in \eqref{eqn:bv-2} and then send
 $\psi\goto \mathds{1}_{\R^d}$ to have 
 $$ \mathbb{E}\Big[  ||v_\eps(t)-v(t)||_{L^1(\R^d)} \Big] =   \mathbb{E}\Big[  ||v_0^\eps-v_0||_{L^1(\R^d)} \Big].$$ This implies that for every 
 $t\in [s,T]$, $v_\eps(t)\goto v(t)$ in $L^1(\Omega \times \R^d)$.
\vspace{.1cm}

\noindent Recall that since $v_0^\eps \in W^{1,2}(\R^d)\cap BV(\R^d)$, 
we have $\mathbb{P}$- a.s. and for all $t \in [0,T]$,
$$v_{\eps}(t)= v_0^{\eps}+ \int_s^t\sigma(v_\eps(r))\,dW(r) + \int_{s}^t \int_{|z|>0} \eta(v_{\eps}(r);z) \widetilde{N}(dz,dr)\quad \text{in}\,\,W^{1,2}(\R^d).$$
Thanks to the linear-continuity of the differential operators $\partial_{x_i}: W^{1,2}(\R^d)\goto L^2(\R^d)$ and the chain-rule formula, $\partial_{x_i}v_\epsilon$ satisfies the following equation in $L^2(\R^d)$: for $1\le i\le d$,
\begin{equation*}
 \begin{cases} 
\displaystyle  d(\partial_{ x_i} v_\eps(t)) = \sigma^{\prime}(v_\eps(t))\partial_{ x_i} v_\eps(t)\,dW(t) + \int_{|z|>0}  \eta^\prime(v_\eps(t);z)\partial_{ x_i} v_\eps(t) \widetilde{N}(dz,dt) \\
 \partial_{ x_i} v_\eps|_{t=s}= \partial_{ x_i} v_0^\eps.
\end{cases}
\end{equation*}
As before, we apply It\^{o}-L\'{e}vy formula  to obtain
 \begin{align}
  \mathbb{E}\Big[ \int_{\R^d}\beta_{\xi} & \big(\partial_{ x_i} v_\eps(t)\big)\psi(x)\,dx \Big]
  - \mathbb{E}\Big[ \int_{\R^d}\beta_{\xi}\big(\partial_{ x_i} v_\eps\big)\psi(x)\,dx \Big]\notag \\ & = \,\frac{1}{2} \mathbb{E}\Big[\int_s^t\int_{\mathbb{R}^d} \big|\sigma^{\prime}(v_\eps(r))\partial_{ x_i} v_\eps\big|^2\beta_{\xi}^{\prime\prime}\big(\partial_{ x_i} v_\eps\big)\psi(x)\,dx\,dr\Big] \notag \\
   &+ \mathbb{E}\Big[\int_{s}^t \int_{|z|>0} \int_{0}^1 \int_{\R^d} (1-\lambda) 
   \beta_{\xi}^{\prime\prime}\Big( \partial_{ x_i} v_\eps(r) + 
   \lambda \eta^\prime(v_\eps(r);z)\partial_{ x_i} v_\eps(r) \Big) \notag \\
  & \hspace{5cm} \times \big|\eta^\prime(v_\eps(r);z)\partial_{ x_i} v_\eps(r)\big|^2 \psi(x)\,dx\,d\lambda \,m(dz)\,dr \Big] \notag \\
   &=:\mathcal{E}_3(\eps,\xi) + \mathcal{E}_4(\eps,\xi).\notag
  \end{align}
In view of the assumption \ref{A4} and \eqref{eq:approx to abosx}, we see that
\begin{align}
    \mathcal{E}_3(\eps,\xi) \le C(T, \psi)\xi.\notag 
\end{align}
To estimate $\mathcal{E}_4(\eps,\xi)$, we proceed as follows: Note that we can rewrite $\mathcal{E}_4(\eps,\xi)$ as
\begin{align*}
\mathcal{E}_4(\eps,\xi)=  \mathbb{E}\Big[\int_{\R_x^d}\int_{s}^t \int_{|z|> 0} \int_{0}^1 (1-\tilde{\theta})\,b^2 \beta_{\xi}^{\prime\prime} 
 \big(a+\tilde{\theta}\,b \big)\psi(x)\,d\tilde{\theta}\,m(dz)\,dr\,dx \Big],
   \end{align*} 
   where $a=\partial_{x_i}v_\eps(r)$ and $b= \eta^\prime(v_\eps(r);z)\partial_{x_i}v_\eps(r)$.
In view of the assumption ~\ref{A6}, it is easy to see that 
   \begin{align}
    b^2  \beta_{\xi}^{\prime\prime} (a+\tilde{\theta}\,b ) &
    \le \big|\partial_{x_i}v_\eps(r)\big|^2(1\wedge |z|^2)  \beta_{\xi}^{\prime\prime} (a+\tilde{\theta}\,b ).\label{estimate-h^2}
   \end{align}
Next we move on to find a suitable upper bound on $a^2  \beta_{\xi}^{\prime\prime} \big(a+\tilde{\theta}\,b \big)$. 
Since $\beta^{\prime\prime}$ is an even function, without loss of generality, we may assume that $a>0$. Then by  Assumption \ref{A6}, it is evident that
   \begin{align*}
    \partial_{x_i}v_\eps(r) + \tilde{\theta} \eta^\prime \big(v_\eps(r);z\big)\partial_{x_i}v_\eps(r) \ge (1-\lambda^*)
    \partial_{x_i}v_\eps(r),
   \end{align*}
   for $\tilde{\theta} \in [0,1]$. In other words
   \begin{align}
     0\le a \le (1-\lambda^*)^{-1} (a+ \tilde{\theta}\,b).\label{estimate-a}
   \end{align}
Combining \eqref{estimate-h^2} and \eqref{estimate-a} yields
   \begin{align*}
     b^2  \beta_{\xi}^{\prime\prime} (a+\tilde{\theta}\,b )   \le  (1\wedge |z|^2) (1-\lambda^*)^{-2} (a+\tilde{\theta} \,b)^2
     \beta_{\xi}^{\prime\prime} (a+\tilde{\theta}\,b )\le C (1\wedge |z|^2) \,\xi.
   \end{align*}
Thanks to the assumption\ref{A8}, we infer that 
\begin{align}
\big|\mathcal{E}_4(\eps,\xi)\big| \le C\, t\,\xi\,||\psi||_{L^1(\R^d)}.\notag
\end{align}
Hence passing to the limit $\eps \goto 0$ and then sending $\psi \goto \mathds{1}_{\R^d}$, we have 
$$ \mathbb{E}\Big[  || \partial_{x_i}v_\eps(t)||_{L^1(\R^d)} \Big] =   \mathbb{E}\Big[  || \partial_{x_i}v_0^\eps||_{L^1(\R^d)} \Big] < + \infty.$$
Thus, for all $t\in [s,T]$, $v_\eps(t) \in BV(\R^d)$ almost surely. Also, we have seen that  
$v_\eps(t)\goto v(t)$ in $L^1(\Omega \times \R^d)$.
Therefore, by lower semi continuity property of $BV$-norm, we have, almost surely
$$TV_x (v(t)) \le \liminf _{\eps \downarrow 0} TV_x (v_\eps(t)).$$
Note that $v(t)\in L^1(\Omega\times \R^d)$ and hence $v(t)$ is measurable with respect to $\mathbb{P}$. Since supremum of measurable functions is again measurable, by definition of total variation, $TV_x(v(t))$ is measurable with respect to the probability measure $\mathbb{P}$.  Consequently, taking expectation and using Fatou's lemma we obtain 
 \begin{align}
  \mathbb{E}\Big[ TV_x (v(t))\Big] \le \liminf _{\eps \downarrow 0}  \mathbb{E}\Big[TV_x (v_\eps(t))\Big]
 \le \liminf _{\eps \downarrow 0}  \mathbb{E}\Big[TV_x (v_0^\eps)\Big]\le   \mathbb{E}\Big[TV_x (v_0)\Big].\label{inequality:tv bound}
 \end{align}
 In view of \eqref{eqn: L1 bound} and \eqref{inequality:tv bound}, we conclude  that
 $ \mathbb{E}\Big[ |v(t)|_{BV(\R^d)}\Big] \le  \mathbb{E}\Big[ |v_0|_{BV(\R^d)}\Big]$, for all $t\in[s,T]$ with $s\in [0,T]$.
This essentially completes the proof.
\end{proof}
\begin{rem}
Observe that, the operator $\bar{R}(t,s)$  also satisfies the properties of Lemmas \ref{lem:propertiesof R} and \ref{lem:BV}.
\end{rem}
\section{ A-Priori Estimates}\label{sec:estimations}

In this section, we derive some necessary a-priori estimates for the approximate solutions $u_{\Delta t}(t,x)$ as well as for viscous solution of  \eqref{eq:fractional} ---which will be useful in establishing the convergence of approximate solutions to an entropy solution of the underlying problem.

\subsection{A-priori estimates for approximate solutions} We denote 
\begin{align}
u^{n+ \frac{1}{2}}:=R(t_{n+1}, t_n)u^n(x). \label{defi:n-half}
\end{align}
Then, for any $n\geq 0$, one has $u^{n+1}(x)=S(\Delta t)u^{n+\frac{1}{2}}(x)$.
\vspace{0.1cm}

We first prove the uniform bound of approximate solutions $u_{\Delta t}(t,x)$ in the following lemma.
\begin{lem}\label{lem:l-infinity bound-approximate-solution}
 There exists a constant $\widetilde{M}>0$, independent of $\Delta t$, such that $\mathbb{P}$-a.s. and for all $t\in [0,T]$, 
 \begin{align}
  ||u_{\Delta t}(t,\cdot)||_{L^\infty(\R^d)} \le \widetilde{M}.\label{l-infty-bound}
 \end{align}
\end{lem}
 
 \begin{proof}
  Recall that  $u^{n+1}= S(\Delta t)u^{n+\frac{1}{2}}$ for $n=0,1,\ldots,N-1$. Hence, in view of Lemma \ref{lem: prpertiesof S}, we have
  $$||u^{n+1}||_{L^\infty(\R^d)} \le ||u^{n+\frac{1}{2}}||_{L^\infty(\R^d)}.$$
  Again, thanks to
  Lemma~\ref{lem:propertiesof R}, we see that 
  $$||R(t,t_n)u^n||_{L^\infty(\R^d)} \le \max \big\{ 2M, ||u^n||_{L^\infty(\R^d)}\big\} \quad \forall\, t\in [t_n, t_{n+1}].$$ Hence 
  \begin{align*}
   ||u^{n+\frac{1}{2}}||_{L^\infty(\R^d)} &= || R(t_{n+1}, t_n)u^n||_{L^\infty(\R^d)}
    \le \max \big\{ 2M, ||u^n||_{L^\infty(\R^d)}\big\} 
   \le  \max \big\{ 2M, ||u_0||_{L^\infty(\R^d)}\big\}.
  \end{align*}
Thus, by construction of $u_{\Delta t}(t,x)$ (cf.~\eqref{approxi:solu}), we get
$$||u_{\Delta t}(t,\cdot)||_{L^\infty(\R^d)} \le  \max \big\{ 2M, ||u_0||_{L^\infty(\R^d)}\big\}\quad \forall\, t\ge 0.$$
Taking $\widetilde{M}= \max \big\{ 2M, ||u_0||_{L^\infty(\R^d)}\big\}$, we  arrive at the assertion of the lemma.
 \end{proof}
 Likewise the deterministic counterpart of stochastic balance laws, one may expect the $BV$ bound for the approximate solutions $u_{\Delta t}(t,x)$. Regarding this, we have the following lemma.
\begin{lem} \label{lem:bv bound-approximate-solution}
  For  all $t\ge 0$, the approximate solutions $u_{\Delta t}(t,x)$ enjoy the following total variation bound
  \begin{align*}
   \mathbb{E}\Big[ |u_{\Delta t}(t,\cdot)|_{BV(\R^d)}\Big] \le \mathbb{E}\Big[ |u_0|_{BV(\R^d)}\Big].
  \end{align*}
 \end{lem}
 
 \begin{proof}
  Note that $u_0 \in BV(\R^d)\cap L^2(\R^d)$ and hence  by virtue of 
  Lemma \ref{lem:BV}, it is easy to see that  
  $\mathbb{E}\Big[ |u^{n+\frac{1}{2}}|_{BV(\R^d)}\Big]\le \mathbb{E}\Big[ |u_0|_{BV(\R^d)}\Big]$. Thus, we have, by Lemma \ref{lem: prpertiesof S}
  \begin{align}
   \mathbb{E}\Big[ |u^{n+1}|_{BV(\R^d)}\Big] = \mathbb{E}\Big[ |S(\Delta t)u^{n+\frac{1}{2}}|_{BV(\R^d)}\Big] 
    \le \mathbb{E}\Big[ |u^{n+\frac{1}{2}}|_{BV(\R^d)}\Big] 
    \le \mathbb{E}\Big[ |u_0|_{BV(\R^d)}\Big].\notag 
  \end{align}
  Again, for any $t\in [t_n, t_{n+1})$, we have
  $$ \mathbb{E}\Big[ |R(t,t_n)u^n|_{BV(\R^d)}\Big] \le \mathbb{E}\Big[ |u^n|_{BV(\R^d)}\Big] \le 
  \mathbb{E}\Big[ |u_0|_{BV(\R^d)}\Big].$$
  Hence, in view of \eqref{approxi:solu}, we conclude that
  $\mathbb{E}\Big[ |u_{\Delta t}(t,\cdot)|_{BV(\R^d)}\Big]\le \mathbb{E}\Big[ |u_0|_{BV(\R^d)}\Big]$.
 \end{proof}
Since we are dealing with L\'{e}vy noise, one cannot get time continuous solution of  \eqref{eq:fractional}. However, we can expect a uniform temporal $L^1$-continuity of the approximate solutions $u_{\Delta t}(t,x)$, independent of the temporal discretization parameter $\Delta t$. In this context, we have the following lemma.
\begin{lem}
\label{lem:average_time_continuity-approximate-solution}
 Let $K$ be any compact subset of $\R^d$. Then, for any $t\in [t_n, t_{n +1}),~~~n=0,1,2,\ldots, N-1$, there exists a constant $C>0$ independent 
 of $\Dt$ and $n$ such that
 \begin{align*}
\mathbb{E}\Big[ \int_{K} \big| u_{\Delta t}(t_{n+1},x)-u_{\Delta t}(t,x)\big| \,dx \Big] \le C(T,K, u_0) \sqrt{\Delta t}.
 \end{align*}
\end{lem}

\begin{proof}
 Let $t\in [t_n, t_{n+1})$. Then, thanks to \eqref{eq:noise} and \eqref{approxi:solu}, we have 
 $$u_{\Delta t}(t,x)= u^n(x)+ \int_{t_n}^t\sigma(u_{\Delta t}(r,x))dW(r) + \int_{t_n}^t \int_{|z|>0} \eta(u_{\Delta t}(r,x);z) \widetilde{N}(dz,dr).$$
Furthermore, recall that $u_{\Delta t}(t_{n+1},x)= u^{n+1}(x)$. 
Let $K$ be any compact subset of $\R^d$. Then, for any $t\in [t_n,  t_{n+1})$
 \begin{align}
  & \mathbb{E}\Big[\int_{K} \big|  u_{\Delta t}(t_{n+1},x) -u_{\Delta t}(t,x)\big|\,dx \Big]  \notag \\
  & \quad \le \mathbb{E}\Big[ \int_{K} \big| u^{n+1}(x)-u^n(x)\big|\,dx \Big] + \mathbb{E}\Big[ \int_{K}\Big| \int_{t_n}^t \sigma(u_{\Delta t}(r,x)) dW(r)\Big|\,dx \Big]\notag \\ & \qquad \quad +
  \mathbb{E}\Big[ \int_{K}\Big| \int_{t_n}^t \int_{|z|>0}  \eta(u_{\Delta t}(r,x);z) \widetilde{N}(dz,dr)\Big|\,dx \Big] 
   =:\mathcal{A} + \mathcal{B} + \mathcal{C}. \label{estim: 1-time}
 \end{align}
 Let us first focus on the term $\mathcal{A}$. One could estimate $\mathcal{A}$ as follows:
 \begin{align}
  \mathcal{A} &\le  \mathbb{E}\Big[ \int_{K} \big| u^{n+1}(x)-u^{n+\frac{1}{2}}(x)\big|\,dx \Big] + 
   \mathbb{E}\Big[ \int_{K} \big| u^{n+\frac{1}{2}}(x)-u^n(x)\big|\,dx \Big] \notag \\
   & = \mathbb{E}\Big[ \int_{K} \big|S(\Delta t)u^{n+\frac{1}{2}}(x) -u^{n+\frac{1}{2}}(x)\big|\,dx \Big]
   +  \mathbb{E}\Big[ \int_{K} \big| R(t_{n+1},t_n)u^n(x)-u^n(x)\big|\,dx \Big] \notag \\
   & = \mathbb{E}\Big[ \int_{K} \big|S(\Delta t)u^{n+\frac{1}{2}}(x) -u^{n+\frac{1}{2}}(x)\big|\,dx \Big]  + \mathbb{E}\Big[ \int_{K} \Big|\int_{t_n}^{t_{n+1}} \sigma(u_{\Delta t}(r,x))dW(r) \Big|\,dx \Big]  \notag \\
   & \quad +  \mathbb{E}\Big[ \int_{K} \Big|\int_{t_n}^{t_{n+1}} \int_{|z|>0}  \eta(u_{\Delta t}(r,x);z) \widetilde{N}(dz,dr) \Big|\,dx \Big] 
    =: \mathcal{\bar{A}} + \mathcal{\bar{B}} + \mathcal{\bar{C}}.\label{estim: a1+a2}
 \end{align}
To proceed further, observe that, invoking 
Lemmas~\ref{lem: prpertiesof S}, \ref{lem:BV} and \ref{lem:bv bound-approximate-solution}, we obtain
 \begin{align}
  \mathcal{\bar{A}}  & \le C\rho(\Delta t)\, \mathbb{E}\Big[ |u^{n+\frac{1}{2}}|_{BV(\R^d)}\Big] 
  = C\rho(\Delta t)\, \mathbb{E}\Big[ |R(t_{n+1},t_n)u^n|_{BV(\R^d)}\Big] \notag \\
  & \le C\rho(\Delta t)\, \mathbb{E}\Big[ |u^{n}|_{BV(\R^d)}\Big] 
  \le C\rho(\Delta t)\, \mathbb{E}\Big[ |u_0|_{BV(\R^d)}\Big].\label{estim: a1}
 \end{align}
 We use Cauchy-Schwartz inequality and It\^{o} isometry to approximate the term $\mathcal{\bar{B}}$ and obtain
 \begin{align}\label{estimate b}
     \mathcal{\bar{B}} &\le \, \mathbb{E}\Big[ \Big(\int_{K}\Big\{ \int_{t_n}^{t_{n+1}} \sigma(u_{\Delta t}(r,x))\,dW(r) \Big\}^2\,dx\Big)^{\frac{1}{2}} \Big]\notag \\
    & \le \,C\Big(\mathbb{E}\Big[ \int_{K}\Big\{ \int_{t_n}^{t_{n+1}} \sigma(u_{\Delta t}(r,x))\,dW(r) \Big\}^2\,dx \Big]\Big)^{\frac{1}{2}}\notag \\
     &\le \,C\Big(\mathbb{E}\Big[ \int_{K} \int_{t_n}^{t_{n+1}} \sigma^2(u_{\Delta t}(r,x))\,dr \,dx \Big]\Big)^{\frac{1}{2}}\notag \\
     & \le \,C(M)\Big(\mathbb{E}\Big[ \int_{K} \int_{t_n}^{t_{n+1}}|u_{\Delta t}(r,x)|^2\,dr \,dx \Big]\Big)^{\frac{1}{2}}\notag \\
     &\le \,C(M, K)\Big(\underset{0 \le r \le T}{sup}\mathbb{E}\Big[||u_{\Delta t}(r,\cdot)||_{L^\infty(\mathbb{R}^d)}^2\Big]|t_{n+1} - t_n| \Big)^{\frac{1}{2}}\, \le  \, C\sqrt{\Delta t}.
 \end{align}
 To estimate the term $\mathcal{\bar{C}}$, we  apply Cauchy-Schwartz inequality along with It\^{o}-L\'{e}vy isometry to get
 \begin{align}
  \mathcal{\bar{C}} & \le C(K) \mathbb{E} \Big[ \Big( \int_{K} \Big\{ \int_{t_n}^{t_{n+1}} \int_{|z|>0} 
  \eta(u_{\Delta t}(r,x);z) \widetilde{N}(dz,dr)\Big\}^2 \,dx \Big)^\frac{1}{2}\Big] \notag \\
  & \le C(K) \Big\{ \mathbb{E}\Big[ \int_{K}   \Big( \int_{t_n}^{t_{n+1}} \int_{|z|>0} 
  \eta(u_{\Delta t}(r,x);z) \widetilde{N}(dz,dr)\Big)^2 \,dx \Big] \Big\}^\frac{1}{2} \notag \\
  & \le C(K) \Big\{ \mathbb{E}\Big[ \int_{K} \int_{t_n}^{t_{n+1}} \int_{|z|>0} 
  \eta^2(u_{\Delta t}(r,x);z)\,m(dz)\,dr \,dx \Big] \Big\}^\frac{1}{2} \notag \\
  & \le C(K) \Big\{ \mathbb{E}\Big[ \int_{K} \int_{t_n}^{t_{n+1}} \int_{|z|>0} 
  \big( 1+ | u_{\Delta t}(r,x)|^2\big) (1\wedge |z|^2)\,m(dz)\,dr \,dx \Big] \Big\}^\frac{1}{2} \notag \\
  & \le C(K) \Big\{ \mathbb{E}\Big[ \int_{K} \int_{t_n}^{t_{n+1}}
  \big( 1+ | u_{\Delta t}(r,x)|^2\big)\,dr \,dx \Big] \Big\}^\frac{1}{2} \notag \\
  & \le C(K) \Big\{ \Big( 1+ \sup_{0\le r \le T} E\big[ ||u_{\Delta t}(r,\cdot)||_{L^\infty(\R^d)}^2\big]\Big)
  |t_{n+1}-t_n| \Big\}^\frac{1}{2} 
  \le C \sqrt{\Delta t}. \label{estim:c}
 \end{align}
 Combining \eqref{estim: a1+a2}, \eqref{estim: a1}, \eqref{estimate b} and \eqref{estim:c} in
\eqref{estim: 1-time} and using the fact that $\rho(\Delta t)$ $\le \sqrt{\Delta t}$, for sufficiently small  $\Delta t$, we conclude that
 \begin{align}
 \mathbb{E}\Big[\int_{K} \big| u_{\Delta t}(t_{n+1},x)-u_{\Delta t}(t,x)\big|\,dx \Big] \le C \sqrt{\Delta t}, \notag 
 \end{align}
 for any $t\in [t_n, t_{n+1})$ with $n=0,1,\ldots, N-1$.  This essentially finishes the proof.
 \end{proof}
 \subsection{A-priori estimates for viscous solution of \eqref{eq:fractional}}
We move on to establish the average time continuity of regularized viscous solution of \eqref{eq:fractional}. To do so, let $u_\eps$ be the unique weak solution of the viscous problem
\begin{align}\label{eq:viscous}
     & du_\eps(t,x) + [\mathcal{L}_{\theta}(u_\eps(t,\cdot))(x) +\text{div}_x f(u_\eps)]\,dt \notag \\
     &\qquad = \sigma(u_\eps(t,x))\,dW(t) + \int_{|z| > 0}  \eta(u_\eps(t,x); z) \,\tilde{N}(dz,dt) + \eps\Delta u_\eps(t,x)\,dt,
\end{align}
for small $\eps > 0$, with initial condition $u_\eps(0,\cdot) = u_0^\eps(\cdot)$ $\in$ $H^1(\mathbb{R}^d)$ such that $u_0^\eps \rightarrow u_0$ in $L^2(\mathbb{R}^d)$. Note that $u_\eps$ $\in$ $H^1(\mathbb{R}^d)$ and satisfies the following a-priori estimates; see \cite{frac lin}.
\begin{align}
    \underset{0 \le t \le T}{sup} \mathbb{E} \Big[||u_\eps(t)||_{L^2(\mathbb{R}^d)}^2\Big] +  \eps \int_0^T \mathbb{E} \Big[||\nabla u_\eps(s)||_{L^2(\mathbb{R}^d)}^2\Big]ds  + \int_0^T \mathbb{E} \Big[||u_\eps(s)||_{H^\theta(\mathbb{R}^d)}^2\Big]ds  \le C. \label{inq:uniform-1st-viscous}
\end{align}
Moreover, an application of It\^o-L\'evy formula to \eqref{eq:viscous} along with the Burkholder-Davis-Gundy~(BDG in short) inequality gives the following estimations (see, \cite[Section 7.3]{frac lin}, \cite{Saibal-2021}).
\begin{lem} Let the assumptions \ref{A1}, \ref{A3}, \ref{A4}, \ref{A6} and \ref{A8} be true. Let $u_\eps$ be the solution to viscous problem \eqref{eq:viscous}. Then for any $p>0$, 
\begin{align}
 \mathbb{E} \Big[ \underset{0 \le t \le T}{sup}||u_\eps(t)||_{L^2(\mathbb{R}^d)}^p\Big] +  \eps^p \mathbb{E} \Big[\Big(\int_0^T ||\nabla u_\eps(s)||_{L^2(\mathbb{R}^d)}^2\,ds\Big)^p\Big] \le C.  \label{inq:uniform-viscous-l2p}
\end{align}
\end{lem}
For technical reason we need $\Delta u_\eps \in L^2(\Omega \times \mathbb{Q}_T)$. To do so, we regularize $u_\eps$ by convolution. Let $\{\tau_k\}$ be a sequence of mollifier in $\mathbb{R}^d.$ Then $u_\eps^\kappa := u_\eps \ast \tau_\kappa$ is a solution to the problem 
\begin{align}\label{eq:regularize}
    &\partial_t \Big[u_\eps^\kappa - \int_0^t\sigma(u_\eps) \ast \tau_\kappa\, dW(r) -\int_0^t \int_{|z| > 0}  \eta(u_\eps; z) \ast \tau_\kappa\, \tilde{N}(dz,dr) \Big] \notag \\
    & = -\mathcal{L}_{\theta}[u_\eps \ast \tau_\kappa] - div_x f(u_\eps) \ast \tau_\kappa + \eps\Delta u_\eps^\kappa(t,x),  \hspace{2mm}a.e. \hspace{2mm} t > 0, \hspace{2mm} x \in \mathbb{R}^d,
\end{align}
for fixed $\eps > 0.$
\begin{lem}\label{lem:average-time-cont-viscous}
Let the assumptions \ref{A1}-\ref{A8} be true. Let $u_\eps^\kappa(t,x)$ be the solution to the problem \eqref{eq:regularize}. Let $\rho_{\delta_0}(r):= \frac{1}{\delta_0}\rho(r/\delta_0)$ be a standard mollifier on $\mathbb{R}$ with $\text{supp}(\rho) \subset [-1,0]$. Then for any compact subset $K \subset \mathbb{R}^d$, 
\begin{align}\label{eq:time continuity of uek}
 \mathbb{E}\left[\int_0^T\int_0^T \int_K \left|u_\eps^\kappa(t,x)-u_\eps^\kappa(s,x)\right|^2\,\rho_{\delta_0}(t-s)\,dx\,dt\,ds\right] \le C(\eps, \kappa)\delta_0.
\end{align}
\end{lem}
\begin{proof}
Let $K\subset \R^d$ be a compact set. From  \eqref{eq:regularize}, one has by using Jensen's inequality
\begin{align}\label{eq:@}
&\mathbb{E}\Big[\int_K |u_\eps^\kappa(t,x)-u_\eps^\kappa(s,x)|^2\,dx\Big]\notag\\
   & \le  C|t-s|\,\mathbb{E}\Big[ \int_K\int_s^t|\text{div}f(u_\eps(r,x))\ast \tau_\kappa|^2\,dr\,dx\Big] +  C|t-s|\,\mathbb{E}\Big[\int_K\int_s^t|\mathcal{L}_{\theta}[u_\eps^\kappa(r,\cdot)](x)|^2\,dr\,dx\Big]\notag\\
   & \quad +  \,C|t-s|\,\eps^2\,\mathbb{E}\Big[\int_K\int_s^t|\Delta u_\eps^\kappa(r,x)|^2\,dr\,dx\Big] + C\,\mathbb{E}\Big[\int_K\Big|\int_s^t\sigma(u_\eps(r,\cdot))\ast \tau_\kappa\, dW(r)\Big|^2\,dx\Big]\notag\\ 
   & \quad\quad + \, C\,\mathbb{E}\Big[\int_K\Big|\int_s^t\int_{|z| > 0}\eta(u_\eps(r,\cdot); z) \ast \tau_\kappa \tilde{N}(dz,dt)\Big|^2\,dx\Big]\,
   =:\sum_{i =1}^{5}\mathcal{A}_i 
   \end{align}
  Using the property of convolution and \eqref{inq:uniform-1st-viscous}, one can estimate $\mathcal{A}_1$ and $\mathcal{A}_2$ as 
   \begin{align}
      \mathcal{A}_1 \le  \,&C(||f'||_{L^\infty}^2)|t-s|\,\mathbb{E}\Big[ \int_K\int_s^t|\nabla_x u_\eps^\kappa(r,x)|^2\,dr\,dx\Big]\notag \\ \le\, &C(||f'||_{L^\infty}^2)|t-s|\,\mathbb{E}\Big[ \int_K\int_s^t||u_\eps(r)||_{L^2(\R^d)}^2||\nabla\tau_\kappa||_{L^2(\R^d)}^2\,dr\,dx\Big]\notag \\ \le \, &
      C(||f'||_{L^\infty}^2, |K|, \kappa)|t-s|^2\,\Big(\underset{0\le t\le T}{sup}\,\mathbb{E}\Big[ ||u_\eps(t)||_{L^2(\R^d)}^2\Big]\Big)\, \le \, C(|K|, \kappa)|t-s|^2,\notag \\
       \mathcal{A}_2 \le\, & C|t-s| \mathbb{E}\left[ \int_K \int_s^t \left( \|u_\eps^\kappa(r)\|_{L^2(\R^d)}^2 + \|u_\eps^\kappa(r)\|_{W^{2,\infty}(\R^d)}^2\right)\,dr\,dx\right] \notag \\
       & \le C(|K|)|t-s|^2\,\sup_{0\le t\le T}\,\mathbb{E}\Big[||u_\eps(t)||_{L^2(\R^d)}^2 \Big] + C|t-s|\,\mathbb{E}\left[\int_K\int_s^t\left( ||u_\eps||_{L^2(\R^d)}^2||\tau_\kappa||_{L^2(\R^d)}^2 \right)\,dr\,dx \right]\ \notag \\
       &\quad  +C|t-s|\,\mathbb{E}\left[\int_K\int_s^t\left(|| u_\eps||_{L^2(\R^d)}^2||\nabla\tau_\kappa||_{L^2(\R^d)}^2 + || u_\eps||_{L^2(\R^d)}^2||\Delta \tau_\kappa||_{L^2(\R^d)}^2\right)\,dr\,dx \right]\notag\\
       \le \,& C(|K|)|t-s|^2 + C(|K|, \kappa)|t-s|^2\Big(\underset{0\le t\le T}{sup}\,\mathbb{E}\Big[||u_\eps(t)||_{L^2(\R^d)}^2 \Big]\Big) \le C(|K|, \kappa)|t-s|^2\notag.
   \end{align}
   In the above estimations, we have used the fact that,  for any $\Psi \in \mathcal{D}(\mathbb{R}^d)$
   $$|\mathcal{L}_\theta(\Psi)| \le C\Big[||\Psi||_{L^2(\mathbb{R}^d)} + ||\Psi||_{W^{2, \infty}(\mathbb{R}^d)}\Big].$$
  Since for each fixed $\eps>0$, $\Delta u_\eps^\kappa\in L^2(\Omega\times \mathbb{Q}_T) $, we have
   \begin{align}
       \mathcal{A}_3 \le \, \eps^2C|t-s|\, \int_0^T\mathbb{E}\Big[||\Delta u_\eps^\kappa(r)||_{L^2(K)}^2\Big]\,dr  \le C(\eps)|t-s|.\notag
   \end{align}
   Using It\^o isometry, the convolution property together with the assumptions \ref{A4} , one has 
   \begin{align}
       \mathcal{A}_4 \le\, & C\,\mathbb{E}\Big[\int_s^t \|\sigma(u_\eps(r))\|_{L^2(\R^d)}^2\,dr\Big] \le\, C|t-s|\,\Big(\underset{0\le t\le T}{sup}\,\mathbb{E}\Big[||u_\eps(t)||_{L^2(\R^d)}^2\Big]\Big) \le\, C|t-s|.\notag
   \end{align}
  An application of It\^o-L\'evy isometry, properties of convolution along with the assumptions \ref{A6} and \ref{A8} gives the following estimation for $\mathcal{A}_6$.
   \begin{align}
       \mathcal{A}_6 \le &\,C \, \mathbb{E}\Big[\int_K\int_s^t\int_{|z| > 0}|\eta(u_\eps; z) \ast \tau_\kappa|^2 \, m(dz)\, dr\,dx\Big]\notag\\
        \le &\,C \,\mathbb{E}\Big[\int_s^t\int_{|z| > 0}\|\eta(u_\eps(r);z)\|_{L^2(\R^d)}^2 \, m(dz)\, dr\Big]\notag \\
        & \le C{\lambda^*} \mathbb{E}\Big[\int_s^t\int_{|z| > 0}\|u_\eps(r)\|_{L^2(\R^d)}^2 (1\wedge |z|^2)\,m(dz)\,dr\Big] \notag \\
        & \le  C|t-s|\Big(  \underset{0\le t\le T}{sup}\,\mathbb{E}\Big[\|u_\eps(t)\|_{L^2(\R^d)}^2\Big]\Big)\le C|t-s|.\notag
   \end{align}
Using the above estimations in \eqref{eq:@}, we have
\begin{align}
\mathbb{E}\Big[\int_K |u_\eps^\kappa(t,x)-u_\eps^\kappa(s,x)|^2\,dx\Big] \le C( \kappa)|t-s|^2 + C(\eps)|t-s|. \label{inq:average-time-cont-l2-}
\end{align}
Considering the non-negative mollifier $\rho_{\delta_0}(t-s)$ on $\R$, one has,  from \eqref{inq:average-time-cont-l2-}
\begin{align}
    &\int_0^T\int_0^T \mathbb{E}\Big[\int_K |u_\eps^\kappa(t,x)-u_\eps^\kappa(s,x)|^2\,dx\Big]\rho_{\delta_0}(t-s)\,dt\,ds\notag \\ 
     & \le \,C( \kappa) \, \int_0^T\int_0^T |t-s|^2\rho_{\delta_0}(t-s)\,dt\,ds + C(\eps)\, \int_0^T\int_0^T |t-s|\rho_{\delta_0}(t-s)\,dt\,ds\notag \\ 
     & \le  C(\kappa)\delta_0^2 + C(\eps)\delta_0 \, \le C(\eps, \kappa)\delta_0.\notag
\end{align}
This completes the proof.
\end{proof}
\begin{rem}\label{rem:time-cont-degenerate-vis}
 Since $\phi$ is Lipschitz continuous and $\phi(0) =0$, the regularized viscous solution of \eqref{eq:fdegenerate} also satisfies the estimation \eqref{eq:time continuity of uek} of  Lemma \ref{lem:average-time-cont-viscous}.
\end{rem}
\begin{rem}
Using Cauchy-Schwartz inequality along with Lemma \ref{lem:average-time-cont-viscous}, one can also have
 \begin{align}
&\mathbb{E}\Big[\int_0^T\int_0^T\int_K|u_\eps^\kappa(t,x)-u_\eps^\kappa(s,x)|\rho_{\delta_0}(t-s)\,dx\,dt\,ds\Big] \le C(\eps, \kappa)\sqrt{\delta_0}. \label{esti:time-cont-viscous-l1}
\end{align}
\end{rem}
\section{Convergence Analysis for fractional Cauchy Problem}\label{sec:convergence u}

 Our main objective in this section is to establish the convergence of approximate solutions $u_{\Delta t}(t,x)$ to a unique $BV$-entropy solution of \eqref{eq:fractional}.  Note that {\em a-priori} estimate on $u_{\Delta t}(t,x)$ given by Lemma  \ref{lem:l-infinity bound-approximate-solution} only guarantees weak compactness of the family $\{u_{\Delta t}(t,x)\}_{\Delta t>0}$, which is inadequate in view of the nonlinearities in the equation.  Thus, to show the convergence of approximate solutions to a unique entropy solution, we follow the strategy of the classical Kru\^zkov's doubling of variables approach as adapted by Bhauryal et. al. in \cite{frac lin, frac non} for proving the uniqueness of entropy solution. However, our approach requires significant changes in the order of passing the limit with respect to various parameters in the proof to obtain the Kato's inequality. We send the parameter $\Delta t \rightarrow 0$, before sending $\delta_0 \rightarrow 0$, (see Subsection \ref{Convergence Analysis}) which effects all the existing terms of the entropy inequality. Therefore, to handle them, suitable estimations are needed.

\subsection{Entropy Formulation}\label{subsec:entropy-formulation}
To establish the entropy inequality for the approximate solutions $u_{\Delta t}(t,x)$, we define a new time interpolant $\widetilde{u}_{\Delta t}(t,x)$--- an entropy solution to \eqref{eq:operator-S} with initial data $u^{n + \frac{1}{2}}$ i.e.,
\begin{equation*}
   \widetilde{u}_{\Delta t}(t,x) := S(t-t_n)u^{n + \frac{1}{2}} \quad t \in [t_n, t_{n+1}],
\end{equation*}
where $u^{n + \frac{1}{2}}$ is defined in \eqref{defi:n-half}. 
For any non-negative test function $\psi$ $\in$ $C_c^2(\mathbb{Q}_T)$ and for any convex entropy flux pair $(\beta, \zeta)$ we have the following inequality,
\begin{multline}{\label{eq:5.2}}
   \int_{\mathbb{R}^d} \beta(\widetilde{u}_{\Delta t}(t_n,x))\psi(t_n,x)dx  -  \int_{\mathbb{R}^d}  \beta(\widetilde{u}_{\Delta t}(t_{n+1},x))\psi(t_{n+1},x)dx \\
   + \int_{\mathbb{R}^d} \int_{t_n}^{t_{n+1}} \big (\beta(\widetilde{u}_{\Delta t}(r,x)))\partial_t\psi(r,x) + \zeta(\widetilde{u}_{\Delta t}(r,x)). \nabla\psi(r,x) \big )\,dr\,dx \\ - \int_{\mathbb{R}^d} \int_{t_n}^{t_{n+1}} \big[\mathcal{L}_\theta^{\bar{r}}[\widetilde{u}_{\Delta t}(r,\cdot)](x)\psi(r,x)\beta'(\widetilde{u}_{\Delta t}(r,x)) + \beta(\widetilde{u}_{\Delta t}(r,x))\mathcal{L}_{\theta,\bar{r}}[\psi(r, .)](x) \big ]\,dr\,dx \geq 0.
\end{multline}
To get entropy formulation for $u_{\Delta t}(t,x)$, we first consider \eqref{eq:noise} with initial condition $\widetilde{u}_{\Delta t}(t_n,x)$. 
Let $$ v(t) =  R(t, t_n)u^n\quad t\in [t_n, t_{n+1}].$$
Note that $v(t_{n+1}) = u^{n + \frac{1}{2}} = \widetilde{u}_{\Delta t}(t_n,x)$, $\widetilde{u}_{\Delta t}(t_{n+1},x) = u_{\Delta t}(t_{n+1},x)$ and $u_{\Delta t}(t,x) = v(t)$, for t $\in$ $(t_n,t_{n+1}).$ Moreover, $u_{\Delta t}(t,x) =  v(t)$  satisfies the equation
\begin{equation*}
       \begin{cases}
         \displaystyle  dv(t,x) = \sigma(v(t,x))\,dW(t) + \int_{|z| > 0} \eta(v(t,x); z)\,\widetilde{N}(dz,dt), \hspace{2mm} t \in (t_n, t_{n+1}]\,, \\
          v(t = t_n) = u^n\,\cdot
         \end{cases}
\end{equation*}
We apply It\^o-L\'evy formula on $\beta(v(t))$ and then multiply the equation by $\psi(t_n, x)$. After integrating the resulting equation with respect to $x$, we get
\begin{align}\label{eq:u1}
  &\int_{\mathbb{R}^d} \beta(v(t_{n+1},x))\psi(t_n, x)\,dx - \int_{\mathbb{R}^d} \beta(v(t_n,x))\psi(t_n, x)\,dx\notag  \\= &\int_{\mathbb{R}^d} \beta(\widetilde{u}_{\Delta t}(t_n,x))\psi(t_n, x)\,dx - \int_{\mathbb{R}^d} \beta(u_{\Delta t}(t_n,x))\psi(t_n, x)\,dx  \notag \\ = 
  &\int_{\mathbb{R}^d} \int_{t_n}^{t_{n+1}}\sigma(u_{\Delta t}(r,x))\beta'(u_{\Delta t}(r,x))\psi(t_n,x)\,dW(r)\,dx \notag  \\ 
  & \quad + \frac{1}{2}\int_{\mathbb{R}^d} \int_{t_n}^{t_{n+1}} \sigma^2(u_{\Delta t}(r,x))\beta''(u_{\Delta t}(r,x))\psi(t_n,x)\,dr\,dx \notag \\
   &  \qquad + \int_{\mathbb{R}^d} \int_{t_n}^{t_{n+1}} \int_{|z| > 0}\Big(\beta \big(u_{\Delta t}(r,x) + \eta(u_{\Delta t}(r,x);z) \big) - \beta(u_{\Delta t}(r,x)) \Big) \psi(t_n,x) \,\widetilde{N}(dz,dr)\,dx\notag
  \\
   & \qquad\quad + \int_{\mathbb{R}^d} \int_{t_n}^{t_{n+1}} \int_{|z| > 0} \Big(\beta \big(u_{\Delta t}(r,x) + \eta(u_{\Delta t}(r,x);z) \big) - \beta(u_{\Delta t}(r,x)) \notag\\
   & \hspace{3cm} -\eta(u_{\Delta t}(r,x);z) \beta'(u_{\Delta t}(r,x))\Big)  \times \psi(t_n,x)\,m(dz)\,dr\,dx\,\cdot
\end{align}
Subtracting \eqref{eq:u1} from \eqref{eq:5.2} and then taking expectation in the resulting inequality, we obtain
\begin{align}\label{eq:u2}
   &\mathbb{E} \Big[ \int_{\mathbb{R}^d} \mathds{1}_{B} \beta(u_{\Delta t}(t_n,x))\psi(t_n, x)\,dx \Big] - \mathbb{E} \Big[ \int_{\mathbb{R}^d}  \mathds{1}_{B}  \beta(\widetilde{u}_{\Delta t}(t_{n+1},x))\psi(t_{n+1},x)\,dx \Big]\notag\\
    &+ \mathbb{E} \Big[  \int_{\mathbb{R}^d} \int_{t_n}^{t_{n+1}}  \mathds{1}_{B} \big (\beta(\widetilde{u}_{\Delta t}(r,x))\partial_t\psi(r,x) + \zeta(\widetilde{u}_{\Delta t}(r,x)). \nabla\psi(r,x) \big )\,dr\,dx \Big] \notag \\  &-\mathbb{E} \Big[ \int_{\mathbb{R}^d} \int_{t_n}^{t_{n+1}}  \mathds{1}_{B} \big[\mathcal{L}_\theta^{\bar{r}}[\widetilde{u}_{\Delta t}(r,\cdot)](x)\psi(r,x)\beta'(\widetilde{u}_{\Delta t}(r,x)) + \beta(\widetilde{u}_{\Delta t}(r,x))\mathcal{L}_{\theta,\bar{r}}[\psi(r, .)](x) \big ]\,dr\,dx \Big]  \notag \\   &+\mathbb{E} \Big[  \int_{\mathbb{R}^d} \int_{t_n}^{t_{n+1}}  \mathds{1}_{B} \sigma(u_{\Delta t}(r,x))\beta'(u_{\Delta t}(r,x))\psi(t_n,x)\,dW(r)\,dx \Big] \notag \\  &+\mathbb{E} \Big[  \frac{1}{2}\int_{\mathbb{R}^d} \int_{t_n}^{t_{n+1}}  \mathds{1}_{B} \sigma^2(u_{\Delta t}(r,x))\beta''(u_{\Delta t}(r,x))\psi(t_n,x)\,dr\,dx \Big] \notag \\ &+ \mathbb{E} \Big[ \int_{\mathbb{R}^d} \int_{t_n}^{t_{n+1}} \int_{|z| > 0}  \mathds{1}_{B}\Big(\beta \big(u_{\Delta t}(r,x) + \eta(u_{\Delta t}(r,x);z) \big) - \beta(u_{\Delta t}(r,x)) \Big) \psi(t_n,x) \,\widetilde{N}(dz,dr)\,dx \Big]\notag\\ & +\mathbb{E} \Big[  \int_{\mathbb{R}^d} \int_{t_n}^{t_{n+1}} \int_{|z| > 0}  \mathds{1}_{B} \Big(\beta \big(u_{\Delta t}(r,x) + \eta(u_{\Delta t}(r,x);z) \big) - \beta(u_{\Delta t}(r,x)) \notag\\ & \hspace{5cm} -\eta(u_{\Delta t}(r,x);z) \beta'(u_{\Delta t}(r,x))\Big)  \times \psi(t_n,x)\,m(dz)\,dr\,dx \Big] \geq 0,
\end{align}
for any $\mathbb{P}$-measurable set $B$.
\vspace{0.1cm}

We introduce a class of entropy function which  will play a crucial role in the subsequent calculations. Let $\beta$ $: \mathbb{R} \goto \mathbb{R}$  be a $C^\infty$ function as defined in the proof of Lemma \ref{lem:BV}.
For any $\xi$, define $\beta_{\xi}$ $: \mathbb{R} \goto \mathbb{R}$ by $\beta_{\xi}:= \xi\beta(\frac{r}{\xi}).$ For $\beta$ = $\beta_\xi$,
we define 
\begin{align}
   & F_k^{\beta}(a,b):= \int_b^a \beta'(r-b)f'_k(r)dr, \quad F_k(a,b):= sign(a-b)(f_k(a) - f_k(b))\quad (1\le k\le d), \notag \\
    & \hspace{3cm} F(a,b):= (F_1(a,b), F_2(a,b), ..., F_d(a,b)). \notag
\end{align}
Let $\rho$ and $\varrho$ be the  non-negative, standard mollifiers on $\mathbb{R}$ and $\mathbb{R}^d$ respectively such that $supp(\rho)$ $\subset$ $[-1,0]$ and $supp (\varrho) = $ $\bar{B}_1(0)$, where $\bar{B}_1(0)$ is the closed unit ball. We define
\[\rho_{\delta_0}(r):= \frac{1}{\delta_0}\rho(\frac{r}{\delta_0}), \hspace{3mm} \varrho_\delta(x):= \frac{1}{\delta^d}\varrho(\frac{x}{\delta}),\]
where $\delta$, $\delta_0$ are two positive constants. Given a non-negative test function $\psi \in C_c^{1,2}([0, \infty) \times \mathbb{R}^d)$ and two positive constant $\delta$, $\delta_0$, define  \[\varphi_{\delta_0, \delta}(t,x,s,y):= \rho_{\delta_0}(t-s)\varrho_\delta(x-y)\psi(t,x).\]
Clearly, $\rho_{\delta 0}(t-s) \neq 0$ only if $s - \delta_0 \le t \le s$ and hence $\varphi_{\delta_0, \delta}(t,x,s,y)$  $= 0$ outside  $s - \delta_0 \le t \le s$.  Let $J$ be the standard symmetric non-negative mollifier on $\mathbb{R}$ with support in $[-1,1]$. For $l>0$, define $J_l(r): = \frac{1}{l}J(\frac{r}{l})$.
\vspace{0.15cm}

  We apply It\^{o}-L\'{e}vy formula to \eqref{eq:regularize}, multiply by $J_l(u_{\Delta t}(t,x) - k)$ and then integrate with respect to $t,x$ and $k$ and taking expectation in the resulting expression, we obtain 
 \begin{align}\label{entropinquality0}
         &0 \le \mathbb{E} \Big[\int_{\mathbb{Q}_T} \int_{\mathbb{R}^d} \int_{\mathbb{R}} \beta(u_\eps^\kappa(0,y)) - k)\varphi_{\delta_0, \delta}(t,x,0,y)J_l(u_{\Delta t}(t,x) - k)\,dk\,dy\,dx\,dt \Big]\notag \\
         &+ \mathbb{E} \Big[\int_{\mathbb{Q}_T^2} \int_{\mathbb{R}} \beta(u_\eps^\kappa(s,y)) - k)\partial_s\varphi_{\delta_0, \delta}J_l(u_{\Delta t}(t,x) - k)\,dk\,ds \,dt\,dy\,dx \Big]\notag \\
         &+\mathbb{E} \Big [\int_{\mathbb{Q}_T^2}\int_\mathbb{R}F^\beta(u_\eps^\kappa(s,y), k). \nabla_y\varphi_{\delta_0, \delta}J_l(u_{\Delta t}(t,x) - k)\,dk\,ds\,dt\,dx\,dy \Big]\notag \\
         &-\mathbb{E} \Big[\int_{\mathbb{Q}_T^2} \int_{\mathbb{R}}\mathcal{L}_\theta^{\bar{r}}[u_\eps^\kappa(s,\cdot)](y)\beta'(u_\eps^\kappa(s,y) - k) \varphi_{\delta_0, \delta}J_l(u_{\Delta t}(t,x) - k)\,dk\,ds\,dt\,dy\,dx \Big]\notag \\
         &- \mathbb{E} \Big[\int_{\mathbb{Q}_T^2} \int_{\mathbb{R}}\beta(u_\eps^\kappa(s,y) - k) \mathcal{L}_{\theta,\bar{r}}[\varphi_{\delta_0, \delta}(t,x,s,\cdot)](y)J_l(u_{\Delta t}(t,x) - k)\,dk\,ds\,dt\,dy\,dx \Big]\notag \\
         &+ \mathbb{E} \Big[\int_{\mathbb{Q}_T^2} \int_{\mathbb{R}}
         (\sigma(u_\eps(s,y)) \ast \tau_\kappa)\beta'(u_\eps^\kappa(s,y)- k)\varphi_{\delta_0, \delta}J_l(u_{\Delta t}(t,x) - k)\,dk\,dt\,dW(s)\,dy\,dx \Big]\notag \\
         &+ \frac{1}{2}\mathbb{E} \Big[\int_{\mathbb{Q}_T^2} \int_{\mathbb{R}}
         (\sigma(u_\eps(s,y)) \ast \tau_\kappa)^2\beta''(u_\eps^\kappa(s,y)- k)\varphi_{\delta_0, \delta}J_l(u_{\Delta t}(t,x) - k)\,dk\,ds\,dt\,dy\,dx \Big] \notag \\
         &+ \mathbb{E} \Big[\int_{\mathbb{Q}_T^2} \int_{\mathbb{R}}
         \int_{|z| > 0} \int_0^1 (\eta(u_\eps(s,y);z) \ast \tau_\kappa)\beta'\left(u_\eps^\kappa(s,y) + \lambda(\eta(u_\eps(s,y);z)\ast \tau_\kappa) - k\right)  \notag \\ &\hspace{6cm} \times \varphi_{\delta_0, \delta}J_l(u_{\Delta t}(t,x) - k)d\lambda \,dk\, \tilde{N}(dz, ds)\,dt\, dx\,dy \Big]\notag \\
         &+ \mathbb{E} \Big[\int_{\mathbb{Q}_T^2} \int_{\mathbb{R}}
         \int_{|z| > 0} \int_0^1 (1 -\lambda) (\eta(u_\eps(s,y);z) \ast \tau_\kappa)^2\beta^{\prime\prime}\left(u_\eps^\kappa(s,y) + \lambda(\eta(u_\eps(s,y);z)\ast \tau_\kappa) - k\right)  \notag \\
          &\hspace{6cm} \times \varphi_{\delta_0, \delta}J_l(u_{\Delta t}(t,x) - k)\,d\lambda\, dk\,m(dz)\,ds\,dt\,dx\,dy \Big] \notag \\
         &- \eps \mathbb{E} \Big[\int_{\mathbb{Q}_T^2} \int_{\mathbb{R}}\beta'(u_\eps^\kappa(s,y) - k)\nabla_y(u_\eps^\kappa(s,y))\cdot  \nabla_y\varphi_{\delta_0, \delta}J_l(u_{\Delta t}(t,x) - k)\,dk\,ds\,dt\,dx\,dy \Big] \notag \\
         &- \eps \mathbb{E} \Big[\int_{\mathbb{Q}_T^2} \int_{\mathbb{R}}\beta''(u_\eps^\kappa(s,y) - k)|\nabla_y u_\eps^\kappa(s,y)|^2\varphi_{\delta_0, \delta}J_l(u_{\Delta t}(t,x) - k)\,dk\,ds\,dt\,dx\,dy \Big]=:\sum_{i=1}^{11} \mathcal{I}_i\,. 
 \end{align}
 
  We multiply the entropy inequality \eqref{eq:u2}, against the entropy pair $(\beta(\cdot-k),F^\beta(\cdot,k))$ and the test function $\varphi_{\delta_0, \delta}$, by $J_l(u_\eps^\kappa(s,y) - k)$ and integrate with respect to $s,y$ and  $k$ to have, after taking expectation
  \begin{align}\label{EntropyIQUE2}
 &0 \le \mathbb{E} \Big [\int_{\mathbb{Q}_T}\int_{\mathbb{R}^d} \int_\mathbb{R} \beta(u_{\Delta t}(0,x) - k)\varphi_{\delta_0, \delta}(0,x,s,y)J_l(u_\eps^\kappa(s,y) - k)\,dk\,dx\,ds\,dy \Big]\notag \\
  &+ \mathbb{E} \Big [\int_{\mathbb{Q}_T^2}\int_\mathbb{R} \big (\beta(u_{\Delta t}(t,x) -k)\partial_t\varphi_{\delta_0, \delta}J_l(u_\eps^\kappa(s,y) - k)\,dk\,ds\,dt\,dx\,dy \Big] \notag \\
  &+ \mathbb{E} \Big [\int_{\mathbb{Q}_T^2}\int_\mathbb{R}F^\beta(u_{\Delta t}(t,x), k)\cdot \nabla_x\varphi_{\delta_0, \delta}J_l(u_\eps^\kappa(s,y) - k)\,dk\,dt\,ds\,dx\,dy \Big]\notag \\
  &-\mathbb{E} \Big [\int_{\mathbb{Q}_T^2}\int_\mathbb{R} \mathcal{L}_\theta^{\bar{r}}[u_{\Delta t}(t,\cdot)](x)\beta'(u_{\Delta t}(t,x) - k)\varphi_{\delta_0, \delta}J_l(u_\eps^\kappa(s,y) - k)\,dk\,ds\,dt\,dx\,dy \Big]\notag\\
  &- \mathbb{E} \Big [\int_{\mathbb{Q}_T^2}\int_\mathbb{R}  \beta(u_{\Delta t}(t,x) -k)\mathcal{L}_{\theta,\bar{r}}[\varphi_{\delta_0, \delta}(t,\cdot,s,y)](x) J_l(u_\eps^\kappa(s,y) - k)\,dk\,ds\,dt\,dx\,dy \Big] \notag \\
  &+ \mathbb{E} \Big [\int_{\mathbb{Q}_T^2}\int_\mathbb{R}  \sigma(u_{\Delta t}(t,x))\beta'(u_{\Delta t}(t,x) -k )\varphi_{\delta_0, \delta}J_l(u_\eps^\kappa(s,y) - k)\,dk\,dW(t)\,ds\,dx\,dy \Big]\notag \\  & + \frac{1}{2}\mathbb{E} \Big [\int_{\mathbb{Q}_T^2}\int_\mathbb{R} \sigma^2(u_{\Delta t}(t,x))\beta''(u_{\Delta t}(t,x) - k)\varphi_{\delta_0, \delta}J_l(u_\eps^\kappa(s,y) - k)\,dk\,ds\,dt\,dy\,dx \Big] \notag \\
  &+ \mathbb{E} \Big [\int_{\mathbb{Q}_T^2}\int_{|z| > 0}\int_\mathbb{R} \int_0^1  \eta(u_{\Delta t}(t,x);z)\beta' \big(u_{\Delta t}(t,x) + \lambda\eta(u_{\Delta t}(t,x);z) - k \big)\notag \\ &\hspace{6cm} \times  \varphi_{\delta_0, \delta}(t,x,s,y)J_l(u_\eps^\kappa(s,y) - k) \,d\lambda\, \,dk\,\tilde{N}(dz,dt)\,ds\,dx\,dy \Big] \notag \\
  &+ \mathbb{E} \Big [\int_{\mathbb{Q}_T^2}\int_\mathbb{R}\int_{|z| > 0} \int_0^1  (1-\lambda)\eta^2(u_{\Delta t}(r,x);z)\beta''(u_{\Delta t}(t,x) + \lambda\eta(u_{\Delta t}(r,x);z) - k) \notag\\ &\hspace{6cm} \times \varphi_{\delta_0, \delta}(t,x,s,y)J_l(u_\eps^\kappa(s,y) - k)\,  d\lambda\, m(dz)dt\,dx\,dk\,ds\,dy \Big] \notag\\
  &+ \mathbb{E}\Big[ \int_{\mathbb{Q}_T}\int_\mathbb{R} \sum_{n=0}^{N-1} \int_{\mathbb{R}^d} \int_{t_n}^{t_{n+1}}\big (\beta(\widetilde{u}_{\Delta t}(t,x) - k ) - \beta(u_{\Delta t}(t_{n+1},x )- k)) \big)\notag \\ &\hspace{6cm} \times \partial_t\varphi_{\delta_0, \delta}(t,x,s,y)J_l(u_\eps^\kappa(s,y) - k)\,dt\,dx\,dk\,ds\,dy \Big]\notag \\
  &+ \mathbb{E}\Big[ \int_{\mathbb{Q}_T}\int_\mathbb{R} \sum_{n=0}^{N-1} \int_{\mathbb{R}^d} \int_{t_n}^{t_{n+1}}\big (\beta(u_{\Delta t}(t_{n+1},x) - k) - \beta(u_{\Delta t}(t,x) - k ) \big) \notag\\ & \hspace{6cm} \times \partial_t\varphi_{\delta_0, \delta}(t,x,s,y)J_l(u_\eps^\kappa(s,y) - k)\,dt\,dx\,dk\,ds\,dy \Big]\notag\\ 
  &+ \mathbb{E}\Big[ \int_{\mathbb{Q}_T}\int_\mathbb{R} \sum_{n=0}^{N-1} \int_{\mathbb{R}^d} \int_{t_n}^{t_{n+1}} \big(F^\beta(\widetilde{u}_{\Delta t}(t,x), k) - F^\beta(u_{\Delta t}(t_{n+1},x), k) \big). \nabla_x \varphi_{\delta_0, \delta}(t,x,s,y)\notag \\ 
  & \hspace{6cm} \times J_l(u_\eps^\kappa(s,y) - k)\,dt\,dx\,dk\,ds\,dy \Big]\notag \\
  &+  \mathbb{E}\Big[\int_{\mathbb{Q}_T}\int_\mathbb{R}\sum_{n=0}^{N-1} \int_{\mathbb{R}^d} \int_{t_n}^{t_{n+1}} \big(F^\beta(u_{\Delta t}(t_{n+1},y), k) - F^\beta(u_{\Delta t}(t,x), k) \big)\cdot \nabla_x \varphi_{\delta_0, \delta}(t,x,s,y) \notag\\ & \hspace{6cm} \times J_l(u_\eps^\kappa(s,y) - k)\,dt\,dx\,dk\,ds\,dy \Big] \notag\\
  &-  \mathbb{E}\Big[\int_{\mathbb{Q}_T}\int_\mathbb{R}\sum_{n=0}^{N-1} \int_{\mathbb{R}^d} \int_{t_n}^{t_{n+1}}  \big[\mathcal{L}_\theta^{\bar{r}}[\widetilde{u}_{\Delta t}(t,\cdot)](x)\beta'(\widetilde{u}_{\Delta t}(t,x) - k)\notag \\
   & \hspace{2cm} -\mathcal{L}_\theta^{\bar{r}}[u_{\Delta t}(t_{n+1},\cdot)](y)\beta'(u_{\Delta t}(t_{n+1},x) -k) ] \varphi_{\delta_0, \delta}(t,x,s,y) J_l(u_\eps^\kappa(s,y) - k)\,dt\,dx\,dk\,ds\,dy \Big] \notag \\
  &-  \mathbb{E}\Big[\int_{\mathbb{Q}_T}\int_\mathbb{R}\sum_{n=0}^{N-1} \int_{\mathbb{R}^d} \int_{t_n}^{t_{n+1}}  \big[\mathcal{L}_\theta^{\bar{r}}[u_{\Delta t}(t_{n+1},\cdot)](x)\beta'(u_{\Delta t}(t_{n+1},x) - k)\notag \\
   & \hspace{2cm} -\mathcal{L}_\theta^{\bar{r}}[u_{\Delta t}(t,\cdot)](y)\beta'(u_{\Delta t}(t,x) - k) ] \varphi_{\delta_0, \delta}(t,x,s,y) J_l(u_\eps^\kappa(s,y) - k)\,dt\,dx\,dk\,ds\,dy \Big]
  \notag\\ 
  &- \mathbb{E}\Big[\int_{\mathbb{Q}_T}\int_\mathbb{R}\sum_{n=0}^{N-1} \int_{\mathbb{R}^d} \int_{t_n}^{t_{n+1}} \big(\beta(\widetilde{u}_{\Delta t}(t,x) -k) - \beta(u_{\Delta t}(t_{n+1},x)-k) \big) \mathcal{L}_{\theta,\bar{r}}[\varphi_{\delta_0, \delta}(t,\cdot,s,y)](x)\notag\notag \\& \hspace{6cm} \times J_l(u_\eps^\kappa(s,y) - k)\,dt\,dx\,dk\,ds\,dy \Big]\notag\\ 
  &- \mathbb{E}\Big[\int_{\mathbb{Q}_T}\int_\mathbb{R} \sum_{n=0}^{N-1} \int_{\mathbb{R}^d} \int_{t_n}^{t_{n+1}} \big(\beta(u_{\Delta t}(t_{n+1},x)-k) - \beta(u_{\Delta t}(t,x)-k) \big) \mathcal{L}_{\theta,\bar{r}}[\varphi_{\delta_0, \delta}(t,\cdot,s,y)](x)\notag \\& \hspace{6cm} \times J_l(u_\eps^\kappa(s,y) - k)\,dt\,dx\,dk\,ds\,dy \Big]\notag \\
  &+ \mathbb{E}\Big[\int_{\mathbb{Q}_T}\int_\mathbb{R}\sum_{n=0}^{N-1}  \int_{\mathbb{R}^d} \int_{t_n}^{t_{n+1}}  \sigma(u_{\Delta t}(t,x))\beta'(u_{\Delta t}(t,x) -k)\{\varphi_{\delta_0, \delta}(t_n,x,s,y) -\varphi_{\delta_0,\delta}(t,x,s,y)\}\notag \\ & \hspace{6cm} \times J_l(u_\eps^\kappa(s,y) - k)\,dW(t)\,dx\,dk\,ds\,dy\Big]\notag \\ 
&+ \frac{1}{2} \mathbb{E}\Big[\int_{\mathbb{Q}_T}\int_\mathbb{R}  \sum_{n=0}^{N-1} \int_{\mathbb{R}^d} \int_{t_n}^{t_{n+1}}   \sigma^2(u_{\Delta t}(t,x))\beta''(u_{\Delta t}(t,x) -k)\{\varphi_{\delta_0, \delta}(t_n,x,s,y) -\varphi_{\delta_0,\delta}(t,x,s,y)\}\notag \\ & \hspace{6cm} \times J_l(u_\eps^\kappa(s,y) - k)\,dt\,dx\,dk\,ds\,dy \Big] \notag \\
  &+ \mathbb{E}\Big[\int_{\mathbb{Q}_T}\int_\mathbb{R}   \sum_{n=0}^{N-1} \int_{\mathbb{R}^d} \int_{t_n}^{t_{n+1}} \int_{|z| > 0} \int_0^1 \eta(u_{\Delta t}(t,x);z) \beta' \big(u_{\Delta t}(t,x) + \lambda\eta(u_{\Delta t}(t,x);z) - k)\notag\\ & \hspace{3.5cm} \times \{\varphi_{\delta_0, \delta}(t_n,x,s,y) -\varphi_{\delta_0,\delta}(t,x,s,y)\}J_l(u_\eps^\kappa(s,y) - k)d\lambda \,\tilde{N}(dz,dt)\,dx\,dk\,ds\,dy \Big]  \notag\\
 &+ \mathbb{E}\Big[\int_{\mathbb{Q}_T}\int_\mathbb{R}  \sum_{n=0}^{N-1}  \int_{\mathbb{R}^d} \int_{t_n}^{t_{n+1}} \int_{|z| > 0}  \int_0^1 (1 - \lambda)\eta^2(u_{\Delta t}(t,x);z)\beta''\big(u_{\Delta t}(t,x) + \lambda\,\eta(u_{\Delta t}(t,x);z) -k \big) \notag\\ & \hspace{3.5cm} \times\{\varphi_{\delta_0, \delta}(t_n,x,s,y) -\varphi_{\delta_0,\delta}(t,x,s,y)\}J_l(u_\eps^\kappa(s,y) - k)d\lambda m(dz)\,dt\,dx\,dk\,ds\,dy \Big] \notag\\
&=:\sum_{i=1}^{21}\mathcal{J}_i\,.
  \end{align}
  Our aim is to show the convergence of approximate solutions to an entropy solution $u(t,x)$. For this purpose, we need to estimate each of the terms in \eqref{entropinquality0} and \eqref{EntropyIQUE2} in terms of  the small parameters $\xi, \delta_0, \delta, \kappa, \eps, l$ and $\Delta t$. In the forthcoming computations, we will frequently use Lemma \ref{lem:average-time-cont-viscous} and the fact that $\displaystyle\int_0^T\rho_{\delta_0}(t-s)ds \le 1 $ and the equality holds for $t \le T - \delta_0.$ 
  \vspace{0.2cm}
  
   Let us consider the terms $\mathcal{I}_1$ and $\mathcal{J}_1$. Since supp($\rho_{\delta_0}$) $\subset$ $[-\delta_0, 0]$, then $\mathcal{I}_1$ = 0. Now we estimate the term $\mathcal{J}_1$. We have
\begin{align}
    \mathcal{J}_1 &\,= \mathbb{E} \Big [\int_{\mathbb{Q}_T}\int_{\mathbb{R}^d} \int_\mathbb{R} \beta(u_{\Delta t}(0,x) - k)\varphi_{\delta_0, \delta}(0,x,s,y)J_l(u_\eps^\kappa(s,y) - k)\,dk\,ds\,dx\,dy \Big]\notag \\
    &\,= \mathbb{E} \Big[\int_{\mathbb{Q}_T} \int_{\mathbb{R}^d} \int_{\mathbb{R}} \Big(\beta(u_{\Delta t}(0,x) - u_\eps^\kappa(s,y) + k) -  \beta(u_{\Delta t}(0,x) - u_\eps^\kappa(0,y)  + k) \Big)\notag \\ &\hspace{6cm} \times \rho_{\delta_0}(-s)\varrho_\delta(x-y) \psi(0,x) J_l( k)\,dk\,ds\,dy\,dx \Big]\notag \\
    &+ \mathbb{E} \Big[\int_{\R^d} \int_{\mathbb{R}^d} \int_{\mathbb{R}} \big ( \beta(u_{\Delta t}(0,x) - u_\eps^\kappa(0,y)+ k) - \beta(u_{\Delta t}(0,x) - u_\eps^\kappa(0,y) \big)\notag\\ &\hspace{6cm} \times \varrho_\delta(x-y)\psi(0,x) J_l( k)\,dk\,dy\,dx \Big]\notag\\
    &+ \mathbb{E} \Big[\int_{\mathbb{R}^d} \int_{\mathbb{R}^d}\beta(u_{\Delta t}(0,x)-u_\eps^\kappa(0,y)) \varrho_\delta(x-y)\psi(0,x)dydx \Big]\notag \\
    &\le\,\mathbb{E} \Big[\int_{\mathbb{Q}_T} \int_{\mathbb{R}^d}  |u_\eps^\kappa(s,y)-u_\eps^\kappa(0,y)|\rho_{\delta_0}(-s)\varrho_\delta(x-y) \psi(0,x)\,ds\,dy\,dx \Big]\notag\\
    &\quad + \mathbb{E} \Big[\int_{\R^d} \int_{\mathbb{R}^d} \int_{\mathbb{R}}k  \varrho_\delta(x-y)\psi(0,x) J_l( k)\,dk\,dy\,dx \Big]\notag \\
    &\qquad + \mathbb{E} \Big[\int_{\mathbb{R}^d} \int_{\mathbb{R}^d}\beta(u_{\Delta t}(0,x)-u_\eps^\kappa(0,y))\varrho_\delta(x-y)\psi(0,x)\,dy\,dx \Big] \notag \\
    &=: \mathcal{J}_{1,1} + Cl +  \mathbb{E} \Big[ \int_{\mathbb{R}^d } \int_{\mathbb{R}^d }\beta(u_{\Delta t}(0,x)-u_\eps^\kappa(0,y)) \varrho_\delta(x-y)\psi(0,x)\,dy\,dx \Big]\,. \notag 
  \end{align}  
    Thanks to Cauchy-Schwartz inequality, compact support of $\psi(t,x)$ and \eqref{inq:average-time-cont-l2-}, we estimate $\mathcal{J}_{1,1}$ as follows.
    \begin{align*}
   \mathcal{J}_{1,1} &\le \|\psi(0,\cdot)\|_{L^\infty} \mathbb{E}\left[ \int_{K_y}\int_0^T  |u_\eps^\kappa(s,y)-u_\eps^\kappa(0,y)|\rho_{\delta_0}(-s)\,ds\,dy \right] \quad(\text{for some compact set $K_y$})\\
   & \le C(\psi) \left( \int_0^T \mathbb{E}\left[ \int_{K_y}   |u_\eps^\kappa(s,y)-u_\eps^\kappa(0,y)|^2\,dy \right] \rho_{\delta_0}(-s)\,ds\right)^\frac{1}{2}\notag \\
   & \le C(\psi) \left( \int_0^T C(\eps, \kappa)\left(|s|^2 + |s|\right)  \rho_{\delta_0}(-s)\,ds\right)^\frac{1}{2}\le C(\eps, \kappa) \sqrt{\delta_0}\,.
    \end{align*}
Combining these estimates and recalling the fact that  $u_{\Delta t}(0,x) = u_0(x)$, we have the following result.
\begin{lem}\label{lem1}
\begin{align}
    \mathcal{I}_1 + \mathcal{J}_1 \le \displaystyle\mathbb{E} \Big[ \int_{\mathbb{R}^d } \int_{\mathbb{R}^d }\beta(u_0(x)-u_\eps^\kappa(0,y)) \varrho_\delta(x-y)\psi(0,x)\,dy\,dx \Big] + C(\eps, \kappa) \sqrt{\delta_0} + Cl.\notag
\end{align}
\end{lem}
Now we focus on the term $(\mathcal{I}_2 + \mathcal{J}_2).$ Since $\beta$ and $J_l$ are even functions, we can observe that,
\begin{align}
    \mathcal{I}_2 + \mathcal{J}_2 &= \,\mathbb{E} \Big[\int_{\mathbb{Q}_T^2}\int_\mathbb{R}\beta(u_{\Delta t}(t,x) -k)\partial_t\psi(t,x)\rho_{\delta_0}(t-s)\varrho_{\delta}(x-y)J_l(u_\eps^\kappa(s,y) - k)\,dk\,ds\,dt\,dx\,dy  \Big]\notag\\
    &= \,\mathbb{E} \Big[\int_{\mathbb{Q}_T^2}\int_\mathbb{R} \Big(\beta(u_{\Delta t}(t,x) - u_\eps^\kappa(s,y)  + k) - \beta(u_{\Delta t}(t,x) - u_\eps^\kappa(t,y) + k)\Big)\notag \\&\hspace{4cm}\times \partial_t\psi(t,x)\rho_{\delta_0}(t-s)\varrho_{\delta}(x-y)J_l(k)\,dk\,ds\,dt\,dx\,dy \Big]\notag\\
    &+\mathbb{E}\Big[\int_{\mathbb{Q}_T}\int_{\mathbb{R}^d}\int_\mathbb{R} \beta(u_{\Delta t}(t,x) - u_\eps^\kappa(t,y) +k)\partial_t\psi(t,x)\notag \\& \hspace{4cm}\times\Big(\int_0^T\rho_{\delta_0} (t-s)ds -1\Big)\varrho_{\delta}(x-y)J_l(k)\,dk\,dt\,dx\,dy  \Big]\notag\\
    &+ \mathbb{E}\Big[\int_{\mathbb{Q}_T}\int_{\mathbb{R}^d}\int_\mathbb{R}\Big(\beta(u_{\Delta t}(t,x) - u_\eps^\kappa(t,y) + k )- \beta(u_{\Delta t}(t,x) - u_\eps^\kappa(t,y)\Big)\notag \\&\hspace{4cm}\times\partial_t\psi(t,x)\varrho_{\delta}(x-y)J_l(k)\,dk\,dt\,dx\,dy  \Big]\notag\\
    &+ \mathbb{E}\Big[\int_{\mathbb{Q}_T}\int_{\mathbb{R}^d}\beta(u_{\Delta t}(t,x) - u_\eps^\kappa(t,y))\partial_t\psi(t,x)\varrho_{\delta}(x-y)\,dt\,dx\,dy  \Big]\notag\\
    &\,\le \mathbb{E} \Big[\int_{\mathbb{Q}_T^2}| u_\eps^\kappa(t,y) - u_\eps^\kappa(s,y)|\,\rho_{\delta_0}(t-s)\, |\partial_t\psi(t,x)|\varrho_{\delta}(x-y)\,ds\,dt\,dx\,dy \Big]\notag\\
    &+\mathbb{E}\Big[\int_{t= T-\delta_0}^{T}\int_{\mathbb{R}^d}\int_{\mathbb{R}^d}\int_\mathbb{R} |u_{\Delta t}(t,x) - u_\eps^\kappa(t,y) +k|\,|\partial_t\psi(t,x)| \varrho_{\delta}(x-y)J_l(k)\,dk\,dt\,dx\,dy  \Big]\notag\\
    &+ \mathbb{E}\Big[\int_{\mathbb{Q}_T}\int_{\mathbb{R}^d}\int_\mathbb{R}|k|\,|\partial_t\psi(t,x)|\varrho_{\delta}(x-y)J_l(k)\,dk\,dt\,dx\,dy  \Big]\notag\\
    &+ \mathbb{E}\Big[\int_{\mathbb{Q}_T}\int_{\mathbb{R}^d}\beta(u_{\Delta t}(t,x) - u_\eps^\kappa(t,y))\partial_t\psi(t,x)\varrho_{\delta}(x-y)\,dt\,dx\,dy  \Big]\notag\\
    &=:\mathcal{M}_{2,1}+ \mathcal{M}_{2,2} + Cl + \mathbb{E}\Big[\int_{\mathbb{Q}_T}\int_{\mathbb{R}^d}\beta(u_{\Delta t}(t,x) - u_\eps^\kappa(t,y))\partial_t\psi(t,x)\varrho_{\delta}(x-y)\,dt\,dx\,dy  \Big]. \notag
    \end{align}
    One can use \eqref{esti:time-cont-viscous-l1} to bound $\mathcal{M}_{2,1}$ as
    \begin{align*}
    \mathcal{M}_{2,1} \le \|\partial_t \psi\|_{L^\infty(\mathbb{Q}_T)} \mathbb{E}\left[\int_0^T \int_0^T \int_{\bar{K}} | u_\eps^\kappa(t,y) - u_\eps^\kappa(s,y)|\,\rho_{\delta_0}(t-s)\,dy\,ds\,dt\right] \le C(\psi, \eps,\kappa) \sqrt{\delta_0}. 
    \end{align*}
By using  Lemma \ref{lem:l-infinity bound-approximate-solution}, Cauchy-Schwartz inequality, the property of convolution together with \eqref{inq:uniform-1st-viscous}, we have
\begin{align*}
   \mathcal{M}_{2,2}& \le \|\partial_t \psi\|_{L^\infty(\mathbb{Q}_T)} \left( \mathbb{E}\left[ \int_{t= T-\delta_0}^{T}\int_{\mathbb{R}^d}\int_{\mathbb{R}^d}
    |u_{\Delta t}(t,x) - u_\eps^\kappa(t,y)| \varrho_{\delta}(x-y)\,dx\,dy\,dt\right] + l \delta_0 \right)\notag \\
    & \le \|\partial_t \psi\|_{L^\infty(\mathbb{Q}_T)} \left( C\widetilde{M}\delta_0 + \mathbb{E}\left[ \int_{t= T-\delta_0}^{T}\int_{\bar{K}} |u_\eps^\kappa(t,y)|\,dy\,dt\right]
    + l\,\delta_0\right) \notag \\
    & \le \|\partial_t \psi\|_{L^\infty(\mathbb{Q}_T)} \left( C\widetilde{M}\delta_0  + C \sqrt{\delta_0} \left\{ \sup_{0\le t\le T} \mathbb{E}\left[
    \|u_\eps^\kappa(t)\|_{L^2(\R^d)}^2\right]\right\}^\frac{1}{2} + l\delta_0 \right) \notag \\
    &  \le \|\partial_t \psi\|_{L^\infty(\mathbb{Q}_T)} \left( C\widetilde{M}\delta_0  + C \sqrt{\delta_0} \left\{ \sup_{0\le t\le T} \mathbb{E}\left[
    \|u_\eps(t)\|_{L^2(\R^d)}^2\right]\right\}^\frac{1}{2} + l\delta_0 \right) \le C \sqrt{\delta_0}\,.
\end{align*}

    We summarize the above result in the following lemma.
\begin{lem}\label{lem 2}
\begin{align}
    \mathcal{I}_2 + \mathcal{J}_2 &\le \mathbb{E}\Big[\int_{\mathbb{Q}_T}\int_{\mathbb{R}^d}\beta(u_{\Delta t}(t,x) - u_\eps^\kappa(t,y))\partial_t\psi(t,x)\varrho_{\delta}(x-y)\,dx\,dt\,dy  \Big] + Cl + C(\eps, \kappa) \sqrt{\delta_0}\,.\notag
\end{align}
\end{lem}
Next, we consider to estimate the terms coming from the associate flux function, i.e. the terms $\mathcal{I}_3$ and  $\mathcal{J}_3.$ We re-write the term $\mathcal{I}_3$ as
\begin{align}
\mathcal{I}_3 =\, &\mathbb{E} \Big [\int_{\mathbb{Q}_T^2}\int_\mathbb{R}\Big(F^\beta(u_\eps^\kappa(s,y), k)-F^\beta(u_\eps^\kappa(t,y), k)\Big)\cdot\nabla_y\varphi_{\delta_0, \delta}J_l(u_{\Delta t}(t,x) - k)\,dk\,ds\,dt\,dx\,dy \Big]\notag\\
+&\mathbb{E} \Big [\int_{\mathbb{Q}_T}\int_{\mathbb{R}^d} \int_{\R}F^\beta(u_\eps^\kappa(t,y), k)\cdot \nabla_y\varrho_\delta(x-y)\psi(t,x)\Big(\int_0^T\rho_{\delta_0}(t-s)ds -1\Big)J_l(u_{\Delta t}(t,x) - k)\,dk\,dt\,dx\,dy \Big]\notag\\
+&\mathbb{E} \Big [\int_{\mathbb{Q}_T}\int_{\mathbb{R}^d}\int_\mathbb{R}\Big(F^\beta(u_\eps^\kappa(t,y), u_{\Delta t}(t,x)-k) - F^\beta(u_\eps^\kappa(t,y), u_{\Delta t}(t,x))\Big)\notag\\&\hspace{6cm}\times \nabla_y\varrho_\delta(x-y)\psi(t,x)J_l(k)\,dk\,dt\,dx\,dy \Big]\notag\\
+&\mathbb{E} \Big [\int_{\mathbb{Q}_T}\int_{\mathbb{R}^d} F^\beta(u_\eps^\kappa(t,y), u_{\Delta t}(t,x)) \cdot \nabla_y\varrho_\delta(x-y)\psi(t,x)\,dt\,dx\,dy \Big]\notag\\
=:&\mathcal{I}_3^1 + \mathcal{I}_3^2 + \mathcal{I}_3^3 + \mathbb{E} \Big [\int_{\mathbb{Q}_T}\int_{\mathbb{R}^d} F^\beta(u_\eps^\kappa(t,y), u_{\Delta t}(t,x)) \cdot \nabla_y\varrho_\delta(x-y)\psi(t,x)\,dt\,dx\,dy \Big]\notag.
\end{align}
Note that, for all $a,b,c \in \mathbb{R}$ 
\begin{align}\label{eqlipF}
\begin{cases}
    |F^\beta(a,b) - F^\beta(c,b)| \le C|a-c|\,, \\
       |F^\beta(a,b) - F^\beta(a,c)| \le C|b-c|\big(1 + |a-b|\big).
       \end{cases}
\end{align}
By using \eqref{eqlipF} and \eqref{esti:time-cont-viscous-l1}, we have 

\begin{align}
    \mathcal{I}_3^1 \le \,&C\,\mathbb{E} \Big [\int_{\mathbb{Q}_T^2}\int_\mathbb{R}|u_\eps^\kappa(s,y)- u_\eps^\kappa(t,y)|\rho_{\delta_0}(t-s)|\nabla_y\varrho_{ \delta}(x-y)|\psi(t,x)J_l(k)\,dk\,ds\,dt\,dx\,dy \Big]\notag\\
    \le \,& C(\delta, \psi)\mathbb{E} \Big [ \int_0^T \int_0^T\int_{K_y}|u_\eps^\kappa(s,y)- u_\eps^\kappa(t,y)||\rho_{\delta_0}(t-s)\,dy\,ds\,dt \Big]
  \le C(\delta, \psi, \eps, \kappa) \sqrt{\delta_0}\,. \notag
\end{align}
Thanks to \eqref{eqlipF}, the property of convolution, the uniform estimates \eqref{l-infty-bound} and \eqref{inq:uniform-1st-viscous}, we bound $\mathcal{I}_3^2$ as follows. 
\begin{align}
    \mathcal{I}_3^2 \le \,&C\,\mathbb{E} \Big [\int_{\mathbb{Q}_T}\int_{\mathbb{R}^d} \int_{\R}\Big(|u_\eps^\kappa(t,y)| + |u_{\Delta t}(t,x) - k|\Big)\,|\nabla_y \varrho_\delta(x-y)|\,\psi(t,x)\notag \\&\hspace{6cm}\times\Big|\int_0^T\rho_{\delta_0}(t-s)ds -1\Big|J_l(k)\,dk\,dx\,dt\,dy \Big]\notag\\
    \le \,&  C(\delta, \psi)\,\mathbb{E} \Big [ \int_{\R}\int_{t = T - \delta_0}^T\int_{K_x}\int_{K_y}\Big(|u_\eps^\kappa(t,y)| + |u_{\Delta t}(t,x) - k|\Big)J_l(k)\,dy\,dx\,dt\,dk \Big]\notag\\
    \le \,&  C(\delta, \psi)\left[l +  \sqrt{\delta_0} \left( \left\{ \underset{0\le t\le T}{sup}\,\mathbb{E}\Big[||u_\eps^\kappa(t)||_{L^2(K_y)}^2\Big]\right\}^\frac{1}{2} +  \left\{\underset{0\le t\le T}{sup}\,\mathbb{E}\Big[||u_{\Delta t}(t,\cdot)||_{L^2(K_x)}^2\Big] \right\}^\frac{1}{2} \right) \right] \notag \\
    \le \,&  C(\delta, \psi)\left[l +  \sqrt{\delta_0} \left( \left\{ \underset{0\le t\le T}{sup}\,\mathbb{E}\Big[||u_\eps(t)||_{L^2(K_y)}^2\Big]\right\}^\frac{1}{2} +  \left\{\underset{0\le t\le T}{sup}\,\mathbb{E}\Big[||u_{\Delta t}(t,\cdot)||_{L^2(K_x)}^2\Big] \right\}^\frac{1}{2} \right) \right] \notag \\
    \le& C(\delta) \sqrt{\delta_0}(1+l)\,. \notag
\end{align}
Similarly, one can estimate $\mathcal{I}_3^3$ as
$$ \mathcal{I}_3^3\le C(\delta) l. $$
We can re-write $\mathcal{J}_3$ as follow.
\begin{align}
        \mathcal{J}_3 =\, &\mathbb{E} \Big [\int_{\mathbb{Q}_T^2}\int_\mathbb{R}\Big(F^\beta(u_{\Delta t}(t,x), u_\eps^\kappa(s,y) - k) - F^\beta(u_{\Delta t}(t,x), u_\eps^\kappa(t,y) - k)\Big). \nabla_x\varphi_{\delta_0, \delta}J_l(k)\,dk\,ds\,dt\,dx\,dy \Big]\notag\\
         +\,&\mathbb{E}\Big[\int_{\mathbb{Q}_T}\int_{\mathbb{R}^d}\int_\mathbb{R}F^\beta(u_{\Delta t}(t,x), u_\eps^\kappa(t,y) - k)\Big(\int_0^T\rho_{\delta_0}(t-s)ds - 1 \Big) 
         \notag\\&\hspace{6cm}\times\nabla_x(\varrho_\delta(x-y)\psi(t,x))J_l(k)dk\,dx\,dt\,dy \Big]\notag\\ 
         +\,&\mathbb{E} \Big [\int_{\mathbb{Q}_T}\int_{\mathbb{R}^d}\int_\mathbb{R}\Big(F^\beta(u_{\Delta t}(t,x), u_\eps^\kappa(t,y) - k) - F^\beta(u_{\Delta t}(t,x), u_\eps^\kappa(t,y))\Big)\notag\\&\hspace{6cm}\times \nabla_x(\varrho_\delta(x-y)\psi(t,x))J_l(k)dk\,dx\,dt\,dy \Big]\notag\\
         +\,&\mathbb{E} \Big [\int_{\mathbb{Q}_T} \int_{\mathbb{R}^d}F^\beta(u_{\Delta t}(t,x), u_\eps^\kappa(t,y)). \nabla_x(\varrho_\delta(x-y)\psi(t,x))\,dx\,dt\,dy \Big]\notag\\
         =: \,& \mathcal{J}_3^1 + \mathcal{J}_3^2 + \mathcal{J}_3^3 + \mathbb{E} \Big [\int_{\mathbb{Q}_T} \int_{\mathbb{R}^d}F^\beta(u_{\Delta t}(t,x), u_\eps^\kappa(t,y)). \nabla_x(\varrho_\delta(x-y)\psi(t,x))\,dx\,dt\,dy \Big]\,. \notag 
 \end{align}
 We use \eqref{eqlipF}, the property of convolution, the uniform estimates \eqref{l-infty-bound} and \eqref{inq:uniform-1st-viscous},  \eqref{esti:time-cont-viscous-l1} and \eqref{eq:time continuity of uek} together with Cauchy-Schwartz inequality to have 
 \begin{align}
     \mathcal{J}_3^1 \le \,&C\,\mathbb{E} \Big [\int_{\mathbb{Q}_T^2}\int_\mathbb{R}|u_\eps^\kappa(t,y) - u_\eps^\kappa(s,y)|\Big( 1 + |u_{\Delta t}(t,x)| + |u_\eps^\kappa(s,y) - k|\Big)|\nabla_x\varphi_{\delta_0, \delta}|J_l(k)\,dk\,ds\,dt\,dx\,dy \Big]\notag\\
     \le\, & C(1 + l)\, \mathbb{E} \Big [\int_{\mathbb{Q}_T^2}|u_\eps^\kappa(t,y) - u_\eps^\kappa(s,y)|\rho_{\delta_0}(t-s) |\nabla_x[\varrho_{\delta}(x-y)\psi(t,x)]|\,ds\,dt\,dx\,dy \Big]\notag \\
     +\, &C\, \Big(\mathbb{E} \Big [\int_{\mathbb{Q}_T^2}|u_\eps^\kappa(t,y) - u_\eps^\kappa(s,y)|^2\rho_{\delta_0}(t-s)|\nabla_x(\varrho_{\delta}(x-y)\psi(t,x))|^2\,ds\,dt\,dx\,dy \Big]\Big)^{\frac{1}{2}}\notag \\&\hspace{2cm}\times\Big(\mathbb{E} \Big [\int_{\mathbb{Q}_T^2}\Big(|u_\eps^\kappa(s,y)|^2 +|u_{\Delta t}(t,x)|^2\Big)\rho_{\delta_0}(t-s) |\nabla_x(\varrho_{\delta}(x-y)\psi(t,x))|^2\,ds\,dt\,dx\,dy \Big]\Big)^{\frac{1}{2}}\notag\\
     \le &\, C(\delta, \psi)(1 + l)\,\mathbb{E} \Big [\int_{K_y \times [0, T]}\int_0^T|u_\eps^\kappa(t,y) - u_\eps^\kappa(s,y)|\rho_{\delta_0}(t-s)\,ds\,dt\,\,dy \Big]\notag\\
      &\hspace{1cm} +\, C(\delta, \psi)\,\Big(\mathbb{E} \Big [\int_{K_y}\int_0^T\int_0^T|u_\eps^\kappa(t,y) - u_\eps^\kappa(s,y)|^2\rho_{\delta_0}(t-s)ds\,dt\,dy \Big]\Big)^{\frac{1}{2}}\notag\\&\hspace{4cm}\times\Big(\underset{0\le s\le T}{sup} \,\mathbb{E}\Big[||u_\eps(s,\cdot)||_{L^2(K_y)}^2\Big] + \underset{0\le t\le T}{sup} \,\mathbb{E}\Big[||u_{\Delta t}(t,\cdot)||_{L^2(K_x)}^2\Big]\Big)^{\frac{1}{2}}\notag\\
     \le &C(\delta,\eps, \kappa)(1+l)\sqrt{\delta_0}\,.\notag 
      \end{align}
 Using a similar lines of argument as done for $\mathcal{I}_3^2$, one can get
 $$\mathcal{J}_3^2 +  \mathcal{J}_3^3  \le C(\delta) \left( \sqrt{\delta_0}(1+l) + l\right)\,.$$
 Since $\Big|F^\beta(x,y)-|y-x|\Big| \le C\xi$, we have
  $$\mathbb{E} \,\Big [\int_{\mathbb{Q}_T} \int_{\mathbb{R}^d}\Big(F^\beta(u_{\Delta t}(t,x), u_\eps^\kappa(t,y))-F^\beta(u_\eps^\kappa(t,y), u_{\Delta t}(t,x)\Big) \cdot \nabla_x\varrho_\delta(x-y)\,\psi(t,x)\,dx\,dt\,dy \Big] \le C\xi\,.$$
 Thus, we have the following lemma.
 \begin{lem}\label{lem3}
 \begin{align}
     \mathcal{I}_3 + \mathcal{J}_3 \le 
      \,& \mathbb{E} \Big [\int_{\mathbb{Q}_T} \int_{\mathbb{R}^d}F^\beta(u_{\Delta t}(t,x), u_\eps^\kappa(t,y))  \cdot \nabla_x\psi(t,x) \varrho_\delta(x-y)\,dx\,dt\,dy \Big]  \notag \\
      & \quad + C(\delta, \eps, \kappa)\left( (1+l)\sqrt{\delta_0} + l \right) + C\xi\, .\notag
 \end{align}
 \end{lem}
 Next, we shift our attention to the non-local terms. $\mathcal{I}_4$ can be re-arranged as follows.
\begin{align}
       \mathcal{I}_4 = &- \mathbb{E} \Big[\int_{\mathbb{Q}_T^2} \int_{\mathbb{R}}\Big(\mathcal{L}_\theta^{\bar{r}}[u_\eps^\kappa(s,\cdot)](y)-\mathcal{L}_\theta^{\bar{r}}[u_\eps^\kappa(t,\cdot)](y)\Big)\beta'(u_\eps^\kappa(s,y) - u_{\Delta t}(t,x) + k)\varphi_{\delta_0, \delta}J_l(k)\,dk\,ds\,dt\,dy\,dx \Big]\notag\\
       &-\mathbb{E} \Big[\int_{\mathbb{Q}_T^2} \int_{\mathbb{R}}\mathcal{L}_\theta^{\bar{r}}[u_\eps^\kappa(t,\cdot)](y)\Big(\beta'(u_\eps^\kappa(s,y) - u_{\Delta t}(t,x) + k)) - \beta'(u_\eps^\kappa(t,y) - u_{\Delta t}(t,x) + k))\Big)\notag\\&\hspace{7cm}\times \varphi_{\delta_0, \delta}J_l(k)\,dk\,ds\,dt\,dy\,dx \Big]\notag\\
       &-\mathbb{E} \Big[\int_{\mathbb{Q}_T}\int_{\mathbb{R}^d} \int_{\mathbb{R}}\mathcal{L}_\theta^{\bar{r}}[u_\eps^\kappa(t,\cdot)](y)\beta'(u_\eps^\kappa(t,y) - u_{\Delta t}(t,x) + k)\notag\\& \hspace{5cm} \times \Big(\int_0^T\rho_{\delta_0}(t-s)ds -1\Big) \varrho_{ \delta}(x-y)\psi(t,x)J_l(k)\,dk\,dy\,dt\,dx \Big]\notag\\
       &-\mathbb{E} \Big[\int_{\mathbb{Q}_T}\int_{\mathbb{R}^d} \int_{\mathbb{R}}\mathcal{L}_\theta^{\bar{r}}[u_\eps^\kappa(t,\cdot)](y)\Big(\beta'(u_\eps^\kappa(t,y)- u_{\Delta t}(t,x) + k) - \beta'(u_\eps^\kappa(t,y) - u_{\Delta t}(t,x))\Big)\notag\\& \hspace{7cm}\times\varrho_{ \delta}(x-y)\psi(t,x)J_l(k)\,dk\,dy\,dt\,dx \Big]\notag\\
       &-\mathbb{E} \Big[\int_{\mathbb{Q}_T}\int_{\mathbb{R}^d} \mathcal{L}_\theta^{\bar{r}}[u_\eps^\kappa(t,\cdot)](y)\beta'(u_\eps^\kappa(t,y) - u_{\Delta t}(t,x))\varrho_{ \delta}(x,y)\psi(t,x)\,dy\,dt\,dx \Big]\notag\\
       &=: \mathcal{I}_4^1 +  \mathcal{I}_4^2 + \mathcal{I}_4^3 +  \mathcal{I}_4^4 + \mathcal{I}_4^5\,. 
       \notag
       \end{align}
    Let us consider the term $\mathcal{I}_4^1$. Note that  for any positive $\bar{r}$,
    \begin{align}
   z\mapsto \frac{1_{|z|\ge \bar{r}}}{|z|^{d+2\theta}} \in L^p(\R^d)\quad (1\le p\le \infty); \quad \text{i.e,}~~\int_{|z| > \bar{r}} \frac{dz}{|z|^{p(d+2\theta)}}\le C(p, \bar{r})\,. \label{inq:bound-frac-z}
    \end{align}
In view of triangle inequality, \eqref{esti:time-cont-viscous-l1}  and \eqref{inq:bound-frac-z}, 
    we obtain
    \begin{align}
        \mathcal{I}_4^1 = &-\mathbb{E} \Big[\int_{\mathbb{Q}_T^2} \int_{\mathbb{R}}\Big(u_\eps^\kappa(s,y)-u_\eps^\kappa(t,y)\Big)\mathcal{L}_\theta^{\bar{r}}[\beta'(u_\eps^\kappa(s,\cdot) - u_{\Delta t}(t,x) + k)\varrho_\delta(x-\cdot)](y)\notag\\&\hspace{6cm }\times\rho_{\delta_0}(t-s)\psi(t,x)J_l(k)\,dk\,ds\,dt\,dy\,dx \Big]\notag\\
        \le \,&C(\beta')\,\mathbb{E} \Big[\int_{\mathbb{Q}_T^2} \int_{\mathbb{R}}|u_\eps^\kappa(s,y)-u_\eps^\kappa(t,y)|\int_{|z| > \bar{r}}\frac{\varrho_\delta(x-y)+\varrho_\delta(x-(y-z))}{|z|^{d+2\theta}}\,dz\notag\\&\hspace{6cm }\times\rho_{\delta_0}(t-s)\psi(t,x)J_l(k)\,dk\,ds\,dt\,dy\,dx \Big]\notag\\
        \le \,&C(\beta', \bar{r}, \psi, \delta )\,\mathbb{E} \Big[\int_0^T\int_0^T\int_{\bar{K}_y} |u_\eps^\kappa(s,y)-u_\eps^\kappa(t,y)|\rho_{\delta_0}(t-s)\,dy\,dt\,ds \Big]\notag\\
        \le \,&C(\bar{r},\delta, \kappa, \eps) \sqrt{\delta_0}.\notag
    \end{align}
Next we estimate the term $\mathcal{I}_4^2$. Using triangle inequality, \eqref{inq:bound-frac-z}, Cauchy-Schwartz inequality, convolution property, Lemma \ref{lem:average-time-cont-viscous} together with \eqref{inq:uniform-1st-viscous}, we see that
    \begin{align}
       \mathcal{I}_4^2 = \,&C(\beta'')\,\mathbb{E} \Big[\int_{\mathbb{Q}_T^2} |\mathcal{L}_\theta^{\bar{r}}[u_\eps^\kappa(t,\cdot)](y)|\,|u_\eps^\kappa(s,y)-u_\eps^\kappa(t,y)|\rho_{\delta_0}(t-s)\varrho_\delta(x-y)\psi(t,x)\,ds\,dt\,dy\,dx \Big]\notag\\
       \le  &C(\beta'', \psi)\,\mathbb{E} \Big[\int_0^T \int_0^T\int_{K_y} |\mathcal{L}_\theta^{\bar{r}}[u_\eps^\kappa(t,\cdot)](y)|\,|u_\eps^\kappa(s,y)-u_\eps^\kappa(t,y)|\rho_{\delta_0}(t-s)\,dy\,dt\,ds \Big]\notag\\
       \le &C(\beta'', \psi)\,\mathbb{E} \Big[\int_0^T\int_{\mathbb{Q}_T} |u_\eps^\kappa(t,y)|\,|u_\eps^\kappa(s,y)-u_\eps^\kappa(t,y)|\rho_{\delta_0}(t-s)\Big(\int_{|z| > \bar{r}}\frac{1}{|z|^{d+2\theta}}\,dz\Big)\,ds\,dy\,dt \Big]\notag\\
       +&C(\beta'', \psi)\,\mathbb{E} \Big[\int_0^T\int_0^T\int_{K_y} ||u_\eps^\kappa(t,\cdot)||_{L^2(\mathbb{R}^d)}\,|u_\eps^\kappa(s,y)-u_\eps^\kappa(t,y)|\rho_{\delta_0}(t-s)\notag\\&\hspace{8cm}\times\Big(\int_{|z|> \bar{r}}\notag\Big(\frac{1}{|z|^{1+2\theta}}\Big)^2\,dz\Big)^{\frac{1}{2}}\,dy\,dt\,ds \Big]\Big)\notag\\
       \le &C(\beta'', \psi, \bar{r})\,\mathbb{E} \Big[\int_0^T\int_{\mathbb{Q}_T} |u_\eps^\kappa(t,y)|\,|u_\eps^\kappa(s,y)-u_\eps^\kappa(t,y)|\rho_{\delta_0}(t-s)\,dy\,dt\,ds \Big]\notag\\
       +&C(\beta'', \psi, \bar{r})\,\mathbb{E} \Big[\int_0^T\int_0^T\int_{K_y} ||u_\eps^\kappa(t,\cdot)||_{L^2(\mathbb{R}^d)}\,|u_\eps^\kappa(s,y)-u_\eps^\kappa(t,y)|\rho_{\delta_0}(t-s)\,dy\,dt\,ds \Big]\Big)\notag\\
       \le 
        &\,C(\beta'', \psi, \bar{r})\Big(\underset{t}{sup}\,\mathbb{E}\Big[||u_\eps(t,\cdot)||_{L^2(\mathbb{R}^d)}^2\Big]\Big)^\frac{1}{2}\Big(\mathbb{E}\Big[\int_0^T\int_0^T\int_{K_y} |u_\eps^\kappa(s,y)-u_\eps^\kappa(t,y)|^2\rho_{\delta_0}(t-s)\,dy\,dt\,ds \Big]\Big)^{\frac{1}{2}}\notag\\
       \le &  C(\xi,\bar{r}, \kappa, \eps) \sqrt{\delta_0}.\notag
    \end{align}
    For $\mathcal{I}_4^3$ and $\mathcal{I}_4^4$, we use triangle inequality, \eqref{inq:bound-frac-z}, Cauchy-Schwartz inequality, convolution property, and  \eqref{inq:uniform-1st-viscous} to have 
    \begin{align}
    \mathcal{I}_4^3  \le  &\,C(\beta', |\psi|)\,\mathbb{E} \Big[\int_{t= T-\delta_0}^T\int_{K_y} |\mathcal{L}_\theta^{\bar{r}}[u_\eps^\kappa(t,\cdot)](y) \,dy\,dt \Big]\notag \le  C(\bar{r}) \sqrt{\delta_0}, \notag \\
        \mathcal{I}_4^4 \le &C(\xi, |\psi|)\,\mathbb{E} \Big[\int_{\mathbb{Q}_T} \int_{\mathbb{R}}|k|\,|\mathcal{L}_\theta^{\bar{r}}[u_\eps^\kappa(t,\cdot)](y)|\,J_l(k)\,dk\,dt\,dy \Big] \le C(\bar{r}, \xi)\,l\,.\notag
    \end{align}
    We rewrite $\mathcal{J}_4$ as follows.  
\begin{align}
         \mathcal{J}_4 =&- \mathbb{E} \Big [\int_{\mathbb{Q}_T^2}\int_\mathbb{R} \mathcal{L}_\theta^{\bar{r}}[u_{\Delta t}(t,\cdot)](x)\Big(\beta'(u_{\Delta t}(t,x) - u_\eps^\kappa(s,y) + k)- \beta'(u_{\Delta t}(t,x) - u_\eps^\kappa(t,y) + k)\Big)\notag\\ &\hspace{7cm} \times \varphi_{\delta_0, \delta}J_l(k)\,dk\,ds\,dt\,dx\,dy \Big]\notag\\
         &- \mathbb{E} \Big [\int_{\mathbb{Q}_T}\int_{\mathbb{R}^d}\int_\mathbb{R} \mathcal{L}_\theta^{\bar{r}}[u_{\Delta t}(t,\cdot)](x)\beta'(u_{\Delta t}(t,x) - u_\eps^\kappa(t,y) + k)\notag\\ &\hspace{5cm} \times \Big(\int_0^T\rho_{\delta_0}(t-s)ds -1 \Big) \varrho_{\delta}(x-y)\psi(t,x)J_l(k)\,dk\,dx\,dt\,dy \Big]\notag\\
         & -\mathbb{E} \Big [\int_{\mathbb{Q}_T}\int_{\mathbb{R}^d}\int_\mathbb{R} \mathcal{L}_\theta^{\bar{r}}[u_{\Delta t}(t,\cdot)](x)\Big(\beta'(u_{\Delta t}(t,x) - u_\eps^\kappa(t,y) +k) - \beta'(u_{\Delta t}(t,x) - u_\eps^\kappa(t,y) )\Big)\notag\\ &\hspace{7cm} \times  \varrho_{\delta}(x-y)\psi(t,x)J_l(k)dk\,dx\,dt\,dy \Big]\notag\\
         &-\mathbb{E} \Big [\int_{\mathbb{Q}_T}\int_{\mathbb{R}^d} \mathcal{L}_\theta^{\bar{r}}[u_{\Delta t}(t,\cdot)](x)  \beta'(u_{\Delta t}(t,x) - u_\eps^\kappa(t,y)) \varrho_{\delta}(x-y)\psi(t,x)\,dx\,dt\,dy \Big] \notag \\
         &=:\sum_{i=1}^4\mathcal{J}_4^i \,. \notag
\end{align}
Next we estimate $\mathcal{J}_4^1.$ By using  \eqref{inq:bound-frac-z}, Lemma \ref{lem:average-time-cont-viscous} along with uniform bound of approximate solutions $u_{\Delta t}(t,x)$, we obtain
\begin{align}
    \mathcal{J}_4^1 \le & \,C(\beta'')\,\mathbb{E} \Big [\int_{\mathbb{Q}_T^2}\int_\mathbb{R} ||u_{\Delta t}(t,\cdot)||_{L^\infty(\mathbb{R}^d)}\,|u_\eps^\kappa(t,y) - u_\eps^\kappa(s,y)|\,\rho_{\delta_0}(t-s)\notag\\&\hspace{3cm}\times \varrho_\delta(x-y)\psi(t,x)\Big|\int_{|z| > \bar{r}}\frac{1}{|z|^{d+2\theta}}dz\Big|J_l(k)\,dk\,ds\,dt\,dx\,dy \Big]\notag\\
    \le & \,C(\beta'', \psi, \bar{r}, \widetilde{M})\,\mathbb{E} \Big [\int_0^T\int_0^T\int_{K_y}|u_\eps^\kappa(t,y) - u_\eps^\kappa(s,y)|\,\rho_{\delta_0}(t-s)\,dy\,ds\,dt \Big]\notag\\
    \le &C(\xi, \bar{r}, \eps, \kappa) \sqrt{\delta_0}.\notag
\end{align}
Similarly, we have 
$$\mathcal{J}_4^2 \le C(\bar{r})\delta_0, \hspace{.5cm}\text{and}\hspace{.5cm} \mathcal{J}_4^3 \le C(\xi, \bar{r})\,l.$$
Since $\beta^\prime$ is odd function, one has 
\begin{align}
    \mathcal{I}_4^5 + \mathcal{I}_4^4 = &-\mathbb{E} \Big[\int_{\mathbb{Q}_T}\int_{\mathbb{R}^d} \mathcal{L}_\theta^{\bar{r}}[u_\eps^\kappa(t,\cdot)](y)\beta'(u_\eps^\kappa(t,y) - u_{\Delta t}(t,x))\varrho_{ \delta}(x,y)\psi(t,x)\,dy\,dt\,dx \Big]\notag\\ &+\mathbb{E} \Big [\int_{\mathbb{Q}_T}\int_{\mathbb{R}^d} \mathcal{L}_\theta^{\bar{r}}[u_{\Delta t}(t,\cdot)](x)  \beta'(u_\eps^\kappa(t,y) - u_{\Delta t}(t,x)) \varrho_{\delta}(x-y)\psi(t,x)\,dx\,dt\,dy \Big]\notag\\
    = &\,  \mathbb{E} \Big [\int_{\mathbb{Q}_T}\int_{\mathbb{R}^d} \Big[\int_{|z|>\bar{r}}\frac{u_\eps^\kappa(t,y+z)-u_\eps^\kappa(t,y)}{|z|^{d+2\theta}}dz-\int_{|z|>\bar{r}}\frac{u_{\Delta t}(t,x+z)-u_{\Delta t}(t,x)}{|z|^{d+2\theta}}dz\Big]\notag\\&\hspace{4cm}\times\beta'( u_\eps^\kappa(t,y)-u_{\Delta t}(t,x)) \varrho_{\delta}(x-y)\psi(t,x)\,dx\,dt\,dy\Big]\notag\\
    =& \, \mathbb{E} \Big [\int_{\mathbb{Q}_T}\int_{\mathbb{R}^d} \Big[\int_{|z|>\bar{r}}\frac{u_\eps^\kappa(t,y+z)-u_{\Delta t}(t,x+z)}{|z|^{d+2\theta}}dz-\int_{|z|>\bar{r}}\frac{u_\eps^\kappa(t,y)-u_{\Delta t}(t,x)}{|z|^{d+2\theta}}\,dz\Big]\notag\\&\hspace{4cm}\times\beta'(u_\eps^\kappa(t,y)-u_{\Delta t}(t,x)) \varrho_{\delta}(x-y)\psi(t,x)\,dx\,dt\,dy\Big]\notag\\
    \le &\,\mathbb{E} \Big [\int_{\mathbb{Q}_T}\int_{\mathbb{R}^d} \Big[\int_{|z|>\bar{r}}\frac{\beta(u_\eps^\kappa(t,y+z)-u_{\Delta t}(t,x+z))}{|z|^{d+2\theta}}\,dz-\int_{|z|>\bar{r}}\frac{\beta(u_\eps^\kappa(t,y)-u_{\Delta t}(t,x))}{|z|^{d+2\theta}}\,dz\Big]\notag\\&\hspace{6cm}\times \varrho_{\delta}(x-y)\psi(t,x)\,dx\,dt\,dy\Big]\notag\\
    = \,&-\,\mathbb{E} \Big [\int_{\mathbb{Q}_T}\int_{\mathbb{R}^d}\beta(u_\eps^\kappa(t,y)-u_{\Delta t}(t,x))\mathcal{L}_\theta^{\bar{r}}[\psi(t,\cdot)](x)\varrho_{\delta}(x-y)\,dx\,dt\,dy\Big].\notag
\end{align}
In the last equality, we have used a change of coordinates for the first integral $x \mapsto x + z$, $y \mapsto y + z$, $z \mapsto -z.$ For the inequality, we have used the fact that $\beta(b) - \beta(a) \ge \beta'(a)(b-a)$ with $a=u_\eps^\kappa(t,y)-u_{\Delta t}(t,x)$ and $b=u_\eps^\kappa(t,y+z)-u_{\Delta t}(t,x+z)$.
\vspace{0.2cm}

Next we move our focus to estimate $\mathcal{I}_5$. We re-arrange $\mathcal{I}_5$ as follows.
 \begin{align}
        \mathcal{I}_5 
        =\,&-\,\mathbb{E}\Big[\int_{\mathbb{Q}_T^2} \int_{\mathbb{R}}\big(\beta(u_\eps^\kappa(s,y) - u_{\Delta t}(t,x) +k) - \beta(u_\eps^\kappa(t,y) - u_{\Delta t}(t,x) + k)\big)\notag\\ & \hspace{5cm} \times \mathcal{L}_{\theta,\bar{r}}[\varrho_\delta(x-\cdot)](y)\psi(t,x)\rho_{\delta_0}(t-s)J_l(k)\,dk\,ds\,dt\,dy\,dx \Big]\notag\\
        &-\mathbb{E}\Big[\int_{\mathbb{Q}_T}\int_{\mathbb{R}^d} \int_{\mathbb{R}}\beta(u_\eps^\kappa(t,y) - u_{\Delta t}(t,x) +k)\mathcal{L}_{\theta,\bar{r}}[\varrho_\delta(x-\cdot)](y)\psi(t,x)\notag\\ & \hspace{5cm} \times\Big(\int_0^T\rho_{\delta_0}(t-s)ds -1\Big) J_l(k)\,dk\,dy\,dt\,dx \Big]\notag\\
        & -\mathbb{E} \Big[\int_{\mathbb{Q}_T}\int_{\mathbb{R}^d} \int_{\mathbb{R}}\Big(\beta(u_\eps^\kappa(t,y) - u_{\Delta t}(t,x) + k ) - \beta(u_\eps^\kappa(t,y) - u_{\Delta t}(t,x))\Big)\notag\\& \hspace{5cm} \times
        \mathcal{L}_{\theta,\bar{r}}[\varrho_\delta(x-\cdot)](y)\psi(t,x)J_l(k)\,dk\,dy\,dt\,dx \Big]\notag\\
        &-\mathbb{E}\Big[\int_{\mathbb{Q}_T} \int_{\mathbb{R}^d}\beta(u_\eps^\kappa(t,y) - u_{\Delta t}(t,x) )\mathcal{L}_{\theta,\bar{r}}[\varrho_\delta(x-\cdot)](y)\psi(t,x)\,dy\,dt\,dx \Big]\notag\\
        &=:\mathcal{I}_5^1 +\mathcal{I}_5^2 +\mathcal{I}_5^3 - \mathbb{E}\Big[\int_{\mathbb{Q}_T} \int_{\mathbb{R}^d}\beta(u_\eps^\kappa(t,y) - u_{\Delta t}(t,x) )\mathcal{L}_{\theta,\bar{r}}[\varrho_\delta(x-\cdot)](y)\psi(t,x)\,dy\,dt\,dx \Big].\notag
\end{align}
 Note that ~(cf.~\cite{cifani}) for $\phi(\cdot,\cdot) \in$ $C_c^2(\mathbb{Q}_T)$
\begin{align*}
|\mathcal{L}_{\theta,\bar{r}}[\phi(t, .)](x)| \le
\begin{cases}
 \displaystyle c_\theta||D\phi||_{L^\infty}\int_{|z| \le \bar{r}}\frac{|z|}{|z|^{d+2\theta}}\,dz, \quad \text{for} \quad \theta \in (0, \frac{1}{2}),\\
 \displaystyle \frac{c_\theta}{2}||D^2\phi||_{L^\infty}\int_{|z| \le \bar{r}}\frac{|z|^2}{|z|^{d+2\theta}}\,dz, \quad \text{for} \quad \theta \in [\frac{1}{2}, 1).
\end{cases}
\end{align*}
Thus, one can see that in both cases 
\begin{align}\label{fractionalbound}
|\mathcal{L}_{\theta,\bar{r}}[\phi(t, .)](x)| \le C(\phi)\bar{r}^a \quad \text{ for some $a >0$}\,.
\end{align}
We use \eqref{fractionalbound}  and Lemma  \ref{lem:average-time-cont-viscous} to estimate $\mathcal{I}_5^1$. We have
\begin{align}
    \mathcal{I}_5^1 \le\, &C(\beta')\,\mathbb{E}\Big[\int_{\mathbb{Q}_T^2} \int_{\mathbb{R}}|u_\eps^\kappa(s,y)- u_\eps^\kappa(t,y)|\rho_{\delta_0}(t-s) \,|\mathcal{L}_{\theta,\bar{r}}[\varrho_\delta(x-\cdot)](y)|\psi(t,x)J_l(k)\,dk\,ds\,dt\,dy\,dx \Big]\notag\\
    \le &C(\beta', \delta,\psi,\bar{r}^a)\,\mathbb{E}\Big[\int_0^T\int_0^T\int_{K_y}|u_\eps^\kappa(s,y)- u_\eps^\kappa(t,y)|\rho_{\delta_0}(t-s)\,dy\,dt\,ds\Big]\notag\\
    \le &C(\delta,\bar{r}^a,\eps, \kappa) \sqrt{\delta_0}.\notag
\end{align}
Again using \eqref{fractionalbound},  \eqref{l-infty-bound} and \eqref{inq:uniform-1st-viscous}, one can easily derive that
\begin{align}
    \mathcal{I}_5^2 \le\, &C(\psi)\,\mathbb{E}\Big[\int_{t=T-\delta_0}^T\int_{K_x} \int_{K_y} \int_{\mathbb{R}}\Big(|u_\eps^\kappa(t,y)| +|u_{\Delta t}(t,x)| + |k|\Big)|\mathcal{L}_{\theta,\bar{r}}[\varrho_\delta(x-\cdot)](y)|\ J_l(k)\,dk\,dy\,dx\,dt \Big]\notag\\
    \le &C(\psi, \delta, \bar{r}^a) l\delta_0 + C(\psi, \delta, \bar{r}^a)\left[ \sqrt{\delta_0} \left( 
    \underset{0 \le t \le T}{sup}\,\mathbb{E}\Big[||u_\eps^\kappa(t)||_{L^2(K_y)} ^2\Big]\right)^\frac{1}{2} + \widetilde{M}\delta_0\right] \notag \\
     \le& C(\delta,\bar{r}^a) \delta_0(1+l)\,. \notag \\
    \mathcal{I}_5^3 \le & C(\beta'')\,\mathbb{E} \Big[\int_{\mathbb{Q}_T}\int_{\mathbb{R}^d} \int_{\mathbb{R}}|k|
        \,|\mathcal{L}_{\theta,\bar{r}}[\varrho_\delta(x-\cdot)](y)|\psi(t,x)J_l(k)\,dk\,dy\,dt\,dx \Big] \le C(\xi,\delta,\bar{r}^a )\, l\,.\notag
\end{align}
Re-arranging $\mathcal{J}_5$, we get
\begin{align}
       \mathcal{J}_5 = & - \mathbb{E} \Big [\int_{\mathbb{Q}_T^2}\int_\mathbb{R}  \big(\beta(u_{\Delta t}(t,x) - u_\eps^\kappa(s,y) + k) - \beta(u_{\Delta t}(t,x) - u_\eps^\kappa(t,y) + k\big)\notag\\ &\hspace{5cm} \times \mathcal{L}_{\theta,\bar{r}}[\varrho_\delta(\cdot-y)\psi(t,\cdot)](x)\rho_{\delta_0}(t-s)J_l(k)\,dk\,ds\,dt\,dx\,dy \Big]\notag\\
       &-\mathbb{E} \Big [\int_{\mathbb{Q}_T}\int_{\mathbb{R}^d}\int_\mathbb{R}  \beta(u_{\Delta t}(t,x) - u_\eps^\kappa(t,y) + k)\mathcal{L}_{\theta,\bar{r}}[\varrho_\delta(\cdot-y)\psi(t,\cdot)](x))\notag\\ &\hspace{5cm} \times\Big(\int_0^T\rho_{\delta_0}(t-s)ds -1 \Big)J_l(k)dk\,dx\,dt\,dy \Big]\notag\\
       &- \mathbb{E} \Big [\int_{\mathbb{Q}_T}\int_{\mathbb{R}^d}\int_\mathbb{R}  \Big(\beta(u_{\Delta t}(t,x) - u_\eps^\kappa(t,y) + k) - \beta(u_{\Delta t}(t,x) - u_\eps^\kappa(t,y))\Big)\notag\\ &\hspace{5cm} \times \mathcal{L}_{\theta,\bar{r}}[\varrho_\delta(\cdot-y)\psi(t,\cdot)](x)J_l(k)\,dk\,dx\,dt\,dy \Big]\notag\\
       &- \mathbb{E} \Big [\int_{\mathbb{Q}_T}\int_{\mathbb{R}^d} \beta(u_{\Delta t}(t,x) - u_\eps^\kappa(t,y)) \mathcal{L}_{\theta,\bar{r}}[\varrho_\delta(\cdot-y)\psi(t,\cdot)](x)\,dx\,dt\,dy \Big]\notag\\
       &=: \mathcal{J}_5^1 +  \mathcal{J}_5^2 +  \mathcal{J}_5^3- \mathbb{E} \Big [\int_{\mathbb{Q}_T}\int_{\mathbb{R}^d} \beta(u_{\Delta t}(t,x) - u_\eps^\kappa(t,y)) \mathcal{L}_{\theta,\bar{r}}[\varrho_\delta(\cdot-y)\psi(t,\cdot)](x)\,dx\,dt\,dy \Big].\notag
\end{align}
Similar to the estimation of $\mathcal{I}_5^1$, one may bound the term $\mathcal{J}_5^1$ as
\begin{align}
  \mathcal{J}_5^1 \le \,& C(\beta')\,\mathbb{E} \Big [\int_{\mathbb{Q}_T^2}|u_\eps^\kappa(s,y) - u_\eps^\kappa(t,y)|\rho_{\delta_0}(t-s)\,| \mathcal{L}_{\theta,\bar{r}}[\varrho_\delta(\cdot-y)\psi(t,\cdot)](x)\,ds\,dt\,dx\,dy \Big]\notag\\  
  \le &C(\delta,\bar{r}^a, \eps, \kappa) \sqrt{\delta_0}.\notag
\end{align}
Using a similar lines of argument as done in the estimations of $\mathcal{I}_5^2$ and $\mathcal{I}_5^3$, one has
$$\mathcal{J}_5^2 \le C(\delta,\bar{r}^a)(1+l)\sqrt{\delta_0}, \hspace{.5cm} \text{and} \hspace{.5cm} \mathcal{J}_5^3 \le C(\delta, \bar{r}^a)l.$$
In the view of the above estimations, we arrive at  the following lemma.
\begin{lem}\label{lem45} We have,
\begin{align}
(\mathcal{I}_4 + \mathcal{J}_4 + \mathcal{I}_5 + \mathcal{J}_5) &\le -\,\mathbb{E} \Big [\int_{\mathbb{Q}_T}\int_{\mathbb{R}^d}\beta(u_\eps^\kappa(t,y)-u_{\Delta t}(t,x))\mathcal{L}_\theta^{\bar{r}}[\psi(t,\cdot)](x)\varrho_{\delta}(x-y)\,dx\,dt\,dy\Big]\notag\\
&\quad- \,\mathbb{E} \Big [\int_{\mathbb{Q}_T}\int_{\mathbb{R}^d} \beta(u_{\Delta t}(t,x) - u_\eps^\kappa(t,y)) \mathcal{L}_{\theta,\bar{r}}[\varrho_\delta(\cdot-y)\psi(t,\cdot)](x)\,dx\,dt\,dy \Big]\notag\\
&\qquad- \,\mathbb{E}\Big[\int_{\mathbb{Q}_T} \int_{\mathbb{R}^d}\beta(u_\eps^\kappa(t,y) - u_{\Delta t}(t,x) )\mathcal{L}_{\theta,\bar{r}}[\varrho_\delta(x-\cdot)](y)\psi(t,x)\,dx\,dt\,dy \Big]\notag\\
& \hspace{2cm} + C(\bar{r}, \delta, \kappa, \xi, \eps) \sqrt{\delta_0} + C(\delta,\bar{r},\xi) l \,. \notag
\end{align}
\end{lem}
We will estimate the error terms arising in inequality \eqref{EntropyIQUE2}. Thanks to Lemmas \ref{lem: prpertiesof S}, \ref{lem:propertiesof R}, and \ref{lem:bv bound-approximate-solution} we see that 
\begin{align}
       \mathcal{J}_{10} &\le \mathbb{E}\Big[ \int_{\mathbb{Q}_T} \sum_{n=0}^{N-1} \int_{\mathbb{R}^d} \int_{t_n}^{t_{n+1}}\big |\widetilde{u}_{\Delta t}(t,x) - u_{\Delta t}(t_{n+1},x)| |\partial_t\psi(t,x)|\rho_{\delta_0}(t-s)\varrho_\delta(x-y)\,dt\,dx\,ds\,dy \Big]\notag\\ 
       &+\mathbb{E}\Big[ \int_{\mathbb{Q}_T} \sum_{n=0}^{N-1} \int_{\mathbb{R}^d} \int_{t_n}^{t_{n+1}}\big |\widetilde{u}_{\Delta t}(t,x) - u_{\Delta t}(t_{n+1},x)|\psi(t,x)|\partial_t\rho_{\delta_0}(t-s)|\varrho_\delta(x-y)\,dt\,dx\,ds\,dy \Big]\notag \\
       & \le  C(\psi)\,\mathbb{E}\Big[ \sum_{n=0}^{N-1} \int_{\mathbb{R}^d}\int_0^T \int_{t_n}^{t_{n+1}}\big |S(t-t_n)u^{n+\frac{1}{2}}(x) - S(t_{n+1}-t_n)u^{n+\frac{1}{2}}(x)||\partial_t\rho_{\delta_0}(t-s)|\,dt\,ds\,dx \Big]\notag\\
       &\quad + C(|\partial_t\psi|)\,\mathbb{E}\Big[ \sum_{n=0}^{N-1} \int_{\mathbb{R}^d} \int_{t_n}^{t_{n+1}}\big |S(t-t_n)u^{n+\frac{1}{2}}(x) - S(t_{n+1}-t_n)u^{n+\frac{1}{2}}(x)|\,dt\,dx \Big]\notag\\
       & \le C(\psi)\,\mathbb{E}\Big[ \sum_{n=0}^{N-1} \int_0^T \int_{t_n}^{t_{n+1}} |u^{n+\frac{1}{2}}|_{BV(\mathbb{R}^d)}\rho(|t-t_{n+1}|)\,|\partial_t\rho_{\delta_0}(t-s)|\,dt\,ds\Big]\notag\\
       & \quad + C(|\partial_t\psi|)\,\mathbb{E}\Big[ \sum_{n=0}^{N-1}  \int_{t_n}^{t_{n+1}} |u^{n+\frac{1}{2}}|_{BV(\mathbb{R}^d)}\rho(|t-t_{n+1}|)dt\Big]\notag\\
       &\le C(\psi)\,\mathbb{E}\Big[ \sum_{n=0}^{N-1} \int_0^T \int_{t_n}^{t_{n+1}} |u^{n+\frac{1}{2}}|_{BV(\mathbb{R}^d)}\rho(\Delta t)\,|\partial_t\rho_{\delta_0}(t-s)|\,dt\,ds\Big]\notag\\
       &  \quad + C(|\partial_t\psi|)\,\mathbb{E}\Big[ \sum_{n=0}^{N-1}  \int_{t_n}^{t_{n+1}} |u^{n+\frac{1}{2}}|_{BV(\mathbb{R}^d)}\rho(\Delta t)dt\Big]
       \le C(\delta_0) \mathbb{E}\Big[|u_0|_{BV(\mathbb{R}^d)}\Big] \sqrt{\Delta t}\,. \notag 
\end{align}

Since $\psi(\cdot,\cdot)\in C_c^{1,2}([0, \infty) \times \mathbb{R}^d)$, we have,  by Lemma \ref{lem:average_time_continuity-approximate-solution} 
\begin{align}
 \mathcal{J}_{11}
       & \le \mathbb{E}\Big[ \int_{\mathbb{Q}_T} \sum_{n=0}^{N-1} \int_{\mathbb{R}^d} \int_{t_n}^{t_{n+1}}\big |u_{\Delta t}(t_{n+1},x) - u_{\Delta t}(t,x)|\psi(t,x)|\partial_t\rho_{\delta_0}(t-s)|\varrho_\delta(x-y)\,dt\,dx\,ds\,dy \Big]\notag \\
       &+ \mathbb{E}\Big[ \int_{\mathbb{Q}_T} \sum_{n=0}^{N-1} \int_{\mathbb{R}^d} \int_{t_n}^{t_{n+1}}\big |u_{\Delta t}(t_{n+1},x) - u_{\Delta t}(t,x)| |\partial_t\psi(t,x)|\rho_{\delta_0}(t-s)\varrho_\delta(x-y)\,dt\,dx\,ds\,dy \Big]\notag \\ 
       & \le  C\mathbb{E}\Big[ \sum_{n=0}^{N-1} \int_{K_x}\int_0^T \int_{t_n}^{t_{n+1}}\big |u_{\Delta t}(t_{n+1},x) - u_{\Delta t}(t,x)||\partial_t\rho_{\delta_0}(t-s)|\,ds\,dt\,dx \Big]\notag\\
       &+ C\mathbb{E}\Big[ \sum_{n=0}^{N-1} \int_{K_x} \int_{t_n}^{t_{n+1}}\big |u_{\Delta t}(t_{n+1},x) - u_{\Delta t}(t,x)|\,dt\,dx \Big]\notag\\
       & \le C\sum_{n=0}^{N-1} \int_0^T \int_{t_n}^{t_{n+1}}\sqrt{|t-t_{n+1}}\,|\partial_t\rho_{\delta_0}(t-s)|\,dt\,ds  + C \sum_{n=0}^{N-1}  \int_{t_n}^{t_{n+1}} \sqrt{|t-t_{n+1}|}\,dt \notag\\
       &\le C(\delta_0)\sqrt{\Delta t}. \notag   
\end{align}
A similar arguments as invoked in  $\mathcal{J}_{10}$ and $\mathcal{J}_{11}$ along with  \eqref{eqlipF} revels that 
\begin{align}
 \mathcal{J}_{12}&\le  \mathbb{E}\Big[ \int_{\mathbb{Q}_T}\sum_{n=0}^{N-1} \int_{\mathbb{R}^d} \int_{t_n}^{t_{n+1}} |\widetilde{u}_{\Delta t}(t,x) - u_{\Delta t}(t_{n+1},x)|\Big\{|\nabla_x\varrho_\delta(x-y)| + |\nabla_x\psi(t,x)|\Big\}\notag\\ & \hspace{3cm} \times\rho_{\delta_0}(t-s)\,dt\,dx\,ds\,dy \Big]
 \le  C(\delta)\sqrt{\Delta t}\,,\notag \\
         \mathcal{J}_{13} &\le  \mathbb{E}\Big[\int_{\mathbb{R}^d}\sum_{n=0}^{N-1} \int_{\mathbb{R}^d} \int_{t_n}^{t_{n+1}} |u_{\Delta t}(t_{n+1},x) - u_{\Delta t}(t,x)| \Big\{|\nabla_x\varrho_\delta(x-y)| + |\nabla_x\psi(t,x)|\Big\}\,dt\,dx\,dy \Big]\notag \\ &\le   C(\delta)\sqrt{\Delta t}.\notag
         \end{align}
Next we approximate the error terms occurring due to non-local terms.  We re-arrange the terms $\mathcal{J}_{14}$ and $\mathcal{J}_{15}$ as follows.
\begin{align}
 \mathcal{J}_{14} = &-  \mathbb{E}\Big[\int_{\mathbb{Q}_T}\int_\mathbb{R}\sum_{n=0}^{N-1} \int_{\mathbb{R}^d} \int_{t_n}^{t_{n+1}}  \big[\mathcal{L}_\theta^{\bar{r}}[\widetilde{u}_{\Delta t}(t,\cdot)](x)  -\mathcal{L}_\theta^{\bar{r}}[u_{\Delta t}(t_{n+1},\cdot)](x)]\notag\\ & \hspace{3cm} \times \beta'(\widetilde{u}_{\Delta t}(t,x) - k) \varphi_{\delta_0, \delta}(t,x,s,y) J_l(u_\eps^\kappa(s,y) - k)\,dt\,dx\,dk\,ds\,dy \Big]\notag\\
 & -  \mathbb{E}\Big[\int_{\mathbb{Q}_T}\int_\mathbb{R}\sum_{n=0}^{N-1} \int_{\mathbb{R}^d} \int_{t_n}^{t_{n+1}} \mathcal{L}_\theta^{\bar{r}}[u_{\Delta t}(t_{n+1},\cdot)](x)\big( \beta'(\widetilde{u}_{\Delta t}(t,x) - k) - \beta'(u_{\Delta t}(t_{n+1},x) -k)\big)\notag\\ & \hspace{3cm} \times  \varphi_{\delta_0, \delta}(t,x,s,y) J_l(u_\eps^\kappa(s,y) - k)\,dt\,dx\,dk\,ds\,dy \Big]\notag \\
 &
=: \mathcal{J}_{14}^1 + \mathcal{J}_{14}^2.\notag
  \end{align}
    \begin{align}
          \mathcal{J}_{15} = &-  \mathbb{E}\Big[\int_{\mathbb{Q}_T}\int_\mathbb{R}\sum_{n=0}^{N-1} \int_{\mathbb{R}^d} \int_{t_n}^{t_{n+1}}  \big[\mathcal{L}_\theta^{\bar{r}}[u_{\Delta t}(t_{n+1},\cdot)](x) -\mathcal{L}_\theta^{\bar{r}}[u_{\Delta t}(t,\cdot)](x)\big]\beta'(u_{\Delta t}(t_{n+1},x) - k)\notag\\ & \hspace{4cm} \times \varphi_{\delta_0, \delta}(t,x,s,y) J_l(u_\eps^\kappa(s,y) - k)\,dt\,dx\,dk\,ds\,dy \Big]\notag\\
          &- \mathbb{E}\Big[\int_{\mathbb{Q}_T}\int_\mathbb{R}\sum_{n=0}^{N-1} \int_{\mathbb{R}^d} \int_{t_n}^{t_{n+1}}  \big[\mathcal{L}_\theta^{\bar{r}}[u_{\Delta t}(t,\cdot)](x)\big(\beta'(u_{\Delta t}(t_{n+1},x) - k) -\beta'(u_{\Delta t}(t,x) - k)\big)\notag \\ & \hspace{4cm} \times \varphi_{\delta_0, \delta}(t,x,s,y) J_l(u_\eps^\kappa(s,y) - k)\,dt\,dx\,dk\,ds\,dy \Big] \notag \\
          &=:\mathcal{J}_{15}^1 + \mathcal{J}_{15}^2.\notag
  \end{align}
  We approximate the terms $\mathcal{J}_{14}^1$ and $\mathcal{J}_{14}^2$. Using Lemmas \ref{lem: prpertiesof S}, \ref{lem:bv bound-approximate-solution} and  \ref{lem:l-infinity bound-approximate-solution} together with triangle inequality and  \eqref{inq:bound-frac-z}, we get
  \begin{align}
          \mathcal{J}_{14}^1 = &- \mathbb{E}\Big[\int_{\mathbb{Q}_T}\int_\mathbb{R}\sum_{n=0}^{N-1} \int_{\mathbb{R}^d} \int_{t_n}^{t_{n+1}} \mathcal{L}_\theta^{\bar{r}}[\beta'(\widetilde{u}_{\Delta t}(t,\cdot) - u_\eps^\kappa(s,y) + k)\varrho_\delta(\cdot-y)\psi(t,\cdot)](x)\notag\\ & \hspace{4cm} \times\big(\widetilde{u}_{\Delta t}(t,x)  -u_{\Delta t}(t_{n+1},x)\big)\rho_{\delta_0}(t-s) J_l(k)\,dt\,dx\,dk\,ds\,dy \Big]\notag\\
          \le\, &C(\beta', \psi, \delta )\,\mathbb{E}\Big[\int_{\mathbb{Q}_T}\sum_{n=0}^{N-1} \int_{\mathbb{R}^d} \int_{t_n}^{t_{n+1}}  |\widetilde{u}_{\Delta t}(t,x)  -u_{\Delta t}(t_{n+1},x)|\Big|\int_{|z|>\bar{r}}\frac{1}{|z|^{d+2\theta}}\,dz\Big|\rho_{\delta_0}(t-s)\,dt\,dx\,ds\,dy \Big]\notag\\
          \le & C(\beta', \psi, \delta, \bar{r} )\,\mathbb{E}\Big[\sum_{n=0}^{N-1}\int_{t_n}^{t_{n+1}}|u^{n+\frac{1}{2}}|_{BV(\mathbb{R}^d)}\rho(|t-t_{n+1}|)dt\Big] \le C(\delta, \bar{r})\sqrt{\Delta t}, \label{eq:j141} \\
 \mathcal{J}_{14}^2 \leq  &\,C(\beta'')\,\mathbb{E}\Big[\int_{\mathbb{Q}_T}\sum_{n=0}^{N-1} \int_{\mathbb{R}^d} \int_{t_n}^{t_{n+1}} |\mathcal{L}_\theta^{\bar{r}}[u_{\Delta t}(t_{n+1},\cdot)](x)| |\widetilde{u}_{\Delta t}(t,x) - u_{\Delta t}(t_{n+1},x)| \varphi_{\delta_0, \delta}(t,x,s,y)\,dt\,dx\,ds\,dy \Big]\notag\\
          \le \,&C(\beta'',\psi, \widetilde{M}, \bar{r})\,\mathbb{E}\Big[\sum_{n=0}^{N-1} \int_{\mathbb{R}^d} \int_{t_n}^{t_{n+1}} |\widetilde{u}_{\Delta t}(t,x) - u_{\Delta t}(t_{n+1},y)|\,dt\,dx \Big] \leq C(\bar{r}, \xi) \sqrt{\Delta t}\,. \label{eq:j142}
  \end{align}
  The error terms $\mathcal{J}_{15}^1$ and $\mathcal{J}_{15}^2$ are approximated as follows.
   \begin{align}
          \mathcal{J}_{15}^1 = &-  \mathbb{E}\Big[\int_{\mathbb{Q}_T}\int_\mathbb{R}\sum_{n=0}^{N-1} \int_{\mathbb{R}^d} \int_{t_n}^{t_{n+1}}  \big(u_{\Delta t}(t_{n+1},x) - u_{\Delta t}(t,x)\big)\rho_{\delta_0}(t-s)\notag\\ & \hspace{1cm}\times \mathcal{L}_\theta^{\bar{r}}[\beta'(u_{\Delta t}(t_{n+1},\cdot) - k)\varrho_\delta(\cdot-y)\psi(t,\cdot)](x) J_l(u_\eps^\kappa(s,y) - k)\,dt\,dx\,dk\,ds\,dy \Big]\notag \\
          \le \,& C(\beta',\delta,\psi)\,\mathbb{E}\Big[\int_{\mathbb{R}^d}\sum_{n=0}^{N-1} \int_{\mathbb{R}^d} \int_{t_n}^{t_{n+1}}  |u_{\Delta t}(t_{n+1},x) - u_{\Delta t}(t,x)|\Big|\int_{|z|>\bar{r}}\frac{1}{|z|^{d+2\theta}}\,dz\Big|\,dt\,dx\,dy \Big]\notag\\
          \le\, & C(\beta',\delta,\psi, \bar{r})\,\mathbb{E}\Big[\sum_{n=0}^{N-1} \int_{\mathbb{R}^d} \int_{t_n}^{t_{n+1}}  |u_{\Delta t}(t_{n+1},x) - u_{\Delta t}(t,x)|\,dt\,dx \Big] \le C(\delta,\bar{r})\sqrt{\Delta t}, \label{eq:j151} \\
           \mathcal{J}_{15}^2 \le  \,&C(\beta'')\, \mathbb{E}\Big[\int_{\mathbb{Q}_T}\sum_{n=0}^{N-1} \int_{\mathbb{R}^d} \int_{t_n}^{t_{n+1}}  |\mathcal{L}_\theta^{\bar{r}}[u_{\Delta t}(t,\cdot)](x)|\,|u_{\Delta t}(t_{n+1},x) - u_{\Delta t}(t,x)|\notag\\ & \hspace{4cm} \times \rho_{\delta_0}(t-s)\varrho_\delta(x-y)\psi(t,x)\,dt\,dx\,ds\,dy \Big]\notag\\
           \le \,& C(\beta'',\psi, \widetilde{M}, \bar{r})\, \mathbb{E}\Big[\sum_{n=0}^{N-1} \int_{\mathbb{R}^d} \int_{t_n}^{t_{n+1}}|u_{\Delta t}(t_{n+1},x) - u_{\Delta t}(t,x) |\,dt\,dx \Big] \le C(\bar{r}, \xi)\sqrt{\Delta t}\,, \label{eq:152}
  \end{align}
  where we have used Lemmas \ref{lem:l-infinity bound-approximate-solution} and \ref{lem:average_time_continuity-approximate-solution}.
  \vspace{0.2cm}
  
  Using \eqref{fractionalbound}, we estimate the following error terms.
  \begin{align}
          \mathcal{J}_{16} 
          \le \,& C(\beta')\, \mathbb{E}\Big[\int_{\mathbb{Q}_T}\sum_{n=0}^{N-1} \int_{\mathbb{R}^d} \int_{t_n}^{t_{n+1}} |\widetilde{u}_{\Delta t}(t,x)) - u_{\Delta t}(t_{n+1},x)| |\mathcal{L}_{\theta,\bar{r}}[\varrho_\delta(\cdot-y)\psi(t,\cdot)](x)|\notag\\ &\hspace{6cm} \times \rho_{\delta_0}(t-s)\,dt\,dx\,ds\,dy \Big]\notag\\
          \le & \,C(\beta',\psi,\delta,\bar{r}^a)\, \mathbb{E}\Big[\sum_{n=0}^{N-1} \int_{\mathbb{R}^d} \int_{t_n}^{t_{n+1}} |\widetilde{u}_{\Delta t}(t,x) - u_{\Delta t}(t_{n+1},x)|\,dt\,dx \Big] \le C(\delta,\bar{r}^a)\sqrt{\Delta t},  \label{eq:j16} \\
  \mathcal{J}_{17}  \le & C(\beta')\mathbb{E}\Big[\int_{\mathbb{Q}_T} \sum_{n=0}^{N-1} \int_{\mathbb{R}^d} \int_{t_n}^{t_{n+1}} |u_{\Delta t}(t_{n+1},x) - u_{\Delta t}(t,x)| |\mathcal{L}_{\theta,\bar{r}}[\varphi_\delta(\cdot-y)\psi(t,\cdot)](x)|\notag \\& \hspace{6cm} \times \rho_{\delta_0}(t-s)\,dt\,dx\,ds\,dy \Big]\notag\\
  \le & \,C(\delta,\bar{r}^a)\,\mathbb{E}\Big[ \sum_{n=0}^{N-1} \int_{\mathbb{R}^d} \int_{t_n}^{t_{n+1}} |u_{\Delta t}(t_{n+1},x) - u_{\Delta t}(t,x)|\,dt\,dx \Big]\le C(\delta,\bar{r}^a)\sqrt{\Delta t}. \label{eq:j17}
  \end{align}
  Next we focus on to estimate the terms $\mathcal{J}_{18}$ and $\mathcal{J}_{20}$ coming from noise terms. An application of  Cauchy-Schwartz inequality, It\^o isometry along with the  assumptions \ref{A4} and \ref{A5} gives, 
  \begin{align}
  \mathcal{J}_{18} &= \mathbb{E}\Big[\int_{\mathbb{Q}_T}\int_\mathbb{R}\sum_{n=0}^{N-1}  \int_{\mathbb{R}^d} \int_{t_n}^{t_{n+1}}  \sigma(u_{\Delta t}(t,x))\beta'(u_{\Delta t}(t,x)- u_\eps^\kappa(s,y) + k)\varrho_\delta(x-y)\notag \\ &  \times\{(\rho_{\delta_0}(t_n-s) - \rho_{\delta_0}(t-s))\psi(t_n, x) + \rho_{\delta_0}(t-s)(\psi(t_n,x) -\psi(t,x))\} J_l(k)\,dW(t)\,dx\,dk\,ds\,dy\Big]\notag\\
  & \le  C \sum_{n=0}^{N-1}\mathbb{E}\Big[\int_{\mathbb{Q}_T} \Big|  \int_{t_n}^{t_{n+1}}  \sigma(u_{\Delta t}(t,x))\{(\rho_{\delta_0}(t_n-s) - \rho_{\delta_0}(t-s))\psi(t_n, x) \notag\\ &\hspace{3cm} + \rho_{\delta_0}(t-s)(\psi(t_n,x) -\psi(t,x))\} \,dW(t)\Big |\,ds\,dx\Big]\notag\\
  &\le C\sum_{n=0}^{N-1}  \Big(\mathbb{E}\Big[\int_{\mathbb{Q}_T}\Big( \int_{t_n}^{t_{n+1}}  \sigma(u_{\Delta t}(t,x))\{(\rho_{\delta_0}(t_n-s) - \rho_{\delta_0}(t-s))\psi(t_n, x) \} dW(t)\Big)^2\,ds\,dx\Big]\Big)^\frac{1}{2}\notag\\ 
  &+ C\sum_{n=0}^{N-1}  \Big(\mathbb{E}\Big[ \int_{\mathbb{Q}_T}\Big( \int_{t_n}^{t_{n+1}}  \sigma(u_{\Delta t}(t,x))\{(\rho_{\delta_0}(t-s)(\psi(t_n, x)  -(\psi(t,x))\} dW(t)\Big)^2\,ds\,dx\Big]\Big)^\frac{1}{2}\notag \\
  & \le  C\sum_{n=0}^{N-1} \Big(\mathbb{E}\Big[\int_{\mathbb{Q}_T} \int_{t_n}^{t_{n+1}}  \sigma^2(u_{\Delta t}(t,x))\{\rho_{\delta_0}(t_n-s) - \rho_{\delta_0}(t-s)\}^2|\psi(t_n, x) |^2\}\,dt\,ds\,dx\Big]\Big)^\frac{1}{2}\notag\\
  &+ C\sum_{n=0}^{N-1} \Big(\mathbb{E}\Big[\int_{\mathbb{Q}_T} \int_{t_n}^{t_{n+1}}  \sigma^2(u_{\Delta t}(t,x))|\rho_{\delta_0}(t-s)|^2|\psi(t_n, x)  -\psi(t,x)|^2\,dt\,ds\,dx\Big]\Big)^\frac{1}{2}\notag \\
  &\le  C\sum_{n=0}^{N-1} \Big(\mathbb{E}\Big[\underset{0 \le t \le T}{sup}||u_{\Delta t}(t,\cdot)||_{L^2(\mathbb{R}^d)}^2\Big]\int_0^T \int_{t_n}^{t_{n+1}} \int_0^1|\partial_t\rho_{\delta_0}(\lambda t +(1-\lambda)t_n - s)|^2|t- t_n|^2\,d\lambda\,dt\,ds \Big)^\frac{1}{2}\notag\\
  & + C\sum_{n=0}^{N-1} \Big(\mathbb{E}\Big[\underset{0 \le t \le T}{sup}||u_{\Delta t}(t,\cdot)||_{L^2(\mathbb{R}^d)}^2\Big]\int_0^T \int_{t_n}^{t_{n+1}} |\rho_{\delta_0}(t-s)|^2|t_n - t|^2\,dt\,ds \Big)^\frac{1}{2} \le C(\delta_0) \sqrt{\Delta t}\,.
  \notag
  \end{align}
  Using  It\^{o}-L\'{e}vy isometry, along with the assumptions \ref{A6}, \ref{A7} and replicating a similar lines of argument as done for $\mathcal{J}_{18}$ yields,
\begin{align}
 \mathcal{J}_{20} \le  C(\delta_0) \sqrt{\Delta t}\,.\notag
 \end{align}
 Using the assumptions \ref{A4}, \ref{A5} and boundedness property of approximate solutions, we estimate $\mathcal{J}_{19}$ as follows.
  \begin{align}
      \mathcal{J}_{19} &  =\mathbb{E}\Big[\int_{\mathbb{Q}_T}\int_\mathbb{R}  \sum_{n=0}^{N-1} \int_{\mathbb{R}^d} \int_{t_n}^{t_{n+1}}   \sigma^2(u_{\Delta t}(t,x))\beta''(u_{\Delta t}(s,x) - u_\eps^\kappa(s,y) + k)\varrho_\delta(x-y) \notag\\ & \hspace{3cm}  + \rho_{\delta_0}(t-s)(\psi(t_n,x) -\psi(t,x))\}J_l(k)\,dt\,dx\,dk\,ds\,dy \Big]\notag \\
      & \le C(\beta'')\sum_{n=0}^{N-1}\mathbb{E}\Big[\int_{\mathbb{Q}_T}\int_{t_n}^{t_{n+1}}   ||(u_{\Delta t}(t,\cdot))||_{L^\infty(\mathbb{R}^d)}^2  \Big\{ |(\rho_{\delta_0}(t_n-s) - \rho_{\delta_0}(t-s))| \psi(t_n, x) \notag \\ & \hspace{3cm}  + \rho_{\delta_0}(t-s)|\psi(t_n,x) -\psi(t,x)|\Big\} \,dt\,ds\,dx \Big]\notag\\
      &\le C(\beta'', \widetilde{M})\sum_{n=0}^{N-1}\Big\{\int_0^T\int_{t_n}^{t_{n+1}} \int_0^1 |\partial_t\rho_{\delta_0}(t\lambda +(1-\lambda)t_n-s)||t-t_n|d\lambda \,dt\,ds \notag\\ & \hspace{4cm}+ \int_0^T\int_{t_n}^{t_{n+1}}\rho_{\delta_0}(t-s)|t-t_n|\,dt\,ds\Big\} \le C(\xi,\delta_0)\,\Delta t . 
         \notag
\end{align}
Similarly using the assumptions \ref{A6} and \ref{A7}, we have
$$\mathcal{J}_{21} \le C(\xi,\delta_0)\,\Delta t .$$
Thus, to summarize the estimations of  error terms, we have the following lemma.
\begin{lem}\label{lem:error} We have
\begin{align}
    \sum_{j=10}^{21}\mathcal{J}_j\le C(\delta, \delta_0, \xi, \bar{r})\sqrt{\Delta t}. \notag
\end{align}
\end{lem}
Note that, since $\beta^{\prime\prime}$ and $J_l$ are non-negative functions  $\mathcal{I}_{11} \le 0$.  Let us focus on the term $\mathcal{I}_{10}$. In view of Cauchy-Schwartz-inequality, convolution property and \eqref{inq:uniform-1st-viscous}, we estimate $\mathcal{I}_{10}$ as follows.
\begin{align}\label{inq:i10}
   \mathcal{I}_{10}    \leq & \,\eps C(\beta',\psi) \mathbb{E} \Big[\int_{K_x}\int_{K_y}\int_0^T |\nabla_x u_\eps^\kappa| |\nabla_x\varrho_\delta(x-y)|\,ds\,dy\,dx \Big]\notag\\
   \leq &\,\eps^{\frac{1}{2}} C(\beta',\psi)\Big(  \eps \mathbb{E} \Big[\int_0^T \|\nabla_x u_\eps^\kappa(s)\|_{L^2(\R^d)}^2\,ds\Big]\Big)^{\frac{1}{2}}\Big( \mathbb{E} \Big[\int_{K_x}\int_{K_y}\int_0^T |\nabla_x\varrho_\delta(x-y)|^2\,ds\,dy\,dx \Big]\Big)^{\frac{1}{2}}\notag\\
   \le&\,C(\delta)\eps^{\frac{1}{2}}.
\end{align}

Let us consider the  additional term $\mathcal{I}_7$ which can be re-written as follows.
\begin{align}
   2\,\mathcal{I}_7 = \,& \mathbb{E} \Big[\int_{\mathbb{Q}_T^2} \int_{\mathbb{R}}
         \Big((\sigma(u_\eps(s,y)) \ast \tau_\kappa)^2-(\sigma(u_\eps(s,y)))^2\Big)\beta''(u_\eps^\kappa(s,y)- k)\notag\\&\hspace{8cm}\times\varphi_{\delta_0, \delta}J_l(u_{\Delta t}(t,x) - k)\,dk\,ds\,dt\,dy\,dx \Big] \notag\\
          +   \,& \mathbb{E} \Big[\int_{\mathbb{Q}_T^2} \int_{\mathbb{R}}
         \Big((\sigma(u_\eps(s,y)))^2-(\sigma(u_\eps^\kappa(s,y)))^2\Big)\beta''(u_\eps^\kappa(s,y)- k)\notag\\&\hspace{8cm}\times\varphi_{\delta_0, \delta}J_l(u_{\Delta t}(t,x) - k)\,dk\,ds\,dt\,dy\,dx \Big] \notag\\
  + \,& \mathbb{E} \Big[\int_{\mathbb{Q}_T^2} \int_{\mathbb{R}}
         \Big((\sigma(u_\eps^\kappa(s,y)))^2-(\sigma(u_\eps^\kappa(t,y)))^2\Big)\beta''(u_\eps^\kappa(s,y)- k)\notag\\&\hspace{8cm}\times\varphi_{\delta_0, \delta}J_l(u_{\Delta t}(t,x) - k)\,dk\,ds\,dt\,dy\,dx \Big] \notag\\
         + \,& \mathbb{E} \Big[\int_{\mathbb{Q}_T^2} \int_{\mathbb{R}}
        (\sigma(u_\eps^\kappa(t,y)))^2\Big(\beta''(u_\eps^\kappa(s,y)- u_{\Delta t}(t,x) +k)- \beta''(u_\eps^\kappa(t,y)- u_{\Delta t}(t,x) +k)\Big)\notag\\&\hspace{8cm}\times\varphi_{\delta_0, \delta}J_l(k)\,dk\,ds\,dt\,dy\,dx \Big] \notag\\ 
        +\,&\mathbb{E} \Big[\int_{\mathbb{R}^d}\int_{\mathbb{Q}_T} \int_{\mathbb{R}}
        (\sigma(u_\eps^\kappa(t,y)))^2 \beta''(u_\eps^\kappa(t,y)- u_{\Delta t}(t,x) +k)\Big(\int_0^T\rho_{\delta_0}(t-s)ds -1\Big)\notag\\&\hspace{8cm}\times\varrho_\delta(x-y)\psi(t,x)J_l(k)\,dk\,dt\,dy\,dx \Big] \notag\\ 
        +\,&  \mathbb{E} \Big[\int_{\mathbb{R}^d}\int_{\mathbb{Q}_T} \int_{\mathbb{R}}
        (\sigma(u_\eps^\kappa(t,y)))^2\Big(\beta''(u_\eps^\kappa(t,y)- u_{\Delta t}(t,x) +k)- \beta''(u_\eps^\kappa(t,y)- u_{\Delta t}(t,x))\Big)\notag\\&\hspace{8cm}\times\varrho_\delta(x-y)\psi(t,x)J_l(k)\,dk\,dt\,dy\,dx \Big] \notag\\
         +\,&  \mathbb{E} \Big[\int_{\mathbb{R}^d}\int_{\mathbb{Q}_T}
        (\sigma(u_\eps^\kappa(t,y)))^2 \beta''(u_\eps^\kappa(t,y)- u_{\Delta t}(t,x))\varrho_\delta(x-y)\psi(t,x)\,dt\,dy\,dx \Big] \notag\\
    &=: \sum_{i = 1}^{6}\mathcal{I}_{7}^i \,+ \,\mathbb{E} \Big[\int_{\mathbb{R}^d}\int_{\mathbb{Q}_T} 
        (\sigma(u_\eps^\kappa(t,y)))^2 \beta''(u_\eps^\kappa(t,y)- u_{\Delta t}(t,x))\varrho_\delta(x-y)\psi(t,x)\,dt\,dy\,dx \Big]. \notag
\end{align}
Using the fact that  $\beta''(r) \le \frac{M_2}{\xi}$ and $\sigma(u_\eps) \in L^2(\Omega \times \R^d)$ (for a.e. $t$) along with Cauchy-Schwartz inequality, properties of convolution and \eqref{inq:uniform-1st-viscous}, we have
\begin{align}
    \mathcal{I}_7^1 \le\,  &C(\xi,\psi)\,\mathbb{E} \Big[\int_0^T\int_{K_y} 
         \big|\sigma(u_\eps(s,y)) \ast \tau_\kappa-\sigma(u_\eps(s,y))\big|  \big|\sigma(u_\eps(s,y)) \ast \tau_\kappa +\sigma(u_\eps(s,y))\big|\,dy\,ds \Big]\notag \\
    \le \, &C(\xi,\psi)\,\Big(\mathbb{E} \Big[\int_0^T\int_{K_y}
         |\sigma(u_\eps(s,y)) \ast \tau_\kappa-\sigma(u_\eps(s,y))|^2\,dy\,ds \Big]\Big)^{\frac{1}{2}}\notag \\ & \hspace{3cm}\times\Big(\mathbb{E} \Big[\int_0^T\int_{K_y} 
           \big(|\sigma(u_\eps(s,y)) \ast \tau_\kappa|^2 +|\sigma(u_\eps(s,y))|^2\big)\,dy\,ds \Big]\Big)^{\frac{1}{2}}\notag\\
    \le\, &C(\xi,\psi)\,\Big(\underset{0 \le s \le T}{sup}\, \mathbb{E}\Big[||u_\eps(s)||_{L^2(\R^d)}^2\Big]\Big)^{\frac{1}{2}}\Big(\mathbb{E} \Big[\int_0^T\int_{K_y}
         |\sigma(u_\eps(s,y)) \ast \tau_\kappa-\sigma(u_\eps(s,y))|^2\,dy\,ds \Big]\Big)^{\frac{1}{2}}\notag \\
         \le \,&C(\xi,\psi)\,\Big(\mathbb{E} \Big[\int_0^T\int_{K_y}
         |\sigma(u_\eps(s,y)) \ast \tau_\kappa-\sigma(u_\eps(s,y))|^2\,dy\,ds \Big]\Big)^{\frac{1}{2}}\notag.
\end{align}
Using a similar lines of argument as done for $\mathcal{I}_7^1$, we bound $\mathcal{I}_7^2$ as follows.
\begin{align}
   \mathcal{I}_7^2 \le \, C(\xi, \psi)\Big(\mathbb{E} \Big[\int_0^T\int_{K_y}
         |u_\eps(s,y)-u_\eps^\kappa(s,y)|^2\,dy \,ds \Big]\Big)^{\frac{1}{2}}\notag.
\end{align}
We use Cauchy-Schwartz inequality, properties of convolutions, assumption \ref{A4}, estimations \eqref{inq:uniform-1st-viscous} and \eqref{eq:time continuity of uek} to get,
\begin{align}
    \mathcal{I}_7^3 \le \,&C(\xi,\psi)\,\mathbb{E} \Big[\int_0^T\int_0^T\int_{K_y}
         \big|\sigma(u_\eps^\kappa(s,y))-\sigma(u_\eps^\kappa(t,y))\big| \,\big|\sigma(u_\eps^\kappa(s,y)) +\sigma(u_\eps^\kappa(t,y))\big|\rho_{\delta_0}(t-s)\,dy\,dt\,ds\Big] \notag\\
        \le \,&C(\xi,\psi)\,\Big(\mathbb{E} \Big[\int_0^T\int_0^T\int_{K_y}|\sigma(u_\eps^\kappa(s,y))-\sigma(u_\eps^\kappa(t,y))|^2\rho_{\delta_0}(t-s)\,dy\,dt\,ds\Big]\Big)^{\frac{1}{2}}\notag \\ &\hspace{3cm}\times\Big(\mathbb{E} \Big[\int_0^T\int_0^T\int_{K_y}\big(|\sigma(u_\eps^\kappa(s,y))|^2+|\sigma(u_\eps^\kappa(t,y))|^2\big)\rho_{\delta_0}(t-s)\,dy\,dt\,ds\Big]\Big)^{\frac{1}{2}}\notag\\
         \le \,&C(\xi)\,\Big(\underset{0 \le s \le T}{sup}\, \mathbb{E}\Big[||u_\eps(s)||_{L^2(\R^d)}^2\Big]\Big)^{\frac{1}{2}}\Big(\mathbb{E} \Big[\int_0^T\int_0^T\int_{K_y}|u_\eps^\kappa(s,y)-u_\eps^\kappa(t,y)|^2\rho_{\delta_0}(t-s)\,dy\,dt\,ds\Big]\Big)^{\frac{1}{2}} \notag \\
        \le &\, C(\xi, \kappa, \eps)\sqrt{\delta_0}.\notag
\end{align}
Using Cauchy-Schwartz inequality, the assumption \ref{A4}, properties of convolutions, \eqref{inq:uniform-viscous-l2p} and \eqref{eq:time continuity of uek}, we have
\begin{align}
    \mathcal{I}_7^4 \le \,& C(\xi, \psi)\, \mathbb{E} \Big[\int_0^T\int_0^T\int_{K_y}
        |u_\eps^\kappa(t,y)|^2\,|u_\eps^\kappa(t,y)-u_\eps^\kappa(s,y)|\rho_{\delta_0}(t-s)\,dy\,ds\,dt\Big] \notag\\
        \le \,&C(\xi, \psi)\,\mathbb{E} \Big[\int_0^T\int_0^T\int_{K_y}
        ||u_\eps(t)||_{L^2(\mathbb{R}^d)}^2\,||\tau_\kappa||_{L^2(\mathbb{R}^d)}^2\,|u_\eps^\kappa(t,y)-u_\eps^\kappa(s,y)|\rho_{\delta_0}(t-s)\,dy\,dt\,ds \Big] \notag\\
        \le \, &C(\xi, \psi,\kappa)\Big(\mathbb{E} \Big[\int_0^T\int_0^T\int_{K_y}
        ||u_\eps(t,\cdot)||_{L^2(\mathbb{R}^d)}^4\rho_{\delta_0}(t-s)\,dy\,dt\,ds \Big]\Big)^{1/2}\notag\\& \hspace{4cm}\times\Big(\mathbb{E} \Big[\int_0^T\int_0^T\int_{K_y}
       |u_\eps^\kappa(t,y)-u_\eps^\kappa(s,y)|^2\rho_{\delta_0}(t-s)\,dy\,dt\,ds \Big]\Big)^{1/2}\notag\\
       \le &\,C(\xi,\psi,\kappa, |K_y|, \eps)\sqrt{\delta_0}\Big(\mathbb{E} \Big[
        \underset{0\le t\le T}{sup}\,||u_\eps(t)||_{L^2(\mathbb{R}^d)}^4\Big]\Big)^{1/2}\, \le \, C(\xi,\kappa,\eps)\sqrt{\delta_0}.\notag
\end{align}
One can use the assumption \ref{A4}, properties of convolutions and \eqref{inq:uniform-1st-viscous} to get the bound for  $\mathcal{I}_7^5$ 
and $\mathcal{I}_7^6$. 
\begin{align}
    \mathcal{I}_7^5\, \le \,&C(\xi)\delta_0 \quad \text{and } \quad \mathcal{I}_7^6\, \le \,C(\xi)l.\notag
\end{align}
We decompose $2\mathcal{J}_7$ as sum of $\mathcal{J}_7^{i},~(1\le i\le 4)$ where 
\begin{align}
    \mathcal{J}_7^1 := &\,\mathbb{E} \Big [\int_{\mathbb{Q}_T^2}\int_\mathbb{R} \sigma^2(u_{\Delta t}(t,x))\Big(\beta''(u_{\Delta t}(t,x) - u_\eps^\kappa(s,y)+k)-\beta''(u_{\Delta t}(t,x) - u_\eps^\kappa(t,y)+k)\Big)\notag \\& \hspace{8cm}\times\varphi_{\delta_0, \delta}J_l(k)\,dk\,ds\,dt\,dy\,dx \Big]\,,\notag\\
       \mathcal{J}_7^2 := &  \, \mathbb{E} \Big [\int_{\mathbb{R}^d}\int_{\mathbb{Q}_T}\int_\mathbb{R} \sigma^2(u_{\Delta t}(t,x))\beta''(u_{\Delta t}(t,x) - u_\eps^\kappa(t,y)+k)\Big(\int_0^T\rho_{\delta_0}(t-s)ds -1\Big)\notag \\&\hspace{4cm}\times\varrho_{\delta}(x-y)\psi(t,x)J_l(k)\,dk\,dt\,dy\,dx \Big]\,, \notag \\
         \mathcal{J}_7^3 := &\,\mathbb{E} \Big [\int_{\mathbb{R}^d}\int_{\mathbb{Q}_T}\int_\mathbb{R} \sigma^2(u_{\Delta t}(t,x))\Big(\beta''(u_{\Delta t}(t,x) - u_\eps^\kappa(t,y)+k)-\beta''(u_{\Delta t}(t,x) - u_\eps^\kappa(t,y))\Big)\notag \\& \hspace{6cm}\times\varrho_{\delta}(x-y)\psi(t,x)J_l(k)\,dk\,dt\,dy\,dx \Big]\,,\notag\\
         \mathcal{J}_7^4 :=  &\,\mathbb{E} \Big [\int_{\mathbb{R}^d}\int_{\mathbb{Q}_T}\sigma^2(u_{\Delta t}(t,x))\beta''(u_{\Delta t}(t,x) - u_\eps^\kappa(t,y))\varrho_{\delta}(x-y)\psi(t,x)\,dt\,dy\,dx \Big]\,. \notag 
     \end{align}
Using Lemmas \ref{lem:l-infinity bound-approximate-solution}, and \ref{lem:average-time-cont-viscous} together with the property \ref{A4}, one has 
\begin{align}
    \mathcal{J}_7^1 \le &C(\xi, \psi)\,\mathbb{E} \Big [\int_0^T\int_0^T\int_{K_y}||u_{\Delta t}(t, \cdot)||_{L^\infty}^2|u_\eps^\kappa(s,y)-u_\eps^\kappa(t,y)|\rho_{\delta_0}(t-s)\,dy\,dt\,ds\Big]\notag\\
    \le \,&\, C(\xi, \psi, \widetilde{M})\,\mathbb{E} \Big [\int_0^T\int_0^T\int_{K_y}|u_\eps^\kappa(s,y)-u_\eps^\kappa(t,y)|\rho_{\delta_0}(t-s)\,dy\,dt\,ds\Big]\notag \, \le C(\xi, \kappa, \eps)\sqrt{\delta_0}.
\end{align}
In the view of Lemma \ref{lem:l-infinity bound-approximate-solution} and  the assumption \ref{A4}, it is easy to observe that
\begin{align}
    \mathcal{J}_7^2 \le C(\xi)\delta_0, \quad \text{and} \quad \mathcal{J}_7^3 \le C(\xi)l.\notag
\end{align}
Combining the above estimations, we have the following Lemma.
\begin{lem}\label{lem:7}
\begin{align}
    \mathcal{I}_7 + \mathcal{J}_7 \le & \,\frac{1}{2}\,\mathbb{E} \Big[\int_{\mathbb{R}^d}\int_{\mathbb{Q}_T} 
        \{(\sigma(u_\eps^\kappa(t,y)))^2 + (\sigma(u_{\Delta t}(t,x))^2\} \beta''(u_\eps^\kappa(t,y)- u_{\Delta t}(t,x))\varrho_\delta(x-y)\psi(t,x)\,dt\,dy\,dx \Big] \notag \\
         & \quad +\,C(\xi)\,\Big(\mathbb{E} \Big[\int_0^T\int_{K_y}
         |\sigma(u_\eps(s,y)) \ast \tau_\kappa-\sigma(u_\eps(s,y))|^2\,dy\,ds \Big]\Big)^{\frac{1}{2}}\notag\\
        & \qquad \quad + \,C(\xi)\Big(\mathbb{E} \Big[\int_0^T\int_{K_y}
         |u_\eps(s,y)-u_\eps^\kappa(s,y)|^2\,dy \,ds \Big]\Big)^{\frac{1}{2}}
        + \, C(\xi, \kappa, \eps)\sqrt{\delta_0} + C(\xi)l.\notag
\end{align}
\end{lem}

Now we consider the additional terms arising due to jump noise. We re-write $\mathcal{I}_9$ to get,
\begin{align}
    \mathcal{I}_9 =\, &\mathbb{E} \Big[\int_{\mathbb{Q}_T^2} \int_{\mathbb{R}}
         \int_{|z| > 0} \int_0^1 (1 -\lambda) \big((\eta(u_\eps(s,y);z) \ast \tau_\kappa)^2- (\eta(u_\eps(s,y);z))^2\big)\notag \\ &\hspace{1cm}\times\beta''(u_\eps^\kappa(s,y) + \lambda(\eta(u_\eps(s,y);z)\ast \tau_\kappa) - k)  \varphi_{\delta_0, \delta}J_l(u_{\Delta t}(t,x) - k)\,d\lambda\,m(dz)\,dk\,ds\,dt\,dx\,dy \Big] \notag\\
        +  &\,\mathbb{E} \Big[\int_{\mathbb{Q}_T^2} \int_{\mathbb{R}}
         \int_{|z| > 0} \int_0^1 (1 -\lambda) \big((\eta(u_\eps(s,y);z))^2- (\eta(u_\eps^\kappa(s,y);z))^2\big)\notag \\ &\hspace{1cm}\times\beta''(u_\eps^\kappa(s,y) + \lambda(\eta(u_\eps(s,y);z)\ast \tau_\kappa) - k)  \varphi_{\delta_0, \delta}J_l(u_{\Delta t}(t,x) - k)\,d\lambda\,m(dz)\,dk\,ds\,dt\,dx\,dy \Big]\notag\\ 
         +  &\,\mathbb{E} \Big[\int_{\mathbb{Q}_T^2} \int_{\mathbb{R}}
         \int_{|z| > 0} \int_0^1 (1 -\lambda) (\eta(u_\eps^\kappa(s,y);z))^2\Big(\beta''(u_\eps^\kappa(s,y) + \lambda(\eta(u_\eps(s,y);z)\ast \tau_\kappa) - k) \notag \\ &\hspace{1cm}- \beta''(u_\eps^\kappa(s,y) + \lambda(\eta(u_\eps(s,y);z)) - k)\Big)\varphi_{\delta_0, \delta}J_l(u_{\Delta t}(t,x) - k)\,d\lambda\,m(dz)\,dk\,ds\,dt\,dx\,dy \Big]\notag\\ 
         + &\,\mathbb{E} \Big[\int_{\mathbb{Q}_T^2} \int_{\mathbb{R}}
         \int_{|z| > 0} \int_0^1 (1 -\lambda) (\eta(u_\eps^\kappa(s,y);z))^2\Big(\beta''(u_\eps^\kappa(s,y) + \lambda(\eta(u_\eps(s,y);z)) - k) \notag \\ &\hspace{1cm}-  \beta''(u_\eps^\kappa(s,y) + \lambda(\eta(u_\eps^\kappa(s,y);z))- k)\Big)\varphi_{\delta_0, \delta}J_l(u_{\Delta t}(t,x) - k)\,d\lambda\, m(dz)\,dk\,ds\,dt\,dx\,dy \Big]\notag\\
         + & \,\mathbb{E} \Big[\int_{\mathbb{Q}_T^2} \int_{\mathbb{R}}
         \int_{|z| > 0} \int_0^1 (1 -\lambda) \big((\eta(u_\eps^\kappa(s,y);z))^2 - (\eta(u_\eps^\kappa(t,y);z))^2\big) \beta''(u_\eps^\kappa(s,y) + \lambda(\eta(u_\eps^\kappa(s,y);z))- k)\notag \\ &\hspace{6cm}\times\varphi_{\delta_0, \delta}J_l(u_{\Delta t}(t,x) - k)\,d\lambda\, \,m(dz)\,dk\,ds\,dt\,dx\,dy \Big]\notag \\
          + & \,\mathbb{E} \Big[\int_{\mathbb{Q}_T^2} \int_{\mathbb{R}}
         \int_{|z| > 0} \int_0^1 (1 -\lambda)  (\eta(u_\eps^\kappa(t,y);z))^2\Big( \beta''(u_\eps^\kappa(s,y) + \lambda(\eta(u_\eps^\kappa(s,y);z))- k)\notag \\ &\hspace{2cm}-\beta''(u_\eps^\kappa(t,y) + \lambda(\eta(u_\eps^\kappa(t,y);z))- k)\Big)\varphi_{\delta_0, \delta}J_l(u_{\Delta t}(t,x) - k)\,d\lambda\, \,m(dz)\,dk\,ds\,dt\,dx\,dy \Big]\notag\\
         + & \,\mathbb{E} \Big[\int_{\mathbb{Q}_T}\int_{\R^d} \int_{\mathbb{R}}
         \int_{|z| > 0} \int_0^1 (1 -\lambda)  (\eta(u_\eps^\kappa(t,y);z))^2\beta''(u_\eps^\kappa(t,y) + \lambda(\eta(u_\eps^\kappa(t,y);z))- k)\Big)\notag\\&\hspace{3cm}\times\Big(\int_0^T\rho_{\delta_0}(t-s)ds -1\Big)\varrho_\delta(x-y)\psi(t,x)J_l(u_{\Delta t}(t,x) - k)\,d\lambda\,\,m(dz)\,dk\,dx\,dt\,dy \Big]\notag\\
         + & \,\mathbb{E} \Big[\int_{\mathbb{Q}_T}\int_{\R^d} \int_{\mathbb{R}}
         \int_{|z| > 0} \int_0^1 (1 -\lambda)  (\eta(u_\eps^\kappa(t,y);z))^2\Big(\beta''(u_\eps^\kappa(t,y) + \lambda(\eta(u_\eps^\kappa(t,y);z))- k)\notag\\&\hspace{1cm}- \beta''(u_\eps^\kappa(t,y) + \lambda(\eta(u_\eps^\kappa(t,y);z)))\Big)\varrho_\delta(x-y)\psi(t,x)J_l(u_{\Delta t}(t,x) - k)\,d\lambda\, \,m(dz)\,dk\,dx\,dt\,dy \Big]\notag\\
         + & \,\mathbb{E} \Big[\int_{\mathbb{Q}_T}\int_{\R^d}
         \int_{|z| > 0} \int_0^1 (1 -\lambda)  (\eta(u_\eps^\kappa(t,y);z))^2\beta''(u_\eps^\kappa(t,y) + \lambda\eta(u_\eps^\kappa(t,y);z)) - u_{\Delta t}(t,x) )\notag\\&\hspace{4cm}\times\varrho_\delta(x-y)\psi(t,x)\,d\lambda\,m(dz)\,dx\,dt\,dy \Big]\notag \\
         &=: \sum_{i=1}^{9}\mathcal{I}_9^{i}.\notag
\end{align}
We will estimate each of the above terms separately. First consider the term $  \mathcal{I}_9^1$. 
Using  Cauchy-Schwartz inequality, properties of convolutions along with the assumptions \ref{A6} and \ref{A8} and the estimation \eqref{inq:uniform-1st-viscous}, we have
\begin{align}
   \mathcal{I}_9^1  \le \,&C(\xi, \psi)\,\mathbb{E} \Big[\int_0^T\int_{K_y}
         \int_{|z| > 0}\big|\eta(u_\eps(s,y);z) \ast \tau_\kappa-
         \eta(u_\eps(s,y);z)\big|\notag\\&\hspace{6cm}\times\big|\eta(u_\eps(s,y);z) \ast \tau_\kappa +
         \eta(u_\eps(s,y);z)\big|\,m(dz)\,dy\,ds \Big]\notag\\
        \le  \,&C(\xi, \psi)\,\Big(\mathbb{E} \Big[\int_0^T\int_{K_y}
         \int_{|z| > 0}|\eta(u_\eps(s,y);z) \ast \tau_\kappa-
         \eta(u_\eps(s,y);z)|^2\,m(dz)\,dy\,ds \Big]\Big)^{\frac{1}{2}}\notag\\&\hspace{2cm}\times\Big(\mathbb{E} \Big[\int_0^T\int_{K_y}
         \int_{|z| > 0}\big(|\eta(u_\eps(s,y);z) \ast \tau_\kappa |^2+
         |\eta(u_\eps(s,y);z)|^2\big)\,m(dz)\,dy\,ds \Big]\Big)^{\frac{1}{2}}\notag\\
         \le  \,&C(\xi)\,\Big(\underset{s}{sup}\,\mathbb{E}\Big[||u_\eps(s)||_{L^2(\R^d)}^2\Big]\Big)^{\frac{1}{2}}\Big(\mathbb{E} \Big[\int_0^T\int_{K_y}
         \int_{|z| > 0}|\eta(u_\eps(s,y);z) \ast \tau_\kappa-
         \eta(u_\eps(s,y);z)|^2\,m(dz)\,dy\,ds \Big]\Big)^{\frac{1}{2}}\notag\\
         \le \,&C(\xi, \psi)\Big(\mathbb{E} \Big[\int_0^T\int_{K_y}
         \int_{|z| > 0}|\eta(u_\eps(s,y);z) \ast \tau_\kappa-
         \eta(u_\eps(s,y);z)|^2\,m(dz)\,dy\,ds \Big]\Big)^{\frac{1}{2}}\,.\notag
\end{align}
We use  Cauchy-Schwartz inequality and the bound of $\beta''(r)$ to have,
\begin{align}
    \mathcal{I}_9^3 \le &C(\xi, \psi)\,\mathbb{E} \Big[\int_0^T\int_{K_y}
         \int_{|z| > 0} (\eta(u_\eps^\kappa(s,y);z))^2|\eta(u_\eps(s,y);z)\ast \tau_\kappa-\eta(u_\eps(s,y);z))|\,m(dz)\,dy\,ds \Big]\notag\\ 
         \le\, & C(\xi, \psi)\,\Big(\mathbb{E} \Big[\int_0^T\int_{K_y}
         \int_{|z| > 0} (\eta(u_\eps^\kappa(s,y);z))^4\,m(dz)\,dy\,ds \Big]\Big)^{\frac{1}{2}}\notag\\&\hspace{2cm}\times\Big(\mathbb{E} \Big[\int_0^T\int_{K_y}
         \int_{|z| > 0}|\eta(u_\eps(s,y);z)\ast \tau_\kappa-\eta(u_\eps(s,y);z))|^2\,m(dz)\,dy\,ds \Big]\Big)^{\frac{1}{2}}\,.\notag
\end{align}
In the view of assumption \ref{A7} and \ref{A8}, we observe that
\begin{align}
&\mathbb{E}\Big[\int_0^T\int_{K_y}\int_{|z| > 0} (\eta(u_\eps^\kappa(s,y);z))^4\,m(dz)\,dy\,ds \Big]\notag\\ &\quad=  \mathbb{E}\Big[\int_0^T\int_{K_y}\int_{|z| > 0} \big(\eta(u_\eps^\kappa(s,y);z)\{\mathbf{1}_{|u_\eps^\kappa(s,y)| \le M}+ \mathbf{1}_{|u_\eps^\kappa(s,y)| > M}\}\big)^4\,m(dz)\,dy\,ds \Big]\notag\\
&\qquad = \mathbb{E}\Big[\int_0^T\int_{K_y}\int_{|z| > 0} |\eta(u_\eps^\kappa(s,y);z)|^4\mathbf{1}_{|u_\eps^\kappa(s,y)| \le M}\,m(dz)\,dy\,ds \Big]\notag\\
&\quad\quad\le\,(\lambda^\star)\mathbb{E} \Big[\int_0^T\int_{K_y}
         \int_{|z| > 0} |u_\eps^\kappa(s,y)|^4\mathbf{1}_{|u_\eps^\kappa(s,y)| \le M}\, (1 \wedge |z|^4)\,m(dz)\,dy\,ds \Big]\notag\\
         & \quad\quad \le\,C(\lambda^\star, M)\mathbb{E} \Big[\int_0^T\int_{K_y}
         \int_{|z| > 0} (1 \wedge |z|^2)\,m(dz)\,dy\,ds \Big]  \le C. \label{eq:etaI93}
         \end{align}
Thus, we have
$$\mathcal{I}_9^3 \le C(\xi, \psi)\,\Big(\mathbb{E} \Big[\int_0^T\int_{K_y}
         \int_{|z| > 0}|\eta(u_\eps(s,y);z)\ast \tau_\kappa-\eta(u_\eps(s,y);z))|^2\,m(dz)\,dy\,ds \Big]\Big)^{\frac{1}{2}}\,. $$
Following the estimations of $\mathcal{I}_9^1$ and $\mathcal{I}_9^3$ respectively, it is easy to see that
\begin{align}
    &\mathcal{I}_9^2,\, \mathcal{I}_9^4 \le C(\xi, \psi)\Big(\mathbb{E} \Big[\int_0^T\int_{K_y}
         |u_\eps^\kappa(s,y)
         -u_\eps(s,y)|^2\,dy\,ds \Big]\Big)^{\frac{1}{2}}\notag.
\end{align}
An application of Cauchy-Schwartz inequality, properties of convolution along with assumption \ref{A6}, \ref{A8}, estimations \eqref{inq:uniform-1st-viscous} and  \eqref{eq:time continuity of uek}, we have
\begin{align}
     \mathcal{I}_9^5 \le  & C(\xi,\psi) \,\mathbb{E} \Big[\int_0^T\int_0^T\int_{K_y}
         \int_{|z| > 0} \big|\eta(u_\eps^\kappa(s,y);z) - \eta(u_\eps^\kappa(t,y);z)\big| \notag \\ &\hspace{4cm}\big|\eta(u_\eps^\kappa(s,y);z)+ \eta(u_\eps^\kappa(t,y);z)\big|\rho_{\delta_0}(t-s)\,m(dz)\,dy\,dt\,ds \Big]\notag\\
         \le & \,C(\xi,\psi)\Big(\mathbb{E} \Big[\int_0^T\int_0^T\int_{K_y}
         \int_{|z| > 0}|\eta(u_\eps^\kappa(s,y);z) - \eta(u_\eps^\kappa(t,y);z)|^2\rho_{\delta_0}(t-s)\,m(dz)\,dy\,dt\,ds \Big]\Big)^{\frac{1}{2}}\notag\\&\hspace{1cm}\times\Big(\mathbb{E} \Big[\int_0^T\int_0^T\int_{K_y}
         \int_{|z| > 0}\big(|\eta(u_\eps^\kappa(s,y);z)|^2 + |\eta(u_\eps^\kappa(t,y);z)|^2\big)\rho_{\delta_0}(t-s)\,m(dz)\,dy\,dt\,ds \Big]\Big)^{\frac{1}{2}}\notag\\
         \le & \,C(\xi,\psi, \lambda^\star)\Big(\mathbb{E} \Big[\int_0^T\int_0^T\int_{K_y}
         \int_{|z| > 0}|u_\eps^\kappa(s,y) - u_\eps^\kappa(t,y)|^2\,(1 \wedge |z|^2)\rho_{\delta_0}(t-s)\,m(dz)\,dy\,dt\,ds \Big]\Big)^{\frac{1}{2}}\notag\\&\hspace{1cm}\times\Big(\mathbb{E} \Big[\int_0^T\int_0^T\int_{K_y}
         \int_{|z| > 0}\big(|u_\eps^\kappa(s,y)|^2 + |u_\eps^\kappa(t,y)|^2\big)(1 \wedge |z|^2)\rho_{\delta_0}(t-s)\,m(dz)\,dy\,dt\,ds \Big]\Big)^{\frac{1}{2}}\notag\\
         \le\, &C(\xi)\Big(\underset{s}{sup}\, \mathbb{E}\Big[||u_\eps(s)||_{L^2(\R^d)}^2\Big]\Big)^{\frac{1}{2}}\Big(\mathbb{E} \Big[\int_0^T\int_0^T\int_{K_y}|u_\eps^\kappa(s,y) - u_\eps^\kappa(t,y)|^2\,\rho_{\delta_0}(t-s)\,dy\,dt\,ds \Big]\Big)^{\frac{1}{2}}\notag \\
           \le \,&\, C(\xi,\kappa,\eps)\sqrt{\delta_0}.\notag
\end{align}
In view of the assumptions \ref{A6}-\ref{A8}, \eqref{eq:etaI93} and  \eqref{eq:time continuity of uek}, we bound $\mathcal{I}_9^6$  as follows.
\begin{align}
\mathcal{I}_9^6 \le & \,C(\xi, \psi)\Big(\,\mathbb{E} \Big[\int_0^T\int_0^T\int_{K_y}
         \int_{|z| > 0}(\eta(u_\eps^\kappa(t,y);z))^2 |u_\eps^\kappa(s,y)-u_\eps^\kappa(t,y)|\rho_{\delta_0}(t-s)\,m(dz)\,dy\,dt\,ds \Big]\notag\\
         + \,& \,\mathbb{E} \Big[\int_0^T\int_0^T\int_{K_y}
         \int_{|z| > 0}(\eta(u_\eps^\kappa(t,y);z))^2 |\eta(u_\eps^\kappa(s,y);z))- \eta(u_\eps^\kappa(t,y);z)|\rho_{\delta_0}(t-s)\,m(dz)\,dy\,dt\,ds \Big]\notag\\
         \le &\, C(\xi)\,\Big(\mathbb{E} \Big[\int_0^T\int_0^T\int_{K_y}
          |u_\eps^\kappa(s,y)-u_\eps^\kappa(t,y)|^2\rho_{\delta_0}(t-s)\,dy\,dt\,ds \Big]\Big)^\frac{1}{2}\notag\\
         \, & \quad + C(\xi)\Big(\mathbb{E} \Big[\int_0^T\int_0^T\int_{K_y}
         \int_{|z| > 0}|\eta(u_\eps^\kappa(s,y);z))- \eta(u_\eps^\kappa(t,y);z)|^2\rho_{\delta_0}(t-s)\,m(dz)\,dy\,dt\,ds \Big]\Big)^{\frac{1}{2}}\notag\\
         \le\, & C(\xi, \kappa, \eps)\sqrt{\delta_0}. \notag 
\end{align}
With assumptions \ref{A6} and \ref{A8} in hand, it is easy to observe that,  
\begin{align}
    \mathcal{I}_9^7 \le C(\xi)\delta_0, \quad \text{and} \quad \mathcal{I}_9^8 \le C(\xi)l.\notag
\end{align}
In the view of Lemma \ref{lem:l-infinity bound-approximate-solution},  the estimation  \eqref{eq:time continuity of uek}, the assumptions \ref{A6}, \ref{A8},  one can approximate  $\mathcal{J}_9$ as follows:
\begin{align}
    \mathcal{J}_9 = \, &\, \mathbb{E} \Big [\int_{\mathbb{Q}_T^2}\int_\mathbb{R}\int_{|z| > 0} \int_0^1  (1-\lambda)\eta^2(u_{\Delta t}(t,x);z)\Big(\beta''(u_{\Delta t}(t,x) + \lambda\eta(u_{\Delta t}(t,x);z)- u_\eps^\kappa(s,y) + k )\notag\\&- \beta''(u_{\Delta t}(t,x) + \lambda\eta(u_{\Delta t}(t,x);z) - u_\eps^\kappa(t,y) + k )\Big)\varphi_{\delta_0, \delta}(t,x,s,y)J_l(k)\,d\lambda m(dz)\,dk\,ds\,dt\,dx\,dy \Big] \notag\\
    + & \,  \mathbb{E} \Big [\int_0^T\int_{\R^d}\int_{\R^d}\int_\mathbb{R}\int_{|z| > 0} \int_0^1  (1-\lambda)\eta^2(u_{\Delta t}(t,x);z) \beta''(u_{\Delta t}(t,x) + \lambda\eta(u_{\Delta t}(t,x);z) - u_\eps^\kappa(t,y) + k )\notag\\&\hspace{2cm}\times\Big(\int_0^T\rho_{\delta_0}(t-s)\,ds - 1\Big)\varrho_\delta(x-y)\psi(t,x)J_l(k)\,d\lambda m(dz)\,dk\,dx\,dy\,dt \Big]\notag \\
    +&\,  \mathbb{E} \Big [\int_0^T\int_{\R^d}\int_{\R^d}\int_{\R}\int_{|z| > 0} \int_0^1  (1-\lambda)\eta^2(u_{\Delta t}(t,x);z) \Big(\beta''(u_{\Delta t}(t,x) + \lambda\eta(u_{\Delta t}(t,x);z) - u_\eps^\kappa(t,y) + k )\notag\\&\hspace{1cm}-\beta''(u_{\Delta t}(t,x) + \lambda\eta(u_{\Delta t}(t,x);z) - u_\eps^\kappa(t,y) )\Big)\varrho_\delta(x-y)\psi(t,x)J_l(k)\,d\lambda\, m(dz)\,dk\,dx\,dy\,dt \Big]\notag\\
    + &\, \mathbb{E} \Big [\int_0^T\int_{\R^d}\int_{\R^d}\int_{|z| > 0} \int_0^1  (1-\lambda)\eta^2(u_{\Delta t}(t,x);z)\beta''(u_{\Delta t}(t,x) + \lambda\eta(u_{\Delta t}(t,x);z) - u_\eps^\kappa(t,y))\notag\\&\hspace{6cm}\times\varrho_\delta(x-y)\psi(t,x)\,d\lambda\, m(dz)\,dx\,dy\,dt \Big]\notag\\
    \le   &\, C(\xi)\Big(\, \mathbb{E} \Big [\int_0^T\int_0^T\int_{K_y}\int_{|z| > 0}||u_{\Delta t}(t,\cdot)||_{L^\infty}^2| u_\eps^\kappa(s,y)- u_\eps^\kappa(t,y)|\,(1 \wedge |z|^2) \, m(dz)\,dy\,dt\,ds \Big] + (\delta_0 + l)\Big)\notag \\
    + &\, \mathbb{E} \Big [\int_{\R^d}\int_{\mathbb{Q}_T}\int_{|z| > 0} \int_0^1  (1-\lambda)\eta^2(u_{\Delta t}(t,x);z)\beta''(u_{\Delta t}(t,x) + \lambda\eta(u_{\Delta t}(t,x);z) - u_\eps^\kappa(t,y))\notag\\&\hspace{6cm}\times\varrho_\delta(x-y)\psi(t,x)\,d\lambda\, m(dz)\,dt\,dx\,dy \Big]\notag\\
   \le \, & \mathbb{E} \Big [\int_{\R^d}\int_{\mathbb{Q}_T}\int_{|z| > 0} \int_0^1  (1-\lambda)\eta^2(u_{\Delta t}(t,x);z)\beta''(u_{\Delta t}(t,x) + \lambda\eta(u_{\Delta t}(t,x);z) - u_\eps^\kappa(t,y))\notag\\&\hspace{3cm}\times\varrho_\delta(x-y)\psi(t,x)\,d\lambda m(dz)\,dt\,dx\,dy \Big] + C(\xi, \kappa, \eps)\sqrt{\delta_0} + C(\xi)l .\notag
\end{align}
We summarize the above estimations in the following Lemma.
\begin{lem}\label{lem:9}
\begin{align}
         &\mathcal{I}_{9} + \mathcal{J}_9 \le \, \mathbb{E} \Big[\int_{\mathbb{Q}_T} \int_{\mathbb{R}^d}
         \int_{|z| > 0} \int_0^1 (1 -\lambda) \big \{ \eta^2(u_{\Delta t}(t,x);z)\beta''(u_{\Delta t}(t,x) + \lambda\eta(u_{\Delta t}(t,x);z) - u_\eps^\kappa(t,y))\notag  \\ & + \eta^2(u_\eps^\kappa(t,y);z)\beta''(u_\eps^\kappa(t,y) + \lambda(\eta(u_\eps^\kappa(t,y);z) - u_{\Delta t}(t,x))  \big \}\varrho_\delta(x-y) \psi(t,x)\,d\lambda \,m(dz)\,dx\,dt\,dy \Big] \notag\\
         &\quad+\,C(\xi)\,\Big(\mathbb{E} \Big[\int_0^T\int_{K_y}
         \int_{|z| > 0}|\eta(u_\eps(s,y);z)\ast \tau_\kappa-\eta(u_\eps(s,y);z))|^2\,m(dz)\,dy\,ds \Big]\Big)^{\frac{1}{2}}\notag\\&\qquad+ \,C(\xi)\Big(\mathbb{E} \Big[\int_0^T\int_{K_y}
         |u_\eps^\kappa(s,y)
         -u_\eps(s,y)|^2\,dy\,ds \Big]\Big)^{\frac{1}{2}}+ C(\xi, \kappa, \eps)\sqrt{\delta_0} + C(\xi)l.\notag
\end{align}
\end{lem}

Next, we consider the It\^{o} integral terms. Note that for any two constant $t_1, t_2$ $\ge 0$ with $t_1 < t_2$,
\begin{align}\label{eq:Ito}
\begin{cases}
       \displaystyle\mathbb{E}\Big[X({t_1})\int_{t_1}^{t_2}J(t)dW(t)\Big] = 0,\\
       \displaystyle\mathbb{E}\Big[X(t_1)\int_{t_1}^{t_2}\int_{|z| > 0}\zeta(t,z)\tilde{N}(dz,dt)\Big] = 0,
\end{cases}
\end{align}
where $J,\, \zeta$ are predictable processes with $\mathbb{E}\Big[\displaystyle\int_0^T\int_{|z| > 0}\zeta^2(t,z)m(dz)dt\Big] < \infty$ and $X(\cdot)$ is an adapted process. In the view of \eqref{eq:Ito}, we have
\begin{align}
        \mathcal{I}_6 & = \int_{\mathbb{Q}_T} \int_{\mathbb{R}^d} \int_{\mathbb{R}} \mathbb{E} \Big[J_l(u_{\Delta t}(t,x) - k) \int_{s = t}^{s = \delta_0 + t}
         (\sigma(u_\eps(s,y)) \ast \tau_\kappa)\beta'(u_\eps^\kappa(s,y)- k)\notag \\&\hspace{8cm}\times\varphi_{\delta_0, \delta}(t,x,s,y)dW(s)\Big]\,dk\,dy\,dt\,dx = 0 \,.\notag
\end{align}
Now, we define 
\begin{align}
    \mathcal{M}[\beta, \varphi_{\delta, \delta_0}](s,y,k): = \int_{\mathbb{Q}_T} \sigma(u_{\Delta t}(t,x))\beta(u_{\Delta t}(t,x) -k )\varphi_{\delta_0, \delta}(t,x,s,y)\,dx\,dW(t). \notag
\end{align}
Regarding $\mathcal{M}[\beta, \varphi_{\delta, \delta_0}](t,x,k)$ we have the following lemma whose proof can be found in \cite{Majee-2014}.
\begin{lem}\label{lem:Ito-Identity}
The following identities hold: 
\begin{align}
    &\partial_k\mathcal{M}[\beta, \varphi_{\delta, \delta_0}](s,y,k) = \mathcal{M}[-\beta', \varphi_{\delta, \delta_0}](s,y,k),\notag\\
    &\partial_y\mathcal{M}[\beta, \varphi_{\delta, \delta_0}](s,y,k) =
    \mathcal{M}[\beta, \partial_y\varphi_{\delta, \delta_0}](s,y,k).\notag\\ 
    \text{Moreover},\notag\\
    &\underset{0 \leq t \leq T }{sup}\,\mathbb{E}\Big[||\mathcal{M}[\beta'', \varphi_{\delta, \delta_0}](t,\cdot,\cdot)||_{L^\infty(\mathbb{R}^d \times \mathbb{R})}^2\Big] \leq \frac{C}{\delta^{\frac{2}{p}}\xi^{q}\delta_0^{\frac{2(p-1)}{p}}},\notag\\
    &\underset{0 \leq t \leq T }{sup}\,\mathbb{E}\Big[||\mathcal{M}[\beta''', \varphi_{\delta, \delta_0}](t,\cdot,\cdot)||_{L^\infty(\mathbb{R}^d \times \mathbb{R})}^2\Big] \leq \frac{C}{\delta^{\frac{2}{p}}\xi^{\bar{q}}\delta_0^{\frac{2(p-1)}{p}}},\notag
\end{align}
where p is a positive integer of the form $p=2^k$ for some $k \in \mathbb{N}$ with $p \geq d + 3 $ and for some $q, \bar{q} > 0$ which only depends on $d$ and hence on $p$.
\end{lem}
Note that, in the view of Fubini's theorem and \eqref{eq:Ito},
\begin{align}
    \mathbb{E} \Big[\int_{\mathbb{Q}_T} \int_{\mathbb{R}} J_l(u_\eps^\kappa(s-\delta_0,y) -k)\int_{t= s-\delta_0}^{t=s}\sigma(u_{\Delta t}(t,x))\beta'(u_{\Delta t}(t,x) -k )\varphi_{\delta_0, \delta}\,dW(t)\,dk\,ds\,dx\Big] =0. \notag
\end{align}
Hence, $\mathcal{J}_6$ can be re-written as follows.
\begin{align}\label{eq:J6}
    \mathcal{J}_6 = \mathbb{E}\Big[ \int_{\mathbb{Q}_T} \int_{\mathbb{R}} \mathcal{M}[\beta', \varphi_{\delta, \delta_0}](s,y,k)\{J_l(u_\eps^\kappa(s,y) -k)-J_l(u_\eps^\kappa(s-\delta_0,y) -k)\,dk\,ds\,dy\}\Big].
\end{align}
To proceed further, we apply the $It\hat{o}$ formula to $J_l(u_\eps^\kappa(s,y)-k)$ to get
\begin{align}
        &J_l(u_\eps^\kappa(s,y) -k)-J_l(u_\eps^\kappa(s-\delta_0,y) -k)\notag \\ &= \int_{s-\delta_0}^{s}J_l'(u_\eps^\kappa(r,y) -k)\big( -\mathcal{L}_{\theta}[u_\eps \ast \tau_\kappa] - \text{div}_x (f(u_\eps) \ast \tau_\kappa) + \eps\Delta u_\eps^\kappa \big)\,dr \notag\\
        &+\int_{s-\delta_0}^{s}J_l'(u_\eps^\kappa(r,y) -k)(\sigma(u_\eps) \ast \tau_\kappa) dW(r)\notag \\
        &+ \frac{1}{2}\int_{s-\delta_0}^{s}J_l''(u_\eps^\kappa(r,y) - k)(\sigma(u_\eps) \ast \tau_\kappa)^2dr\notag\\
        &+ \int_{s-\delta_0}^{s}\int_{|z| >0 }\big(J_l(u_\eps^\kappa + (\eta(u_\eps;z)\ast\tau_\kappa) - k) - (J_l(u_\eps^\kappa(r,y) - k) \big) \tilde{N}(dz,dr) \notag\\
        &+ \int_{s-\delta_0}^{s}\int_{|z| >0 }\int_0^1(1-\lambda)J_l''(u_\eps^\kappa(r,y) - k +\lambda(\eta(u_\eps;z)\ast\tau_\kappa ) (\eta(u_\eps;z)\ast\tau_\kappa)^2\,d\lambda\, m(dz)\,dr\,. \notag
\end{align}
Therefore, \eqref{eq:J6} can be re-written as
\begin{align}
        \mathcal{J}_6 = & \mathbb{E}\Big[ \int_{\mathbb{Q}_T} \int_{\mathbb{R}} \mathcal{M}[\beta'', \varphi_{\delta, \delta_0}](s,y,k)\int_{s-\delta_0}^{s}J_l(u_\eps^\kappa(r,y) -k)\big( -\mathcal{L}_{\theta}[u_\eps \ast \tau_\kappa])\,dr\,dk\,ds\,dy \Big]\notag\\
        \,&+ \mathbb{E}\Big[ \int_{\mathbb{Q}_T} \int_{\mathbb{R}} \mathcal{M}[\beta'', \varphi_{\delta, \delta_0}](s,y,k)\int_{s-\delta_0}^{s}J_l(u_\eps^\kappa(r,y) -k)\big( - \text{div}_x (f(u_\eps) \ast \tau_\kappa)  \big)\,dr\,dk\,ds\,dy \Big]\notag\\
        &+\,\mathbb{E}\Big[ \int_{\mathbb{Q}_T} \int_{\mathbb{R}} \mathcal{M}[\beta'', \varphi_{\delta, \delta_0}](s,y,k)\int_{s-\delta_0}^{s}J_l(u_\eps^\kappa(r,y) -k)\eps\Delta u_\eps^\kappa\,dr\,dk\,ds\,dy \Big]\notag\\
        &+\frac{1}{2}\mathbb{E}\Big[ \int_{\mathbb{Q}_T} \int_{\mathbb{R}} \mathcal{M}[\beta''', \varphi_{\delta, \delta_0}](s,y,k)\int_{s-\delta_0}^{s}J_l(u_\eps^\kappa(r,y) - k)(\sigma(u_\eps) \ast \tau_\kappa)^2\,dr\,dk\,ds\,dy\Big]\notag\notag\\
        &+ \mathbb{E}\Big[ \int_{\mathbb{R}^d}\int_{\mathbb{Q}_T}\int_{s-\delta_0}^{s}\int_{\mathbb{R}} \sigma(u_{\Delta t}(r,x))(\sigma(u_\eps(r,y)) \ast \tau_\kappa)\beta'(u_{\Delta t}(r,x) -k )\notag\\&\hspace{6cm}\times\varphi_{\delta_0, \delta}(r,x,s,y) J_l'(u_\eps^\kappa(r,y) -k)\,dk\,dr\,ds\,dx\,dy\Big]\notag\\
        & =: \mathcal{J}_6^1 + \mathcal{J}_6^2 + \mathcal{J}_6^3 + \mathcal{J}_6^4 + \mathcal{J}_6^5 .\notag
\end{align} 
We have the following estimations of $\mathcal{J}_6^2$ and $\mathcal{J}_6^3$ due to \cite[ Lemma 5.4]{Majee-2015}.
$$\mathcal{J}_6^2 \le C(\delta, \xi)\delta_0^{1/p}, \hspace{.5cm} \text{and} \hspace{.5cm} \mathcal{J}_6^3 \le C(\delta,\xi)\delta_0^a.$$
for some $a>0$.
Using Lemma \ref{lem:Ito-Identity} and estimation of the term $\mathcal{A}_2$ in Lemma \ref{lem:average-time-cont-viscous}, we have
\begin{align}
    \mathcal{J}_6^1 \le &\,\mathbb{E}\Big[ \int_{\mathbb{Q}_T} \int_{\mathbb{R}} ||\mathcal{M}[\beta'', \varphi_{\delta, \delta_0}](s,\cdot,\cdot)||_{L^\infty(\mathbb{R}^d\times\mathbb{R})}\int_{s-\delta_0}^{s}J_l(u_\eps^\kappa(r,y) -k)|\mathcal{L}_{\theta}[u_\eps \ast \tau_\kappa]| \big)\,dr\,dk\,ds\,dy \Big]\notag\\
    \le &\,\mathbb{E}\Big[ \int_{\mathbb{Q}_T} \int_{s-\delta_0}^{s} ||\mathcal{M}[\beta'', \varphi_{\delta, \delta_0}](s,\cdot,\cdot)||_{L^\infty(\mathbb{R}^d\times\mathbb{R})}\,|\mathcal{L}_{\theta}[u_\eps \ast \tau_\kappa]|\,dr\,ds\,dy \Big]\notag\\
    \le \, &C\,\int_0^T\int_{s-\delta_0}^{s}\Big(\mathbb{E}\Big[  ||\mathcal{M}[\beta'', \varphi_{\delta, \delta_0}](s,\cdot,\cdot)||_{L^\infty(\mathbb{R}^d\times\mathbb{R})}^2\,\Big]\Big)^{1/2}\Big(\mathbb{E}\Big[\int_{K_y}|\mathcal{L}_{\theta}[u_\eps \ast \tau_\kappa]|^2\,dy \Big]\Big)^{1/2}\,dr\,ds\notag\\
    \le \, & \, C(\delta, \xi,\kappa, |K_y|)\delta_0^{1/p}.\notag
\end{align}
In the view of Lemma \ref{lem:Ito-Identity} and estimations \eqref{inq:uniform-viscous-l2p}, one can
approximate $\mathcal{J}_6^4$ as follows.
\begin{align}
    2\,\mathcal{J}_6^4 
    \le  &\,C\,\mathbb{E}\Big[ \int_{\mathbb{Q}_T} \int_{s-\delta_0}^{s} ||\mathcal{M}[\beta''', \varphi_{\delta, \delta_0}](s,\cdot,\cdot)||_{L^\infty(\mathbb{R}^d \times \mathbb{R})}|\sigma(u_\eps(r))\ast\tau_\kappa|^2\,dr\,ds\,dy\Big]\notag\\
    \le  &\,C\,\mathbb{E}\Big[ \int_0^T \int_{s-\delta_0}^{s} ||\mathcal{M}[\beta''', \varphi_{\delta, \delta_0}](s,\cdot,\cdot)||_{L^\infty(\mathbb{R}^d \times \mathbb{R})}||\sigma(u_\eps(r))\ast\tau_\kappa||_{L^2(\mathbb{R}^d)}^2\,dr\,ds\Big]\notag\\
    \le &C(\kappa) \int_0^T \int_{s-\delta_0}^{s} \Big(\mathbb{E}\Big[ ||\mathcal{M}[\beta''', \varphi_{\delta, \delta_0}](s,\cdot,\cdot)||_{L^\infty(\mathbb{R}^d \times \mathbb{R})}^2 \Big]\Big)^{1/2} \Big(\mathbb{E}\Big[\underset{r}{\sup} ||u_\eps(r)||_{L^2(\mathbb{R}^d)}^4\Big]\Big)^{1/2} \,dr\,ds\notag\\
    \le \,&C(\delta, \xi, \kappa, |K_y|)\delta_0^{1/p}.\notag
\end{align}
Now, we want to estimate the term $\mathcal{J}_6^5$ which is re-arranged as follows.
\begin{align}
        \mathcal{J}_6^5 
        & =\mathbb{E}\Big[\int_{\mathbb{R}^d} \int_{\mathbb{R}^d}\int_{s-\delta_0}^s\int_{\mathbb{R}} \sigma(u_{\Delta t}(r,x))(\sigma(u_\eps(r,y))\ast \tau_\kappa)\beta''(u_{\Delta t}(r,x) -k )\Big(\int_0^T\rho_{\delta_0}(r-s)ds -1\Big)\notag\\&\hspace{6cm}\times\varrho_\delta(x-y)\psi(r,x)J_l(u_\eps^\kappa(r,y) -k)\,dk\,dx\,dr\,dy\Big]\notag\\ 
        &+\mathbb{E}\Big[\int_{\mathbb{Q}_T} \int_{\mathbb{R}^d}\int_{\mathbb{R}} \sigma(u_{\Delta t}(r,x))(\sigma(u_\eps(r,y))\ast \tau_\kappa)\Big(\beta''(u_{\Delta t}(r,x) - u_\eps^\kappa(r,y) + k)\notag\\&\hspace{5cm}- \beta''(u_{\Delta t}(r,x) - u_\eps^\kappa(r,y)) \Big)\varrho_\delta(x-y)\psi(r,x)J_l(k)\,dk\,dx\,dr\,dy\Big]\notag\\ 
        &+\mathbb{E}\Big[\int_{\mathbb{Q}_T} \int_{\mathbb{R}^d}\sigma(u_{\Delta t}(r,x))\Big((\sigma(u_\eps(r,y))\ast \tau_\kappa)- \sigma(u_\eps(r,y))\Big)\beta''(u_{\Delta t}(r,x) - u_\eps^\kappa(r,y) )\notag \\&\hspace{6cm}\times\varrho_\delta(x-y)\psi(r,x)\,dx\,dr\,dy\Big]\notag\\ 
        &+\mathbb{E}\Big[\int_{\mathbb{Q}_T} \int_{\mathbb{R}^d}\sigma(u_{\Delta t}(r,x))\Big(\sigma(u_\eps(r,y))- \sigma(u_\eps^\kappa(r,y))\Big)\beta''(u_{\Delta t}(r,x) - u_\eps^\kappa(r,y) )\notag \\&\hspace{6cm}\times\varrho_\delta(x-y)\psi(r,x)\,dx\,dr\,dy\Big]\notag\\ 
        &+\mathbb{E}\Big[\int_{\mathbb{Q}_T} \int_{\mathbb{R}^d}\sigma(u_{\Delta t}(r,x)) \sigma(u_\eps^\kappa(r,y))\beta''(u_{\Delta t}(r,x) - u_\eps^\kappa(r,y) )\notag \\&\hspace{6cm}\times\varrho_\delta(x-y)\psi(r,x)\,dx\,dr\,dy\Big]\notag \\
        &=: \sum_{i=1}^{5} \mathcal{J}_6^{5,i}.\notag
\end{align}
We use Lemma \ref{lem:l-infinity bound-approximate-solution}, the assumption \ref{A4}, Cauchy-Schwartz inequality and \eqref{inq:uniform-1st-viscous} to have,
\begin{align}
    \mathcal{J}_6^{5,1} &\le C(\xi, \psi)\mathbb{E}\Big[\int_{K_x} \int_{K_y}\int_0^{\delta_0}||u_{\Delta t}(r,x)||_{L^\infty}(\sigma(u_\eps(r,y))\ast \tau_\kappa)\,dr\,dx\,dy\Big]\notag \\ &\le C(\xi, \psi, \widetilde{M})\delta_0\Big(\underset{r}{sup}\,\mathbb{E}\Big[||u_\eps(r)||_{L^2(\R^d)}^2\Big]\Big) \le C(\xi)\delta_0.\notag
\end{align}
Similarly, one can observe that $$\mathcal{J}_6^{5,2} \le C(\xi)l.$$
A replications of estimations of $\mathcal{I}_7^1$ and Lemma \ref{lem:l-infinity bound-approximate-solution} yields,
\begin{align}
    &\mathcal{J}_6^{5,3} \le C(\xi)\Big(\mathbb{E}\Big[\int_0^T\int_{K_y} |\sigma(u_\eps(r,y))\ast \tau_\kappa- \sigma(u_\eps(r,y))|^2\,dy\,dr\Big]\Big)^{\frac{1}{2}},\notag\\ 
    &\mathcal{J}_6^{5,4}\le C(\xi)\Big(\mathbb{E}\Big[\int_0^T\int_{K_y} |u_\eps^\kappa(r,y) - u_\eps(r,y)|^2\,dy\,dr\Big]\Big)^{\frac{1}{2}}.\notag
\end{align}
Combining Lemma \ref{lem:7} and the above estimations  for $\mathcal{I}_6$ and $\mathcal{J}_6$, we have the following result.
\begin{lem}\label{lem67}
\begin{align}
       &\mathcal{I}_6+ \mathcal{J}_6 + \mathcal{I}_7 + \mathcal{J}_7 \notag \\
       & \leq  \frac{1}{2}\,\mathbb{E} \Big[\int_{\mathbb{Q}_T} \int_{\mathbb{R}^d} \Big( \sigma(u_\eps^\kappa(t,y)) - \sigma(u_{\Delta t}(t,x) \Big)^2 \beta''(u_{\Delta t}(t,x) - u_\eps^\kappa(t,y)) \varrho_\delta(x-y)\psi(t,x)\,dx\,dt\,dy \Big]\notag \\
       & \quad +C\,(\xi)\Big(\mathbb{E}\Big[\int_0^T\int_{K_y} |\sigma(u_\eps(r,y))\ast \tau_\kappa- \sigma(u_\eps(r,y))|^2\,dy\,dr\Big]\Big)^{\frac{1}{2}} + \,C(\delta, \xi, \kappa, \eps)\delta_0^{\frac{1}{p}} \notag \\& \qquad +\,C(\xi)\,\Big(\mathbb{E}\Big[\int_0^T\int_{K_y} |u_\eps^\kappa(r,y) - u_\eps(r,y)|^2\,dy\,dr\Big]\Big)^{\frac{1}{2}} + C(\delta, \xi)\delta_0^a  + C(\xi)l\notag\\
       & \le C\,(\xi)\Big(\mathbb{E}\Big[\int_0^T\int_{K_y} |\sigma(u_\eps(r,y))\ast \tau_\kappa- \sigma(u_\eps(r,y))|^2\,dy\,dr\Big]\Big)^{\frac{1}{2}} + \,C(\delta, \xi, \kappa, \eps)\delta_0^{\frac{1}{p}}\notag \\&\quad+\,C(\xi)\,\Big(\mathbb{E}\Big[\int_0^T\int_{K_y} |u_\eps^\kappa(r,y) - u_\eps(r,y)|^2\,dy\,dr\Big]\Big)^{\frac{1}{2}}\, + C(\delta, \xi)\delta_0^a  + C(\xi)l + C\xi.\notag
\end{align}
\end{lem}

Next, we consider the stochastic integral terms having jump noise. Thanks to \eqref{eq:Ito}
\begin{align}
        \mathcal{I}_8 = &\int_{\mathbb{Q}_T}\int_{\mathbb{R}^d}\int_\mathbb{R}\int_0^1  \mathbb{E} \Big [ J_l(u_{\Delta t}(t,x) - k)\int_{s=t}^{s = \delta_0 + t} \int_{|z| > 0} (\eta(u_\eps(s,y);z) \ast \tau_\kappa)\beta'(u_\eps^\kappa(s,y)\notag  \\ &\hspace{5cm}+ \lambda(\eta(u_\eps(s,y);z)\ast \tau_\kappa) - k) \varphi_{\delta_0, \delta} \tilde{N}(dz, ds)\Big] \,d\lambda\,dk\,dx\,dt\,dy
         = 0.\notag
\end{align}
We define
\begin{align}
        \mathcal{K}[\beta, \varphi_{\delta, \delta_0}](s,y,k) &:= \int_{\mathbb{Q}_T} \int_{|z| > 0} \Big( \beta \big(u_{\Delta t}(t,x) + \eta(u_{\Delta t}(t,x);z) - k \big) - \beta\big(u_{\Delta t}(t,x)-k\big)\Big)\notag\\ &\hspace{6cm} \times \varphi_{\delta_0, \delta}(t,x,s,y)\,\tilde{N}(dz,dt)\,dx \notag
\end{align}
Following \cite[Lemmas 5.4, 5.5]{Majee-2015}, one can deduce the following:
\begin{align}
         &\partial_k\mathcal{K}[\beta, \varphi_{\delta, \delta_0}](s,y,k) = \mathcal{K}[-\beta', \varphi_{\delta, \delta_0}](s,y,k),\notag\\
    &\partial_y\mathcal{K}[\beta, \varphi_{\delta, \delta_0}](s,y,k) =
    \mathcal{K}[\beta, \partial_y\varphi_{\delta, \delta_0}](s,y,k).\notag
    \end{align}
    \text{Moreover,}
    \begin{align}\label{eq:Jumpnoise}
    &\underset{0 \leq s \leq T }{sup}\mathbb{E}\Big[||\mathcal{K}[\beta', \varphi_{\delta, \delta_0}](s,\cdot,\cdot)||_{L^\infty(\mathbb{R}^d \times \mathbb{R})}^2\Big] \leq \frac{C}{\delta^\frac{2}{p}\xi^{m}\delta_0^\frac{2(p-1)}{p}},\notag\\
    &\underset{0 \leq s \leq T }{sup}\mathbb{E}\Big[||\mathcal{K}[\beta'', \varphi_{\delta, \delta_0}](s,\cdot,\cdot)||_{L^\infty(\mathbb{R}^d \times \mathbb{R})}^2\Big] \leq \frac{C}{\delta^\frac{2}{p}\xi^{\bar{m}}\delta_0^\frac{2(p-1)}{p}},
\end{align}
where $p$ is a positive integer of the form $p = 2^k$ for some $k \in \mathbb{N}$ and $p \geq d +3$ and some $m, \bar{m} >0$.
In view of Fubini's theorem and \eqref{eq:Jumpnoise}, one has
\begin{align}
    & \mathbb{E} \Big [\int_{\mathbb{Q}_T}\int_{\mathbb{R}^d}\int_\mathbb{R}  J_l(u_\eps^\kappa(s -\delta_0,y ) - k)\int_0^1\int_{t=s-\delta_0}^{t=s} \int_{|z| > 0} \eta(u_{\Delta t}(t,x);z) \beta'(u_{\Delta t}(t,x)\notag \\ &\hspace{4cm}+ \lambda\eta(u_{\Delta t}(t,x);z) - k) \varphi_{\delta_0, \delta}\, \tilde{N}(dz, dt) \,d\lambda \,dk\,dx\,ds\,dy\Big] = 0. \notag
\end{align}
Using  It\^o-L\'evy formula  and following the same lines of argument as done for  $\mathcal{J}_6$, we can re-write $\mathcal{J}_8$ as follows.
\begin{align}
        \mathcal{J}_8  &= \mathbb{E}\Big[ \int_{\mathbb{Q}_T} \int_{\mathbb{R}} \mathcal{K}[\beta', \varphi_{\delta, \delta_0}](s,y,k)\int_{s-\delta_0}^{s}J_l(u_\eps^\kappa(r,y) -k)\big( -\mathcal{L}_{\theta}[u_\eps \ast \tau_\kappa] - div_x (f(u_\eps) \ast \tau_\kappa)\notag\\ & \hspace{8cm} + \eps\Delta u_\eps^\kappa \big)\ dr\,dk\,ds\,dy \Big]\notag\\
        &+\mathbb{E}\Big[ \int_{\mathbb{Q}_T} \int_{\mathbb{R}} \mathcal{K}[\beta'', \varphi_{\delta, \delta_0}](s,y,k)\int_{s-\delta_0}^{s}\int_{|z|>0}\int_0^1 (1-\lambda) J_l(u_\eps^\kappa(r,y) - k + \lambda (\eta(u_\eps); z)\ast\tau_\kappa)\notag\\ & \hspace{6cm} \times (\eta(u_\eps(r,y))\ast\tau_\kappa)^2\,d\lambda \,m(dz)\,dr \,dk\,ds\,dy\Big]\notag\\
        &+ \mathbb{E}\Big[ \int_{\mathbb{Q}_T} \int_{\mathbb{R}^d}\int_{\mathbb{R}}\int_{s-\delta_0}^{s}\int_{|z| > 0}\Big(\beta(u_{\Delta t}(r,x) - k +  \eta(u_{\Delta t}(r,x); z))-\beta(u_{\Delta t}(r,x) - k)\Big)\notag \\ & \times \Big(J_l(u_\eps^\kappa(r,y) - k +  (\eta(u_\eps(r,y); z))\ast\tau_\kappa) - J_l(u_\eps^\kappa(r,y) - k)\Big)\rho_{\delta_ 0}(r-s)\varrho_\delta(x-y)\psi(t,x)\,dk\,dr\,ds\,dx\,dy\Big]\notag\\
        & =: \mathcal{J}_8^1 + \mathcal{J}_8^2 + \mathcal{J}_8^3\notag
\end{align}
Following the estimations of $\mathcal{J}_6^1$ along with \eqref{eq:Jumpnoise}, and also due to the result \cite[Lemma 5.6]{Majee-2015}, we have 
\begin{align}
    \mathcal{J}_8^1 \leq  C(\delta,\xi,\bar{r},\kappa)\delta_0^{1/p} + C(\delta,\xi)\delta_0^{1/p} + C (\delta, \xi)\delta_0^a, \hspace{.2cm}\text{and} \hspace{.2cm} \mathcal{J}_8^2 \le C(\xi,\eps)\delta_0^{1/p}.\notag
\end{align}
Next we want to estimate the term $\mathcal{J}_8^3$. A similar lines of argument as done in \cite[Lemma 5.6]{Majee-2015} leads to
\begin{align}
        \mathcal{J}_8^3 
        \leq\, & \,\mathbb{E}\Big[ \int_{\mathbb{Q}_T} \int_{\mathbb{R}^d}\int_{|z| > 0}\Big(\beta \big(u_{\Delta t}(t,x)  + \eta(u_{\Delta t}(t,x); z) - u_\eps^\kappa(t,y) - (\eta(u_\eps(t,y); z)\ast\tau_\kappa)\big)\notag \\ & - \beta\big(u_{\Delta t}(t,x) - u_\eps^\kappa(t,y) - (\eta(u_\eps(t,y); z)\ast\tau_\kappa)\big) -\beta\big(u_{\Delta t}(t,x) +\eta(u_{\Delta t}(t,x); z)) - u_\eps^\kappa(t,y)\big)\notag \\& + \beta\big(u_{\Delta t}(t,x) -  u_\eps^\kappa(t,y)\big) \Big) \varrho_\delta(x-y)\psi(t,x)m(dz)\,dx\,dt\,dy\Big] \,+\, C(\xi)(\sqrt{\delta_0} + l)\notag \\& =:\mathcal{J}_8^{3,1} + C(\xi)(\sqrt{\delta_0} + l).\notag
\end{align}
Observe that, $\mathcal{J}_8^{3,1}$ can be re-arranged as follows.
\begin{align}
    &\mathbb{E}\Big[ \int_{\mathbb{Q}_T} \int_{\mathbb{R}^d}\int_{|z| > 0}\Big(\beta \big(u_{\Delta t}(t,x)  + \eta(u_{\Delta t}(t,x); z) - u_\eps^\kappa(t,y) - (\eta(u_\eps(t,y); z)\ast\tau_\kappa)\big)\notag \\ & - \beta\big(u_{\Delta t}(t,x) -  u_\eps^\kappa(t,y)\big) - \big(\eta(u_{\Delta t}(t,x);z) - \eta(u_\eps(t,y);z)\ast\tau_\kappa\big)\beta'\big(u_{\Delta t}(t,x) -  u_\eps^\kappa(t,y)\big) \Big) \varrho_\delta(x-y)\notag\\ &\times\psi(t,x)m(dz)\,dx\,dt\,dy\Big]\notag \\
    =\,&\, \mathbb{E}\Big[ \int_{\mathbb{Q}_T} \int_{\mathbb{R}^d}\int_{|z| > 0}\int_0^1(1-\theta) |\eta(u_{\Delta t}(t,x);z)-\eta(u_\eps(t,y); z)\ast\tau_\kappa)|^2\notag\\&\times\beta''(u_{\Delta t}(t,x) -u_\eps^\kappa(t,y)+\theta(\eta(u_{\Delta t}(t,x);z)-\eta(u_\eps(t,y); z)\ast\tau_\kappa)) \varrho_\delta(x-y)\psi(t,x)\,d\theta\,m(dz)\,dx\,dt\,dy\Big]\notag\\
    = &\,  \mathbb{E}\Big[ \int_{\mathbb{Q}_T} \int_{\mathbb{R}^d}\int_{|z| > 0}\int_0^1(1-\theta) |\eta(u_\eps(t,y); z)-\eta(u_\eps(t,y); z)\ast\tau_\kappa|^2\notag\\&\times\beta''(u_{\Delta t}(t,x) -u_\eps^\kappa(t,y)+\theta(\eta(u_{\Delta t}(t,x);z)-\eta(u_\eps(t,y); z)\ast\tau_\kappa)) \varrho_\delta(x-y)\psi(t,x)\,d\theta\,m(dz)\,dx\,dt\,dy\Big]\notag\\
    + &\,  \mathbb{E}\Big[ \int_{\mathbb{Q}_T} \int_{\mathbb{R}^d}\int_{|z| > 0}\int_0^1(1-\theta) |\eta(u_\eps^\kappa(t,y); z)-\eta(u_\eps(t,y); z)|^2\notag\\&\times\beta''(u_{\Delta t}(t,x) -u_\eps^\kappa(t,y)+\theta(\eta(u_{\Delta t}(t,x);z)-\eta(u_\eps(t,y); z)\ast\tau_\kappa)) \varrho_\delta(x-y)\psi(t,x)\,d\theta\,m(dz)\,dx\,dt\,dy\Big]\notag\\
    + &\,  \mathbb{E}\Big[ \int_{\mathbb{Q}_T} \int_{\mathbb{R}^d}\int_{|z| > 0}\int_0^1(1-\theta) |\eta(u_{\Delta t}(t,x);z)-\eta(u_\eps^\kappa(t,y); z))|^2\notag\\&\times\beta''(u_{\Delta t}(t,x) -u_\eps^\kappa(t,y)+\theta(\eta(u_{\Delta t}(t,x);z)-\eta(u_\eps(t,y); z)\ast\tau_\kappa)) \varrho_\delta(x-y)\psi(t,x)\,d\theta\,m(dz)\,dx\,dt\,dy\Big]\notag \notag \\
    =: & \mathfrak{A}_1 + \mathfrak{A}_2 + \mathfrak{A}_3.\notag
\end{align}
Following the estimation of $\mathcal{I}_9^1$ and $\mathcal{I}_9^3$, we find the upper bound of the terms $\mathfrak{A}_1$ and $\mathfrak{A}_2$ as follows. 
\begin{align}
    \mathfrak{A}_1 \le & C(\xi, \psi)\,\mathbb{E}\Big[ \int_0^T\int_{K_y}\int_{|z| > 0} |\eta(u_\eps(t,y); z)-\eta(u_\eps(t,y); z)\ast\tau_\kappa|^2 \,m(dz)\,dy\,dt\Big],\notag\\
    \mathfrak{A}_2 \le & C(\xi, \psi)\,\mathbb{E}\Big[\int_0^T\int_{K_y} |u_\eps^\kappa(t,y)-u_\eps(t,y)|^2\,dy\,dt\Big].\notag
\end{align}
We now focus on the term $\mathfrak{A}_3$. It can be decomposed as a sum of three terms $ \mathfrak{A}_3^i~(1\le i\le 3)$ where 
 can be re-arranged as follows.
\begin{align}
    \mathfrak{A}_3^1& =  \, \mathbb{E}\Big[ \int_{\mathbb{Q}_T} \int_{\mathbb{R}^d}\int_{|z| > 0}\int_0^1(1-\theta) |\eta(u_{\Delta t}(t,x);z)-\eta(u_\eps^\kappa(t,y); z))|^2\notag\\&\times\Big\{\beta''(u_{\Delta t}(t,x) -u_\eps^\kappa(t,y)+\theta(\eta(u_{\Delta t}(t,x);z)-\eta(u_\eps(t,y); z)\ast\tau_\kappa))\notag\\ & - \beta''(u_{\Delta t}(t,x) -u_\eps^\kappa(t,y)+\theta(\eta(u_{\Delta t}(t,x);z)-\eta(u_\eps(t,y); z)))\Big\}\notag \\&\hspace{4cm}\times \varrho_\delta(x-y)\psi(t,x)\,d\theta\,m(dz)\,dx\,dt\,dy\Big]\,, \notag \\
   \mathfrak{A}_3^2&=\mathbb{E}\Big[ \int_{\mathbb{Q}_T} \int_{\mathbb{R}^d}\int_{|z| > 0}\int_0^1(1-\theta) |\eta(u_{\Delta t}(t,x);z)-\eta(u_\eps^\kappa(t,y); z))|^2\notag\\&\times\Big\{\beta''(u_{\Delta t}(t,x) -u_\eps^\kappa(t,y)+\theta(\eta(u_{\Delta t}(t,x);z)-\eta(u_\eps(t,y); z)))\notag\\ & - \beta''(u_{\Delta t}(t,x) -u_\eps^\kappa(t,y)+\theta(\eta(u_{\Delta t}(t,x);z)-\eta(u_\eps^\kappa(t,y); z)))\Big\}\notag \\&\hspace{4cm}\times \varrho_\delta(x-y) \psi(t,x)\,d\theta\,m(dz)\,dx\,dt\,dy\Big] \,,\notag \\
   \mathfrak{A}_3^3 &= \mathbb{E}\Big[ \int_{\mathbb{Q}_T} \int_{\mathbb{R}^d}\int_{|z| > 0}\int_0^1(1-\theta) |\eta(u_{\Delta t}(t,x);z)-\eta(u_\eps^\kappa(t,y); z))|^2\notag\\&\times\beta''(u_{\Delta t}(t,x) -u_\eps^\kappa(t,y)+\theta(\eta(u_{\Delta t}(t,x);z)-\eta(u_\eps^\kappa(t,y); z))) \varrho_\delta(x-y)\notag\\&\hspace{2cm}\times\psi(t,x)\,d\theta\,m(dz)\,dx\,dt\,dy\Big] .\notag
\end{align}
Following the estimation of $\mathcal{I}_9^3$, one can easily derive the bound of  $\mathfrak{A}_3^1$ and  $\mathfrak{A}_3^2$:
\begin{align}
    \mathfrak{A}_3^1 \le &
     C(\xi)\, \Big(\mathbb{E}\Big[\int_0^T\int_{K_y}\int_{|z| > 0}|\eta(u_\eps(t,y); z)\ast\tau_\kappa -\eta(u_\eps(t,y); z)|^2\,m(dz)\,dy\,dt\Big]\Big)^{1/2},\notag\\
    \mathfrak{A}_3^2 \le  &\,C(\xi)\Big(\mathbb{E}\Big[\int_{\mathbb{Q}_T} |u_\eps(t,y) - u_\eps^\kappa(t,y)|^2\,dt\,dy\Big]\Big)^{1/2}.\notag
\end{align}
Using Lemma \ref{lem:9}, along with above estimations for $\mathcal{J}_8$ and keeping in mind that $\mathcal{I}_8=0$, we have the following Lemma.
\begin{lem}\label{lem89}
\begin{align}
       &\mathcal{I}_{8} + \mathcal{J}_8 + \mathcal{I}_9 + \mathcal{J}_9 \, \leq \,\mathbb{E}\Big[ \int_{\mathbb{Q}_T} \int_{\mathbb{R}^d}\int_{|z| > 0}\Big(\beta \big(u_{\Delta t}(t,x)  + \eta(u_{\Delta t}(t,x); z) - u_\eps^\kappa(t,y) - \eta(u_\eps^\kappa(t,y); z)\big)\notag \\ & \hspace{5cm} - \beta\big(u_{\Delta t}(t,x) -  u_\eps^\kappa(t,y)\big) - \big(\eta(u_{\Delta t}(t,x);z) - \eta(u_\eps^\kappa(t,y);z)\big)\notag\\ & \hspace{5.5cm}\times \beta'\big(u_{\Delta t}(t,x) -  u_\eps^\kappa(t,y)\big) \Big) \varrho_\delta(x-y)\psi(t,x)m(dz)\,dx\,dt\,dy\Big]\,\notag \\
       &  +\, C(\xi)\,\mathbb{E}\Big[\int_0^T\int_{K_y}\int_{|z| > 0}|\eta(u_\eps(t,y); z)\ast\tau_\kappa -\eta(u_\eps(t,y); z)|^2\,m(dz)\,dy\,dt\Big]\notag \\&\,+\,C(\xi)\mathbb{E}\Big[\int_{\mathbb{Q}_T} |u_\eps(t,y) - u_\eps^\kappa(t,y)|^2\,dy\,dt\Big] + C(\xi)\Big(\mathbb{E}\Big[\int_{\mathbb{Q}_T} |u_\eps(t,y) - u_\eps^\kappa(t,y)|^2\,dy\,dt\Big]\Big)^{1/2}\notag \\
       &\,+\, C(\xi)\,\Big(\mathbb{E}\Big[\int_0^T\int_{K_y}\int_{|z| > 0}|\eta(u_\eps(t,y); z)\ast\tau_\kappa -\eta(u_\eps(t,y); z)|^2\,m(dz)\,dy\,dt\Big]\Big)^{1/2}\notag \\& \qquad \quad+ \,C(\delta, \xi, \kappa, \eps)\delta_0^{\frac{1}{p}} + C(\delta, \xi)\delta_0^a  + C(\xi)l\notag \\
       &\le \, C(\xi)\,\mathbb{E}\Big[\int_0^T\int_{K_y}\int_{|z| > 0}|\eta(u_\eps(t,y); z)\ast\tau_\kappa -\eta(u_\eps(t,y); z)|^2\,m(dz)\,dy\,dt\Big]\notag \\&\,+\,C(\xi)\mathbb{E}\Big[\int_{\mathbb{Q}_T} |u_\eps(t,y) - u_\eps^\kappa(t,y)|^2\,dy\,dt\Big] + C(\xi)\Big(\mathbb{E}\Big[\int_{\mathbb{Q}_T} |u_\eps(t,y) - u_\eps^\kappa(t,y)|^2\,dy\,dt\Big]\Big)^{1/2}\notag \\
       &\,+\, C(\xi)\,\Big(\mathbb{E}\Big[\int_0^T\int_{K_y}\int_{|z| > 0}|\eta(u_\eps(t,y); z)\ast\tau_\kappa -\eta(u_\eps(t,y); z)|^2\,m(dz)\,dy\,dt\Big]\Big)^{1/2}\notag \\
       &\qquad\quad + \,C(\delta, \xi, \kappa, \eps)\delta_0^{\frac{1}{p}} + C(\delta, \xi)\delta_0^a  + C(\xi)l + C\xi.\notag
\end{align}
\end{lem}
Adding \eqref{entropinquality0} and \eqref{EntropyIQUE2} and using the Lemmas \ref{lem1}, \ref{lem 2}, \ref{lem3}, \ref{lem45}, \ref{lem67}, \ref{lem89} and \ref{lem:error} along with estimation \eqref{inq:i10}, we get 
\begin{align}\label{first-hand-ineqaulity}
 0  \le & \displaystyle\mathbb{E} \Big[ \int_{\mathbb{R}^d } \int_{\mathbb{R}^d }\beta(u_0(x)-u_\eps^\kappa(0,y)) \varrho_\delta(x-y)\psi(0,x)\,dx\,dy \Big]\notag\\
+\, &\mathbb{E}\Big[\int_{\mathbb{Q}_T}\int_{\mathbb{R}^d}\beta(u_{\Delta t}(t,x) - u_\eps^\kappa(t,y))\partial_t\psi(t,x)\varrho_{\delta}(x-y)\,dx\,dt\,dy  \Big]\notag\\
+\,&\mathbb{E} \Big [\int_{\mathbb{Q}_T} \int_{\mathbb{R}^d}F^\beta(u_{\Delta t}(t,x), u_\eps^\kappa(t,y))\cdot \nabla_x\psi(t,x)\varrho_\delta(x-y)\,dx\,dt\,dy \Big]\notag\\
-\,&\mathbb{E} \Big [\int_{\mathbb{Q}_T}\int_{\mathbb{R}^d}\beta(u_\eps^\kappa(t,y)-u_{\Delta t}(t,x))\mathcal{L}_\theta^{\bar{r}}[\psi(t,\cdot)](x)\varrho_{\delta}(x-y)\,dx\,dt\,dy\Big]\notag\\
-& \,\mathbb{E} \Big [\int_{\mathbb{Q}_T}\int_{\mathbb{R}^d} \beta(u_{\Delta t}(t,x) - u_\eps^\kappa(t,y)) \mathcal{L}_{\theta,\bar{r}}[\varrho_\delta(\cdot-y)\psi(t,\cdot)](x)\,dx\,dt\,dy \Big]\notag\\
- &\,\mathbb{E}\Big[\int_{\mathbb{Q}_T} \int_{\mathbb{R}^d}\beta(u_\eps^\kappa(t,y) - u_{\Delta t}(t,x) )\mathcal{L}_{\theta,\bar{r}}[\varrho_\delta(x-\cdot)](y)\psi(t,x)\,dx\,dt\,dy \Big]\notag\\
 +&C(\xi)\Big\{\mathbb{E}\Big[ \int_0^T \int_{K_y}|u_\eps^\kappa(t,y)-u_\eps(t,y)|^2 \,dy\,dt\Big]
+ \Big(\mathbb{E}\Big[ \int_0^T \int_{K_y}|u_\eps^\kappa(t,y)-u_\eps(t,y)|^2 \,dy\,dt\Big]\Big)^{\frac{1}{2}} \Big\}\notag\\ 
+\,&\, C(\xi)\,\mathbb{E} \Big[\int_0^T\int_{K_y} | \sigma(u_\eps(t,y)\ast\tau_\kappa) - \sigma(u_\eps(t,y))|^2\,dy\,dt \Big]\notag \\
+ &\,C(\xi)\,\mathbb{E}\Big[\int_0^T \int_{K_y} \int_{|z| > 0} |\eta(u_\eps(t,y); z)-\eta(u_\eps(t,y); z)\ast\tau_\kappa|^2 \,m(dz)\,dt\,dy\Big].\notag\\ +\, \,&C(\xi)\,\Big(\mathbb{E}\Big[\int_0^T \int_{K_y} \int_{|z| > 0}|\eta(u_\eps(t,y); z)\ast\tau_\kappa -\eta(u_\eps(t,y); z)|^2 \,m(dz)\,dy\,dt\Big]\Big)^{1/2}\notag\\
+&  C(\delta, \delta_0, \xi, \bar{r})\sqrt{\Delta t} +C(\delta, \xi, \kappa, \bar{r})(\sqrt{\delta_0} + \delta_0^{1/p}) + C(\delta, \xi)\delta_0^a + C(\delta, \xi, \bar{r})l + C(\delta)\sqrt{\eps} + C\xi \notag \\
 =:& \sum_{i=1}^{11} {\tt B}_i +  C(\delta, \delta_0, \xi, \bar{r})\sqrt{\Delta t} +C(\delta, \xi, \kappa, \bar{r})(\sqrt{\delta_0} + \delta_0^{1/p}) + C(\delta, \xi)\delta_0^a + C(\delta, \xi, \bar{r})l + C(\delta)\sqrt{\eps} + C\xi\,.
\end{align}

\subsection{Convergence analysis}\label{Convergence Analysis}
As we mentioned, our main goal is to establish  the convergence of approximate solutions $u_{\Delta t}(t,x)$ to an entropy solution for Cauchy problem \eqref{eq:fractional}. For this purpose, we will use Young measure theory in stochastic setup to send $\Delta t$ to $0$. We refer to see 
Dafermos \cite{dafermos} for deterministic setting, and Balder \cite{Balder} for the 
stochastic version of Young measure theory. To proceed further, we recall  the definition of Young measure. Let $(\Theta, \Sigma, \mu)$ be a $\sigma$-finite measure space and $\mathcal{P}(\R)$ be the space of probability measures on $\R$.
\begin{defi}[Young Measure]
A Young measure from $\Theta$ into $\R$ is a map $\tau \mapsto \mathcal{P}(\R)$ such that
$\tau(\cdot): \gamma \mapsto \tau(\gamma)(B)$ is $\Sigma$-measurable for every Borel subset $B$ of $\R$.
The set of all Young measures from $\Theta$ into $\R$ is denoted by $\mathcal{R}(\Theta, \Sigma, \mu).$
\end{defi}
In this context, we mention that with an appropriate choice of $(\Theta, \Sigma, \mu)$, the family
$\{u_{\Delta t}(t,x)\}_{\Delta t>0}$ can be thought of as a family of Young measures. To this end, we consider the predictable
$\sigma$-field of $\Omega\times(0,T)$ with the respect to $\{\mathcal{F}_t\}$, denoted by $\mathcal{P}_T$, and set 
\begin{align*}
  \Theta = \Omega\times (0,T)\times \R^d,\quad \Sigma = \mathcal{P}_T \times \mathcal{L}(\R^d)\quad \text{and} \quad \mu= P\otimes \lambda_t\otimes \lambda_x,
\end{align*}
where $\mathcal{L}(\R^d)$ is the Lebesgue $\sigma$-algebra on $\R^d$, $\lambda_t$ and $\lambda_x$ are respectively the Lebesgue measures on $(0,T)$ and $\R^d$. Moreover, for $M\in \mathbb{N}$, set
$\Theta_M = \Omega\times (0,T)\times B_M,$ where $B_M$ be the ball of radius $M$ around zero in $\R^d$. With the above setting at hand,
 we sum up the necessary results in the following proposition to carry over the subsequent analysis, see \cite{Majee-2014}.
\begin{prop} \label{prop:young-measure}
  Let  $\{u_{\Delta t}(t,x)\}_{{\Delta t}> 0}$ be a sequence of predictable processes such that
  \eqref{l-infty-bound} holds. Then there exist a sub-sequence $\{{\Delta t}_n\}$ with ${\Delta t}_n\goto 0$ and a Young measure
  $\tau\in \mathcal{R}(\Theta, \Sigma, \mu) $ such that the followings hold:
  \begin{itemize}
   \item [(A)] If  $g(\gamma,\nu)$ is a Carath\'eodory function on $\Theta\times \R$ such that $\mbox{supp}(g)\subset \Theta_M\times \R$
 for some $M \in \mathbb{N}$ and $\{g(\gamma, u_{\Delta t}(\gamma))\}_n$ (where $\gamma \equiv(\omega; t, x))$ is uniformly integrable,  then 
 \begin{align}
    \lim_{\Delta t_n\rightarrow 0} \int_{\Theta}g(\gamma, u_{\Delta t_n}(\gamma))\,\mu(d\gamma)
    = \int_\Theta\Big[\int_{\R} g(\gamma,\nu)\tau(\gamma)(\,d\nu)\Big]\,\mu(d\gamma) \notag 
    \end{align}
\item [(B)]  Denoting a triplet $(\omega,x,t)\in \Theta$ by $\gamma$, we define 
\begin{align}
 {\bf u}(\gamma,\alpha)=\inf\Big\{ c\in \R: \tau(\gamma)\Big((-\infty,c)\Big)>\alpha \Big\}\quad \text{for}\quad \alpha \in (0,1)
 ~\text{and}~\gamma\in \Theta.\notag 
\end{align} 
 Then, the function ${\bf u}(\gamma,\alpha)$ is  non-decreasing, right continuous on $(0,1)$ and
 $ \mathcal{P}_T\times \mathcal{L}(\R^d\times (0,1))$- measurable.
  Moreover, if $g(\gamma,\nu)$ is a  Carath\'eodory function on $\Theta\times \R$, then 
  \begin{align}
 \int_{\Theta}\Big[\int_\R g(\gamma,\nu)\tau(\gamma)(\,d\nu)\Big]\, \mu \,(d\gamma)
 = \int_\Theta\int_{\alpha=0}^1 g(\gamma, {\bf u}(\gamma,\alpha))\,d\alpha\, \mu(d\gamma)\,\cdot \notag
 \end{align}
\end{itemize}
\end{prop}
\noindent{\bf Step I:~(passage to the limit as $\lim_{l\goto 0} \lim_{\delta_0\goto 0} \lim_{\Delta t \goto 0}$).} We will pass to the limit as $\Delta t \goto 0$ in \eqref{first-hand-ineqaulity} (keeping all other involved parameters fixed) in the sense of Young measure (stochastic setting). Let $u(t,x, \alpha)$  be the Young measure valued limit of $u_{\Delta t}(t,x)$. We define
\begin{align}
    &\mathcal{H}_1(t,x,\omega;\nu) := \int_{\mathbb{R}^d}\beta(\nu - u_\eps^\kappa(s,y))\partial_t\psi(t,x)\varrho_{\delta}(x-y)\,dy\,,\notag\\
    &\mathcal{H}_2(t,x,\omega;\nu) := \int_{\mathbb{R}^d}F^\beta(\nu, u_\eps^\kappa(t,y))\cdot \nabla_x \psi(t,x)\varrho_\delta(x-y)\,dy.\notag
\end{align}
Note that $\mathcal{H}_1$,  
$\mathcal{H}_2$ are Carath\'edory function on $\Theta \times \mathbb{R}$ and the families $\mathcal{H}_1(., u_{\Delta t})$, 
and $\mathcal{H}_2(., u_{\Delta t})$ are uniformly integrable; for more details see \cite{Bauzet-2012}. Thus, invoking Proposition \ref{prop:young-measure} we have
\begin{align}
    \underset{\Delta t \rightarrow 0}{lim} {\tt B}_2& = \,\mathbb{E}\Big[\int_{\mathbb{Q}_T}\int_{\mathbb{R}^d}\int_0^1\beta(u(t,x,\alpha) - u_\eps^\kappa(s,y))\partial_t\psi(t,x)\varrho_{\delta}(x-y)\,d\alpha\,dx\,dt\,dy  \Big]\,, \notag\\
     \underset{\Delta t \rightarrow 0}{lim}\,{\tt B}_3&=
 \mathbb{E} \Big [\int_{\mathbb{Q}_T} \int_{\mathbb{R}^d}\int_0^1F^\beta(u(t,x,\alpha), u_\eps^\kappa(t,y))\varrho_\delta(x-y)\partial_x\psi(t,x))\,d\alpha\,dx\,dt\,dy \Big].\notag
\end{align}
To send the limit as $\Delta t \goto 0$ in the terms ${\tt B}_4$ and ${\tt B}_5$, we define  Carath\'eodory functions  $\mathcal{H}_4$ and $\mathcal{H}_5$ on $\Theta \times \mathbb{R}$,
\begin{align}
    &\mathcal{H}_4(t,x,\omega; \nu) := \int_{\mathbb{R}^d}\beta(u_\eps^\kappa(t,y)-\nu)\mathcal{L}_\theta^{\bar{r}}[\psi(t,\cdot)](x)\varrho_{\delta}(x-y)\,dy\,, \notag\\
     &\mathcal{H}_5(t,x,\omega; \nu) := \int_{\mathbb{R}^d} \beta(\nu - u_\eps^\kappa(t,y)) \mathcal{L}_{\theta,\bar{r}}[\varrho_\delta(\cdot-y)\psi(t,\cdot)](x)\,dy.\notag
\end{align}
We claim that  $\mathcal{H}_4(.,u_{\Delta t})$ and $\mathcal{H}_5(.,u_{\Delta t})$ are uniformly bounded sequence in $L^2(\mathbb{Q}_T \times \Omega)$ and hence uniformly integrable. Indeed, one can use triangle inequality, \eqref{inq:bound-frac-z} and \eqref{fractionalbound} to have
\begin{align}
  &\mathbb{E}\Big[\int_{\mathbb{Q}_T}|\mathcal{H}_4(.,u_{\Delta t})|^2\,dx\,dt\Big]
  = \mathbb{E}\Big[\int_{\mathbb{Q}_T}\Big(\int_{\mathbb{R}^d}\beta(u_\eps^\kappa(t,y)-u_{\Delta t}(t,x))\mathcal{L}_\theta^{\bar{r}}[\psi(t,\cdot)](x)\varrho_{\delta}(x-y)\,dy\Big)^2\,dx\,dt\Big]\notag\\
  &\le \, C(\bar{r}, \psi, \beta')\,\mathbb{E}\Big[\int_{\mathbb{Q}_T}\int_{\mathbb{R}^d}|u_\eps^\kappa(t,y)-u_{\Delta t}(t,x)|^2\varrho_{\delta}^2(x-y)\,dy \,dx\,dt\Big] \le C(\bar{r}, \delta),\notag
\end{align}
and
\begin{align}
&\mathbb{E}\Big[\int_{\mathbb{Q}_T}|\mathcal{H}_5(.,u_{\Delta t})|^2\,dx\,dt\Big]  =  \mathbb{E}\Big[\int_{\mathbb{Q}_T}\Big(\int_{\mathbb{R}^d} \beta(u_{\Delta t}(t,x) - u_\eps^\kappa(t,y)) \mathcal{L}_{\theta,\bar{r}}[\varrho_\delta(\cdot-y)\psi(t,\cdot)](x)\,dy\Big)^2\,dx\,dt\Big]\notag\\
   &\le C(\delta, \psi, \beta')\bar{r}^a\mathbb{E}\Big[\int_{\mathbb{Q}_T}\int_{\mathbb{R}^d} |(u_{\Delta t}(t,x) - u_\eps^\kappa(t,y))|^2\,dy\,dx\,dt\Big] \le C(\delta)\bar{r}^a. \notag
\end{align}
 Hence, in the view of Proposition \ref{prop:young-measure} we have
\begin{align}
    \underset{\Delta t \rightarrow 0}{lim}\,{\tt B}_4
    &= \mathbb{E} \Big [\int_{\mathbb{Q}_T}\int_{\mathbb{R}^d}\int_0^1\beta(u_\eps^\kappa(t,y)-u(t,x,\alpha))\mathcal{L}_\theta^{\bar{r}}[\psi(t,\cdot)](x)\varrho_{\delta}(x-y)\,d\alpha \,dx\,dt\,dy\Big]\,, \notag \\
    \underset{\Delta t \rightarrow 0}{lim}\,{\tt B}_5&=
  \mathbb{E} \Big [\int_{\mathbb{Q}_T}\int_{\mathbb{R}^d}\int_0^1 \beta(u(t,x,\alpha)- u_\eps^\kappa(t,y)) \mathcal{L}_{\theta,\bar{r}}[\varrho_\delta(\cdot-y)\psi(t,\cdot)](x)\,d\alpha\,dx\,dt\,dy \Big].\notag
\end{align}
Similarly it can be shown that,
\begin{align}
    \underset{\Delta t \rightarrow 0}{lim}\,{\tt B}_6 = \mathbb{E}\Big[\int_{\mathbb{Q}_T} \int_{\mathbb{R}^d}\int_0^1\beta(u_\eps^\kappa(t,y) - u(t,x,\alpha) )\mathcal{L}_{\theta,\bar{r}}[\varrho_\delta(x-\cdot)](y)\psi(t,x)\,d\alpha\,dx\,dt\,dy \Big].\notag
\end{align}
Hence, passing  the limit as $\Delta t$ tends to zero first, then subsequently sending $\delta_0 \rightarrow 0$ and $l \rightarrow 0$ in the inequality \eqref{first-hand-ineqaulity},  we get
\begin{align}\label{eq:A}
0\le {\tt B}_1 + \sum_{i=1}^5 \bar{\tt B}_i + \sum_{i=7}^{11} {\tt B}_i + C(\delta)\eps^{1/2} + C\xi\,,
\end{align}
where
\begin{align*}
 \bar{\tt B}_1&:= \mathbb{E}\Big[\int_{\mathbb{Q}_T}\int_{\mathbb{R}^d}\int_0^1\beta(u(t,x,\alpha) - u_\eps^\kappa(s,y))\partial_t\psi(t,x)\varrho_{\delta}(x-y)\,d\alpha\,dx\,dt\,dy  \Big]\,, \notag\\
 \bar{\tt B}_2&:= \mathbb{E} \Big [\int_{\mathbb{Q}_T} \int_{\mathbb{R}^d}\int_0^1F^\beta((u(t,x,\alpha), u_\eps^\kappa(t,y))\cdot \nabla_x\psi(t,x))\varrho_\delta(x-y)\,d\alpha\,dx\,dt\,dy \Big]\,,\notag\\
 \bar{\tt B}_3&:=-\mathbb{E} \Big [\int_{\mathbb{Q}_T}\int_{\mathbb{R}^d}\int_0^1\beta(u_\eps^\kappa(t,y)-u(t,x,\alpha))\mathcal{L}_\theta^{\bar{r}}[\psi(t,\cdot)](x)\varrho_{\delta}(x-y)\,d\alpha \,dx\,dt\,dy\Big]\,,\notag \\
 \bar{\tt B}_4&:=- \mathbb{E} \Big [\int_{\mathbb{Q}_T}\int_{\mathbb{R}^d}\int_0^1 \beta(u(t,x,\alpha)- u_\eps^\kappa(t,y)) \mathcal{L}_{\theta,\bar{r}}[\varrho_\delta(\cdot-y)\psi(t,\cdot)](x)\,d\alpha\,dx\,dt\,dy \Big]\,,\notag \\
 \bar{\tt B}_5&:=-\mathbb{E}\Big[\int_{\mathbb{Q}_T} \int_{\mathbb{R}^d}\int_0^1\beta(u_\eps^\kappa(t,y) - u(t,x,\alpha) )\mathcal{L}_{\theta,\bar{r}}[\varrho_\delta(x-\cdot)](y)\psi(t,x)\,d\alpha\,dx\,dt\,dy \Big]\,. 
 \end{align*}
 \noindent{\bf Step II:~(passage to the limit as $\kappa\goto 0$).}
Observe that, thanks to \eqref{inq:bound-frac-z} and convolution property, 
\begin{align}
    &\bigg| \bar{\tt B}_3 - \, \mathbb{E} \Big [\int_{\mathbb{Q}_T}\int_{\mathbb{R}^d}\int_0^1\beta(u_\eps(t,y)-u(t,x,\alpha))\mathcal{L}_\theta^{\bar{r}}[\psi(t,\cdot)](x)\varrho_{\delta}(x-y)\,d\alpha \,dx\,dt\,dy\Big] \bigg|\notag\\
    \le&\, \mathbb{E} \Big [\int_{\mathbb{Q}_T}\int_{\mathbb{R}^d}\int_0^1\Big|\beta(u_\eps^\kappa(t,y)-u(t,x,\alpha))- \beta(u_\eps(t,y)-u(t,x,\alpha))\Big|\,|\mathcal{L}_\theta^{\bar{r}}[\psi(t,\cdot)](x)|\varrho_{\delta}(x-y)\,d\alpha \,dx\,dt\,dy\Big]\notag \\
    \le \,&\, C(\beta')\mathbb{E} \Big [\int_{\mathbb{Q}_T}\int_{\mathbb{R}^d}|u_\eps^\kappa(t,y)-u_\eps(t,y)||\mathcal{L}_\theta^{\bar{r}}[\psi(t,\cdot)](x)|\varrho_{\delta}(x-y)\,dx\,dt\,dy\Big]\notag \\
    \le \, &C(\beta', |\psi|,\bar{r})\mathbb{E} \Big [\int_{\mathbb{Q}_T}\int_{\mathbb{R}^d}|u_\eps^\kappa(t,y)-u_\eps(t,y)|\varrho_{\delta}(x-y)\,dt\,dx\,dy\Big]\notag\\ \le &\, C(\bar{r})\, \Big(\mathbb{E} \Big [\int_{\mathbb{Q}_T}|u_\eps^\kappa(t,y)-u_\eps(t,y)|^2\,dt\,dy\Big]\Big)^{1/2} \underset{\kappa \rightarrow 0}{\longrightarrow} 0.\notag
\end{align}
Thus, we have 
$$\underset{\kappa \rightarrow 0}{lim}\, \bar{\tt B}_3 =\, - \, \mathbb{E} \Big [\int_{\mathbb{Q}_T}\int_{\mathbb{R}^d}\int_0^1\beta(u_\eps(t,y)-u(t,x,\alpha))\mathcal{L}_\theta^{\bar{r}}[\psi(t,\cdot)](x)\varrho_{\delta}(x-y)\,d\alpha \,dx\,dt\,dy\Big].$$
 Similarly, using properties of convolution and \eqref{fractionalbound}, we have 
 \begin{align*}
 \underset{\kappa \rightarrow 0}{lim}\,\bar{\tt B}_4& =- \mathbb{E} \Big [\int_{\mathbb{Q}_T}\int_{\mathbb{R}^d}\int_0^1 \beta(u(t,x,\alpha)- u_\eps^\kappa(t,y)) \mathcal{L}_{\theta,\bar{r}}[\varrho_\delta(\cdot-y)\psi(t,\cdot)](x)\,d\alpha\,dx\,dt\,dy \Big]\,, \\
  \underset{\kappa \rightarrow 0}{lim}\,\bar{\tt B}_5&= -\mathbb{E}\Big[\int_{\mathbb{Q}_T} \int_{\mathbb{R}^d}\int_0^1\beta(u_\eps(t,y) - u(t,x,\alpha) )\mathcal{L}_{\theta,\bar{r}}[\varrho_\delta(x-\cdot)](y)\psi(t,x)\,d\alpha\,dx\,dt\,dy \Big]\,. 
 \end{align*}
 Since $\sigma$ is Lipschitz continuous with $\sigma(0)=0$ and $u_\eps(\cdot)$ satisfies the uniform bound \eqref{inq:uniform-1st-viscous}, one can easily show that
 \begin{align*}
 \lim_{\kappa\goto 0} {\tt B}_7=0= \lim_{\kappa\goto 0} {\tt B}_8= \lim_{\kappa\goto 0} {\tt B}_9\,. 
 \end{align*}
 Note that $\mathbb{P}$-a.s. and for fixed $z\in \R$, 
 $\eta(u_\eps;z)\ast \tau_{\kappa} \goto \eta(u_\eps;z)$ in $L^2(K_x)$. Hence, 
 thanks to the assumptions \ref{A6} and \ref{A8}, estimation \eqref{inq:uniform-1st-viscous},
one can use dominated convergence theorem to conclude that 
$$\underset{\kappa \rightarrow 0}{lim}\,{\tt B}_{10} = \,0\, = \underset{\kappa \rightarrow 0}{lim}\,{\tt B}_{11}.$$
By using Lemma \ref{lem:l-infinity bound-approximate-solution} along with Fatou's Lemma, we have
 \begin{align}\label{eq:boundofyoungslimit}
&\underset{0\leq t \leq T}{sup}\mathbb{E}\Big[||{u}(t,\cdot,\cdot)||_2^2\Big] \leq \underset{0\leq t \leq T}{sup}  \Big(\liminf _{\eps \downarrow 0}\mathbb{E}\Big[|| u_{\Delta t}(t,\cdot)||_2^2 \Big]\Big)  \leq  \underset{0\leq t \leq T}{sup}\mathbb{E} \Big[||u_{\Delta t}(t,\cdot)||_2^2 \Big] < \infty.
 \end{align}
Since $\beta$ is Lipschitz continuous and $F^\beta$ satisfies \eqref{eqlipF} and the inequality \eqref{eq:boundofyoungslimit} in hand, it is easy to observe that, (see also \cite{Bauzet-2015})
\begin{align*}
   &\underset{\kappa \rightarrow 0}{lim} \,{\tt B}_1 =\mathbb{E} \Big[ \int_{\mathbb{R}^d } \int_{\mathbb{R}^d }\beta(u_0(x)-u_\eps(0,y)) \varrho_\delta(x-y)\psi(0,x)\,dx\,dy \Big], \\
   &\underset{\kappa \rightarrow 0}{lim} \,{\tt \bar{B}}_1 = \mathbb{E}\Big[\int_{\mathbb{Q}_T}\int_{\mathbb{R}^d}\int_0^1\beta(u(t,x,\alpha) - u_\eps(t,y))\partial_t\psi(t,x)\varrho_{\delta}(x-y)\,d\alpha\,dx\,dt\,dy  \Big], \\
   &\underset{\kappa \rightarrow 0}{lim} \,{\tt \bar{B}}_2 =  \mathbb{E} \Big [\int_{\mathbb{Q}_T} \int_{\mathbb{R}^d}\int_0^1F^\beta((u(t,x,\alpha), u_\eps(t,y))\varrho_\delta(x-y)\cdot\nabla_x\psi(t,x))\,d\alpha\,dx\,dt\,dy \Big].
\end{align*}
 Thus, making use of above estimations and sending $\kappa \rightarrow 0$ in \eqref{eq:A}, we have
\begin{align}\label{eq:A1}
    0 \le &\,\mathbb{E} \Big[ \int_{\mathbb{R}^d } \int_{\mathbb{R}^d }\beta(u_0(x)-u_\eps(0,y)) \varrho_\delta(x-y)\psi(0,x)\,dx\,dy \Big]\notag \\
    +\, &\mathbb{E}\Big[\int_{\mathbb{Q}_T}\int_{\mathbb{R}^d}\int_0^1\beta(u(t,x,\alpha) - u_\eps(t,y))\partial_t\psi(t,x)\varrho_{\delta}(x-y)\,d\alpha\,dx\,dt\,dy  \Big]\notag\\
+\,&\mathbb{E} \Big [\int_{\mathbb{Q}_T} \int_{\mathbb{R}^d}\int_0^1F^\beta((u(t,x,\alpha), u_\eps(t,y))\varrho_\delta(x-y)\cdot\nabla_x\psi(t,x))\,d\alpha\,dx\,dt\,dy \Big]\notag\\
-\, &\mathbb{E} \Big [\int_{\mathbb{Q}_T}\int_{\mathbb{R}^d}\int_0^1\beta(u_\eps(t,y)-u(t,x,\alpha))\mathcal{L}_\theta^{\bar{r}}[\psi(t,\cdot)](x)\varrho_{\delta}(x-y)\,d\alpha \,dx\,dt\,dy\Big]\notag \\
-\, & \mathbb{E} \Big [\int_{\mathbb{Q}_T}\int_{\mathbb{R}^d}\int_0^1 \beta(u(t,x,\alpha)- u_\eps(t,y)) \mathcal{L}_{\theta,\bar{r}}[\varrho_\delta(\cdot-y)\psi(t,\cdot)](x)\,d\alpha\,dx\,dt\,dy \Big]\notag \\
-\, &\mathbb{E}\Big[\int_{\mathbb{Q}_T} \int_{\mathbb{R}^d}\int_0^1\beta(u_\eps(t,y) - u(t,x,\alpha) )\mathcal{L}_{\theta,\bar{r}}[\varrho_\delta(x-\cdot)](y)\psi(t,x)\,d\alpha\,dx\,dt\,dy \Big] + C(\delta)\eps^{\frac{1}{2}} + C\xi\notag\\
&=: \sum_{i=1}^{6}\mathsf{E}_i + C(\delta)\eps^{\frac{1}{2}} + C\xi.
\end{align}
 \noindent{\bf Step III:~(passage to the limit as $\lim_{\eps\goto 0}\lim_{\xi\goto 0}$).}
 Recall that, $\big|\beta_\xi(r)- |r|\big| \le M_1\xi$.
Therefore, we have
\begin{align}\label{inq:E4}
    &\bigg|\, \mathsf{E}_4 - \mathbb{E} \Big [\int_{\mathbb{Q}_T}\int_{\mathbb{R}^d}\int_0^1|u_\eps(t,y)-u(t,x,\alpha)|\,\mathcal{L}_\theta^{\bar{r}}[\psi(t,\cdot)](x)\varrho_{\delta}(x-y)\,d\alpha \,dx\,dt\,dy\Big]\bigg|\notag\\
    \le \,&\mathbb{E} \Big [\int_{\mathbb{Q}_T}\int_{\mathbb{R}^d}\int_0^1\big|\beta(u_\eps(t,y)-u(t,x,\alpha)) - |u_\eps(t,y)-u(t,x,\alpha)|\big| \mathcal{L}_\theta^{\bar{r}}[\psi(t,\cdot)](x)\varrho_{\delta}(x-y)\,d\alpha \,dx\,dt\,dy\Big]\notag \\
    \le \, & C(\bar{r}){\xi}\,.
\end{align}
Thus, $$\underset{\xi \rightarrow 0}{lim}\,\mathsf{E}_4 = - \, \mathbb{E} \Big [\int_{\mathbb{Q}_T}\int_{\mathbb{R}^d}\int_0^1|u_\eps(t,y)-u(t,x,\alpha)|\,\mathcal{L}_\theta^{\bar{r}}[\psi(t,\cdot)](x)\varrho_{\delta}(x-y)\,d\alpha \,dx\,dt\,dy\Big].$$
Similarly, we get
\begin{align*}
&\underset{\xi \rightarrow 0}{lim}\, \mathsf{E}_1 = \, \mathbb{E} \Big[ \int_{\mathbb{R}^d } \int_{\mathbb{R}^d }|u_0(x)-u_\eps(0,y)| \varrho_\delta(x-y)\psi(0,x)\,dy\,dx \Big],\\ &\underset{\xi \rightarrow 0}{lim}\, \mathsf{E}_2 =  \mathbb{E}\Big[\int_{\mathbb{Q}_T}\int_{\mathbb{R}^d}\int_0^1|u(t,x,\alpha) - u_\eps(t,y)|\partial_t\psi(t,x)\varrho_{\delta}(x-y)\,d\alpha\,dx\,dt\,dy  \Big].
\end{align*}
Following \eqref{inq:E4} and using \eqref{fractionalbound}, we have
\begin{align*}
&\underset{\xi \rightarrow 0}{lim}\, \mathsf{E}_5 = -\,  \mathbb{E} \Big [\int_{\mathbb{Q}_T}\int_{\mathbb{R}^d}\int_0^1 |u(t,x,\alpha)- u_\eps(t,y)| \mathcal{L}_{\theta,\bar{r}}[\varrho_\delta(\cdot-y)\psi(t,\cdot)](x)\,d\alpha\,dx\,dt\,dy \Big], \\
&\underset{\xi \rightarrow 0}{lim}\, \mathsf{E}_6 =\,-\, \mathbb{E}\Big[\int_{\mathbb{Q}_T} \int_{\mathbb{R}^d}\int_0^1|(u_\eps(t,y) - u(t,x,\alpha) |\mathcal{L}_{\theta,\bar{r}}[\varrho_\delta(x-\cdot)](y)\psi(t,x)\,d\alpha\,dx\,dt\,dy \Big].
\end{align*}
Note that $\big|F^\beta(a,b) - F(a,b)\big| \le C\xi$ for any $a, b \in \R$. Thus, one has
\begin{align*}
    \underset{\xi \rightarrow 0}{lim}\, \mathsf{E}_3 = \mathbb{E} \Big [\int_{\mathbb{Q}_T} \int_{\mathbb{R}^d}\int_0^1F(u(t,x,\alpha), u_\eps(t,y))\varrho_\delta(x-y)\cdot\nabla_x\psi(t,x)\,d\alpha\,dx\,dt\,dy \Big].
\end{align*}
 Thanks to \cite[Section 3.4]{frac lin}, we can pass to the limit  $\eps \rightarrow 0$ in \eqref{eq:A1}. Thus in the view of above estimations, passing the limit as $\xi$ goes to 0 and subsequently sending $\eps \rightarrow 0$, we get
\begin{align}\label{eq:A2}
 0 \le \,&\,\mathbb{E} \Big[ \int_{\mathbb{R}^d } \int_{\mathbb{R}^d }|u_0(x)-u_0(y)| \varrho_\delta(x-y)\psi(0,x)\,dx\,dy \Big]\notag \\
+\, &\mathbb{E}\Big[\int_{\mathbb{Q}_T}\int_{\mathbb{R}^d}\int_0^1|u(t,x,\alpha) - \bar{u}(t,y)|\partial_t\psi(t,x)\varrho_{\delta}(x-y)\,d\alpha\,dx\,dt\,dy  \Big]\notag\\
+\,&\mathbb{E} \Big [\int_{\mathbb{Q}_T} \int_{\mathbb{R}^d}\int_0^1F((u(t,x,\alpha), \bar{u}(t,y))\varrho_\delta(x-y)\cdot\nabla_x\psi(t,x)\,d\alpha\,dx\,dt\,dy \Big]\notag\\
-\, &\mathbb{E} \Big [\int_{\mathbb{Q}_T}\int_{\mathbb{R}^d}\int_0^1|\bar{u}(t,y)-u(t,x,\alpha)|\mathcal{L}_\theta^{\bar{r}}[\psi(t,\cdot)](x)\varrho_{\delta}(x-y)\,d\alpha \,dx\,dt\,dy\Big]\notag \\
-\, & \mathbb{E} \Big [\int_{\mathbb{Q}_T}\int_{\mathbb{R}^d}\int_0^1|u(t,x,\alpha)- \bar{u}(t,y)| \mathcal{L}_{\theta,\bar{r}}[\varrho_\delta(\cdot-y)\psi(t,\cdot)](x)\,d\alpha\,dx\,dt\,dy \Big]\notag \\
-\, &\mathbb{E}\Big[\int_{\mathbb{Q}_T} \int_{\mathbb{R}^d}\int_0^1|\bar{u}(t,y) - u(t,x,\alpha) |\mathcal{L}_{\theta,\bar{r}}[\varrho_\delta(x-\cdot)](y)\psi(t,x)\,d\alpha\,dx\,dt\,dy \Big],
\end{align}
where $\bar{u}(t,y)$ is the unique entropy solution of the Cauchy problem \eqref{eq:fractional}. Moreover,  $u_\eps$ converges weakly to $\bar{u}$ in $L^2(\Omega \times \mathbb{Q}_T)$ and $u_{\eps}\goto \bar{u}$ in $L^p(\Omega \times(0,T)\times B(0,M))$ for any $M>0$ and any $1\le p<2$. 
\vspace{0.2cm}

\noindent{\bf Step IV:~(passage to the limit as $\bar{r}\goto 0$).} In view of \eqref{fractionalbound} and the fact that $\mathcal{L}_\theta[\Psi] := \mathcal{L}_{\theta,\bar{r}}[\Psi] + \mathcal{L}_\theta^{\bar{r}}[\Psi]$,  we have~(see also \cite[Step 2, Lemma $6$]{frac lin})
\begin{align}
     & \mathbb{E} \Big [\int_{\mathbb{Q}_T}\int_{\mathbb{R}^d}\int_0^1 |u(t,x,\alpha)- \bar{u}(t,y)| \mathcal{L}_{\theta,\bar{r}}[\varrho_\delta(\cdot-y)\psi(t,\cdot)](x)\,d\alpha\,dx\,dy \Big]\notag \\  &\hspace{1cm}+\,\mathbb{E}\Big[\int_{\mathbb{Q}_T} \int_{\mathbb{R}^d}\int_0^1|\bar{u}(t,y) - u(t,x,\alpha) |\mathcal{L}_{\theta,\bar{r}}[\varrho_\delta(x-\cdot)](y)\psi(t,x)\,d\alpha\,dx\,dt\,dy \Big] \, \underset{\bar{r} \rightarrow 0}{\longrightarrow} 0\,, \notag \\
    &\mathbb{E} \Big [\int_{\mathbb{Q}_T}\int_{\mathbb{R}^d}\int_0^1|\bar{u}(t,y)-u(t,x,\alpha)|\mathcal{L}_\theta^{\bar{r}}[\psi(t,\cdot)](x)\varrho_{\delta}(x-y)\,d\alpha
 \,dx\,dt\,dy\Big]\notag\\
 &\underset{\bar{r} \rightarrow 0}{\longrightarrow}\, \mathbb{E}\Big[\int_{\mathbb{Q}_T}\int_{\mathbb{R}^d}\int_0^1|\bar{u}(t,y)-u(t,x,\alpha)|\mathcal{L}_\theta[\psi(t,\cdot)](x)\varrho_{\delta}(x-y)\,d\alpha
 \,dx\,dt\,dy\Big].\notag
\end{align}
Substituting the above estimations in \eqref{eq:A2}, we get
\begin{align}\label{eq:A3}
 0 \le \, & \,\mathbb{E} \Big[ \int_{\mathbb{R}^d } \int_{\mathbb{R}^d }|u_0(x)-u_0(y)| \varrho_\delta(x-y)\psi(0,x)\,dy\,dx \Big]\notag\\
 +\,&\mathbb{E}\Big[\int_{\mathbb{Q}_T}\int_{\mathbb{R}^d}\int_0^1|u(t,x,\alpha) - \bar{u}(t,y)|\partial_t\psi(t,x)\varrho_{\delta}(x-y)\,d\alpha\,dx\,dt\,dy  \Big]\notag\\
+\,&\,\mathbb{E} \Big [\int_{\mathbb{Q}_T} \int_{\mathbb{R}^d}\int_0^1F(u(t,x,\alpha), \bar{u}(t,y))\varrho_\delta(x-y)\cdot\nabla_x\psi(t,x)\,d\alpha\,dx\,dt\,dy \Big]\notag\\
-\, &\mathbb{E} \Big [\int_{\mathbb{Q}_T}\int_{\mathbb{R}^d}\int_0^1|\bar{u}(t,y)-u(t,x,\alpha)|\mathcal{L}_\theta[\psi(t,\cdot)](x)\varrho_{\delta}(x-y)\,d\alpha \,dx\,dt\,dy\Big].
\end{align}
\noindent{\bf Step V:~(passage to the limit as $\delta \goto 0$).} Following the calculations  as invoked in \cite[Lemmas $1$, $3$ and Step $3$ of Lemma $6$]{frac lin}, one can pass to the limit as $\delta \goto 0$ in \eqref{eq:A3} to have the following Kato inequality:
\begin{align}\label{eq:kato}
    0 \le \, &\mathbb{E}\Big[\int_{\mathbb{Q}_T}\int_0^1|u(t,x,\alpha) - \bar{u}(t,x)|\partial_t\psi(t,x)\,d\alpha\,dt\,dx  \Big]\notag\\
+\,&\,\mathbb{E} \Big [\int_{\mathbb{Q}_T} \int_0^1F(u(t,x,\alpha), \bar{u}(t,x)) \cdot\nabla_x\psi(t,x)\,d\alpha\,dt\,dx \Big]\notag\\
-\, &\mathbb{E} \Big [\int_{\mathbb{Q}_T}\int_0^1|\bar{u}(t,x)-u(t,x,\alpha)|\mathcal{L}_\theta[\psi(t,\cdot)](x)\,d\alpha \,dt\,dx\Big].
\end{align}
At this point, one can closely follow the analysis as mentioned in \cite[Subsection $3.2$]{frac lin} to arrive at 
\begin{align}
    \mathbb{E}\Big[\int_{\mathbb{Q}_T}\int_0^1|u(t,x,\alpha) - \bar{u}(t,x)|\,d\alpha\,dt\,dx  \Big] = 0.\label{eq:final-conv-frac-lin}
\end{align}
\subsection{Proof of Theorem \ref{thm: mainthmf}}
Thanks to the uniqueness result  of Theorem \ref{thm: fractional}, we conclude from \eqref{eq:final-conv-frac-lin} that $u(t,x, \alpha)$, the Young measure-valued limit of  the approximate solutions $u_{\Delta t}(t,x)$, is an independent function of the variable $\alpha$ and equal to $\bar{u}(t,x)$ pointwise, the unique entropy solution of \eqref{eq:fractional}. Thus, $u_{\Delta t}$ converges weakly to $\bar{u}$ in $L^2(\Omega \times (0,T), \mathbb{R}^d)$. Moreover, since $u_{\Delta t}\in$ $L^\infty(\Omega \times \mathbb{Q}_T),$ by using the Carath\'eodory function $g(\omega, t, x; \nu)=\mathds{1}_{B}(x)|u(\omega, t,x) - \nu|^p$ in  Proposition \ref{prop:young-measure} for any compact set $B \subset \mathbb{R}^d$ and $1 \leq p < \infty$, one can easily conclude that $u_{\Delta t}(\cdot,\cdot)$ converges to $\bar{u}(\cdot,\cdot)$ in $L_{loc}^p(\mathbb{R}^d; L^p(\Omega \times (0,T)))$.  In other words, the approximate solutions $u_{\Delta t}(t, x)$ given by \eqref{approxi:solu} and \eqref{eq:sequence} converges to a unique entropy solution of \eqref{eq:fractional} in $L_{loc}^p(\mathbb{R}^d; L^p(\Omega \times (0,T)))$ for $1 \le p < \infty.$ This completes the proof of Theorem \ref{thm: mainthmf}. 


\section{Fractional degenerate Cauchy Problem: Proof of Theorem \ref{thm: mainthmd}}\label{sec:6}

In this section, we will analyze the convergence of  $u_{\Delta t}(t,x)$, given by \eqref{approxi:solu} and \eqref{eq:sequence-nonlinear}, to an entropy solution of the Cauchy problem \eqref{eq:fdegenerate}. Like in linear fractional Cauchy problem, let $u_\eps(\cdot)$ be the viscous solution of associated to \eqref{eq:fdegenerate}
with initial data $u_\eps(0,x) = u_0^\eps \in H^1(\mathbb{R}^d).$ Then thanks to \cite[Theorem 2.1]{frac non},  $u_\eps \in C([0,T]; L^2(\mathbb{R}^d))$ $\mathbb{P}$-a.s, and there exists a Constant $C > 0$ such that
\begin{align}\label{eq:estimations-fd}
    \underset{0 \le t \le T}{sup}\, \mathbb{E} \Big[||u_\eps(t)||_{L^2(\mathbb{R}^d)}^2\Big] +  \eps \int_0^T \mathbb{E} \Big[||\nabla u_\eps(s)||_{L^2(\mathbb{R}^d)}^2\Big]ds  + \int_0^T \mathbb{E} \Big[||\phi(u_\eps(s))||_{H^\theta(\mathbb{R}^d)}^2\Big]ds  \le C.
\end{align}
Observe that $u_\eps \in H^1(\mathbb{R}^d)$, while as in previous section for technical reason we need higher regularity of $u_\eps.$ Thus, we regularize $u_\eps$ by a space convolution with a molliffier-sequence $\{\tau_\kappa\}_\kappa$. Note that $u_\eps^\kappa = u_\eps \ast \tau_\kappa$ satisfies,
\begin{align}\label{eq:regularizefd}
     &\partial_t \Big[u_\eps - \int_0^t\sigma(u_\eps) \ast \tau_\kappa dW(r)\Big] = -\mathcal{L}_{\theta}[\phi(u_\eps) \ast \tau_\kappa] - div_x f(u_\eps) \ast \tau_\kappa + \eps\Delta (u_\eps^\kappa).
\end{align}

We choose $\varphi_{\delta,{\delta_0}}$  and $J_l$ as defined in Sections~\ref{sec:convergence u}. We apply It\^{o} formula to \eqref{eq:regularizefd}, multiply by $J_l(u_{\Delta t}(t,x) - k)$ and then integrate with respect to $t,x$ and $k$ and take expectation in the resulting expression to obtain the following inequality.
 \begin{align}\label{uepsdfrac}
         &0 \le \sum_{i=1}^3 \mathcal{I}_i + \sum_{i=1}^2 \mathcal{X}_i + \sum_{i=6}^7 \mathcal{I}_i +  \sum_{i=10}^{11} \mathcal{I}_i\,
         \end{align}
         where $\mathcal{I}_i$ are given in  \eqref{entropinquality0} and $\mathcal{X}_i ~(i=1,2)$ are given by
         \begin{align*}
         \mathcal{X}_1:=
         &-\mathbb{E} \Big[\int_{\mathbb{Q}_T^2} \int_{\mathbb{R}}\mathcal{L}_\theta^{\bar{r}}[\phi(u_\eps)^\kappa(s,\cdot))](y)\beta'(u_\eps^\kappa(s,y) - k) \varphi_{\delta_0, \delta}J_l(u_{\Delta t}(t,x) - k)\,dk\,ds\,dt\,dy\,dx \Big]\,,\notag \\
      \mathcal{X}_2:=   &- \mathbb{E} \Big[\int_{\mathbb{Q}_T^2} \int_{\mathbb{R}}\beta'(u_\eps^\kappa(s,y) - k) \mathcal{L}_{\theta,\bar{r}}[\phi(u_\eps)^\kappa(s,\cdot))](y)\varphi_{\delta_0, \delta} J_l(u_{\Delta t}(t,x) - k)\,dk\,ds\,dt\,dy\,dx \Big]\,.
    \end{align*}
    Similar to \eqref{EntropyIQUE2}, we obtain entropy inequality for approximate solution $u_{\Delta t}(t,x)$ of \eqref{eq:fdegenerate}  as 
\begin{align}
 &0 \le \sum_{i=1}^3 \mathcal{J}_i + \sum_{i=1}^2 \mathcal{Y}_i + \sum_{i=6}^{7} \mathcal{J}_i + \sum_{i=10}^{13} \mathcal{J}_i 
 +  \sum_{i=3}^6 \mathcal{Y}_i +  \sum_{i=18}^{19} \mathcal{J}_i\,, \label{udeldfrac}
 \end{align}
 where
 \begin{align*}
 \mathcal{Y}_1 &:= 
  -\mathbb{E} \Big [\int_{\mathbb{Q}_T^2}\int_\mathbb{R} \mathcal{L}_\theta^{\bar{r}}[\phi(u_{\Delta t}(t,\cdot))](x)\beta'(u_{\Delta t}(t,x) - k)\varphi_{\delta_0, \delta}J_l(u_\eps^\kappa(s,y) - k)\,dk\,ds\,dt\,dx\,dy \Big]\,, \notag\\
  \mathcal{Y}_2 &:=  - \mathbb{E} \Big [\int_{\mathbb{Q}_T^2}\int_\mathbb{R}\phi_k^\beta(u_{\Delta t}(t,x))\mathcal{L}_{\theta,\bar{r}}[\varphi_{\delta_0, \delta}(t,\cdot,s,\cdot)](x) J_l(u_\eps^\kappa(s,y) - k)\,dk\,ds\,dt\,dx\,dy \Big], \notag \\
   \mathcal{Y}_3 &:= -  \mathbb{E}\Big[\int_{\mathbb{Q}_T}\int_\mathbb{R}\sum_{n=0}^{N-1} \int_{\mathbb{R}^d} \int_{t_n}^{t_{n+1}}  \big[\mathcal{L}_\theta^{\bar{r}}[\phi(\widetilde{u}_{\Delta t}(t,\cdot))](x)\beta'(\widetilde{u}_{\Delta t}(t,x) - k)\notag \\ & \hspace{1cm} -\mathcal{L}_\theta^{\bar{r}}[\phi(u_{\Delta t}(t_{n+1},\cdot))](x)\beta'(u_{\Delta t}(t_{n+1},x) -k) ] \varphi_{\delta_0, \delta}(t,x,s,y) J_l(u_\eps^\kappa(s,y) - k)\,dt\,dx\,dk\,ds\,dy \Big],\notag \\
  \mathcal{Y}_4 &:=  -  \mathbb{E}\Big[\int_{\mathbb{Q}_T}\int_\mathbb{R}\sum_{n=0}^{N-1} \int_{\mathbb{R}^d} \int_{t_n}^{t_{n+1}}  \big[\mathcal{L}_\theta^{\bar{r}}[\phi(u_{\Delta t}(t_{n+1},\cdot))](x)\beta'(u_{\Delta t}(t_{n+1},x) - k)\notag \\ & \hspace{2cm} -\mathcal{L}_\theta^{\bar{r}}[\phi(u_{\Delta t}(t,\cdot))](x)\beta'(u_{\Delta t}(t,x) - k) ] \varphi_{\delta_0, \delta}(t,x,s,y) J_l(u_\eps^\kappa(s,y) - k)\,dt\,dx\,dk\,ds\,dy \Big],
  \notag\\ 
   \mathcal{Y}_5 &:= - \mathbb{E}\Big[\int_{\mathbb{Q}_T}\int_\mathbb{R}\sum_{n=0}^{N-1} \int_{\mathbb{R}^d} \int_{t_n}^{t_{n+1}} \big(\phi_k^\beta(\widetilde{u}_{\Delta t}(t,x)) - \phi_k^\beta(u_{\Delta t}(t_{n+1},x)) \big) \mathcal{L}_{\theta,\bar{r}}[\varphi_{\delta_0, \delta}(t,\cdot,s,y)](x)\notag\notag \\& \hspace{6cm} \times J_l(u_\eps^\kappa(s,y) - k)\,dt\,dx\,dk\,ds\,dy \Big],\notag\\ 
   \mathcal{Y}_6&:= - \mathbb{E}\Big[\int_{\mathbb{Q}_T}\int_\mathbb{R} \sum_{n=0}^{N-1} \int_{\mathbb{R}^d} \int_{t_n}^{t_{n+1}} \big(\phi_k^\beta(u_{\Delta t}(t_{n+1},x)) - \phi_k^\beta(u_{\Delta t}(t,x)) \big) \mathcal{L}_{\theta,\bar{r}}[\varphi_{\delta_0, \delta}(t,\cdot,s,y)](x)\notag \\& \hspace{6cm} \times J_l(u_\eps^\kappa(s,y) - k)\,dt\,dx\,dk\,ds\,dy \Big]\,,
  \end{align*}
  and $\mathcal{J}_i$ are given in  \eqref{EntropyIQUE2}. 
Next, we want to approximate each of the terms in \eqref{uepsdfrac}-\eqref{udeldfrac} in terms of small parameters $\delta, \delta_0, l, \xi, \eps, \kappa$ and $\Delta t$ to achieve our main result. Let us start with the non-local fractional terms. We rearrange  $\mathcal{X}_1$ as follows.
\begin{align}
       \mathcal{X}_1 = &- \mathbb{E} \Big[\int_{\mathbb{Q}_T^2} \int_{\mathbb{R}}\Big(\mathcal{L}_\theta^{\bar{r}}[\phi(u_\eps)^\kappa(s,\cdot))](y)-\mathcal{L}_\theta^{\bar{r}}[\phi(u_\eps)(s,\cdot))](y)\Big)\beta'(u_\eps^\kappa(s,y) - u_{\Delta t}(t,x) + k)\notag \\ &\hspace{8cm}\times\varphi_{\delta_0, \delta}J_l(k)\,dk\,ds\,dt\,dy\,dx \Big]\notag\\
       &- \mathbb{E} \Big[\int_{\mathbb{Q}_T^2} \int_{\mathbb{R}}\Big(\mathcal{L}_\theta^{\bar{r}}[\phi(u_\eps)(s,\cdot))](y)-\mathcal{L}_\theta^{\bar{r}}[\phi(u_\eps^\kappa)(s,\cdot)](y)\Big)\beta'(u_\eps^\kappa(s,y) - u_{\Delta t}(t,x) + k)\notag \\ &\hspace{8cm}\times\varphi_{\delta_0, \delta}J_l(k)\,dk\,ds\,dt\,dy\,dx \Big]\notag\\
       &- \mathbb{E} \Big[\int_{\mathbb{Q}_T^2} \int_{\mathbb{R}}\Big(\mathcal{L}_\theta^{\bar{r}}[\phi(u_\eps^\kappa)(s,\cdot))](y)-\mathcal{L}_\theta^{\bar{r}}[\phi(u_\eps^\kappa)(t,\cdot)](y)\Big)\beta'(u_\eps^\kappa(s,y) - u_{\Delta t}(t,x) + k)\notag \\ &\hspace{8cm}\times\varphi_{\delta_0, \delta}J_l(k)\,dk\,ds\,dt\,dy\,dx \Big]\notag\\
       &-\mathbb{E} \Big[\int_{\mathbb{Q}_T^2} \int_{\mathbb{R}}\mathcal{L}_\theta^{\bar{r}}[\phi(u_\eps^\kappa)(t,\cdot)](y)\Big(\beta'(u_\eps^\kappa(s,y) - u_{\Delta t}(t,x) + k)) - \beta'(u_\eps^\kappa(t,y) - u_{\Delta t}(t,x) + k))\Big)\notag\\&\hspace{7cm}\times \varphi_{\delta_0, \delta}J_l(k)\,dk\,ds\,dt\,dy\,dx \Big]\notag\\
       &-\mathbb{E} \Big[\int_{\mathbb{Q}_T}\int_{\mathbb{R}^d} \int_{\mathbb{R}}\mathcal{L}_\theta^{\bar{r}}[\phi(u_\eps^\kappa)(t,\cdot)](y)\beta'(u_\eps^\kappa(t,y) - u_{\Delta t}(t,x) + k)\notag\\& \hspace{5cm} \times \Big(\int_0^T\rho_{\delta_0}(t-s)ds -1\Big) \varrho_{ \delta}(x-y)\psi(t,x)J_l(k)\,dk\,dx\,dt\,dy \Big]\notag\\
       &-\mathbb{E} \Big[\int_{\mathbb{Q}_T}\int_{\mathbb{R}^d} \int_{\mathbb{R}}\mathcal{L}_\theta^{\bar{r}}[\phi(u_\eps^\kappa(t,\cdot))](y)\Big(\beta'(u_\eps)^\kappa(t,y)- u_{\Delta t}(t,x) + k) - \beta'(u_\eps^\kappa(t,y) - u_{\Delta t}(t,x))\Big)\notag\\& \hspace{7cm}\times\varrho_{ \delta}(x-y)\psi(t,x)J_l(k)\,dk\,dx\,dt\,dy \Big]\notag\\
       &-\mathbb{E} \Big[\int_{\mathbb{Q}_T}\int_{\mathbb{R}^d} \mathcal{L}_\theta^{\bar{r}}[\phi(u_\eps^\kappa)(t,\cdot)](y)\beta'(u_\eps^\kappa(t,y) - u_{\Delta t}(t,x))\varrho_{ \delta}(x-y)\psi(t,x)\,dx\,dt\,dy \Big]\notag\\
       &=: \mathcal{X}_1^1 +  \mathcal{X}_1^2 + \mathcal{X}_1^3 +  \mathcal{X}_1^4 + \mathcal{X}_1^5 + \mathcal{X}_1^6 + \mathcal{X}_1^7.        \notag
       \end{align}
    Using the triangle inequality and \eqref{inq:bound-frac-z}, we bound $\mathcal{X}_1^1$ as follows.
    \begin{align}
        \mathcal{X}_1^1 = &- \,\mathbb{E} \Big[\int_{\mathbb{Q}_T^2} \int_{\mathbb{R}}\big(\phi(u_\eps)^\kappa(s,y)-\phi(u_\eps)(s,y)\big)\mathcal{L}_\theta^{\bar{r}}[\beta'(u_\eps^\kappa(s,.) - u_{\Delta t}(t,x) + k)\varrho_{\delta}(x-\cdot)](y)\notag \\ &\hspace{8cm}\times\rho_{\delta_0}(t-s)\psi(t,x)J_l(k)\,dk\,ds\,dt\,dy\,dx \Big]\notag\\
        \le \,&\, C(\beta', \psi)\mathbb{E} \Big[\int_0^T\int_{K_y}\int_{K_x} \big|\phi(u_\eps)^\kappa(s,y)-\phi(u_\eps)(s,y)\big|\Big(\int_{|z| > \bar{r}}\frac{\varrho_\delta(x-y) + \varrho_{\delta}(x-(y-z))}{|z|^{d+2\theta}}\,dz\Big)\,dx\,dy\,ds \Big]\notag\\
        \le \,&\,C(\beta', \psi, \delta, \bar{r})\mathbb{E} \Big[\int_0^T\int_{K_y} \big|\phi(u_\eps)^\kappa(s,y)-\phi(u_\eps)(s,y)\big|\,dy\,ds\,\Big].\notag
    \end{align}
    Similarly, since $\phi$ is Lipschitz continuous, we get
    \begin{align*}
        \mathcal{X}_1^2 \le C(\delta, \bar{r})\mathbb{E} \Big[\int_0^T\int_{K_y} \big|u_\eps^\kappa(s,y)-u_\eps(s,y)\big|\,ds\,dy\,\Big].
    \end{align*}
    We use the assumption \ref{A2}, triangle inequality and estimation \eqref{inq:bound-frac-z}, and Remark \ref{rem:time-cont-degenerate-vis} to approximate $\mathcal{X}_1^3$.
    \begin{align}
        \mathcal{X}_1^3 = &-\mathbb{E} \Big[\int_{\mathbb{Q}_T^2} \int_{\mathbb{R}}\Big(\phi(u_\eps^\kappa)(s,y)-\phi(u_\eps^\kappa)(t,y)\Big)\mathcal{L}_\theta^{\bar{r}}[\beta'(u_\eps^\kappa(s,\cdot) - u_{\Delta t}(t,x) + k)\varrho_\delta(x-\cdot)](y)\notag\\&\hspace{6cm }\times\rho_{\delta_0}(t-s)\psi(t,x)J_l(k)\,dk\,ds\,dt\,dy\,dx \Big]\notag\\
        \le \,&C(\beta',||\phi'||_{L^\infty} , |K_x|)\,\mathbb{E} \Big[\int_0^T\int_0^T\int_{K_y} \,|u_\eps^\kappa(s,y)-u_\eps^\kappa(t,y)| \rho_{\delta_0}(t-s) \notag \\
        & \hspace{4cm } \times \Big(\int_{|z| \ge \bar{r}}\frac{\varrho_\delta(x-y)+\varrho_\delta(x-(y-z))}{|z|^{d+2\theta}}\,dz\Big)\,dy\,dt\,ds\Big]\notag\\
        \le \,&C(\beta',||\phi'||_{L^\infty} , \delta,\bar{r})\,\mathbb{E}\Big[\int_0^T\int_0^T\int_{K_y}\,|u_\eps^\kappa(s,y)-u_\eps^\kappa(t,y)| \rho_{\delta_0}(t-s)\,dy\,dt\,ds\Big] \notag \\
         \le\,&\, C(\delta,\bar{r}, \kappa, \eps)\sqrt{\delta_0}.\notag
    \end{align}
Using the assumption \ref{A3}, Cauchy-Schwartz inequality, estimations  \eqref{inq:bound-frac-z}, \eqref{eq:estimations-fd} and Remark \ref{rem:time-cont-degenerate-vis}, we get
    \begin{align}
       \mathcal{X}_1^4  
       \le \, &C(\xi)\,\mathbb{E} \Big[\int_0^T\int_0^T\int_{K_y} |\mathcal{L}_\theta^{\bar{r}}[\phi(u_\eps)^\kappa(t,\cdot)](y)|\,|u_\eps^\kappa(s,y)-u_\eps^\kappa(t,y)|\rho_{\delta_0}(t-s)\,dy\,dt\,ds \Big]\notag\\
       \le &C(\xi)\,\mathbb{E} \Big[\int_0^T\int_0^T\int_{K_y} |u_\eps^\kappa(t,y)|\,|u_\eps^\kappa(s,y)-u_\eps^\kappa(t,y)|\rho_{\delta_0}(t-s)\Big(\int_{|z| > \bar{r}} \frac{1}{|z|^{d+2\theta}}\,dz\Big)\,dy\,dt\,ds \Big]\notag\\
       +&\,C(\xi)\mathbb{E} \Big[\int_0^T\int_0^T\int_{K_y} ||u_\eps^\kappa(t,\cdot)||_{L^2(\mathbb{R}^d)}\,|u_\eps^\kappa(s,y)-u_\eps^\kappa(t,y)|\rho_{\delta_0}(t-s)\Big(\int_{|z|> \bar{r}}\Big(\frac{1}{|z|^{1+2\theta}}\Big)^2dz\Big)^{1/2}\,dy\,dt\,ds \Big]\notag\\
       \le &C(\xi,\psi, \bar{r})\Big(\underset{t}{sup}\,\mathbb{E}\Big[||u_\eps(t,\cdot)||_{L^2(\R^d)}^2\Big]\Big)^{\frac{1}{2}}\Big(\mathbb{E}\Big[\int_0^T\int_{K_y\times [0,T]} |u_\eps^\kappa(s,y)-u_\eps^\kappa(t,y)|^2\rho_{\delta_0}(t-s)\,dy\,dt\,ds \Big]\Big)^{1/2}\notag\\
       + &\,C(\xi, \psi, \bar{r}, |K_x|)\Big(\underset{t}{sup}\,\mathbb{E}\Big[||u_\eps(t,\cdot)||_{L^2(\mathbb{R}^d)}^2\Big]\Big)^{\frac{1}{2}}\Big(\mathbb{E}\Big[\int_0^T\int_{K_y\times [0,T]} |u_\eps^\kappa(s,y)-u_\eps^\kappa(t,y)|^2\rho_{\delta_0}(t-s)\,dy\,dt\,ds \Big]\Big)^{\frac{1}{2}}\notag \\
       \le \,&C(\xi,\bar{r}, \eps, \kappa)\sqrt{\delta_0}.\notag
    \end{align}
 A simple calculation reveals that
     \begin{align}
    &\mathcal{X}_1^5 \le  \,C(\beta', \psi)\mathbb{E} \Big[\int_0^T\int_{K_y} |\mathcal{L}_\theta^{\bar{r}}[\phi(u_\eps^\kappa)(t,\cdot)](y)|\Big|\int_0^T\rho_{\delta_0}(t-s)ds -1\Big| \,dy\,dt\Big]\notag \\
    &\le \,C(\beta', \psi, ||\phi'||_{L^\infty})\,\mathbb{E} \Big[\int_{t= T-\delta_0}^T\int_{K_y}\Big(\int_{|z| > \bar{r}} \frac{|u_\eps^\kappa(t,x) - u_\eps^\kappa(t,x+z)|}{|z|^{d+2\theta}}\,dz\Big)\,dy\,dt \Big]\notag \le C(\bar{r}, \kappa, \eps) \sqrt{\delta_0}.
    \end{align}
    Similarly, it is easy to observe that
    $$\mathcal{X}_1^6 \le C(\xi, \bar{r})l.$$
    We rearrange $\mathcal{Y}_1$ as follows. 
\begin{align}
         \mathcal{Y}_1 =&- \mathbb{E} \Big [\int_{\mathbb{Q}_T^2}\int_\mathbb{R} \mathcal{L}_\theta^{\bar{r}}[\phi(u_{\Delta t}(t,\cdot))](x)\Big(\beta'(u_{\Delta t}(t,x) - u_\eps^\kappa(s,y) + k)- \beta'(u_{\Delta t}(t,x) - u_\eps^\kappa(t,y) + k)\Big)\notag\\ &\hspace{7cm} \times \varphi_{\delta_0, \delta}J_l(k)\,dk\,ds\,dt\,dx\,dy \Big]\notag\\
         &- \mathbb{E} \Big [\int_{\mathbb{Q}_T}\int_{\mathbb{R}^d}\int_\mathbb{R} \mathcal{L}_\theta^{\bar{r}}[\phi(u_{\Delta t}(t,\cdot))](x)\beta'(u_{\Delta t}(t,x) - u_\eps^\kappa(t,y) + k)\notag\\ &\hspace{5cm} \times \Big(\int_0^T\rho_{\delta_0}(t-s)ds -1 \Big) \varphi_{\delta}(x-y)\psi(t,x)J_l(k)\,dk\,dx\,dt\,dy \Big]\notag\\
         & -\mathbb{E} \Big [\int_{\mathbb{Q}_T}\int_{\mathbb{R}^d}\int_\mathbb{R} \mathcal{L}_\theta^{\bar{r}}[\phi(u_{\Delta t}(t,\cdot))](x)\Big(\beta'(u_{\Delta t}(t,x) - u_\eps^\kappa(t,y) +k) - \beta'(u_{\Delta t}(t,x) - u_\eps^\kappa(t,y) )\Big)\notag\\ &\hspace{7cm} \times  \varphi_{\delta}(x-y)\psi(t,x)J_l(k)dk\,dx\,dt\,dy \Big]\notag\\
         &-\mathbb{E} \Big [\int_{\mathbb{Q}_T}\int_{\mathbb{R}^d} \mathcal{L}_\theta^{\bar{r}}[\phi(u_{\Delta t}(t,\cdot))](x)  \beta'(u_{\Delta t}(t,x) - u_\eps^\kappa(t,y)) \varphi_{\delta}(x-y)\psi(t,x)\,dx\,dt\,dy \Big]\notag\\
         &=: \mathcal{Y}_1^1 + \mathcal{Y}_1^2 + \mathcal{Y}_1^3+ \mathcal{Y}_1^4\,.
         \notag
\end{align}
We use Lemmas \ref{lem:l-infinity bound-approximate-solution}, Remark \ref{rem:time-cont-degenerate-vis} and estimation  \eqref{inq:bound-frac-z} to bound $\mathcal{Y}_1^1$. 
\begin{align}
    \mathcal{Y}_1^1 \le & \,C(\xi, ||\phi'||_{L^\infty}, \psi)\,\mathbb{E} \Big [\int_0^T\int_0^T\int_{K_y}||u_{\Delta t}(t,\cdot)||_{L^\infty(\mathbb{R}^d)}\,|u_\eps^\kappa(t,y) - u_\eps^\kappa(s,y)| \notag \\
    & \hspace{3cm} \times \,\rho_{\delta_0}(t-s)\Big|\int_{|z| > r}\frac{1}{|z|^{d+2\theta}}dz\Big|\,dy\,dt\,ds \Big]\notag\\
    \le & \,C(\xi,\bar{r})\,\mathbb{E} \Big [\int_0^T\int_0^T\int_{K_y}|u_\eps^\kappa(t,y) - u_\eps^\kappa(s,y)|\,\rho_{\delta_0}(t-s)\,dy\,dt\,ds \Big] \le C(\xi, \bar{r}, \eps, \kappa)\sqrt{\delta_0}.\notag
\end{align}
Similarly using Lemma \ref{lem:l-infinity bound-approximate-solution} and \eqref{inq:bound-frac-z}, we get 
$$\mathcal{Y}_1^2 \le C(\bar{r})\delta_0, \hspace{.5cm}\text{and}\hspace{.5cm} \mathcal{Y}_1^3 \le C(\xi, \bar{r})l.$$
Since $\beta'$ is odd function, we combine $\mathcal{X}_1^7$ and $\mathcal{Y}_1^4$ to get
\begin{align}
    \mathcal{X}_1^7 + \mathcal{Y}_1^4
    =&\,  \mathbb{E} \Big [\int_{\mathbb{Q}_T}\int_{\mathbb{R}^d} \Big[\int_{|z|>\bar{r}}\frac{\phi(u_\eps^\kappa)(t,y+z)-\phi(u_\eps^\kappa)(t,y)}{|z|^{d+2\theta}}dz-\int_{|z|>\bar{r}}\frac{\phi(u_{\Delta t}(t,x+z))-\phi(u_{\Delta t}(t,x))}{|z|^{d+2\theta}}dz\Big]\notag\\&\hspace{4cm}\times\beta'( u_\eps^\kappa(t,y)-u_{\Delta t}(t,x)) \varrho_{\delta}(x-y)\psi(t,x)\,dx\,dt\,dy\Big]\notag\\
     =&\, \mathbb{E} \Big [\int_{\mathbb{Q}_T}\int_{\mathbb{R}^d} \Big[\int_{|z|>\bar{r}}\frac{\phi(u_\eps^\kappa)(t,y+z)-\phi(u_{\Delta t}(t,x+z))}{|z|^{d+2\theta}}dz-\int_{|z|>\bar{r}}\frac{\phi(u_\eps^\kappa)(t,y)-\phi(u_{\Delta t}(t,x))}{|z|^{d+2\theta}}dz\Big]\notag\\&\hspace{4cm}\times\beta'(u_\eps^\kappa(t,y)-u_{\Delta t}(t,x)) \varrho_{\delta}(x-y)\psi(t,x)\,dx\,dt\,dy\Big]\notag\\
    \le &\,\mathbb{E} \Big [\int_{\mathbb{Q}_T}\int_{\mathbb{R}^d} \Big[\int_{|z|>\bar{r}}\frac{|\phi(u_\eps^\kappa)(t,y+z)-\phi(u_{\Delta t}(t,x+z))|}{|z|^{d+2\theta}}dz-\int_{|z|>\bar{r}}\frac{|\phi(u_\eps^\kappa)(t,y)-\phi(u_{\Delta t}(t,x))|}{|z|^{d+2\theta}}dz\Big]\notag\\&\hspace{6cm}\times \varrho_{\delta}(x-y)\psi(t,x)\,dx\,dt\,dy\Big]\notag\\
    +\, &\mathbb{E} \Big[\int_{\mathbb{Q}_T}\int_{\mathbb{R}^d} \big(sign(u_\eps^\kappa(t,y) - u_{\Delta t}(t,x)) -\beta'(u_\eps^\kappa(t,y) - u_{\Delta t}(t,x))\big)\notag \\ &\hspace{3cm} \times \Big[\int_{|z|>\bar{r}}\frac{\phi(u_\eps^\kappa)(t,y)-\phi(u_{\Delta t}(t,x))}{|z|^{d+2\theta}}dz\Big]\varrho_{\delta}(x-y)\psi(t,x)\,dx\,dt\,dy\Big]\notag\\
    =\,& -\,\mathbb{E} \Big [\int_{\mathbb{Q}_T}\int_{\mathbb{R}^d}|\phi(u_\eps^\kappa)(t,y)-\phi(u_{\Delta t}(t,x)|\mathcal{L}_\theta^{\bar{r}}[\psi(t,\cdot)](x)\varrho_{\delta}(x-y)\,dx\,dt\,dy\Big]\notag\\
    + \, &\mathbb{E} \Big[\int_{\mathbb{Q}_T}\int_{\mathbb{R}^d} \big(sign(u_\eps^\kappa(t,y) - u_{\Delta t}(t,x)) -\beta'(u_\eps^\kappa(t,y) - u_{\Delta t}(t,x))\big)\notag \\ &\hspace{3cm} \times \Big[\int_{|z|>\bar{r}}\frac{\phi(u_\eps^\kappa)(t,y)-\phi(u_{\Delta t}(t,x))}{|z|^{d+2\theta}}dz\Big]\varrho_{\delta}(x-y)\psi(t,x)\,dx\,dt\,dy\Big] \notag \\
    & =: -\,\mathbb{E} \Big [\int_{\mathbb{Q}_T}\int_{\mathbb{R}^d}|\phi(u_\eps^\kappa)(t,y)-\phi(u_{\Delta t}(t,x)|\mathcal{L}_\theta^{\bar{r}}[\psi(t,\cdot)](x)\varrho_{\delta}(x-y)\,dx\,dt\,dy\Big] + \mathcal{C}\,. \notag
\end{align}
In the above ~(the penultimate inequality), we have used the fact that $\phi(a) - \phi(b)$ and $sign(a-b)$ have the same sign ~(as $\phi$ is non-decreasing) and  a change of coordinates for the first integral $x \mapsto x + z$, $y \mapsto y + z$, $z \mapsto -z.$ 

Thanks to the properties of $\beta_\xi^\prime$ and \eqref{inq:bound-frac-z}, one has
\begin{align}
  \mathcal{C} \le \,&C(\bar{r}, \psi)\,\mathbb{E} \Big[\int_{\mathbb{Q}_T}\int_{\mathbb{R}^d} \big|sign(u_\eps^\kappa(t,y) - u_{\Delta t}(t,x)) -\beta'(u_\eps^\kappa(t,y) - u_{\Delta t}(t,x))\big|\notag\\ &\hspace{4cm}\times|u_\eps^\kappa(t,y)-u_{\Delta t}(t,x)|\varrho_{\delta}(x-y)\,dx\,dt\,dy\Big]\notag \\
   \le \, &C(\bar{r}, \psi)\,\mathbb{E} \Big[\int_{\mathbb{Q}_T}\int_{\mathbb{R}^d} \big|sign(u_\eps^\kappa(t,y) - u_{\Delta t}(t,x)) -\beta'(u_\eps^\kappa(t,y) - u_{\Delta t}(t,x))\big|\textbf{1}_{\{|u_\eps^\kappa(t,y) - u_{\Delta t}(t,x)| < \xi\}}\notag\\ &\hspace{6cm}\times|u_\eps^\kappa(t,y)-u_{\Delta t}(t,x)|\varrho_{\delta}(x-y)\,dx\,dt\,dy\Big]\notag\\
    \le \,&C(\bar{r},\psi)\,\mathbb{E} \Big[\int_{\mathbb{Q}_T}\int_{K_x} 2\xi\varrho_{\delta}(x-y)\,dx\,dt\,dy\Big]\notag \le C(\bar{r})\xi.\notag
\end{align}
Next we move our focus on the term $\mathcal{X}_2$ which can be re-written as,
 \begin{align}
        \mathcal{X}_2 
        =\,&-\mathbb{E} \Big[\int_{\mathbb{Q}_T^2} \int_{\mathbb{R}}\beta'(u_\eps^\kappa(s,y) - u_{\Delta t}(t,x) +k) \Big(\mathcal{L}_{\theta,\bar{r}}[\phi(u_\eps)^\kappa(s,\cdot)](y)- \mathcal{L}_{\theta,\bar{r}}[\phi(u_\eps)^\kappa(t,\cdot)](y)\Big)\notag\\&\hspace{8cm}\times \varphi_{\delta_0, \delta}J_l(k)\,dk\,ds\,dt\,dy\,dx \Big]\notag
         \\
        &- \mathbb{E} \Big[\int_{\mathbb{Q}_T^2} \int_{\mathbb{R}}\big(\beta'(u_\eps^\kappa(s,y) - u_{\Delta t}(t,x) +k )-\beta'(u_\eps^\kappa(t,y) - u_{\Delta t}(t,x) +k ) \big)\notag\\&\hspace{2cm}\times \mathcal{L}_{\theta,\bar{r}}[\phi(u_\eps)^\kappa(t,\cdot))](y)\varphi_{\delta_0, \delta}(t,x,s,y) J_l(k)\,dk\,ds\,dt\,dy\,dx \Big]\notag \\
        &-\mathbb{E}\Big[\int_{\mathbb{Q}_T^2} \int_{\mathbb{R}}\beta'(u_\eps^\kappa(t,y) - u_{\Delta t}(t,x) +k)\big[\mathcal{L}_{\theta,\bar{r}}[\phi(u_\eps)^\kappa(s,\cdot) - \phi(u_\eps^\kappa)(t,\cdot)](y) \big]\notag\\ & \hspace{5cm}\times \varphi_{\delta_0, \delta}(t,x,s,y)J_l(k)\,dk\,dt\,ds\,dy\,dx \Big]\notag\\
        &-\mathbb{E}\Big[\int_{\mathbb{Q}_T}\int_{\mathbb{R}^d} \int_{\mathbb{R}}\beta'(u_\eps^\kappa(t,y) - u_{\Delta t}(t,x) +k)\mathcal{L}_{\theta,\bar{r}}[\phi(u_\eps)^\kappa(t,\cdot)](y)\Big(\int_0^T\rho_{\delta_0}(t-s)\,ds -1\Big)\notag\\ & \hspace{5cm}\times \varrho_\delta(x-y)\psi(t,x)J_l(k)\,dk\,dx\,dt\,dy \Big]\notag\\
        & -\mathbb{E} \Big[\int_{\mathbb{Q}_T}\int_{\mathbb{R}^d} \int_{\mathbb{R}}\Big(\beta'((u_\eps^\kappa(t,y) - u_{\Delta t}(t,x) + k ) - \beta'(u_\eps^\kappa(t,y) - u_{\Delta t}(t,x))\Big)\mathcal{L}_{\theta,\bar{r}}[\phi(u_\eps)^\kappa(t,\cdot)](y)\notag\\& \hspace{5cm} \times\varrho_\delta(x-y)\psi(t,x)J_l(k)\,dk\,dx\,dt\,dy \Big]\notag\\
        &-\mathbb{E}\Big[\int_{\mathbb{Q}_T} \int_{\mathbb{R}^d}\beta'(u_\eps^\kappa(t,y) - u_{\Delta t}(t,x) )\mathcal{L}_{\theta,\bar{r}}[\phi(u_\eps)^\kappa(t,\cdot)](y)\varrho_\delta(x-y)\psi(t,x)\,dx\,dt\,dy \Big]\notag\\
        &=:\mathcal{X}_2^1 +\mathcal{X}_2^2 +\mathcal{X}_2^3 + \mathcal{X}_2^4 \,.\notag
        \end{align}
Using properties of convolution it is easy to observe that, 
\begin{align}\label{eq:betadelta}
& |D^2(\beta'(u_\eps^\kappa(s,y) - u_{\Delta t}(t,x) +k)\varrho_\delta(x-y)| \notag \\
& = \Big|\beta^{\prime\prime\prime}(u_\eps^\kappa(s,y)-\cdot)|\nabla_xu_\eps^\kappa(s,y)|\varrho_\delta(x-y)+ \beta''(u_\eps^\kappa(s,y)-\cdot)|\Delta_x u_\eps^\kappa(s,y)|\varrho_\delta(x-y) \notag \\
& \quad + \beta''(u_\eps^\kappa(s,y)-\cdot)|\Delta_x u_\eps^\kappa(s,y)|\nabla_x\varrho_\delta(x-y) + \beta'(u_\eps^\kappa(s,y) - .)\Delta_x\varrho_\delta(x-y)\Big| \notag \\
&  \le C(\xi, \delta, \kappa)||u_\eps(s,\cdot)||_{L^2(\R^d)}\notag.
\end{align}
Hence, by \eqref{fractionalbound}
\begin{align}
&|\mathcal{L}_{\theta,\bar{r}}[\beta'(u_\eps^\kappa(s,\cdot) - u_{\Delta t}(t,x) +k)\varrho_\delta(x-\cdot)](y)|\notag \\
& \le C||D^2(\beta'(u_\eps^\kappa(s,\cdot) - u_{\Delta t}(t,x) +k)\varrho_\delta(x-y)||_{L^\infty}\bar{r}^a  \le C(\xi, \delta, \kappa)||u_\eps(s,\cdot)||_{L^2(\R^d)}.
\end{align}
We use \eqref{eq:betadelta}, Cauchy-Schwartz inequality, properties of Convolution and \eqref{eq:estimations-fd} to bound $\mathcal{X}_2^1$ as
\begin{align}
    \mathcal{X}_2^1 = &-\mathbb{E}\Big[\int_{\mathbb{Q}_T^2} \int_{\mathbb{R}}\big(\phi(u_\eps)^\kappa(s,y) - \phi(u_\eps)^\kappa(t,y)\big)\mathcal{L}_{\theta,\bar{r}}[\beta'(u_\eps^\kappa(s,\cdot) - u_{\Delta t}(t,x) +k)\varrho_{\delta}(x-\cdot)](y)\notag\\ & \hspace{3cm}\times\rho_{\delta_0}(t-s)\psi(t,x)J_l(k)\,dk\,dt\,ds\,dy\,dx \Big]\notag \\
    \le \,& C(\xi,\delta,\kappa, \psi)\mathbb{E}\Big[\int_0^T\int_0^T\int_{K_y} ||u_\eps(s)||_{L^2(\R^d)}\,|\phi(u_\eps)^\kappa(s,y) - \phi(u_\eps)^\kappa(t,y)|\rho_{\delta_0}(t-s)\,dy\,ds\,dt \Big]\notag\\
    \le \,& C(\xi,\delta,\kappa, \psi, |K_y|)\Big(\underset{0 \le s \le T}{sup}\mathbb{E}\Big[||u_\eps(s)||_{L^2(\R^d)}^2\Big]\Big)^{1/2}\notag \\&\hspace{2cm}\times\Big(\mathbb{E}\Big[\int_0^T\int_0^T\int_{K_y} \,\phi(u_\eps)^\kappa(s,y) - \phi(u_\eps)^\kappa(t,y)|^2\rho_{\delta_0}(t-s)\,dy\,ds\,dt \Big]\Big)^{\frac{1}{2}}\notag\\
    \le \, & C(\xi,\delta,\kappa)\Big(\mathbb{E}\Big[\int_0^T\int_0^T\int_{K_y} \,|| u_\eps(s,y) - u_\eps(t,y)||_{L^2(\R^d)}^2\rho_{\delta_0}(t-s)\,dy\,ds\,dt \Big]\Big)^{\frac{1}{2}}.\notag
\end{align}
Applying \eqref{fractionalbound}, Cauchy-Schwartz inequality, properties of convolution and \eqref{eq:estimations-fd},
we compute $\mathcal{X}_2^2$.
\begin{align}
    \mathcal{X}_2^2 \le\,& C(\beta'')\,\mathbb{E} \Big[\int_0^T\int_0^T\int_{K_y} |u_\eps^\kappa(s,y) - u_\eps^\kappa(t,y)| |\mathcal{L}_{\theta,\bar{r}}[\phi(u_\eps)^\kappa(t,\cdot))](y)|\rho_{\delta_0}(t-s)\,dy\,ds\,dt \Big]\notag \\
    \le \,& C\,(\beta'', \psi, \bar{r}^a)\,\mathbb{E} \Big[\int_0^T\int_0^T\int_{K_y}|u_\eps^\kappa(s,y) - u_\eps^\kappa(t,y)|\, ||\phi(u_\eps)(t,\cdot)\ast D^2\tau_\kappa||_{L^\infty}\rho_{\delta_0}(t-s)\,dy\,ds\,dt \Big]\notag \\
   \le &\, C(\beta'',\psi, \bar{r}^a)\,\Big(\mathbb{E} \Big[\int_0^T\int_0^T\int_{K_y} |u_\eps^\kappa(s,y) - u_\eps^\kappa(t,y)|^2\rho_{\delta_0}(t-s)\,dy\,ds\,dt \Big]\Big)^{1/2}\notag\\&\hspace{2cm}\times\mathbb{E}\Big( \Big[\int_0^T\int_0^T\int_{K_y} ||\phi(u_\eps)(t)||_{L^2(\R^d)}^2 ||D^2\tau_\kappa||_{L^2(\R^d)}^2\rho_{\delta_0}(t-s)\,dy\,ds\,dt\Big]\Big)^{1/2}\notag\\
   \le &\,C(\xi,\bar{r}^a,\kappa,\eps)\sqrt{\delta_0}\Big(\underset{0\le t\le T}{sup}\,\mathbb{E}\Big[||u_\eps(t)||_{L^2(\R^d)}^2\Big]\Big)^{\frac{1}{2}} \le C(\xi,\bar{r}^a,\kappa,\eps)\sqrt{\delta_0}.\notag
\end{align}
In the view of \eqref{fractionalbound} and \eqref{eq:estimations-fd}, it is easy to see that
$$\mathcal{X}_2^3 \le C(\kappa, \bar{r}^a)\sqrt{\delta_0}, \quad \text{and} \quad \mathcal{X}_2^4 \le  C(\xi,\kappa, \bar{r}^a)l.$$

Re-arranging $\mathcal{Y}_2$, we get
\begin{align}
       \mathcal{Y}_2 = & -  \mathbb{E} \Big [\int_{\mathbb{Q}_T^2}\int_\mathbb{R}\Big(\phi_{u_\eps^\kappa(s,y) - k}^\beta(u_{\Delta t}(t,x)) - \phi_{u_\eps^\kappa(t,y) - k}^\beta(u_{\Delta t}(t,x))\Big )\mathcal{L}_{\theta,\bar{r}}[\varphi_{\delta_0, \delta}(t,\cdot,s,y)](x) J_l(k)\,dk\,ds\,dt\,dx\,dy \Big] \notag \\
       &-\mathbb{E} \Big [\int_{\mathbb{Q}_T}\int_{\mathbb{R}^d}\int_\mathbb{R}  \phi_{u_\eps^\kappa(t,y) - k}^\beta(u_{\Delta t}(t,x)) \mathcal{L}_{\theta,\bar{r}}[\varrho_\delta(\cdot-y)\psi(t,\cdot)](x))\Big(\int_0^T\rho_{\delta_0}(t-s)ds -1 \Big)J_l(k)dk\,dx\,dt\,dy \Big]\notag\\
       &- \mathbb{E} \Big [\int_{\mathbb{Q}_T}\int_{\mathbb{R}^d}\int_\mathbb{R}  \Big(\phi_{u_\eps^\kappa(t,y) - k}^\beta(u_{\Delta t}(t,x))- \phi_{u_\eps^\kappa(t,y)}^\beta(u_{\Delta t}(t,x)\Big) \mathcal{L}_{\theta,\bar{r}}[\varrho_\delta(\cdot-y)\psi(t,\cdot)](x)J_l(k)\,dk\,dx\,dt\,dy \Big]\notag\\
       &- \mathbb{E} \Big [\int_{\mathbb{Q}_T}\int_{\mathbb{R}^d} \phi_{u_\eps^\kappa(t,y)}^\beta(u_{\Delta t}(t,x)) \mathcal{L}_{\theta,\bar{r}}[\varrho_\delta(\cdot-y)\psi(t,\cdot)](x)\,dx\,dt\,dy \Big]\notag\\
       &=: \mathcal{Y}_2^1 +  \mathcal{Y}_2^2 +  \mathcal{Y}_2^3 - \mathbb{E} \Big [\int_{\mathbb{Q}_T}\int_{\mathbb{R}^d} \phi_{u_\eps^\kappa(t,y)}^\beta(u_{\Delta t}(t,x)) \mathcal{L}_{\theta,\bar{r}}[\varrho_\delta(\cdot-y)\psi(t,\cdot)](x)\,dx\,dt\,dy \Big].\notag
\end{align}
Note that for all $a, b, c \in \mathbb{R}$
\begin{align}\label{eq:inqphi}
  |\phi_b^\beta(a) - \phi_c^\beta(a)| \le C(1 + |a-b|)|b-c|\,.
\end{align}
To estimate $\mathcal{Y}_2^1$, we  use \eqref{eq:inqphi}, \eqref{fractionalbound},  \eqref{eq:estimations-fd}, Lemma \ref{lem:l-infinity bound-approximate-solution}, Remark \ref{rem:time-cont-degenerate-vis} and Cauchy- Schwartz inequality. We have
\begin{align}
  \mathcal{Y}_2^1 &\le C \,\mathbb{E} \Big [\int_{\mathbb{Q}_T^2}(1 + |u_{\Delta t}(t,x) - u_\eps^\kappa(s,y) +k|)|u_\eps^\kappa(s,y) - u_\eps^\kappa(t,y)|\,|\mathcal{L}_{\theta,\bar{r}}[\varphi_{\delta_0, \delta}(t,\cdot,s,y)](x)| \,ds\,dt\,dx\,dy \Big] \notag \\
  \le &\, C(\delta,\psi, \bar{r}^a) \mathbb{E}\Big [\int_{K_x \times [0,T]}\int_{K_y \times [0,T]}(1 + |u_{\Delta t}(t,x) - u_\eps^\kappa(s,y) +k|)|u_\eps^\kappa(s,y) - u_\eps^\kappa(t,y)|\notag\\&\hspace{7cm}\times\rho_{\delta_0}(t-s)\,ds\,dy\,dt\,dx \Big]\notag\\
  &\le \,C(\delta,\psi, \bar{r}^a, \widetilde{M}, \kappa,\eps)\sqrt{\delta_0} (1+l)\Big(\mathbb{E}\Big[\int_0^T\int_{K_x \times [0,T]}|u_\eps^\kappa(s,y) - u_\eps^\kappa(t,y)|^2\rho_{\delta_0}(t-s)\,dx\,dt\,ds\Big]\Big)^\frac{1}{2}\notag\\&\hspace{6cm}\times\Big(\mathbb{E}\Big[\int_{K_y \times [0,T]}|u_\eps^\kappa(s,y)|^2\,ds\,dy |\Big]\Big)^\frac{1}{2}\notag\\
  &\le C(\delta,\psi, \bar{r}^a, \widetilde{M}, \kappa,\eps)\sqrt{\delta_0}\Big(\underset{0\le s\le T}{sup}\,\mathbb{E}\Big[||u_\eps(s)||_{L^2(\R^d)}^2\Big]\Big)^\frac{1}{2} \le C(\delta,\bar{r}^a, \kappa,\eps)\sqrt{\delta_0}.\notag
\end{align}
Applying the Lipschitz continuity of $\phi_k^\beta$ and \eqref{fractionalbound},  \eqref{eq:estimations-fd}, \eqref{eq:inqphi}, Lemma \ref{lem:l-infinity bound-approximate-solution}, one can easily estimate the following.
\begin{align}
\mathcal{Y}_2^2 \le C(\delta, \bar{r}^a)(\sqrt{\delta_0} + l) \quad \text{and} \quad \mathcal{Y}_2^3 \le C(\delta)l.\notag
\end{align}
Combining all the above estimations, we arrive at the following Lemma.
\begin{lem}\label{eq:x4y4x5y5}
We have,
\begin{align}
    &\mathcal{X}_1 + \mathcal{Y}_1 + \mathcal{X}_2 + \mathcal{Y}_2 
    \le -\,\mathbb{E} \Big [\int_{\mathbb{Q}_T}\int_{\mathbb{R}^d}|\phi(u_\eps^\kappa)(t,y)-\phi(u_{\Delta t}(t,x))|\mathcal{L}_\theta^{\bar{r}}[\psi(t,\cdot)](x)\varrho_{\delta}(x-y)\,dx\,dt\,dy\Big]\notag \\
    -\,&\,\mathbb{E}\Big[\int_{\mathbb{Q}_T} \int_{\mathbb{R}^d}\beta'(u_\eps^\kappa(t,y) - u_{\Delta t}(t,x) )\mathcal{L}_{\theta,\bar{r}}[\phi(u_\eps)^\kappa(t,\cdot)](y)\varrho_\delta(x-y)\psi(t,x)\,dx\,dt\,dy \Big]\notag\\
-\, & \,\mathbb{E} \Big [\int_{\mathbb{Q}_T}\int_{\mathbb{R}^d} \phi_{u_\eps^\kappa(t,y)}^\beta(u_{\Delta t}(t,x)) \mathcal{L}_{\theta,\bar{r}}[\varrho_\delta(\cdot-y)\psi(t,\cdot)](x)\,dx\,dt\,dy \Big]\notag\\
+\,&\,C(\delta, \bar{r})\mathbb{E} \Big[\int_0^T\int_{K_y} \big|\phi(u_\eps)^\kappa(s,y)-\phi(u_\eps)(s,y)\big|\,dy\,ds\,\Big]
+C(\delta, \bar{r})\mathbb{E} \Big[\int_0^T\int_{K_y} \big|u_\eps^\kappa(s,y)-u_\eps(s,y)\big|\,dy\,ds\,\Big]\notag\\
+\,&\,C(\xi,\delta,\kappa)\Big(\mathbb{E}\Big[\int_0^T\int_0^T\int_{K_y} \,||u_\eps(s,y) - u_\eps(t,y)||_{L^2(\R^d)}^2\rho_{\delta_0}(t-s)\,dx\,dt\,ds \Big]\Big)^{\frac{1}{2}}\notag\\
& +C(\delta,\xi, \kappa, \eps, \bar{r}^a)\sqrt{\delta_0} + C(\xi,\bar{r}, \delta)l + C(\bar{r})\xi.
\end{align}
\end{lem}
Now we will estimates the error terms appearing  due to the fractional terms in the inequalities.
We re-write $\mathcal{Y}_{3}$ and $\mathcal{Y}_{4}$ as follows.
\begin{align}
    \mathcal{Y}_{3} =  &-  \mathbb{E}\Big[\int_{\mathbb{Q}_T}\int_\mathbb{R}\sum_{n=0}^{N-1} \int_{\mathbb{R}^d} \int_{t_n}^{t_{n+1}}  \big[\mathcal{L}_\theta^{\bar{r}}[\phi(\widetilde{u}_{\Delta t}(t,\cdot))](x) -\mathcal{L}_\theta^{\bar{r}}[\phi(u_{\Delta t}(t_{n+1},\cdot))](y)\big]\notag \\ & \hspace{3cm}\times\beta'(\widetilde{u}_{\Delta t}(t,x) - k) \varphi_{\delta_0, \delta}(t,x,s,y) J_l(u_\eps^\kappa(s,y) - k)\,dt\,dx\,dk\,ds\,dy \Big] \notag \\
    \, - &\mathbb{E}\Big[\int_{\mathbb{Q}_T}\int_\mathbb{R}\sum_{n=0}^{N-1} \int_{\mathbb{R}^d} \int_{t_n}^{t_{n+1}}  \big[\mathcal{L}_\theta^{\bar{r}}[\phi(u_{\Delta t}(t_{n+1},\cdot))](x)\big(\beta'(\widetilde{u}_{\Delta t}(t,x) - k)-\beta'(u_{\Delta t}(t_{n+1},x) -k)\big)\notag \\ & \hspace{3cm} \times\varphi_{\delta_0, \delta}(t,x,s,y) J_l(u_\eps^\kappa(s,y) - k)\,dt\,dx\,dk\,ds\,dy \Big] \notag:= \mathcal{Y}_{3}^1 + \mathcal{Y}_{3}^2\notag \\
    \mathcal{Y}_{4} = &-  \mathbb{E}\Big[\int_{\mathbb{Q}_T}\int_\mathbb{R}\sum_{n=0}^{N-1} \int_{\mathbb{R}^d} \int_{t_n}^{t_{n+1}}  \big[\mathcal{L}_\theta^{\bar{r}}[\phi(u_{\Delta t}(t_{n+1},\cdot))](x) -\mathcal{L}_\theta^{\bar{r}}[\phi(u_{\Delta t}(t,\cdot))](y)\Big]\notag \\ & \hspace{3cm} \times\beta'(u_{\Delta t}(t_{n+1},x) - k) \varphi_{\delta_0, \delta}(t,x,s,y) J_l(u_\eps^\kappa(s,y) - k)\,dt\,dx\,dk\,ds\,dy \Big]\notag \\
    \,  &- \, \mathbb{E}\Big[\int_{\mathbb{Q}_T}\int_\mathbb{R}\sum_{n=0}^{N-1} \int_{\mathbb{R}^d} \int_{t_n}^{t_{n+1}}  \mathcal{L}_\theta^{\bar{r}}[\phi(u_{\Delta t}(t,\cdot))](x)\big(\beta'(u_{\Delta t}(t_{n+1},x) - k) -\beta'(u_{\Delta t}(t,x) - k)\big)\notag \\ & \hspace{3cm}\times \varphi_{\delta_0, \delta}(t,x,s,y) J_l(u_\eps^\kappa(s,y) - k)\,dt\,dx\,dk\,ds\,dy \Big] := \mathcal{Y}_{4}^1 + \mathcal{Y}_{4}^2.\notag
\end{align}
Since $\phi$ is Lipschitz continuous function, a similar lines of argument ( as  done in \eqref{eq:j141}, \eqref{eq:j142}, \eqref{eq:j151}, and \eqref{eq:152}, reveals that
\begin{align}\label{eq:y11y12}
    \mathcal{Y}_{3}^1 ,\, \mathcal{Y}_{4}^1 \le C(\delta, \bar{r})\sqrt{\Delta t}, \quad  \text{and}  \quad  \mathcal{Y}_{3}^2,\, \mathcal{Y}_{4}^2 \le C(\bar{r}, \xi)\sqrt{\Delta t}.  
\end{align}
Observe that, 
\begin{align}\label{eq:phikbeta}
    |\phi_k^\beta(a) - \phi_k^\beta(b)| \le C|a-b|.
\end{align}
In the view \eqref{eq:phikbeta}, \eqref{fractionalbound} and following the estimations  \eqref{eq:j16} and \eqref{eq:j17}, we have
\begin{align}\label{eq:y13y14}
    \mathcal{Y}_{5} ,\, \mathcal{Y}_{6} \le C(\delta,\bar{r}^a)\Delta t\,.
\end{align}

In view of the analysis of Subsection \ref{Convergence Analysis}, one can estimate the terms $\mathcal{I}_i$ and $\mathcal{J}_i$
appeared in  \eqref{uepsdfrac}-\eqref{udeldfrac} and pass to the limit  as $\Delta t \goto 0$ in the surviving terms in the sense of Young measure.  It remains to 
 send the limit as $\Delta t$ tends to zero in the first three terms in Lemma \ref{eq:x4y4x5y5}. For this purpose, we define,
\begin{align}
    &\mathcal{G}_1(t,x, \omega; \nu) := \, \int_{\mathbb{R}^d}|\phi(u_\eps^\kappa)(t,y)-\phi(\nu)|\mathcal{L}_\theta^{\bar{r}}[\psi(t,\cdot)](x)\varrho_{\delta}(x-y)\,dy\,, \notag \\
    &\mathcal{G}_2(t,x, \omega; \nu) := \,  \int_{\mathbb{R}^d}\beta'(u_\eps^\kappa(t,y) - \nu)\mathcal{L}_{\theta,\bar{r}}[\phi(u_\eps)^\kappa(t,\cdot)](y)\varrho_\delta(x-y)\psi(t,x)\,dy\,,\notag\\
    &\mathcal{G}_3(t,x,\omega; \nu) := \, \int_{\mathbb{R}^d} \phi_{u_\eps^\kappa(t,y)}^\beta(\nu) \mathcal{L}_{\theta,\bar{r}}[\varrho_\delta(\cdot-y)\psi(t,\cdot)](x)\,dy\,.\notag
\end{align}
Observe that $\mathcal{G}_1$, $\mathcal{G}_2$ , and $\mathcal{G}_3$ are  Carath\'eodory functions on $\Theta \times \mathbb{R}$.  Moreover, $\mathcal{G}_1(.,u_{\Delta t})$ ,$\mathcal{G}_2(.,u_{\Delta t})$ and $\mathcal{G}_1(.,u_{\Delta t})$ are uniformly bounded sequence in $L^2(\mathbb{Q}_T \times \Omega)$. 
We show it only for $\mathcal{G}_2$. 
Using Lipschitz continuity of $\phi$ and \eqref{fractionalbound}, we have
\begin{align}
    &\mathbb{E}\Big[\int_{\mathbb{Q}_T}|\mathcal{G}_2(.,u_{\Delta t})|^2\,dx\,dt \Big]\notag \\&=\,
    \mathbb{E}\Big[\int_{\mathbb{Q}_T}\Big|\int_{\mathbb{R}^d}\beta'(u_\eps^\kappa(t,y) - u_{\Delta t}(t,x) )\mathcal{L}_{\theta,\bar{r}}[\phi(u_\eps)^\kappa(t,\cdot)](y)\varrho_\delta(x-y)\psi(t,x)\,dy\Big|^2\,dx\,dt \Big]\notag\\
    \le\,& C(\beta',\psi, \bar{r}^a)\,  \mathbb{E}\Big[\int_{\mathbb{Q}_T}\int_{K_y}|D^2\phi(u_\eps)^\kappa(t,y)|_{L^\infty}^2\varrho_\delta(x-y)^2\,dy\,dx\,dt \Big]\notag\\
    \le\,& C(\beta',\psi, ||\phi'||_{L^\infty}, \bar{r}^a)\,  \mathbb{E}\Big[\int_{\mathbb{Q}_T}\int_{K_y}||u_\eps(t,\cdot)||_{L^2}^2||\tau_\kappa||_{L^2}^2\varrho_\delta(x-y)^2\,dy\,dx\,dt \Big]\notag \le C(\delta, \kappa,\bar{r}^a).
\end{align}
Thus invoking Proposition \eqref{prop:young-measure}, we have
\begin{align}\label{conv:degenerate-frac-1}
\begin{cases}
   \displaystyle \underset{\Delta t \rightarrow 0}{lim}\,\mathbb{E} \Big [\int_{\mathbb{Q}_T}\int_{\mathbb{R}^d}|\phi(u_\eps)^\kappa(t,y)-\phi(u_{\Delta t}(t,x)|\mathcal{L}_\theta^{\bar{r}}[\psi(t,\cdot)](x)\varrho_{\delta}(x-y)\,dx\,dt\,dy\Big] \\ 
     \displaystyle \,= \,\mathbb{E} \Big [\int_{\mathbb{Q}_T}\int_{\mathbb{R}^d}\int_0^1|\phi(u_\eps^\kappa)(t,y)-\phi(u(t,x, \alpha)\mathcal{L}_\theta^{\bar{r}}[\psi(t,\cdot)](x)\varrho_{\delta}(x-y)\,d\alpha\,dx\,dt\,dy\Big]\,,\\
    \displaystyle  \underset{\Delta t \rightarrow 0}{lim}\,\mathbb{E}\Big[\int_{\mathbb{Q}_T} \int_{\mathbb{R}^d}\beta'(u_\eps^\kappa(t,y) - u_{\Delta t}(t,x) )\mathcal{L}_{\theta,\bar{r}}[\phi(u_\eps)^\kappa(t,\cdot)](y)\varrho_\delta(x-y)\psi(t,x)\,dx\,dt\,dy \Big] \\ 
      \displaystyle \,= \,\mathbb{E}\Big[\int_{\mathbb{Q}_T} \int_{\mathbb{R}^d}\int_0^1\beta'(u_\eps^\kappa(t,y) - u(t,x,\alpha)\mathcal{L}_{\theta,\bar{r}}[\phi(u_\eps)^\kappa(t,\cdot)](y)\varrho_\delta(x-y)\psi(t,x)\, d\alpha\,dx\,dt\,dy \Big]\,,\\
         \displaystyle \underset{\Delta t \rightarrow 0}{lim}\mathbb{E} \Big [\int_{\mathbb{Q}_T}\int_{\mathbb{R}^d} \phi_{u_\eps^\kappa(t,y)}^\beta(u_{\Delta t}(t,x)) \mathcal{L}_{\theta,\bar{r}}[\varrho_\delta(\cdot-y)\psi(t,\cdot)](x)\,dx\,dt\,dy \Big]\\ 
           \displaystyle\,=\,  \mathbb{E} \Big [\int_{\mathbb{Q}_T}\int_{\mathbb{R}^d}\int_0^1 \phi_{u_\eps^\kappa(t,y)}^\beta(u(t,x, \alpha)) \mathcal{L}_{\theta,\bar{r}}[\varrho_\delta(\cdot-y)\psi(t,\cdot)](x)\,d\alpha\,dx\,dt\,dy \Big] \,,
        \end{cases}
\end{align}
where $u(t,x,\alpha)$ is the Young-measure valued limit of the approximate solutions $u_{\Delta t}(t,x)$. Thanks to Lemmas \ref{lem1}, \ref{lem 2}, \ref{lem3}, \ref{eq:x4y4x5y5}, \ref{lem67}, and the  convergence  results in \eqref{conv:degenerate-frac-1} together with \eqref{eq:y11y12} and \eqref{eq:y13y14}, one can 
send $\Delta t$ to $0$, and subsequently  pass to the limit $\delta_0 \rightarrow 0$ and $l \rightarrow 0$ in in  \eqref{uepsdfrac}-\eqref{udeldfrac} to have
\begin{align}
    0 \le  &\,\mathbb{E} \Big[ \int_{\mathbb{R}^d } \int_{\mathbb{R}^d }\beta(u_0(x)-u_\eps^\kappa(0,y)) \varrho_\delta(x-y)\psi(0,x)\,dy\,dx \Big]\notag \\
    +\, &\mathbb{E}\Big[\int_{\mathbb{Q}_T}\int_{\mathbb{R}^d}\int_0^1\beta(u(t,x,\alpha) - u_\eps^\kappa(t,y))\partial_t\psi(t,x)\varrho_{\delta}(x-y)\,d\alpha\,dx\,dt\,dy  \Big].\notag\\
+\,&\mathbb{E} \Big [\int_{\mathbb{Q}_T} \int_{\mathbb{R}^d}\int_0^1F^\beta((u(t,x,\alpha), u_\eps^\kappa(t,y))\varrho_\delta(x-y)\cdot\nabla_x\psi(t,x)\,d\alpha\,dx\,dt\,dy \Big]\notag\\
-\,&\mathbb{E} \Big [\int_{\mathbb{Q}_T}\int_{\mathbb{R}^d}\int_0^1|\phi(u_\eps^\kappa)(t,y)-\phi((u(t,x,\alpha))|\mathcal{L}_\theta^{\bar{r}}[\psi(t,\cdot)](x)\varrho_{\delta}(x-y)\,d\alpha\,dx\,dt\,dy\Big]\notag \\
    -\,&\,\mathbb{E}\Big[\int_{\mathbb{Q}_T}\int_{\mathbb{R}^d}\int_0^1 \beta'(u_\eps^\kappa(t,y) - (u(t,x,\alpha) )\mathcal{L}_{\theta,\bar{r}}[\phi(u_\eps)^\kappa(t,\cdot)](y)\varrho_\delta(x-y)\psi(t,x)\,d\alpha\,dx\,dt\,dy \Big]\notag\\
-\, & \,\mathbb{E} \Big [\int_{\mathbb{Q}_T}\int_{\mathbb{R}^d} \phi_{u_\eps^\kappa(t,y)}^\beta(u(t,x,\alpha)) \mathcal{L}_{\theta,\bar{r}}[\varrho_\delta(\cdot-y)\psi(t,\cdot)](x)\,d\alpha\,dx\,dt\,dy \Big]\notag\\
+\,&\,C(\delta, \bar{r})\mathbb{E} \Big[\int_0^T\int_{K_y} \big|\phi(u_\eps)^\kappa(s,y)-\phi(u_\eps)(s,y)\big|\,dy\,ds\,\Big]\notag\\
+\,&\,C(\delta, \bar{r})\mathbb{E} \Big[\int_0^T\int_{K_y} \big|u_\eps^\kappa(s,y)-u_\eps(s,y)\big|\,dy\,ds\,\Big]\notag\\
+\, &C(\xi)\,\mathbb{E} \Big[\int_0^T\int_{K_y}  \Big( (\sigma(u_\eps(t,y))\ast\tau_\kappa) - \sigma(u_\eps(t,y))\Big)^2 \varrho_\delta(x-y)\,dy\,dt \Big]\notag\\
+ \, & C(\xi)\,\mathbb{E} \Big[\int_0^T\int_{K_y} \Big( \sigma(u_\eps^\kappa(t,y)) - \sigma(u_\eps(t,y))\Big)^2 \,dy\,dt\Big] + C(\delta)\eps^{1/2} + C(\bar{r})\xi\notag \\
=:\, &\mathcal{B}_1 + \mathcal{B}_2 + \mathcal{B}_3 + \mathcal{B}_4 + \mathcal{B}_5 + \mathcal{B}_6 + \mathcal{B}_7 + \mathcal{B}_8 + \mathcal{B}_9 + \mathcal{B}_{10} + C(\delta)\eps^{1/2} + C\xi. \label{inq:degenerate-frac-2}
\end{align}
Observe that
\begin{align}
    &\Bigg|\mathcal{B}_4- \mathbb{E} \Big [\int_{\mathbb{Q}_T}\int_{\mathbb{R}^d}\int_0^1|\phi(u_\eps)(t,y)-\phi(u(t,x, \alpha)|\mathcal{L}_\theta^{\bar{r}}[\psi(t,\cdot)](x)\varrho_{\delta}(x-y)\,d\alpha\,dx\,dt\,dy\Big]\Bigg|\notag \\
    \le \, & \,\mathbb{E} \Big [\int_{\mathbb{Q}_T}\int_{\mathbb{R}^d}\int_0^1|\phi(u_\eps^\kappa)(t,y)-\phi(u_\eps(t,y))|\,|\mathcal{L}_\theta^{\bar{r}}[\psi(t,\cdot)](x)|\varrho_{\delta}(x-y)\,d\alpha\,dx\,dt\,dy\Big]\notag \\ \le &\, C(\bar{r}, \psi)\, \Big(\mathbb{E} \Big [\int_0^T\int_{K_y}|u_\eps^\kappa(t,y)-u_\eps(t,y)|^2\,dt\,dy\Big]\Big)^{1/2} \underset{\kappa \rightarrow 0}{\longrightarrow} 0. \label{conv-kappa-b4-mathcal}
    \end{align}
Following the calculations \cite[Step $2$, Lemma $3.4$]{frac non}, one has
\begin{align*}
    &\underset{\kappa \rightarrow 0}{lim}\,\mathcal{B}_5 \notag \\
    & = \, -\mathbb{E}\Big[\int_{\mathbb{Q}_T}\ \int_{\mathbb{R}^d}\int_0^1\beta'(u_\eps(t,y) - u(t,x,\alpha))\mathcal{L}_{\theta,\bar{r}}[\phi(u_\eps)(t,\cdot)](y)\varrho_\delta(x-y)\psi(t,x)\, d\alpha\,dx\,dt\,dy \Big]\,. 
\end{align*}
 In view of  \cite[Appendix A]{frac non} and the fact that
$$(\phi(a)-\phi(b))\beta'(a-k) \ge \phi_k^\beta(a) - \phi_k^\beta(b) \ge (\phi(a) - \phi(b))\beta'(b-k)\quad \forall~a,b \in \R\,,$$
we have
\begin{align}
    &\underset{\kappa \rightarrow 0}{lim}\,\mathcal{B}_5 = - \mathbb{E}\Big[\int_{\mathbb{Q}_T}\ \int_{\mathbb{R}^d}\int_0^1\beta'(u_\eps(t,y) - u(t,x,\alpha))\mathcal{L}_{\theta,\bar{r}}[\phi(u_\eps)(t,\cdot)](y)\varrho_\delta(x-y)\psi(t,x)\, d\alpha\,dx\,dt\,dy \Big]\notag \\ &\le\, -\,
    \mathbb{E}\Big[\int_{\mathbb{Q}_T}\ \int_{\mathbb{R}^d}\int_0^1\phi_{u(t,x,\alpha)}^\beta(u_\eps(t,y))\mathcal{L}_{\theta,\bar{r}}[\varrho_\delta(x-\cdot)](y)\psi(t,x)\, d\alpha\,dx\,dt\,dy \Big]\,. \label{conv-kappa-b5-mathcal}
\end{align}
We now focus on the term $\mathcal{B}_6$. Using \eqref{eq:inqphi}, we have
\begin{align}
    &\bigg|\mathcal{B}_6 - \mathbb{E} \Big [\int_{\mathbb{Q}_T}\int_{\mathbb{R}^d}\int_0^1 \phi_{u_\eps(t,y)}^\beta(u(t,x, \alpha)) \mathcal{L}_{\theta,\bar{r}}[\varrho_\delta(\cdot-y)\psi(t,\cdot)](x)\,d\alpha\,dx\,dt\,dy \Big]\bigg|\notag\\
    \le & \, \mathbb{E} \Big [\int_{\mathbb{Q}_T}\int_{\mathbb{R}^d}\int_0^1 |\phi_{u_\eps^\kappa(t,y)}^\beta(u(t,x, \alpha))- \phi_{u_\eps(t,y)}^\beta(u(t,x, \alpha))| \mathcal{L}_{\theta,\bar{r}}[\varrho_\delta(\cdot-y)\psi(t,\cdot)](x)\,d\alpha\,dx\,dt\,dy \Big]\notag \\
    \le \, &\, C\,\mathbb{E}\Big[\int_{\mathbb{Q}_T}\int_{\mathbb{R}^d}\int_0^1 \big(1 + |u(t,x, \alpha)- u_\eps^\kappa(t,y)|\big)\,|u_\eps^\kappa(t,y) - u_\eps(t,y)|\,|\mathcal{L}_{\theta,\bar{r}}[\varrho_\delta(\cdot-y)\psi(t,\cdot)](x)|\,d\alpha\,dx\,dt\,dy \Big]\notag\\
    \le\, &C(\delta, \psi)\bar{r}^a\Big(1 + \underset{t}{sup}\,\mathbb{E}\Big[||u(t,\cdot,\cdot)||_{L^2}^2\Big] + \underset{t}{sup}\,\mathbb{E}\Big[||u_\eps(t,\cdot)||_{L^2}^2\Big]\Big)^{\frac{1}{2}}\Big(\mathbb{E}\Big[\int_0^T\int_{K_y}|u_\eps^\kappa(t,y) - u_\eps(t,y)|^2\,dy\,dt \Big]\Big)^{1/2}\notag\\
    &\hspace{13cm}\underset{\kappa \rightarrow 0}{\longrightarrow} 0.\notag
\end{align}
Thus, we have 
\begin{align}
    &\underset{\kappa \rightarrow 0}{lim} \mathcal{B}_6 =\,- \,\mathbb{E} \Big [\int_{\mathbb{Q}_T}\int_{\mathbb{R}^d}\int_0^1 \phi_{u_\eps(t,y)}^\beta(u(t,x, \alpha)) \mathcal{L}_{\theta,\bar{r}}[\varrho_\delta(\cdot-y)\psi(t,\cdot)](x)\,d\alpha\,dx\,dt\,dy \Big]. \label{conv-kappa-b6-mathcal}
\end{align}
Since $\big|\phi_k^\beta(a) -|\phi(a)- \phi(k)|\big| \le \xi||\phi'||_{L^\infty}$, we have
\begin{align}
    &\bigg|\,\mathbb{E}\Big[\int_{\mathbb{Q}_T}\ \int_{\mathbb{R}^d}\int_0^1\phi_{u(t,x,\alpha)}^\beta(u_\eps(t,y)\mathcal{L}_{\theta,\bar{r}}[\varrho_\delta(x-\cdot)](y)\psi(t,x)\, d\alpha\,dx\,dt\,dy \Big]\notag \\&\hspace{1cm}- \mathbb{E}\Big[\int_{\mathbb{Q}_T}\ \int_{\mathbb{R}^d}\int_0^1|\phi(u_\eps) - \phi(u(t,x, \alpha)) |\,\mathcal{L}_{\theta,\bar{r}}[\varrho_\delta(x-\cdot)](y)\psi(t,x)\, d\alpha\,dx\,dt\,dy \Big]\bigg|\notag\\
    & \, \le \mathbb{E}\Big[\int_{\mathbb{Q}_T}\ \int_{\mathbb{R}^d}\int_0^1\Big|\phi_{u(t,x,\alpha)}^\beta(u_\eps(t,y)) - |\phi(u_\eps) - \phi(u(t,x, \alpha))|\Big|\notag\\ &\hspace{8cm}\times\mathcal{L}_{\theta,\bar{r}}[\varrho_\delta(x-\cdot)](y)|\psi(t,x)\, d\alpha\,dx\,dt\,dx \Big]\notag\\
     \,& \le\,\xi||\phi'||_{L^\infty}C(|\psi|)\bar{r}^a \underset{\xi \rightarrow 0}{\longrightarrow} 0.\label{conv-xi-corresponding-b5-mathcal}
\end{align}
Similarly, we have
\begin{align}
&\mathbb{E} \Big [\int_{\mathbb{Q}_T}\int_{\mathbb{R}^d}\int_0^1 \phi_{u_\eps(t,y)}^\beta(u(t,x, \alpha)) \mathcal{L}_{\theta,\bar{r}}[\varrho_\delta(\cdot-y)\psi(t,\cdot)](x)\,d\alpha\,dt\,dx\,dy \Big]\notag \\ &\hspace{1cm} \underset{\xi \rightarrow 0}{\longrightarrow} \mathbb{E} \Big [\int_{\mathbb{Q}_T}\int_{\mathbb{R}^d}\int_0^1 |\phi(u_\eps) - \phi(u(t,x, \alpha))| \mathcal{L}_{\theta,\bar{r}}[\varrho_\delta(\cdot-y)\psi(t,\cdot)](x)\,d\alpha\,dt\,dx\,dy \Big].
\label{conv-xi-corresponding-b6-mathcal}
\end{align}
Repeating the arguments as done in {\bf Steps II $\&$ III} of Subsection \ref{Convergence Analysis} and using \eqref{conv-kappa-b4-mathcal}-\eqref{conv-xi-corresponding-b6-mathcal}, we first pass to the limit in $\kappa \rightarrow 0$ and then $\xi \rightarrow 0$ and $\eps \rightarrow 0$ in \eqref{inq:degenerate-frac-2}. The resulting inequality reads as
\begin{align}
    0 \le  &\,\mathbb{E} \Big[ \int_{\mathbb{R}^d } \int_{\mathbb{R}^d }|u_0(x)-u_0(y)| \varrho_\delta(x-y)\psi(0,x)\,dy\,dx \Big]\notag \\
    +\, &\mathbb{E}\Big[\int_{\mathbb{Q}_T}\int_{\mathbb{R}^d}\int_0^1|u(t,x,\alpha) - \bar{u}(t,y))|\partial_t\psi(t,x)\varrho_{\delta}(x-y)\,d\alpha\,dx\,dt\,dy  \Big].\notag\\
   +\,&\mathbb{E} \Big [\int_{\mathbb{Q}_T} \int_{\mathbb{R}^d}\int_0^1F((u(t,x,\alpha), \bar{u}(t,y))\varrho_\delta(x-y)\cdot\nabla_x\psi(t,x)\,d\alpha\,dx\,dt\,dy \Big]\notag\\
-&\,\mathbb{E} \Big [\int_{\mathbb{Q}_T}\int_{\mathbb{R}^d}\int_0^1|\phi(\bar{u}(t,y))-\phi(u(t,x, \alpha)|\mathcal{L}_\theta^{\bar{r}}[\psi(t,\cdot)](x)\varrho_{\delta}(x-y)\,d\alpha\,dx\,dt\,dy\Big]\notag \\
-\,& \mathbb{E}\Big[\int_{\mathbb{Q}_T}\ \int_{\mathbb{R}^d}\int_0^1|\phi(\bar{u}(t,y)) - \phi(u(t,x, \alpha))|\mathcal{L}_{\theta,\bar{r}}[\varrho_\delta(x-\cdot)](y)\psi(t,x)\, d\alpha\,dx\,dt\,dy \Big]\notag\\
-\, & \,\mathbb{E} \Big [\int_{\mathbb{Q}_T}\int_{\mathbb{R}^d}\int_0^1|\phi(\bar{u}(t,y)) - \phi(u(t,x, \alpha))|\mathcal{L}_{\theta,\bar{r}}[\varrho_\delta(\cdot-y)\psi(t,\cdot)](x)\,\,d\alpha\,dx\,dt\,dy \Big]\,,\label{inq:degenerate-frac-3}
\end{align}
where $\bar{u}(t,y)$ is the unique entropy solution of \eqref{eq:fdegenerate} and $u_\eps(\cdot,\cdot)$ converges to $\bar{u}(\cdot,\cdot)$ weakly in $L^2(\Omega \times \mathbb{Q}_T)$ and strongly in $L^p(\Omega\times (0,T)\times B(0,M))$ for any $M>0$ and $1\le p<2$. 
\vspace{0.2cm}

Since $\phi(\cdot)$ is Lipschitz, by following the same arguments as invoked in {\bf Steps IV $\&$ V} of Subsection \ref{Convergence Analysis}, we can pass to the limit as $\bar{r} \goto 0$ and $\delta \goto 0$ in \eqref{inq:degenerate-frac-3} and obtain the following Kato type inequality:

\begin{align}\label{katofd}
    0 \le     \, &\mathbb{E}\Big[\int_{\mathbb{Q}_T}\int_0^1|u(t,x,\alpha) - \bar{u}(t,x)|\partial_t\psi(t,x)\,d\alpha\,dt\,dx  \Big]\notag\\
   +\,&\mathbb{E} \Big [\int_{\mathbb{Q}_T} \int_0^1F((u(t,x,\alpha), \bar{u}(t,x))\cdot\nabla_x\psi(t,x)\,d\alpha\,dt\,dx \Big]\notag\\
-&\,\mathbb{E} \Big [\int_{\mathbb{Q}_T}\int_0^1|\phi(\bar{u}(t,x))-\phi(u(t,x, \alpha)|\mathcal{L}_\theta[\psi(t,\cdot)](x)\,d\alpha\,dt\,dx\Big]\,. 
\end{align}
In view of  \cite[Lemma $2.6$, Corollary $5.8$, Lemma $5.9$, Lemma $5.2$]{Endal-2014} and following the arguments of \cite[Subsection $3.2.1$]{frac non}, we arrive from \eqref{katofd}
\begin{align}
\mathbb{E}\Big[\int_{\mathbb{Q}_T}\int_0^1 |u(t,x,\alpha) - \bar{u}(t,x)|\,d\alpha \,dx \,dt\Big] = 0. \label{eq:final-con-degenerate-frac}
\end{align}
Since $u(t,x,\alpha)$ is the Young measure valued limit of the approximate solutions $u_{\Delta t}(t,x)$, generated by \eqref{approxi:solu} and \eqref{eq:sequence-nonlinear} and $\bar{u}(t,x)$ is the unique entropy solution of  \eqref{eq:fdegenerate}, we conclude from \eqref{eq:final-con-degenerate-frac} that 
 $u_{\Delta t}(\cdot,\cdot)$ converges to the unique entropy solution of \eqref{eq:fdegenerate} in $L_{loc}^p(\mathbb{R}^d; L^p(\Omega \times (0,T)))$ for $1 \le p < \infty.$
 This essentially completes the proof.
 
\section{Numerical Experiments}\label{sec:numerical}
In this section, the operator splitting scheme given by \eqref{approxi:solu}-\eqref{eq:sequence}
has been tested on suitable test cases in order to demonstrate its effectiveness in one space dimension.
However, in order to implement our scheme, one needs to discretize both space and time variables. We begin by introducing some notation needed to define the
fully discrete finite difference schemes. To that context, we
reserve $\Dx, \Dt$ to denote small positive numbers that represent the
spatial and temporal discretization parameter of the numerical scheme respectively. Given $\Dx>0$, we set $x_j=j\Dx$ for $j\in \Z$, to denote the spatial mesh points. Similarly,  we set $t^n = n \Dt$ for $n= 0,1,\cdots,N$, where $N\Dt=T$ for some fixed time horizon $T>0$. Moreover, for any function $u=u(x,t)$ admitting point values, we write $u^n_j = u(x_j, t^n)$. With the help of the above notations, we consider the following explicit numerical scheme for the fractional conservation laws \eqref{eq:operator-S}, (see \cite[Section 7]{cifani fds}).
\begin{align}\label{eq:discrtization-op-s}
    u_j^{n+1} = u_j^n - \Delta t\,D_{-}F(u_j^n, u_{j+1}^n) + \Delta t\sum_{i \neq 0}G_i(u_{i+j}^n - u_j^n)\,.
\end{align}
Similarly for \eqref{eq:operator Sbar} we have,
\begin{align}\label{eq:discritization-op-sbar}
    u_j^{n+1} = u_j^n - \Delta t\,D_{-}F(u_j^n, u_{j+1}^n) + \Delta t\sum_{i \neq 0}G_i(\phi(u_{i+j}^n) - \phi(u_j^n)).
\end{align}
In \eqref{eq:discrtization-op-s}
and \eqref{eq:discritization-op-sbar}, $u_j^n$ is the approximate solution in the cell $[t^n, t^{n+1}) \times [x_{j-\frac{1}{2}}, x_{j+\frac{1}{2}})$, $F:\R^2 \goto \R$ is the monotone numerical flux corresponding to the given flux function $f$, and $D_{-}$ is the spatial difference operator given by 
$$D_{-}(\cdot)_j := \frac{((\cdot)_{j}-(\cdot)_{j-1})}{\Delta x},$$
and $G_i$ is the approximate diffusion operator given in \eqref{eq:compute-G-i}.
For simplicity, we have chosen to demonstrate the Engquist-Osher method for the flux function $f : \R \rightarrow \R$, which has the following form.
\begin{align}
    &\Delta t\,D_{-}F(u_j^n, u_{j+1}^n) =\frac{\Dt}{\Dx} \Big\{ F_{\rm EO}(u^n_j, u^n_{j+1}) - F_{\rm EO}(u^n_{j-1}, u^n_j) \Big\}\,,\notag 
\end{align}
where 
\begin{align}
     F_{\rm EO}(u,v):= \int_{0}^{u} \max{ \Big( f'(s), 0 \Big)} \,ds + \int_{0}^{v} \min{ \Big(f'(s), 0 \Big)} \,ds + f(0).
 \label{EO_flux}
\end{align}
The approximate diffusion operator in \eqref{eq:discrtization-op-s} and \eqref{eq:discritization-op-sbar} is given by (see, \cite{cifani fds, Koley-2021-Monte Carlo})
\begin{align}\label{eq:compute-G-i}
    G_j = a_\theta\int_{x_{j - \frac{1}{2}}}^{x_{j + \frac{1}{2}}} \frac{dz}{|z|^{1 + 2\theta}}, \quad j \in \Z-\{0\},
\end{align}
where $a_\theta = \frac{2^{2\theta-1}\Gamma(\frac{1+2\theta}{2})}{\pi^{\frac{1}{2}}\Gamma(1-\theta)}$ with $\Gamma(\xi) = \int_0^\infty e^{-x}x^{\xi-1}\,dx, \quad \xi > 0.$
The above explicit numerical scheme converges under the following CFL condition \cite[Equation 4.4]{cifani}, \cite[Section 7]{cifani fds}. 
\begin{align}\label{eq:CFL-condition}
    2L_f\frac{\Delta t}{\Delta x} + \Big(a_\theta2^{2\theta}L_{\phi}\int_{|z| >0}\frac{dz}{|z|^{1+2\theta}}\Big)\frac{\Delta t}{(\Delta x)^{2\theta}} < 1,
\end{align}
where $L_f$ and $L_{\phi}$ are the Lipschitz constants of $f$ and $\phi$ respectively.\\
\vspace{0.1cm}

 Meanwhile, for the numerical scheme of \eqref{eq:noise}, we consider the following  Euler-Maruyama approximations.
\begin{align}
 u_j^{n+1}= u_j^n + \sigma(u_j^n)\big[W(t_{n+1})-W(t_{n})\big] + \int_{t_n}^{t_{n+1}} \int_{|z|>0} \eta(u_j^n,z) \,\widetilde{N}(dz,ds). \label{EM-scheme-1}
\end{align}
Here we assume that $\alpha:= m(\{|z|>0\}) < + \infty$. Moreover, 
observe that, $N(t):= N(\{|z|>0\} \times [0,t])$ is a stochastic process which  counts the number of jumps until  time $t$.
 Also, $N(dz,dt)$ generates a pair of sequence of random variables
 $\Big\{ (\tau_i, \xi_i): ~~ i \in \{ 1,2,\cdots, N(T)\}\Big\}$,
 where $\tau_i:\Omega \mapsto \R_{+}, ~~ i \in \{ 1,2,\cdots, N(T)\}$, are increasing sequence of 
 nonnegative random variable which 
 represents the jump times of a Poisson process with intensity $\alpha$, and 
 $\xi_i:\Omega \mapsto \R, ~~ i \in \{ 1,2,\cdots, N(T)\}$, are i.i.d sequence of random variables with distribution
 as $\frac{m(dz)}{\alpha}$. Then, Euler-Maruyama scheme \eqref{EM-scheme-1} reduces to
\begin{align}
u_j^{n+1}= u_j^n + \sigma(u_j^n)\big[W(t_{n+1})-W(t_n)\big] -  \Delta t\int_{|z|>0} \eta(u_j^n, z)\,m(dz) +{ \displaystyle  \sum \limits_{i=N(t_n)+1}^{N(t_{n+1})} \eta(u_j^n, \xi_i)}.  \label{EM-scheme-2}
\end{align}
The initial data can be approximated as 
\begin{align*}
    u_j^0 = \frac{1}{\Delta x}\int_{x_{j -\frac{1}{2}}}^{x_{j +\frac{1}{2}}}u_0(x)\,dx, \quad j\in \mathbb{Z}\,.
\end{align*}
\subsection{Computation of non-local diffusion term}
To simulate the numerical experiments, it is not useful to consider the whole domain $\R$. Following \cite[Subsection 5.1]{Koley-2021-Monte Carlo}, we consider a symmetric truncated domain $[-K, K]$ for a given $K>0$. For computational purposes, one must appropriately extend the domain outside the prescribed interval. Keeping this in mind, the solution can be prolonged beyond the initial interval in a constant way, such that $u_j = u_{-K}$ for all $j \le -K$ and $u_j = u_K$ for all $j \ge K$;~see also \cite{koley-2022}. Thus, we reformulate the non-local term in \eqref{eq:operator Sbar} as follows.
\begin{align*}
    &\sum_{i \neq 0}G_i(\phi(u_{i+j}^n) - \phi(u_j^n))\\ &= \sum_{i < -K-j}G_i(\phi(u_{i+j}^n) - \phi(u_j^n)) + \sum_{\underset{i \neq 0}{i = -K-j}}^{K-j}G_i(\phi(u_{i+j}^n)-\phi(u_j^n)) + \sum_{i > K-j}G_i(\phi(u_{i+j}^n)-\phi(u_j^n))\\
    &= (\phi(u_{-K}^n) -\phi(u_j^n))\sum_{i < -K-j}G_i +  \sum_{\underset{i \neq 0}{i = -K-j}}^{K-j}G_i(\phi(u_{i+j}^n)-\phi(u_j^n)) + (\phi(u_K^n) -\phi(u_j^n))\sum_{i > K-j}G_i\\
    &= \frac{a_\theta}{2\theta(\Delta x)^{2\theta}}\frac{\phi(u_{-K}^n)-\phi(u_j^n)}{(K + j +\frac{1}{2})^{{2\theta}}}\, + \, \frac{a_\theta}{2\theta(\Delta x)^{2\theta}}\frac{\phi(u_{K}^n)-\phi(u_j^n)}{(K - j +\frac{1}{2})^{{2\theta}}}\, +  \sum_{\underset{i \neq 0}{i = -K-j}}^{K-j}G_i(\phi(u_{i+j}^n)-\phi(u_j^n))\,.
\end{align*}
The above estimation requires the value of $G_i$ for $|i| < K$ which can be computed by \eqref{eq:compute-G-i}. A similar formulation can be done for the non-local term of \eqref{eq:discritization-op-sbar} with $\phi(x) = x$.
\subsection{Numerical Examples} As it has been mentioned earlier, in the first step, we solve the stochastic equation \eqref{eq:operator-S} using the Euler-Maruyama method 
\eqref{EM-scheme-2}, and in the second step, we solve the scalar fractional conservation 
laws  using explicit numerical scheme \eqref{eq:discritization-op-sbar} and \eqref{eq:discrtization-op-s}, 
with initial condition as an obtained solution from the first step.
We consider the truncated spatial domain $[-1,1]$ and the time simulation $T =1$. All the computations are performed on an AMD Ryzen $5$-$3500$U, $2.10$ GHz processor with
$8$ GB RAM. 

\begin{ex}\label{exm:1} Here, we consider \eqref{eq:fdegenerate} with $f(u) = \frac{u^2}{2}$ and the diffusion term $\phi$, noise co-coefficient $\sigma$ and initial function $u_0^1$ defined as follows.
\begin{equation*}
\phi(x)= 
\begin{cases}
 (x-\frac{1}{2})^+ \quad  if \quad 0 \le x \le 1, \\
 0 \quad \text{else},
\end{cases}\, 
 \sigma(x) =
\begin{cases}
 x(1-x) \quad \text{if} -1 \le x \le 1\\
 0 \quad \text{else},
\end{cases}\,
u_0^1(x)=
 \begin{cases}
  -\frac{1}{2} \quad   -1 \le x < 0, \\
  \frac{1}{2}   \quad  0 \le x \le 1\\
  0,  \quad \text{else}\,.
 \end{cases}
\end{equation*}
We have simulated the numerical solution for different values of $\theta\in (0,1)$. Figure \ref{fig:1}, ${\rm (A)}$ displays the numerical solution of \eqref{eq:fdegenerate} for the above mentioned datum with $\theta=0.1$~(red doted curve) and $\theta= 0.3$~(blue doted curve) for the uniform temporal discretization parameter $\Delta t=0.002$ and corresponding spatial discretization parameter $\Delta x$ obtained from the CFL condition \eqref{eq:CFL-condition}, whereas Figure \ref{fig:1}, ${\rm (B)}$ depicts the numerical solution of \eqref{eq:fdegenerate} with $\theta=0.6$~(red doted curve) and $\theta= 0.8$~(blue doted curve) for $\Delta t=0.001$ and corresponding $\Delta x$ given via the CFL condition \eqref{eq:CFL-condition}.
\begin{figure}[htbp]
\begin{subfigure} [b]{0.49\textwidth}
        \includegraphics[scale=0.5]{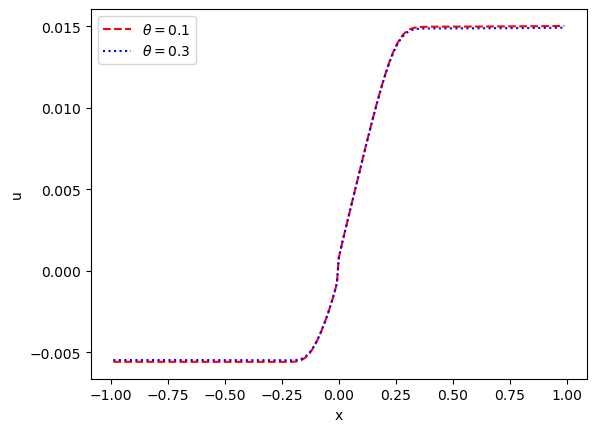} 
         \caption{}
  \end{subfigure}
  \begin{subfigure}[b]{0.49\textwidth}
        \includegraphics[scale=0.5]{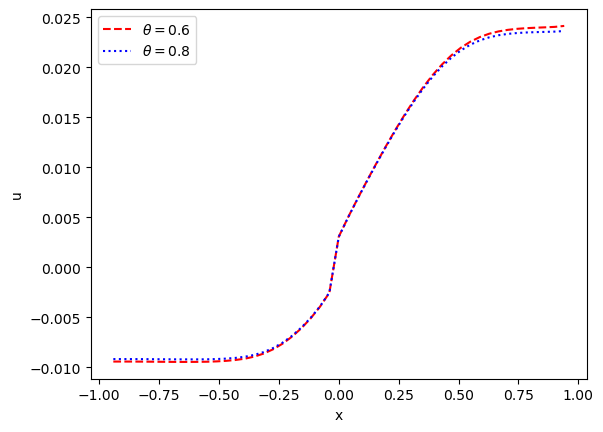}
       \caption{}
  \end{subfigure}
   \caption{Numerical solution of \eqref{eq:fdegenerate} for Burger's flux at final time $T=1$:
    (A) for $\theta=0.1,~0.3$, (B) for $\theta=0.6,~0.8$.}
  \label{fig:1}
\end{figure}
\end{ex}

\begin{ex}\label{exm:2} 
Here, we again consider \eqref{eq:fdegenerate} with $f, \phi, \sigma$ as defined in Example \ref{exm:1} and initial condition $u_0^2$ given by
\begin{equation*}
u_0^2(x)=
 \begin{cases}
  2\exp{(\frac{1}{x^2 -1})} \quad \text{if} \quad -1 < x < 1,\\
  0  \qquad \text{else}.
 \end{cases}
\end{equation*}
\begin{figure}[htbp]
\begin{subfigure} [b]{0.49\textwidth}
        \includegraphics[scale=0.5]{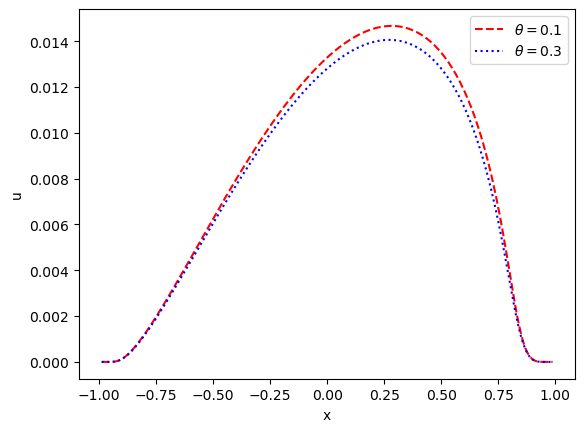} 
         \caption{For $\Delta t=0.002$ and $\Delta x$ using \eqref{eq:CFL-condition} }
  \end{subfigure}
  \begin{subfigure}[b]{0.49\textwidth}
        \includegraphics[scale=0.5]{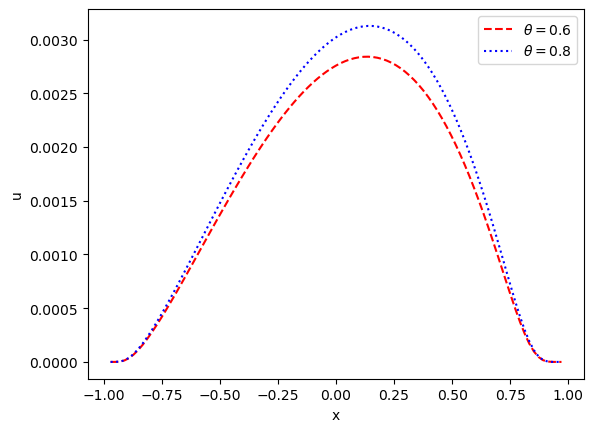}
        \caption{For $\Delta t=0.001$ and $\Delta x$ using \eqref{eq:CFL-condition} }
  \end{subfigure}
   \caption{ Numerical solution of \eqref{eq:fdegenerate} for Burger's flux at final time $T=1$:
    (A) for $\theta=0.1,~0.3$, (B) for $\theta=0.6, 0.8$.}
  \label{fig:2}
\end{figure}

\end{ex}
\vspace{0.5cm}

\noindent{\bf Acknowledgement:} The authors wish to thank Prof. Ujjwal Koley for his valuable discussions and suggestions. The first author would like to acknowledge the financial support by CSIR, India.
The second author is supported by Department of Science and Technology, Govt. of India-the INSPIRE fellowship~(IFA18-MA119).

\vspace{.2cm}

\noindent{\bf Data availability:}  Data sharing is not applicable to this article as no data-sets were generated or analyzed during the current study.

\vspace{0.2cm}

\noindent{\bf Conflict of interest:}  The authors have not disclosed any competing interests.

\end{document}